\documentclass[12pt]{article}
\usepackage[utf8]{inputenc}
\usepackage{amsmath}            
\usepackage{amssymb}          
\usepackage{amsfonts}           
\usepackage{mathrsfs}          
\usepackage{amsthm}
\usepackage{mathtools}
\usepackage{commath}
\usepackage{bm}
\usepackage{graphicx}
\usepackage{geometry}
\usepackage{float}
\usepackage{listings}
\usepackage{subfigure}
\usepackage{multirow}
\usepackage{diagbox}
\usepackage[export]{adjustbox}
\usepackage[all]{xy}
\usepackage{fancyhdr}
\usepackage{dsfont}
\usepackage{hyperref}
\usepackage{enumitem}
\usepackage{stmaryrd}
\usepackage{tikz-cd}
\usepackage{bbm}
\DeclareMathAlphabet{\mathbbb}{U}{bbold}{m}{n}
\usepackage{bold-extra}
\usepackage[T1]{fontenc}

\numberwithin{equation}{section}

\theoremstyle{definition}
\newtheorem{defi}{Definition}[section]

\newtheorem{cond}[defi]{Condition}%[section]
\newtheorem{ex}[defi]{Example}
\newtheorem{notn}[defi]{Notation}

\theoremstyle{plain}
\newtheorem{conj}[defi]{Conjecture}%[section]
\newtheorem{construction}[defi]{Construction}%[section]
\newtheorem{cor}[defi]{Corollary}
\newtheorem{cor-defi}[defi]{Corollary-Definition}
\newtheorem{lemma}[defi]{Lemma}
\newtheorem{prop}[defi]{Proposition}
\newtheorem{prop-defi}[defi]{Proposition-Definition}
\newtheorem{theo}[defi]{Theorem}

\theoremstyle{remark}
\newtheorem*{rmk}{Remark}

\title{Rational Hodge--Tate prismatic crystals of quasi-l.c.i algebras and non-abelian $p$-adic Hodge theory}
\author{Xiaoyu Qu\and Jiahong Yu}

\newcommand\dA{\mathbb{A}}
\newcommand\dB{\mathbb{B}}
\newcommand\dC{\mathbb{C}}

\newcommand\dE{\mathbb{E}}

\newcommand\dL{\mathbb{L}}

\newcommand\dQ{\mathbb{Q}}
\newcommand\dR{\mathbb{R}}
\newcommand\dZ{\mathbb{Z}}

\newcommand\calA{\mathcal{A}}
\newcommand\calB{\mathcal{B}}
\newcommand\calC{\mathcal{C}}
\newcommand\calD{\mathcal{D}}
\newcommand\calE{\mathcal{E}}
\newcommand\calF{\mathcal{F}}
\newcommand\calK{\mathcal{K}}

\newcommand\calJ{\mathcal{J}}
\newcommand\calO{{\mathcal{O}}}

\newcommand\calP{{\mathcal{P}}}
\newcommand\calS{{\mathcal{S}}}
\newcommand\calT{{\mathcal{T}}}
\newcommand\calV{{\mathcal{V}}}

\newcommand\bfC{\mathbf{C}}
\newcommand\bL{\mathbf{L}}
\newcommand\cbigoplus{\mathop{\widehat{\bigoplus}}}

\newcommand\cbigwedge{\mathop{\widehat{\bigwedge}}\nolimits}
\newcommand\cD{\widehat D}
\newcommand\cdL{\widehat\dL}
\newcommand\cGamma{\widehat\Gamma}
\newcommand\cOmega{\widehat\Omega}
\newcommand\cotimes{\mathbin{\widehat{\otimes}}}
\newcommand\ep{\epsilon}

\newcommand\hol{\text{\normalfont{H}}}

\newcommand\Prism{\mathbbb{\Delta}}

\newcommand\simp{\mathbf{\Delta}}

\DeclareMathOperator{\cofib}{cofib}
\DeclareMathOperator{\coker}{coker}
\DeclareMathOperator*{\colim}{colim}
\DeclareMathOperator*{\ccolim}{\widehat{\colim}}
\DeclareMathOperator{\cone}{cone}
\DeclareMathOperator{\End}{End}
\DeclareMathOperator{\Ext}{Ext}
\DeclareMathOperator{\fib}{fib}
\DeclareMathOperator{\Hom}{Hom}
\DeclareMathOperator{\im}{im}
\DeclareMathOperator{\Lie}{Lie}

\DeclareMathOperator{\sgn}{sgn}
\DeclareMathOperator{\Shv}{Shv}
\DeclareMathOperator{\Spa}{Spa}
\DeclareMathOperator{\Spd}{Spd}
\DeclareMathOperator{\Spec}{Spec}
\DeclareMathOperator{\Spf}{Spf}

\DeclareMathOperator{\Sym}{Sym}
\DeclareMathOperator{\Tor}{Tor}

\newcommand\ad{\text{\normalfont{ad}}}

\newcommand\an{\text{\normalfont{an}}}
\newcommand\Aut{\text{\normalfont{Aut}}}
\newcommand\bd{\text{\normalfont{bd}}}
\newcommand\CAfd{\text{\normalfont{CAfd}}}
\newcommand\CAlg{\text{\normalfont{CAlg}}}
\newcommand\Comp{\text{\normalfont{Comp}}}
\newcommand\cplt{\text{\normalfont{cplt}}}
\newcommand\crys{\text{\normalfont{crys}}}
\newcommand\cyc{\text{\normalfont{cyc}}}
\newcommand\derham{\text{\normalfont{dR}}}
\newcommand\dHodge{\text{\normalfont{dHodge}}}
\newcommand\discrete{\text{\normalfont{discrete}}}
\newcommand\DR{\text{\normalfont{DR}}}

\newcommand\Fil{\text{\normalfont{Fil}}}
\newcommand\fil{\text{\normalfont{fil}}}
\newcommand\Flat{\text{\normalfont{flat}}}
\newcommand\Frob{\text{\normalfont{Frob}}}
\newcommand\Fun{\text{\normalfont{Fun}}}
\newcommand\Gal{\text{\normalfont{Gal}}}
\newcommand\GCart{\text{\normalfont{GCart}}}
\newcommand\gr{\text{\normalfont{gr}}}
\newcommand\Higgs{\text{\normalfont{Higgs}}}
\newcommand\HT{\text{\normalfont{HT}}}
\newcommand\id{\text{\normalfont{id}}}
\newcommand\Kos{\text{\normalfont{Kos}}}
\newcommand\la{\text{\normalfont{la}}}
\newcommand\Map{\text{\normalfont{Map}}}
\newcommand\MIC{\text{\normalfont{MIC}}}
\newcommand\Mor{\text{\normalfont{Mor}}}

\newcommand\op{\text{\normalfont{op}}}
\newcommand\PD{\text{\normalfont{PD}}}
\newcommand\Perf{\text{\normalfont{Perf}}}
\newcommand\perf{\text{\normalfont{perf}}}

\newcommand\Res{\text{\normalfont{{Res}}}}
\newcommand\Sptr{\text{\normalfont{Sptr}}}
\newcommand\Strat{\text{\normalfont{Strat}}}
\newcommand\tn{\text{\normalfont{tn}}}
\newcommand\Tot{\text{\normalfont{Tot}}}
\newcommand\Tr{\text{\normalfont{Tr}}}

\newcommand\Vect{\text{\normalfont{Vect}}}

\newcommand\Alg{\mathsf{Alg}}
\newcommand\cAlg{\widehat{\mathsf{Alg}}}
\newcommand\Cat{\mathsf{Cat}}

\newcommand\FinPoly{\mathsf{FinPoly}}

\newcommand\Mod{\mathsf{Mod}}

\newcommand\RKE{\mathsf{RKE}}
\newcommand\SCR{\mathsf{SCR}}

\newcommand\AAinf{\dA_{\inf}}

\newcommand\Ainf{A_{\inf}}
\newcommand\Bdrp{\dB_\derham^+}
\newcommand\Fin{\calF\text{\normalfont{in}}}
\newcommand\grAlg{\gr\Alg}

\newcommand\Rep{\text{\normalfont{\textbf{Rep}}}}
\newcommand\RHom{\mathop{R\text{\normalfont{Hom}}}\nolimits}
\newcommand\SCRModcn{\mathsf{SCRMod}^{\text{\normalfont{cn}}}}
\newcommand\taulezRHom{\mathop{\tau^{\le 0}R\text{\normalfont{Hom}}}\nolimits}

\newcommand\clim[1]{\mathop{{\lim}^{#1}}}
\newcommand\clims[2]{\mathop{{\lim_{#2}}^{#1}}}
\newcommand\fJsm[1]{\text{\normalfont{fake-$#1$-small}}}
\newcommand\sm[1]{\text{\normalfont{$#1$-small}}}
\newcommand\smet[1]{\text{\normalfont{$#1$-small-\'et}}}
\newcommand\tors[1]{\text{\normalfont{$#1$-tors}}}

\renewcommand\dif{\text{\normalfont{d}}}

\begin{document}

\maketitle

\begin{abstract}
    Consider a bounded prism $(A,I)$ and a bounded quasi-l.c.i algebra $R$ over $\overline{A}$. In this paper, for any prism $S/A$ with a surjection $S\to R$ such that $\cdL_{\overline{S}/\overline{A}}$ is a $p$-completely flat module over $\overline{S}$, we establish an equivalence of categories between rational Hodge-Tate crystals on $(R/A)_{\Prism}$ and topologically nilpotent integrable connections on the Hodge--Tate cohomology ring $\overline{\Prism}_{R/S}$. As an application, for a non-zero divisor $a\in \overline{A}$, we introduce the concept of $a$-smallness for a rational Hodge-Tate prismatic crystal on $(R/A)_{\Prism}$. Finally, we focus on some special algebras $R$ over $\calO_{\dC_p}$ (or generally, the ring of integers of an algebraic closed and complete non-archimedean field) including all $p$-completely smooth algebras, $p$-complete algebras with semi-stable reductions and geometric valuation rings. By using our equivalence, we analyze the restriction functor from the category of $a$-small rational Hodge-Tate prismatic crystals to the category of $v$-vector bundles. This yields some new results in $p$-adic non-abelian Hodge Theory.
\end{abstract}

\tableofcontents

\section{Introduction}
%Fix a bounded prism $\left( A,I \right)$. We impose the following condition on an $\overline A:=A/I$-algebra $R$:
%\begin{cond}[{\cite[Remark 4.1.18]{bhatt2022absolute}}]\label{cond: tor -1,0}
%    The ring $R$ is of bounded $p^\infty$-torsion, $p$-adically complete, and the complete contangent complex $\cdL_{R/\overline A}$ is of $p$-complete $\Tor$-amplitude $[-1,0]$ over $R$ (see Proposition-Definition \ref{prop-defi: complete flatness} for complete flatness and complete $\Tor$-amplitudes). Such an algebra is is called \emph{quasi-lci} over $\overline{A}$.
%\end{cond}

%Also, we have the following variant:
%\begin{cond}[{\cite[Remark 4.1.19]{bhatt2022absolute}}]\label{cond: tor -1}
%    The ring $R$ is of bounded $p^\infty$-torsion, $p$-adically complete, and $\cdL_{R/\overline A} \left[ -1 \right]$ is $p$-completely flat over $R$.
%\end{cond}

%Let $S$ be a bounded prism over $A$. We also consider the following condition:
%\begin{cond}\label{cond: flat + tor 0}
%    The ring $S$ is $\left( p,I \right)$-completely flat over $A$, and the complete contangent complex $\cdL_{S/A}$ is $p$-completely flat over $S$.
%\end{cond}

\subsection{Background and overview}

Let $p$ be a prime number. One of the goals of $p$-adic Hodge theory is to establish, for any analytic space $X$ over $\mathbb{Q}_p$, a correspondence between vector bundles with coefficients in various period rings (such as $\widehat{\calO}$, $\Bdrp$, $\mathbf{B}_{\crys}^+$, etc.) on the $v$-topology of $X$ (denoted by $X_v$) and classical geometric objects on $X$ (such as connections, Higgs fields, etc.). In recent years, Bhatt--Scholze (cf.\cite{Bhatt_2022}) have developed prismatic theory, providing a clearer and more natural framework for the $p$-adic Hodge theory. Specifically, let $\mathfrak{X}$ be a $p$-adic formal scheme, and let $X$ be its generic fiber. Bhatt--Scholze defined the prismatic site $\mathfrak{X}_{\Prism}$ of $\mathfrak{X}$. This site includes all perfectoid objects in $X_v$, as well as all `deperfections' of these perfectoid objects (so-called prisms). Moreover, the definitions of various period rings can be extended to all objects in $\mathfrak{X}_{\Prism}$. Fixing a period ring $\mathbf{B}$, by restricting to the perfectoid objects in $X_v$, one can associate to a $\mathbf{B}$-crystals on the prismatic site $\mathfrak{X}_{\Prism}$ a $\mathbf{B}$-vector bundles on $X_v$. Correspondingly, by restricting to the `deperfections', one hopes to associate to a $\mathbf{B}$-crystals on $\mathfrak{X}_{\Prism}$ a classical geometric object. This picture is described by the following diagram:
\begin{equation}\label{eq: intro diagram of prismatic view of p adic hodge}
\begin{tikzcd}[column sep=1em]
& \text{$\mathbf B$-crystals on } \mathfrak{X}_{\Prism}\arrow[dl,"?" swap] \arrow[dr, "{\Res}_{\mathfrak{X}}"] &\\ 
\text{Classical geometric objects}\arrow[rr,dashrightarrow,"\text{$p$-adic Hodge}"] & & \text{$\mathbf{B}$-vector bundles on $X_v$}\ ,
\end{tikzcd}
\end{equation}
where the arrow `?' is generally constructed by restriction to `deperfections' but is mostly unknown, and the arrow `${\Res}_\mathfrak{X}$' is given by restriction. Thus, the study of $p$-adic Hodge theory with coefficients translates into the following questions:
\begin{enumerate}
    \item Understanding the construction of the arrow `?'.
    \item Understanding the properties of the functor ${\Res}_{\mathfrak{X}}$ (e.g., fully faithfulness, essential image, etc.).
\end{enumerate}

In this paper, we focus on the case where $\mathbf{B}$ is taken to be $\widehat{\mathcal{O}}$ in diagram (\ref{eq: intro diagram of prismatic view of p adic hodge}). In prismatic theory, this corresponds to rational Hodge--Tate prismatic crystals. To make our introduction clearer, let us quickly recall the definitions:
\begin{itemize}
    \item Given a bounded prism $(A,I)$ and a formal scheme $\mathfrak{X}/\overline{A}$ \footnote{We follow the notation in \cite{Bhatt_2022} that $\overline{A}=A/I$ for any prism $(A,I)$.}, define a rational Hodge--Tate prismatic crystal over $\mathfrak{X}$ as a crystal on $(\mathfrak{X}/A)_{\Prism}$ with coefficients in $\overline{\mathcal{O}}_{\Prism}[\frac{1}{p}]$. In the following, we will use $\Vect\bigl((\mathfrak{X}/A)_\Prism,\overline\calO[\frac{1}{p}]\bigr)$ to denote the category of rational Hodge--Tate prismatic crystals over $\mathfrak{X}$. Also, when $\mathfrak X=\Spf(R)$ is affine, we will also use the notation $\Vect\bigl((R/A)_\Prism,\overline\calO[\frac{1}{p}]\bigr)$.
\end{itemize}
At this point, the geometric objects in the diagram (\ref{eq: intro diagram of prismatic view of p adic hodge}) (cf. \cite{Min_2025} and \cite{anschütz2023smallpadicsimpsoncorrespondence}) should be taken to be Higgs bundles, and the bottom arrow is often known as the $p$-adic non-abelian Hodge theory (aka. $p$-adic Simpson theory). Correspondingly, the objects on $X_v$ then become $v$-vector bundles. Therefore, the diagram (\ref{eq: intro diagram of prismatic view of p adic hodge}) specializes to the following:
\begin{equation}\label{eq: intro diagram of prismatic view of Simpson}
\begin{tikzcd}[column sep=3em]
& \Vect\bigl((\mathfrak{X}/A)_\Prism,\overline\calO[\frac{1}{p}]\bigr)\arrow[dl,"?" swap] \arrow[dr, "{\Res}_{\mathfrak{X}}"] & \\
\text{Higgs bundles}\arrow[rr,dashrightarrow,"\text{$p$-adic non-abelian Hodge}"] & & \text{$v$-vector bundles on }X\ .
\end{tikzcd}
\end{equation}
Similarly, our work involves constructing the arrow `?' and understanding the properties of $\Res_{\mathfrak{X}}$.

We first recall previous constructions of the arrow `?' in some special cases. An important observation by Tian in \cite{Tian_2023} showed that if $\mathfrak{X}$ is a $(p,I)$-completely smooth $\delta$-formal scheme over a bounded prism $(A,I)$, then the Hodge--Tate crystals of $\mathfrak{X}:=\overline{A}\otimes_{A}\widetilde{\mathfrak{X}}$ are classified by topologically nilpotent Higgs bundles. Later, Ogus (\cite{ogus2024crystallineprismsreflectionsdiffractions}) generalized this method to locally complete intersection schemes in the setting of crystalline prisms. The method of Ogus also works for some other prisms (cf. \cite{liu2025stackyapproachprismaticcrystals}).

However, it can be seen that these results are only applicable to topologically finitely generated algebras. In further research, it was found that many results in \cite{Bhatt_2022} can be generalized to non-topologically finitely generated algebras. For example, in the further study of prismatic theory by Bhatt--Lurie (cf. \cite{bhatt2022absolute}, \cite{bhatt2022prismatization}), it was discovered that for a class of algebras satisfying only certain cohomological conditions, the prismatic cohomology is isomorphic to the derived prismatic cohomology. Another example is that Bouis in \cite{Bouis_2023} found that the computations in \cite{Bhatt_2022} can be generalized to Cartier smooth algebras (which is also a class of algebras satisfying only some cohomological conditions). Meanwhile, as an analogy to prismatic theory, Bhatt--de Jong in \cite{Bhatt2011CrystallineCA} discovered that the correspondence between crystalline crystals and integrable connections established by Berthelot can also be generalized to non-finitely generated algebras.

Based on these observations, we believe that Tian's classification method can be generalized to non-topologically finitely generated algebras. This constitutes the main content of our theory. Specifically, we generalize this classification to `cohomological' locally complete intersection algebras over $\overline{A}=A/I$ (cf. Condition \ref{cond: tor -1,0}). Unlike the previously mentioned results, the algebras to which our theory applies are not necessarily topologically finitely generated. Therefore, our theory applies to some big rings such as valuation rings.

Furthermore, we find that our construction is compatible with Faltings' local construction of $p$-adic Simpson correspondence in \cite{Faltings_2005} and \cite{Faltings_2011}. Using this result, we study the properties of ${\Res}_{\mathfrak{X}}$ for some $\mathfrak{X}$ that have not been considered in previous research. Specifically, assume that $\mathbf{C}$ is a $p$-adically complete non-Archimedean field that is itself algebraically closed, and let $\mathcal{O}_{\mathbf{C}}$ be its ring of integers. First, we consider the case where $\mathfrak{X}$ is a syntomic formal scheme over $\mathcal{O}_{\mathbf{C}}$ with smooth generic fiber $X$. Furthermore, we assume that the sheaf of differentials of $\mathfrak{X}$ is a flat $\mathcal{O}_{\mathbf{C}}$-module. For such $\mathfrak{X}$, we first prove a cohomological comparison theorem for Hodge--Tate crystals and generalize the results of \cite{Min_2025} to $\mathfrak{X}$. Second, we consider a geometric valuation ring $\mathcal{K}^+$ over $\mathcal{O}_{\mathbf{C}}$ (cf. Definition \ref{defi:geo val ring}), establish a Sen theory on it, and prove an analogue of the main theorem of \cite{gao2023hodgetateprismaticcrystalssen}.

\subsection{Main results}

Instead of some complex and abstract conditions, we hope to first explain an example and our main results on $p$-adic non-abelian Hodge theory to help readers better understand our theory.\vspace{0.3em}

\noindent\textbf{\scshape An example: irrational radius disks}: We first consider the rational Hodge--Tate prismatic crystals on the so-called \emph{disks of irrational radius}. Let $\mathbf{C}=\dC_p$. We consider the $p$-adic norm on $\mathbf{C}$ such that $|p|=p^{-1}$. Let $\calO_{\mathbf{C}}$ be the ring of integers of $\mathbf{C}$. Let $\Ainf$ be the prism $(\AAinf(\calO_{\mathbf{C}}),\ker\theta)$.

\begin{defi}
    Let $r\in \dR$ be a real number. Define the \emph{ring of integral valued analytic functions on the $p$-adic disk of radius $p^r$} as
    \[\calO_{\mathbf{C}}\{\tfrac{T}{p^r}\}=\{f=\sum_{n=0}^\infty a_nT^n\in\mathbf{C}[[T]]:p^{nr}|a_n|\leq1\text{ and }\lim_{n\to \infty}p^{nr}a_n=0\}.\]
\end{defi}

Note that if $p^r$ lies in the valuation group of $|-|$, we can choose a $t\in \mathbf{C}$ such that $|t|=p^{r}$ and then $\calO_{\mathbf{C}}\{\frac{T}{p^r}\}$ is isomorphic to the usual Tate algebra of variable $tT$. If such a $t$ does not exist, the algebra $\calO_{\mathbf{C}}\{\frac{T}{p^r}\}$ is not even topologically finitely generated. In this case, we call it \emph{the irrational radius disk}. Let $\mathbf{C}\{\frac{T}{p^r}\}=\calO_{\mathbf{C}}\{\frac{T}{p^r}\}[\frac{1}{p}]$.

Using our theory, we can give a classification of rational Hodge--Tate prismatic crystals on $(\calO_{\mathbf{C}}\{\frac{T}{p^r}\}/\Ainf)_{\Prism}$.

\begin{theo}[Example \ref{ex: dR realisation for irrational disks}]
    There exists an equivalence between categories
    \[\Vect\bigl((\calO_{\mathbf{C}}\{\tfrac{T}{p^r}\}/\Ainf)_\Prism,\overline{\calO}[\tfrac{1}{p}]\big)\cong \Higgs^{\tn}(\mathbf{C}\{\tfrac{T}{p^r}\}, \cOmega_{\calO_{\mathbf{C}}\{\frac{T}{p^r}\}/\calO_{\mathbf{C}}}),\]
    where:
    \begin{itemize}
        \item $\Vect\bigl((\calO_{\mathbf{C}}\{\frac{T}{p^r}\})_\Prism,\overline{\calO}[\frac{1}{p}]\big)$ is the the category of rational Hodge--Tate prismatic crystals on the prismatic site $(\calO_{\mathbf{C}}\{\frac{T}{p^r}\}/\Ainf)_\Prism$,
        \item $\Higgs^\tn(\mathbf{C}\{\frac{T}{p^r}\},\cOmega_{\calO_{\mathbf{C}}\{\frac{T}{p^r}\}/\calO_{\mathbf{C}}})$ is the category of Higgs bundles $(\calE,\theta:\calE\to \calE \dif T)$ such that $a\frac{\dif}{\dif T}$ acts topologically nilpotently for any $a\in\mathbf{C}$ such that $|a|\leq p^r$.
    \end{itemize}
\end{theo}\vspace{0.3em}

\noindent\textbf{\scshape Generalized representations and rational Hodge--Tate prismatic crystals of geometric valuation rings}: 
In \cite{Sen_1980}, Sen established a systematic framework for studying semi-linear continuous $\dC_p$-representations of the absolute Galois group of a local field $K/\dQ_p$. The core idea is the introduction of the Sen operator (also known as the Sen derivative), which essentially represents the infinitesimal action of the Galois group $\Gal_K$. Through this operator, semi-linear continuous $\dC_p$-representations of $\Gal_K$ can be `linearized', allowing for their classification and study using linear algebra tools (such as eigenvalues and eigenspaces).

For any local field $K/\dQ_p$ (more generally, any closed subfield of $\dC_p$), semilinear $\dC_p$ continuous representations of $K$ are naturally identified with $v$-vector bundles over $\Spa(K,\calO_K)$. Hence, we can view the functor ${\Res}_{\mathfrak X}$, where $\mathfrak X$ is $\Spf(\calO_{K})$, as a functor from $\Vect\bigl(\Spf(\calO_K)_{\Prism},\overline{\calO}[\frac{1}{p}]\bigr)$ to the category of semi-linear continuous $\dC_p$-representations of $\Gal_K$. In \cite{gao2023hodgetateprismaticcrystalssen}, Gao--Min--Wang studied this functor and found that:
\begin{itemize}
    \item ${\Res}_{\mathfrak X}$ is fully faithful;
    \item The essential image of ${\Res}_{\mathfrak X}$ is the so-called `nearly Hodge--Tate' representations, which is a condition of the eigenvalues of the Sen operator.
\end{itemize}

More recently, Tongmu He (cf. \cite{He_2025}, \cite{he2025padicgaloiscohomologyvaluation}) generalized Sen's method to some valuation rings other than $\calO_K$. In particular, combining his results and Camargo's method on constructing geometric Sen operators, we can define Sen operators for some other valuation rings. Our main result in this part is an analogue of Gao--Min--Wang's result for a new class of valuation rings, which are called `geometric valuation rings' by us.

Again, let $\mathbf{C}$ be an algebraic closed $p$-adic non-archimedean field, and let $\calO_{\mathbf{C}}$ be the ring of integers of $\mathbf{C}$. Put $\Ainf$ the prism $(\AAinf(\calO_{\mathbf{C}}),\ker\theta)$. Fix a norm on $\bfC$.

Geometric valuation rings are valuation rings that can appear as the valuation rings of complete residue fields of some rigid analytic spaces. Explicitly, we make the following definition.

\begin{defi}[Definition \ref{defi:geo val ring}]
    A \emph{geometric valuation ring} over $\calO_{\mathbf{C}}$ is a $p$-complete valuation ring $\calK^+$ over $\calO_{\mathbf{C}}$ such that for the complete cotangent complex $\cdL_{\calK^+/\calO_{\mathbf{C}}}$, the vector space (it is automatically concentrated on degree 0 by Theorem \ref{theo: valuation rings are cartier smooth})
    \[ \cdL_{\calK^+/\calO_{\mathbf{C}}}[\frac{1}{p}] \]
    is of finite dimension over $\calK :=\calK^+ \bigl[ \frac 1 p \bigr]$.
\end{defi}

\begin{rmk}
     We will see that $\calK$ is indeed the fractional field of $\calK^+$ by Theorem \ref{theo: almost finitely generated}.
\end{rmk}

Fix a geometric valuation ring $\calK^+$ with fractional field $\calK$. Put $\widehat{\overline\calK}$ the complete algebraic closure of $\calK$. This time, the $v$-vector bundles on $\Spa(\calK,\calK^+)$ are identified with the finite dimensional continuous semilinear representations of $\Gal_{\calK}$ with coefficients in $\widehat{\overline{\calK}}$ (aka. generalized representation). Denote by $\Rep_{\Gal(\calK)}(\widehat{\overline{\calK}})$ the category of finite dimensional continuous semilinear representations of $\Gal_{\calK}$ with coefficients in $\widehat{\overline{\calK}}$.

\begin{theo}[Special case of Proposition \ref{prop-defi: sen theory}]\label{intro:sen of val rins}
    Keep the notation, for any $(V,\rho)\in \Rep_{\Gal_\calK}(\widehat{\overline \calK})$. There exists a canonical $\Gal_\calK$-equivariant $\widehat{\overline \calK}$-linear operator \[\theta_V:\cOmega_{\calK^+/\calO_{\mathbf{C}}}[\frac{1}{p}]^\vee(1)\to \End_{\widehat{\overline{\calK}}}(V),\]
    such that $\theta_V=0$ if and only if $V$ is trivial.
\end{theo}

Keep the notation in Theorem \ref{intro:sen of val rins}. The operator $\theta_V$ is called the \emph{Sen operator} of $V$. For any $\partial\in \cOmega_{\calK^+/\calO_{\mathbf{C}}}[\frac{1}{p}]^\vee(1)$, the linear operator $\theta_V(\partial)$ is called the \emph{Sen action} of $\partial$.

\begin{theo}[Theorem \ref{theo: prism simpson val ring}]\label{intro: prism non abelian hodge of geo val rings}
    The functor $\mathcal{R}:={\Res}_{\Spf(\calK^+)}$ from $\Vect\big((\calK^+/A)_{\Prism},\overline\calO[\frac{1}{p}]\big)$ to $\Rep_{\Gal_\calK} (\widehat{\overline\calK})$ satisfies the following:
    \begin{enumerate}[label=(\arabic*)]
        \item For any $a\in \mathbf{C}^\times$ such that $|a|<1$, there exists a full subcategory \[\Vect^{\sm a}\big((\calK^+/A)_{\Prism},\overline\calO[\frac{1}{p}]\big)\] of  $\Vect\big((\calK^+/A)_{\Prism},\overline\calO[\frac{1}{p}]\big)$, called the category of $a$-small crystals, on which $\cal{R}$ is fully faithful.
        \item The essential image of the full subcategory
        \[\bigcup_{a\in\bfC^\times,|a|<1}\Vect^{\sm a}\big((\calK^+/A)_{\Prism},\overline\calO[\frac{1}{p}]\big)\]
        is the representations $(M,\rho)$ such that there exists $a\in\bfC^\times$, $|a|<1$ satisfying that the Sen action of $a^{-1}(\zeta_{p}-1)^{-1}\cOmega_{\calK^+/\calO_{\mathbf{C}}}^\vee(1)$ on $M$ is topologically nilpotent.
    \end{enumerate}
\end{theo}

We call a crystal in $\Vect\big((\calK^+/A_{\inf})_{\Prism},\overline\calO[\frac{1}{p}]\big)$ \emph{strongly topologically nilpotent} if it is $a$-small for some topologically nilpotent $a\in \calK^\times$. The above theorem then gives a classification of all strongly topologically nilpotent crystals by generalized representations.

We conjecture that the class of strongly topologically nilpotent crystals is so big that it contains all crystals from rigid analytic geometry.

\begin{conj}
    Let $R$ be a topologically finitely generated and $p$-complete $\calO_{\mathbf{C}}$-algebra, and \[f:R\to \calK^+\] be an $\calO_{\mathbf C}$-homomorphism. Then for any rational Hodge--Tate crystal $\calE$ on $(R/A)_\Prism$, %\footnote{We still consider rational Hodge--Tate prismatic crystals as crystals on the prismatic site as we defined earlier even it is possibly better to define a rational Hodge--Tate crystal as a vector bundle (with certain coefficients) on the Hodge--Tate stack $\Spf(R)^{\HT}_{A}$. We emphasize that by abstract nonsense this conjecture  still holds if we change the definition.}
    the base change of $\calE$ to $\calK^+$ is $a$-small for some topologically nilpotent $a\in \calK^\times$.
\end{conj}
\vspace{0.3em}

\noindent\textbf{\scshape A prismatic $p$-adic non-abelian Hodge theory for pre-smooth formal schemes: }
Let $\mathbf{C}$ be an algebraic closed $p$-adic non-archimedean field, and let $\calO_{\mathbf{C}}$ be the ring of integers of $\mathbf{C}$. Put $\Ainf$ the prism $(\AAinf(\calO_{\mathbf{C}}),\ker\theta)$. Fix a norm $|\cdot|$ on $\bfC$. In this subsection, all $p$-complete $\calO_{\mathbf{C}}$-algebras are assumed to be topologically finitely generated. For any smooth formal scheme $\mathfrak{X}$ over $\calO_{\mathbf{C}}$ with generic fiber $X$, Min--Wang studied the functor ${\Res}_{\mathfrak{X}}$ in \cite{Min_2025}\footnote{Although in the paper, they worked on the arithmetic setting, their method also works for the geometric setting. They pointed this out in \cite[Remark 1.5]{Min_2025}.}\footnote{Their results was later reproved in \cite{bhatt2022prismatization} and \cite{anschütz2023smallpadicsimpsoncorrespondence} using stacky approach. However, we will not use this approach here.}, and found that ${\Res}_{\mathfrak X}$ is fully faithful with essential image described by certain smallness conditions.

Using our theory, we can analyze the functor ${\Res}_{\mathfrak{X}}$ for some non-smooth formal schemes. Specifically,
we will define a new class of $\calO_{\mathbf{C}}$-algebras --- pre-smooth algebras --- containing all $p$-completely smooth algebras as well as those $p$-complete algebras with semi-stable reductions. The explicit definitions will be introduced later (cf. Definition \ref{intro: pre smooth}).

Our first main result is a Hodge--Tate type comparison theorem for pre-smooth algebras.

\begin{theo}[Corollary \ref{cor: HT comparison revisited}]
    Let $R/\calO_{\mathbf{C}}$ be a pre-smooth algebra. Then, the Hodge--Tate filtration on $\overline\Prism_{R/\calO_{\mathbf C}}$ is exhaustive in $D(R)$ (i.e. in the uncompleted sense), and it induces an isomorphism of graded commutative rings
    \[\bigwedge\nolimits^*_{R[\frac 1 p]} \bigl( \cdL_{R/\calO_{\mathbf{C}}} \{-1\}\bigl[ \frac 1 p \bigr] \bigr) \to \hol^*\bigl((\Spf(R)/\Ainf)_{\Prism},\overline\calO \bigl[ \frac 1 p \bigr] \bigr). \]
\end{theo}

\begin{rmk}
    Indeed, the theorem also holds for a larger class of algebras, weakly pre-smooth algebras. (cf. Definition \ref{intro: weak pre smooth})
\end{rmk}

\begin{rmk}
    This theorem is non-trivial, because for this theorem to hold (for some $R$ satisfying Condition \ref{cond: tor -1,0} for which $\cdL_{R/\calO_{\mathbf{C}}}$ is of bounded $p^\infty$-torsion and that $\cdL_{R/\calO_{\mathbf{C}}} \bigl[ \frac 1 p \bigr]$ is finite projective over $R \bigl[ \frac 1 p \bigr]$), a necessary and sufficient condition is that the Hodge--Tate filtration on $\overline\Prism_{R/\calO_{\mathbf C}}$ is exhaustive, i.e. the map
    \[ \colim_n \Fil_n(\overline\Prism_{R/\calO_{\mathbf C}}) \to \overline\Prism_{R/\calO_{\mathbf C}} \]
    is an isomorphism in $D(R)$. In general, the Hodge-Tate filtration only guarantees that the above map is an isomorphism \emph{after taking the derived $p$-adic completion}.

    We give a short proof that the condition is necessary and sufficient. Note that the fibre of the map
    \[ \colim_n \Fil_n(\overline\Prism_{R/\calO_{\mathbf C}}) \to \overline\Prism_{R/\calO_{\mathbf C}} \]
    is identified with $\RHom_\dZ(\dZ[1/p], \colim_n \Fil_n(\overline\Prism_{R/\calO_{\mathbf C}}))$, so the map is an isomorphism if and only if it is an isomorphism after inverting $p$. Now, the cohomology ring of $\colim_n \Fil_n(\overline\Prism_{R/\calO_{\mathbf C}}) \bigl[ \frac 1 p \bigr]$ is (using techniques of Corollary \ref{cor: HT comparison revisited}) $\bigwedge\nolimits^*_{R[\frac 1 p]} \bigl( \cdL_{R/\calO_{\mathbf{C}}} \{-1\} \bigl[ \frac 1 p \bigr] \bigr)$, and the cohomology ring of $\overline\Prism_{R/\calO_{\mathbf C}} \bigl[ \frac 1 p \bigr]$ is $\hol^*\bigl((\Spf(R)/\Ainf)_{\Prism},\overline\calO \bigl[ \frac 1 p \bigr] \bigr)$.
\end{rmk}

\begin{theo}[Corollary \ref{cor: fully faithfulness on small crystal l.c.i case revisited}]\label{intro: main theorem p adic Simpson}
    Let $R/\calO_{\mathbf{C}}$ be a pre-smooth algebra. Then:
    \begin{enumerate}[label=(\arabic*)]
        \item There exists an $a\in \mathbf{C}^\times$ which is topologically nilpotent such that the essential image of ${\Res}_{\Spf(R)}$ contains all $v$-vector bundles such that the Sen action of $a^{-1}(\zeta_p-1)^{-1}\cdL_{R/\calO_{\mathbf{C}}}^\vee(1)$ is topologically nilpotent. (Here $\cdL_{R/\calO_{\mathbf{C}}} := \Hom_R(\cdL_{R/\calO_{\mathbf{C}}},R)$ is non-derived.)
        \item We may choose $a\in \bfC^\times$ in (1) such that there exists a full subcategory $\calC$ of $\Vect\bigl((R/\Ainf),\overline{\calO}[\frac{1}{p}]\bigr)$ on which ${\Res}_{\Spf(R)}$ is fully faithful, and the essential image of $\calC$ contains all $v$-vector bundles mentioned in (1).
    \end{enumerate}
\end{theo}

To make it clearer, we provide a moduli description. In \cite{Heuer_2025}, Heuer defined for any smooth rigid space $X/\mathbf{C}$ (or more general, a smoothiod space), the moduli stack $\mathcal{B}un_n(X)$ of rank-$n$ $v$-vector bundles. Also, like the complex analytic geometry, he defined the Hitchin base $\calA_n(X)$ of $X$. He pointed out that there exists a canonical morphism
\[\widetilde H:\calB un_n(X)\to \calA_n(X)\]
called the Hitchin morphism on the Betti side (cf.\cite[Definition 8.12]{Heuer_2025}). The morphism $\widetilde H$ describes the eigenvalues of the Sen operators. From this point of view, Theorem \ref{intro: main theorem p adic Simpson} ensures that the essential image of ${\Res}_{\Spf(R)}$ contains $\widetilde H^{-1}(U)$ for some neighborhood of $0$ in the Hitchin base of $\Spa(R[\frac{1}{p}])$.

\vspace{0.3em}

\noindent\textbf{\scshape De rham realization of rational Hodge Tate crystals of quasi-l.c.i algebras: } Now we turn to introducing our main result on classifying rational Hodge--Tate crystals of quasi-l.c.i
algebras.

Fix a bounded prism $\left( A,I \right)$. We impose the following condition on an $\overline A:=A/I$-algebra $R$:
\begin{cond}[{\cite[Remark 4.1.18]{bhatt2022absolute}}]\label{cond: tor -1,0}
    The ring $R$ is of bounded $p^\infty$-torsion, $p$-adically complete, and the complete contangent complex $\cdL_{R/\overline A}$ is of $p$-complete $\Tor$-amplitude $[-1,0]$ over $R$ (see Proposition-Definition \ref{prop-defi: complete flatness} for complete flatness and complete $\Tor$-amplitudes). Such an algebra is is called \emph{quasi-l.c.i} over $\overline{A}$.
\end{cond}

Also, we have the following variant:
\begin{cond}[{\cite[Remark 4.1.19]{bhatt2022absolute}}]\label{cond: tor -1}
    The ring $R$ is of bounded $p^\infty$-torsion, $p$-adically complete, and $\cdL_{R/\overline A} \left[ -1 \right]$ is $p$-completely flat over $R$.
\end{cond}

Let $S$ be a bounded prism over $A$. We also consider the following condition:
\begin{cond}\label{cond: flat + tor 0}
    The ring $S$ is $\left( p,I \right)$-completely flat over $A$, and the complete contangent complex $\cdL_{S/A}$ is $p$-completely flat over $S$.
\end{cond}

Under the condition above, we have the following theorem:
\begin{theo}[Theorem \ref{theo: main dR realization revisit}]\label{theo: main dR realization}
    Let $R$ be an $\overline A$-algebra, and $S$ be a bounded prism over $A$, equipped with a morphism of $A$-algebras $S \to R.$ Assume that $S/A$ satisfies Condition \ref{cond: flat + tor 0} and $R/S$ satisfies Condition \ref{cond: tor -1}. Then, $D := \Prism_{R/S}$ is concentrated on degree 0, and a bounded prism over $S$.
    
    Define $\overline D = D/ID$. Then, we have an $R$-linear derivation on $\overline D$ (integrable as a Higgs field over $R$):
    \[ \dif: \overline D \to \overline D \cotimes_R \bigl( R \cotimes_{\overline S} \cOmega_{\overline S/\overline A} \bigr) \left\{ -1 \right\}, \]
    satisfying the Griffith transversality, with $\gr_*(\dif)$ as in Example \ref{ex: differential dHodge}, such that we have the symmetric monoidal correspondences below:
    \begin{align*}
        \Vect \bigl( \left( R/A \right)_\Prism, \overline\calO \bigr) & \cong \MIC^\tn \bigl( \overline D, \dif \bigr), \\
        \Vect \bigl( \left( R/A \right)_\Prism, \overline\calO \bigl[ \frac 1 p \bigr] \bigr) & \cong \MIC^\tn \bigl( \overline D \bigl[ \frac 1 p \bigr], \dif \bigr).
    \end{align*}
    Moreover, if $\mathcal E$ corresponds to $\left( M,\nabla \right)$ under the correspondence above, then we have a natural isomorphism of cohomologies (compatible with the symmetric monoidal structure)
    \[ R\Gamma \bigl( \left( R/A \right)_\Prism, \mathcal E \bigr) \cong \DR \left( M,\nabla \right). \]
    When $\calE=\overline\calO$ is trivial, then the above isomorphism can be made into a strict isomorphism of filtered complete $\dE_\infty$-$R$-algebras, and its graded pieces is identified (under Hodge-Tate comparison) with the map
    \[ \dHodge_{R/A} \to \DR(\dHodge_{R/S}, \dif) \]
    as in Corollary \ref{cor: natural isomorphism of dR and dHodge}.
\end{theo}

The proof of this theorem is based on the techniques of \cite{Bhatt_2023}. First, using Theorem \cite[Corollary 4.3.14]{bhatt2022absolute}, one can obtain that under condition \ref{cond: tor -1}, the prismatic cohomology $\Prism_{R/S}$ is concentrated in degree 0 and is the initial object of the category ${R/S}_{\Prism}$. Through the techniques of \cite{Bhatt_2023}, one can transform a rational Hodge--Tate crystal into a stratification on the cosimplicial ring $\{\Prism_{R/S^{\cotimes\bullet}}\}$ (cf. \cite[Subsection 7.3]{Bhatt_2023}). Here, the complete tensor product is taken over $A$. In order to further obtain a connection from the stratification, we need to concretely compute the structure of $\Prism_{R/S^{\cotimes n}}$. An important observation is as follows:

\begin{theo}[Theorem \ref{theo: prism=pd derived construction}]\label{intro: prism=pd derived construction}
    Let $(A,I)$ be a bounded prism. Let $(R,IR)$ be a bounded prism over $(A,I)$ such that $A \to R$ is surjective, and $\overline R/\overline A$ satisfies Condition \ref{cond: tor -1}. Denote $J = \fib(A \to R) = \ker(A \to R)$ and $\overline J = J \otimes_A^\bL \overline A$. Then $\overline J = \fib(\overline A \to \overline R)$ is concentrated on degree 0, and there is a canonical isomorphism compatible with the Hodge-Tate comparison
    \[\overline\Prism_{\overline{R}/A} \cong \cGamma_{\overline R}^*\bigl((\overline J/\overline J^2)^\land \{-1\}\bigr),\]
    such that the map $\bigl(\overline J/\overline J^2\bigr)^\land \left\{ -1 \right\} \to \overline \Prism_{\overline R/A}$ is given by the twist of the reduction of
    \[ J \to I\Prism_{\overline R/A} \]
    (note that by the universal property of $\Prism_{\overline R/A},$ the image of $J$ lies in $I\Prism_{\overline R/A}$).

    By naturality, the map $\overline\Prism_{\overline R/A} \to \overline\Prism_{\overline R/R} = \overline R$ is identified with $\cGamma_{\overline R}^*\bigl((\overline J/\overline J^2)^\land \{-1\}\bigr) \to \overline R$ (sending the positive degree part to zero). Therefore, the positive degree part of $\cGamma_{\overline R}^*\bigl((\overline J/\overline J^2)^\land \{-1\}\bigr) \cong \overline\Prism_{\overline R/A}$ agrees with the reduction of the kernel of $\Prism_{\overline R/A} \to R.$
\end{theo}

The notation $\widehat\Gamma^*_{\overline{R}}$ is a complete version of PD-polynomial (See Section \ref{sect: remarks on infty-cat}).
In fact, when $J$ is generated by a Koszul regular sequence, this theorem was already proven in \cite{Tian_2023}. Tian's proof is based on explicit computations on $\delta$-polynomial rings. However, his method does not directly work in our situation, primarily because we do not know the form of Theorem \cite[Theorem 3.13]{Bhatt_2022} in general. Here we will use a new (and more natural) method to prove this theorem. Moreover, it can be shown that our construction coincides with Tian's construction when $J$ is generated by a Koszul regular sequence.

Furthermore, applying Theorem \ref{intro: prism=pd derived construction}, we will prove that for any $n$,
\[
\overline\Prism_{R/S^{\cotimes{n+1}}} \cong \overline\Prism_{R/S}\cotimes_R \cGamma_R^* \bigl(R \cotimes_{\overline S}\cOmega_{\overline S/\overline A}^{\oplus n}\{-1\}\bigr).
\]
Finally, by generalizing Tian's proof to the case of general PD polynomials, we obtain the proof of Theorem \ref{theo: main dR realization}.

\begin{rmk}[Links with previous results]
    As we mentioned above, Theorem \ref{theo: main dR realization} is a generalization of many previous results:
    \begin{enumerate}[label=(\arabic*)]
        \item When $R$ is $p$-completely smooth and $S/IS=R$, this is proved in \cite{Tian_2023}.
        \item When $A$ is crystalline, $S$ is smooth and $\ker(S/p\to R)$ is generated by a regular sequence, this is proved in \cite{ogus2024crystallineprismsreflectionsdiffractions}.
        \item When $A$ is the $q$-crystalline prism (cf. Section \cite{Bhatt_2022}), $S$ is smooth and $\ker(\overline{S}\to R)$ is generated by a Koszul regular sequence, this is proved in \cite{tsuji2024prismaticcrystalsqhiggsfields} with a stacky approach in \cite{liu2025stackyapproachprismaticcrystals}.
        \item When $A=\AAinf(\calO_{\mathbb{C}_p})$ and $R$ is smooth, this is proved in \cite{Min_2025} with a stacky approach in \cite{anschütz2023smallpadicsimpsoncorrespondence}.
    \end{enumerate}
\end{rmk}

\begin{rmk}[Generalizing to other coefficients]
    It is natural to ask whether Theorem \ref{theo: main dR realization} can be generalized to other coefficients (such as $\Bdrp$). It seems to be hard for general prisms $A$. But for special $A$, this can be done, and we are working on this.
\end{rmk}
\vspace{0.3em}
\noindent\textbf{\scshape Local calculation of the functor ${\Res}$: }
In this subsection, we focus on the local algebraic framework necessary for the $p$-adic non-abelian Hodge theory. The primary technical obstacle in extending results like those of Min–Wang (\cite{Min_2025}) lies in performing reliable local computations when the integral model is not smooth. Our goal is to develop an alternative local method that accommodates controlled singularities.

Let $\mathbf{C}$ be an algebraic closed $p$-adic non-archimedean field, and let $\calO_{\mathbf{C}}$ be the ring of integers of $\mathbf{C}$. Put $\Ainf$ the prism $(\AAinf(\calO_{\mathbf{C}}),\ker\theta)$.

\begin{defi}\label{intro: weak pre smooth}
    Let $R/\calO_{\mathbf{C}}$ be a $p$-complete algebra. Call $R$ \emph{weakly pre-smooth} if it satisfies the following conditions:
    \begin{enumerate}[label=(\arabic*)]
        \item $R/\calO_{\mathbf{C}}$ is quasi-syntomic in the sense of \cite[Definition 1.1]{Ansch_tz_2023}.
        \item The complete cotangent complex $\widehat\dL_{R/\calO_{\mathbf{C}}}$ is concentrated a $p$-torsion free module in degree $0$.
        \item $\cdL_{R/\calO_{\mathbf{C}}}[\frac{1}{p}]$ is a finite-rank projective module over $R[\frac{1}{p}]$.
    \end{enumerate}
\end{defi}

\begin{defi}\label{intro: pre smooth}
    Let $R/\calO_{\mathbf{C}}$ be a $p$-complete weakly pre-smooth algebra. Call $R$ \emph{pre-smooth} if there exists $f_1,f_2,\dots,f_n\in R$ that generate the unit ideal, and for any $i$, there exists a homomorphism of $\calO_{\mathbf{C}}$-algebras
    \[ \calO_{\mathbf{C}}\langle T_1,T_2,\dots,T_{d_i}\rangle \to \widehat{R[f_i^{-1}]} \]
    such that the images of $\dif T_i$ in $\widehat{R[f_i^{-1}]} \cotimes_R \cdL_{R/\calO_{\mathbf{C}}}$ freely generate $\cdL_{\widehat{R[f_i^{-1}]}/\calO_{\mathbf C}} \bigl[ \frac 1 p \bigr] \cong \widehat{R[f_i^{-1}]} \cotimes_R \cdL_{R/\calO_{\mathbf{C}}} \bigl[\frac{1}{p}\bigr]$ as a $R \bigl[\frac{1}{p}\bigr]$-module.
\end{defi}

The pre-smoothness is significantly weaker than the smoothness assumption used in previous literature, yet it is arithmetically essential:
\begin{itemize}
    \item Quasi-syntomicness ensures that we can use our theory to classify rational Hodge--Tate crystals by connections.
    \item The other assumptions on cotangent complex ensures that the generic fiber $R[\frac{1}{p}]$ is `nearly' smooth. Hence, the classical calculations on toric charts (see, for example \cite[Section 5]{Min_2025}) still works.
\end{itemize}

\begin{ex}
    \begin{itemize}
        \item Any $p$-completely smooth algebra over $\calO_\mathbf{C}$ is pre-smooth.
        \item A $p$-complete $\calO_{\mathbf{C}}$-algebra with semi-stable reduction is pre-smooth (cf. Example \ref{ex: semistable are presmooth}).
        \item A geometric valuation ring over $\calO_{\mathbf{C}}$ is pre-smooth.
    \end{itemize}
\end{ex}

The local computations in the smooth setting (cf. \cite[Section 5]{Min_2025}) fundamentally rely on the existence of a $p$-adically complete étale map from a standard toric ring:
\[\calO_{\mathbf{C}}\langle T_1^{\pm1},T_{2}^{\pm1},\dots,T_{d}^{\pm1} \rangle\to R.\]
Under our pre-smooth assumption, such an étale map cannot exist unless $R$ is smooth. Consequently, the reliance on this strong local model must be abandoned.

Fix a pre-smooth algebra $R/\calO_{\mathbf{C}}$

\begin{defi}[Definition \ref{defi: toric chart}]\label{defi: intro, toric chart}
    A toric chart of $R$ is a homomorphism of $\calO_{\mathbf{C}}$-algebras
    \[f:\calO_{\mathbf{C}}\langle T_1,T_2,\dots,T_d\rangle\to R\]
    satisfying the following conditions:
    \begin{enumerate}[label=(\arabic*)]
        \item $\cdL_{R/\calO_{\mathbf{C}}}[\frac{1}{p}]=\bigoplus_{j=1}^dR\bigl[\frac{1}{p}\bigr]\dif f(T_j)$;
        \item The ring 
        \[\bigl(\calO_{\mathbf{C}}\langle T_1^{\frac{1}{p^\infty}},T_2^{\frac{1}{p^\infty}},\dots, T_d^{\frac{1}{p^\infty}} \rangle\cotimes_{\calO_{\mathbf{C}}\langle T_1,T_2,\dots, T_d\rangle}R\bigr)\bigl[\frac{1}{p}\bigr]\]
        is locally perfectoid (Definition \ref{defi: locally perfectoid huber pairs});
        \item For any $i$, $f(T_i)$ is invertible in $R$.
    \end{enumerate}
    For such a toric chart, define $R_\infty$ as the integral closure of
    \[ \calO_{\mathbf{C}}\langle T_1^{\frac{1}{p^\infty}},T_2^{\frac{1}{p^\infty}},\dots, T_d^{\frac{1}{p^\infty}} \rangle\cotimes_{\calO_{\mathbf{C}}\langle T_1,T_2,\dots, T_d\rangle}R\]
    inside
    \[\bigl(\calO_{\mathbf{C}}\langle T_1^{\frac{1}{p^\infty}},T_2^{\frac{1}{p^\infty}},\dots, T_d^{\frac{1}{p^\infty}} \rangle\cotimes_{\calO_{\mathbf{C}}\langle T_1,T_2,\dots, T_d\rangle}R\bigr)\bigl[\frac{1}{p}\bigr].\]
\end{defi}

From now on, fix for any $n\geq 1$ a primitive $p^n$-root of unity $\zeta_{p^n}$ such that $\zeta_{p^n}^p=\zeta_{p^{n-1}}$.

\begin{rmk}
    Keep the notation in Definition \ref{defi: intro, toric chart}. Obviously, there is a continuous action of $\Gamma:=\mathbb{Z}_p^d$ on $R_\infty[\frac{1}{p}]$. The `pro-finite \'etale descent' provides an equivalence between the $v$-vector bundles on $\Spd(R,R^+)$ and continuous semi-linear representation of $\Gamma$ over $R_{\infty}[\frac{1}{p}]$.    
\end{rmk}

Our main calculation is the following:

\begin{prop}[Special case of Proposition \ref{prop: Higgs to perfect crystals}]
Keep the notations and assume $f:\calO_{\mathbf{C}}\langle T_1,T_2,\dots,T_d\rangle\to R$ is a toric chart. Then we have the a commutative diagram as follows
\begin{equation}\label{eq: intro highs to perf}
    \xymatrix{
\Vect\bigl((R/A_{\inf})_\Prism, \overline\calO\bigl[\frac 1 p \bigr]\bigr) \ar[r]^-{{\Res}_{\Spf(R)}} & \Vect\bigl( \Spd(R,R^+)_v,\widehat\calO\bigr)\ar[d] \\
\Higgs^\tn(R[\frac{1}{p}],R \otimes_{\overline P}\cOmega_{\overline P/\calO_{\mathbf{C}}} \{-1\}) \ar[r] \ar[u] & \Rep_\Gamma \bigl( R_\infty \bigl[ \frac 1 p \bigr] \bigr)}
\end{equation}
where the arrow on the left hand side is a `natural base change' (see Proposition \ref{prop: Higgs to crystal}), and the arrow at the bottom is given by $(M,\nabla) \mapsto M_\infty = R_\infty \otimes_R M$, where $\gamma_i$ acts semilinearly on $M_\infty$, and acts on $M$ by $\exp((\zeta_p-1) f(T_i)\nabla_i)$, in which
\[ \nabla(m) = \sum_{i=1}^d \nabla_i(m) \otimes \frac{\dif x_i}{\xi}, \]
and $\xi := \dfrac{[\ep]-1}{[\ep]^{1/p}-1}$.
\end{prop}

\begin{rmk}
    We explain the `natural base change' in Diagram (\ref{eq: intro highs to perf}). 
    
    View $\Ainf\langle T_1,T_2,\dots,T_d \rangle$ as a $\delta$-algebra over $A$ such that $\delta(T_i)=0$. Choose a $\delta$-algebra $S/A$ satisfying Condition \ref{cond: flat + tor 0} and $f:S\to R$ satisfying Condition \ref{cond: tor -1}. In addition, choose a $\delta$-homomorphism $i:\Ainf\langle T_1,T_2,\dots,T_d \rangle\to S$. By Theorem \ref{theo: main dR realization}, we have an equivalence
    \[\Vect \bigl( \left( R/A \right)_\Prism, \overline\calO \bigl[ \frac 1 p \bigr] \bigr)  \cong \MIC^\tn \bigl( \overline D \bigl[ \frac 1 p \bigr], \dif \bigr)\]
    where $D=\Prism_{R/S}$. Consider the $R$-linear homomorphism
    \[\iota:R\otimes_{\overline{P}}\cOmega_{\overline{P}/\calO_{\mathbf{C}}}\to R\cotimes_{\overline S}\cOmega_{\overline{S}/\calO_{\mathbf{C}}}\]
    sending $\dif T_i$ to $\dif(i(T_i))$ for all $i$. Then for any topologically nilpotent Higgs bundle $(M,\nabla)$ over $R[\frac{1}{p}]$, $\bigl(\overline{D}\otimes_{R}M,\dif_{\overline{D}}\otimes \id + \id \otimes (\iota\circ\nabla)\bigr)$ defines an object in $\MIC^\tn \bigl( \overline D \bigl[ \frac 1 p \bigr], \dif \bigr)$. This is the construction of the `natural base change'.

    We will see that this functor does not depend on the choice of $S$. Hence, if $R$ is $p$-completely smooth, we can choose $S=\Ainf\langle T_1,T_2,\dots,T_d \rangle$ and $f:\overline{S}\to R$ is $p$-completely \'etale. This time our construction coincides with \cite{Min_2025}.
\end{rmk}

\vspace{0.3em}
\noindent\textbf{\scshape{The smallness condition of a rational Hodge--Tate prismatic crystal: }}
Finally, we will describe the essential image of the natural base change in the diagram (\ref{eq: intro highs to perf}). This is provided by the so-called `smallness' of a rational Hodge--Tate prismatic crystal. Roughly speaking, keep the notation in Theorem \ref{theo: main dR realization} and assume $A=\Ainf(\calO_{\mathbb{C}_p})$. The ring $\overline{\Prism}_{R/S}$ has a canonical increasing filtration---the conjugate filtration. For any topologically nilpotent $a\in \dC_{p}^\times$, define a new ring 
$\overline{\Prism}_{R/S}^a$ by the completion of a suitable subring of $\colim \Fil_n (\overline{\Prism}_{R/S})$ (as rings, not $p$-complete rings). Similarly, there exists a differential $\dif:\overline{\Prism}_{R/S}^a \to \overline{\Prism}_{R/S}^a\cotimes_{R}(R\cotimes_{\overline{S}}\cOmega_{\overline{S}/\calO_{\mathbb{C}_p}})$. The rings $\overline{\Prism}_{R/S}^a$ satisfy the following properties:
\begin{itemize}
    \item There exists a canonical homomorphism $\overline{\Prism}_{R/S}^a\to \overline{\Prism}_{R/S}$.
    \item For any topologically nilpotent $a'\in \dC_p^\times$, there exists a canonical homomorphism $\overline{\Prism}_{R/S}^{aa'}\to \overline{\Prism}_{R/S}^a$.
    \item For any $a'\in \calO_{\dC_p}^\times$, there exists a canonical isomorphism
    $\overline{\Prism}_{R/S}^a\cong \overline{\Prism}_{R/S}^{aa'}$.
    \item All homomorphisms appearing in the above three properties are compatible with each other.
\end{itemize}
In particular, there exists a base change functor from topologically nilpotent connections on $\overline{\Prism}_{R/S}^a$ to $\overline{\Prism}_{R/S}$. We will define the $a$-small crystals as the essential image of this functor and prove that this does not depend on the choice of $S$.

\subsection{Acknowledgment}

We thank Mr. Tian Qiu for carefully reading and pointing out some typos of the draft of this paper. We also thank Mr. Lin Chen for his help on $\infty$-category theory. 

We sincerely thank Professor Ruochuan Liu for inviting us to give a talk of our article and for the valuable discussion. The idea for this paper was conceived during the Shenzhen Arithmetic Geometry Conference (2024), and we are grateful to the organizers of this conference as well as the speakers for their excellent talks. 

\newpage

\section*{Notations and conventions}\label{sect: notations}

\noindent\textbf{\scshape{Tate algebras, (pre-)adic spaces and diamonds:}}

All Tate algebras are assumed to be complete and over $\dZ_p$.

\begin{itemize}
    \item \fbox{$\Spa$}: adic spectrum. Given a Tate--Huber pair $(R,R^+)$, we will use $|\Spa(R,R^+)|$ to denote the associate adic spectrum (as a topological space). When the structure presheaf $\calO_{\Spa(R,R^+)}$ is a sheaf, we will use $\Spa(R,R^+)$ to denote the adic space. We will not use the notation $\Spa(R,R^+)$ unless $(R,R^+)$ is sheafy.

    \item \fbox{$X^\diamond$, $\Spd$}: associated diamond. In this paper, we will freely use the language of diamonds (cf. \cite[Chapter 8]{Scholze_2020}). For any pre-adic space $X$ (see \cite[Definition 3.4.2]{Scholze_2020} for the definition of pre-adic spaces), we will use $X^\diamond$ to denote its associated diamond (cf. \cite[Definition 10.1.1]{Scholze_2020}). For an affinoid pre-adic space (see \cite[Definition 3.4.1]{Scholze_2020}) $X=\Spa^{\mathrm{ind}}(R,R^+)$, denote by $\Spd(R,R^+)$ the diamond $X^\diamond$.

    \item \fbox{$\Spec (R), \Spf (R)$}: the spectrum and the formal spectrum of a ring. We will only use $\Spf(R)$ in the case when $R$ is a $p$-complete ring of bounded $p^\infty$-torsion.

    \item \fbox{Formal schemes}: We will only use bounded $p$-adic formal schemes, i.e. $\mathfrak X$ is covered by affine opens $\Spf(R)$ where $R$ is a $p$-complete ring of bounded $p^\infty$-torsion.

    \item \fbox{Affinoid perfectoid spaces}: An affinoid perfectoid space is an affinoid adic space $X=\Spa(R,R^+)$ such that $R^+$ integral perfectoid.

    \item \fbox{$\CAfd$}: category of Tate-Huber pairs.

    \item \fbox{$X_v$}: the $v$-topology. For any diamond $X$, denote by $X_v$ the ecategory of pairs $(T:=\Spa(R,R^+),f:T^\diamond\to X)$ where $T$ is an affinoid perfectoid space and $f$ is a morphism between diamonds. Topology on $X_v$ is the $v$-topology of perfectoid spaces.

    \item \fbox{$\mathrm{an}$}: analytic vectors

    \item \fbox{$\la$}: locally analytic vectors.
\end{itemize}
\vspace{0.3em}
\noindent\textbf{\scshape{Prisms and $p$-adic Hodge theory:}}

\begin{itemize}
\item \fbox{$(R/A)_\Prism,\calO,\overline\calO$}: Let $(A,I)$ be a bounded prism in the sense of \cite[Definition 3.2]{Bhatt_2022} and $R$ be a $p$-complete ring over $\overline A :=A/I$ of bounded $p^\infty$-torsion. We use $(R/A)_\Prism$ to denote the category of pairs $((B,IB),R \to \overline B)$, where $(B,IB)$ is a bounded prism over $A$, and $R \to \overline B := B/IB$ is a map of $\overline A$-algebras. Equip $(R/A)_\Prism$ with the $p$-complete fpqc topology. The structure sheaves $\calO, \overline\calO$ are defined by $(B,IB) \mapsto B$ and $(B,IB) \mapsto \overline B$, so $\calO$ is a sheaf of derived $(p,I)$-complete $A$-algebras, and $\overline\calO$ is a sheaf of derived $p$-complete $R$-algebras.

\item \fbox{$\Prism_{R/A}$}: The derived prismatic cohomology in \cite[Construction 7.6]{Bhatt_2022}; here, $(A,I)$ is a bounded prism and $R$ is a derived $p$-complete simplicial ring over $\overline A := A/I$.

\item \fbox{$W(R), W_n(R)$}: the Witt vectors and the truncated Witt vectors.

\item \fbox{$R^\flat, \AAinf(R), \theta: \AAinf(R) \to R, \Ainf$}: Let $R$ be integral perfectoid in the sense of \cite[Definition 3.5]{bhatt2018integral}. We use $\displaystyle R^\flat = \clims 0 \varphi R/p$ to denote the tilt of $R$, and $\AAinf (R) = W(R^\flat)$ and $\theta: \AAinf(R) \to R$ to denote the $\Ainf$-construction. By \cite[Theorem 3.10]{Bhatt_2022}, $(\AAinf(R),\ker \theta)$ is a perfect prism, and the functor $R \mapsto (\AAinf(R),\ker \theta)$ gives an equivalence between the category of integral perfectoid rings and the category perfect prisms. We usually use $\Ainf$ to denote the perfect prism $(\AAinf(K^+),\ker \theta)$ for some perfectoid field $(K,K^+)$.
\end{itemize}
\vspace{0.3em}
\noindent\textbf{\scshape{Homological algebras and $\infty$-categories:}}

All rings are commutative ($\dE_\infty$) unless otherwise stated.

The readers may refer to \cite{lurie2009higher}, \cite{lurie2017higher} and \cite{lurie2018spectral} for the theory of $\infty$-categories. In this article, all $\infty$-categories refer to $(\infty,1)$-categories. We use $\infty$-groupoids to refer to $(\infty,0)$-categories.

\begin{itemize}
\item \fbox{$\simp, \Delta^n, [n]$}: We use $\simp$ to denote the category of combinatorial simplices (see \cite[Section A.2.7]{lurie2009higher}), and $[n] = \{0,\ldots,n\}$ the objects in $\simp$. We use $\Delta^n$ to denote the presheaf on $\simp$ represented by $[n]$, so $\Delta^n$ are the standard simplices.

%\item \fbox{$\Sigma(-), \Omega(-)$}: suspensions and loops of pointed spaces.

\item \fbox{$\calC^\op$}: the opposite ($\infty$-)category of $\calC$.

\item \fbox{$\Hom, \Mor, \Map$}: the set (space) of morphisms. Usually we use $\Hom$ and $\Mor$ in case of ordinary categories, and $\Map$ in case of $\infty$-categories.

\item \fbox{$\Fun(\calC,\calD), \Shv(\calC), \Shv_\calC(\calD)$}: ($\infty$-)categories of functors, sheaves on $\calC$, and $\calD$-valued sheaves on $\calC$, respectively.

\item \fbox{$\calC_{x/}, \calC_{/x}$}: slicing categories of $\calC$. More precisely, for a functor $F: \calC \to \calD$ of $\infty$-categories and $x \in \calD$, $\calC_{x/}$ is the $\infty$-category of pairs $(y \in \calC, f:x \to Fy)$, and $\calC_{/x}$ is the $\infty$-category of pairs $(y\in \calC, f: Fy \to x)$.

\item \fbox{$\Fil(\calC) = \Fun(\mathbb Z_{\ge 0}, \calC), \gr(\calC) = \Fun(\mathbb Z_{\ge 0, \discrete}, \calC)$}: $\infty$-category of filtered and graded objects. See Sect. \ref{sect: remarks on infty-cat} for definitions and details.

\item \fbox{$\pi_n, \hol^n, \tau^{\le n}, \tau^{\ge n}, \calC^{\le n}, \calC^{\ge n}$}: homotopy (homology) groups and truncations. We always use the cohomological convention except for $\pi_n$. Let $\calC$ be a stable $\infty$-category, an object $X \in \calC$ is called \emph{$n$-truncated} (resp. \emph{$n$-connective}) if $X \in \calC^{\ge -n}$ (resp. $\calC^{\le -n}$). In particular, $X\in\calC$ is called coconnective (resp. connective) if $X \in \calC^{\ge 0}$ (resp. $\calC^{\le 0}$).

\item \fbox{$X[n], X\{n\}$}: shiftings and Breuil-Kisin twists. For abelian groups $M$, sometimes we use $M[p^n]$ to denote the $p^n$-torsion part of $M$, and $M[p^\infty]$ to denote the union of all $M[p^n]$. Hopefully this would not cause any confusion.

\item \fbox{$\Sptr$}: stable $\infty$-category of spectra, equipped with the natural $t$-structure (both left complete and right complete).

\item \fbox{$\RHom, \taulezRHom, \underline{\Hom}, \underline{\RHom}, \underline{\taulezRHom}$}: inner $\Hom$'s in different categories. $\RHom$ and $\taulezRHom$ are inner $\Hom$s in $\Sptr$ and $\Sptr^{\le 0}$, $\underline{\Hom}$ denotes inner $\Hom$ of sheaves of spaces (which also works for abelian groups), $\underline{\RHom}$ and $\underline{\taulezRHom}$ denotes the inner $\Hom$ of sheaves of spectra and sheaves of connective spectra.

\item \fbox{$\SCR, \SCRModcn$}: $\SCR$ denotes the $\infty$-category of simplicial commutative rings; $\SCRModcn$ denotes the category of pairs $(A,M)$ where $A\in \SCR$ and $M$ is a \emph{connective} $A$-module.

\item \fbox{$\calP_\Sigma(\calC)$}: For a small $\infty$-category $\calC$ with finite coproducts, denote $\calP_\Sigma(\calC)$ the full subcategory of presheaves on $\calC$ consisting of all presheaves sending coproducts to products. See \cite[Definition 5.5.8.8]{lurie2009higher}, and the discussions in \cite[Section 5.5.8]{lurie2009higher}.

For example, if $\calC$ is the ordinary category of finite type polynomial rings over $\mathbb Z$, then $\calP_\Sigma(\calC) = \SCR$; if $\calC$ is the ordinary category of pairs $(P,F)$ where $P$ is a finite type polynomial ring over $\mathbb Z$ and $F$ is a finite free $P$-module, then $\calP_\Sigma(\calC) = \SCRModcn$. See \cite[Sect. 25.2]{lurie2018spectral} for a complete discussion.

\item \fbox{$\calC^\otimes, \CAlg(\calC), \Alg_\calO(\calC), \mathsf{Op}(\infty),\dE_n,\Fin_*$}: We use $\calC^\otimes$ to denote an $\infty$-operad (e.g. a symmetric monoidal $\infty$-category), $\CAlg(\calC)$ to denote the $\infty$-category of $\dE_\infty$-algebras in $\calC$, and $\Alg_\calO(\calC)$ to denote the $\infty$-category of algebras $\calO \to \calC$. We use $\mathsf{Op}(\infty)$ to denote the $(\infty,1)$-category of $\infty$-operads, and $\Fin_*$ to denote the ordinary category of pointed finite sets (i.e. $\Fin_* = \dE_\infty$). We use $\dE_n^\otimes / \Fin_*$ ($0 \le n \le \infty$) to denote the little cube $\infty$-operads defined in \cite[Definition 5.1.0.2]{lurie2017higher} and \cite[Definition 5.1.1.6]{lurie2017higher}, so $\dE_1$-algebras are the same as associative algebras, and $\dE_\infty$-algebras are the same as commutative algebras.

\item \fbox{$D(A), \cD(A)$}: $D(A)$ is the stable $\infty$-category of left $A$-modules, where $A$ is a connective $\dE_1$-ring, equipped with the natural $t$-structure (both left complete and right complete). See \cite[Proposition 7.1.1.13]{lurie2017higher} for details.

When $A$ is a simplicial commutative ring and we are given a finitely generated ideal $I \subseteq \pi_0(A)$, $\cD(A)$ will denote the $\infty$-category of derived $I$-adically complete modules over $A$ (usually $I = pA$). See Definition \ref{defi: derived complete}.

\item \fbox{$\Mod_A, \widehat\Mod_{A,\Flat}, \widehat\Mod_{A,\Flat,\fil}$}: When $A$ is a classical ring, $\Mod_A = D(A)^\heartsuit$ denotes the (ordinary) abelian category of $A$-modules. When $A$ is $p$-adically complete of bounded $p^\infty$-torsion, $\widehat\Mod_{A,\Flat}$ denote the ordinary category of $p$-complete and $p$-completely flat $A$-modules, and $\widehat\Mod_{A,\Flat,\fil}$ denotes the full subcategory of $\Fil(\cD(A))$ all of whose graded pieces lie in $\widehat\Mod_{A,\Flat}$.

\item \fbox{$\lim,\colim, \ccolim$}: limits and colimits. Unless other wise stated, we will always use them in the $\infty$-categorical sense (e.g. in the $\infty$-category of spaces, $\infty$-categories, or spectra). For clarity, we use $\ccolim$ to denote colimits in $\cD(A)$, and $\colim$ to denote colimits in $D(A)$.

\item \fbox{$\oplus,\fib,\cofib$}: direct sums, fibres and cofibres in a stable $\infty$-category.

\item \fbox{$\clim n, \ker, \coker$}: We reserve the notations $\ker, \coker$ for kernels and cokernels in abelian categories, and $\clim n = \hol^n(\lim)$ will denote the $n$th cohomology of the limit taken in the stable $\infty$-category. For example, for abelian groups $(M_i)$, $\clim 0 M_i$ will denote the classical limit of $M_i$.

\item \fbox{$\widehat{(-)}, (-)^\land$}: Unless otherwise stated, we will always use $\widehat{(-)}$ for derived completion and $(-)^\land$ for classical completion. The ideal of completion is usually clear from context.

\item \fbox{$\otimes, \cotimes, \otimes^\bL, \cotimes^\bL$}: classical and derived tensor products, uncompleted or completed. The operator $\cotimes$ is defined in Corollary \ref{cor: tensor conc on degree 0} and Proposition-Definition \ref{prop-defi: base change of adic modules}. We will not use it in any other case.

\item \fbox{$\Sym^*, \bigwedge^*, \Gamma^*, \widehat{\Sym}^*, \cbigwedge^*, \cGamma^*$}: (uncompleted and completed) symmetric powers, exterior powers, PD powers in the derived sense, defined on $\SCRModcn$ see Sect \ref{sect: remarks on infty-cat}. The readers should always keep in their minds of Lemma \ref{lemma: concentrated on degree zero}.

\item \fbox{$\dL_{-/-}, \cdL_{-/-}, \Omega_{-/-}, \cOmega_{-/-}$}: We use $\dL_{B/A}$ to denote the (uncompleted) cotangent complex of simplicial commutative rings, and $\cdL_{B/A}$ its completed version (see Definition \ref{defi: complete cotangent complex}). When and only when $\dL_{B/A}$ is flat over $B$ (resp. $B$ is of bounded $p^\infty$-torsion and $\cdL_{B/A}$ is $p$-completely flat over $B$), we write $\Omega_{B/A}$ for $\cdL_{B/A}$ (resp. $\cOmega_{B/A}$ for $\cdL_{B/A}$). See the discussions in Notation \ref{notn: cOmega}.
\end{itemize}
\vspace{0.3em}
\noindent\textbf{\scshape{Miscellanies}}

\begin{itemize}
\item \fbox{$\Sigma_n$}: the symmetric group on the finite set of $n$ elements.

\item \fbox{$\gamma_n$}: PD powers in a PD algebra, see for example \cite[\href{https://stacks.math.columbia.edu/tag/07GU}{Tag 07GU}]{stacks-project}.

\item \fbox{$X^G$}: Let $G$ be a group, $X$ be a $G$-set, and then $X^G$ denotes the subset of $X$ consisting of $G$-fixed points. This notation is always in the classical sense, even for abelian groups. In other words, for a $G$-module $M$, $M^G$ will always take the meaning of $\displaystyle \clims 0 G M$.

\item \fbox{$M[p^n],M[p^\infty]$}: For an abelian group $M$, we use $M[p^n]$ to denote the $p^n$-torsion part of $M$, and $M[p^\infty]$ to denote the union of all $M[p^n]$. Hopefully this would not cause any confusion with cohomological shiftings.

\item \fbox{$\DR(M,\nabla)$}: the de Rham complex. We will only use this notation in the case where $R$ is a classical ring (resp. $p$-complete classical ring of bounded $p^\infty$-torsion), $M$ is a classical module over $R$ (resp. $M$ is $p$-complete of bounded $p^\infty$-torsion over $R$, or $M$ is adic over $R \bigl[ \frac 1 p \bigr]$), and $\nabla: M \to M \otimes_R \Omega$ (resp. $\nabla: M \to M \cotimes_R \Omega$) is a Higgs field, where $\Omega$ is flat (resp. $p$-complete and $p$-completely flat) over $R$. See Definition \ref{defi: Higgs field} for the $p$-complete case (we won't define this for the uncompleted case, as it is classical, and the definition is virtually the same as Definition \ref{defi: Higgs field}).

Unless otherwise stated, we will always use $\DR(M,\nabla)$ for the $p$-complete case.

\item \fbox{$\Vect,\Strat,\MIC,\Higgs,\MIC^\tn,\Higgs^\tn$}: see Definition \ref{defi: crystal} and Notation \ref{notn: Strat and MIC}. For any (classical) ring $R$, we also use $\Vect(R)$ to denote the category of finite projective $R$-modules.

We will always use $\MIC,\Higgs, \MIC^\tn,\Higgs^\tn$ for the $p$-complete case. Note that $\Strat (S^\bullet)$ is the same as $\Vect(\simp,S^\bullet)$, so there is no need to differ between $p$-complete case or uncompleted case for $\Strat$.
\end{itemize}

\section*{Remarks on set theoretic issues}
We will use big categories such as $(R/A)_\Prism$ and `calculate' cohomologies on it. This is usually a trouble. However, in our case, this will cause no harm. Readers who care about this may constantly keep in their mind of a strong limit cardinal $\kappa$, and only consider $\kappa$-small rings. Since we only use countable limits and colimits in constructing rings (and we have the Hodge--Tate comparison theorem Proposition \ref{prop: HT comparison} to control the size of $\Prism_{R/A}$), our constructions will never exceed the bound of $\kappa$. Our theorems (e.g. Theorem \ref{theo: main dR realization revisit}) will show that the categories and cohomologies we are interested in also does not depend on $\kappa$, so taking a truncation by a cardinal causes no trouble.

\newpage

\section{Preliminaries on $\infty$-category theory}\label{sect: remarks on infty-cat}
In this article, we use \cite{lurie2009higher}, \cite{lurie2017higher} and \cite{lurie2018spectral} as our reference for $\infty$-category theory.

We will often speak of exactness and $t$-exactness of functors between stable $\infty$-categories. We give the definitions as follows:
\begin{defi}\label{defi: exactness and t-exactness}
Let $\calC, \calD$ be stable $\infty$-categories and $F: \calC \to \calD$ be a functor. Then,
\begin{enumerate}[label=(\arabic*)]
\item $F$ is called \emph{exact} if it preserves the zero object and exact triangles. By \cite[Proposition 1.1.4.1]{lurie2017higher}, this holds if and only if $F$ preserves finite limits, if and only if $F$ preserves finite colimits;

\item assume that $\calC$ and $\calD$ are equipped with $t$-structures and $F$ is exact. Then, we call $F$ \emph{right $t$-exact} if $F$ sends $\calC^{\le 0}$ to $\calD^{\le 0}$; we call \emph{left $t$-exact} if $F$ sends $\calC^{\ge 0}$ to $\calD^{\ge 0}$. We call $F$ \emph{$t$-exact} if it is both right $t$-exact and left $t$-exact.
\end{enumerate}
\end{defi}
\vspace{0.3em}
\noindent\textbf{\scshape{Filtered and graded objects and symmetric monoidal $\infty$-categories}}

In this article, unless otherwise stated, whenever we mention a `filtered object', we mean that the object has an increasing filtration by $\mathbb Z_{\ge 0}$, and the filtration is exhaustive. Similarly, a graded object is always graded by $\mathbb Z_{\ge 0}$ unless otherwise stated. More precisely, for any stable $\infty$-category $\calC$ with countable colimits, we define its filtered objects as
\[ \Fil(\calC) = \Fun(\mathbb Z_{\ge 0}, \calC), \quad \gr(\calC) = \Fun(\mathbb Z_{\ge 0, \discrete}, \calC). \]
We have functors of underlying objects:
\begin{align*}
\Fil(\calC) \to \calC, & \quad (X_i)_{i \ge 0} \mapsto \colim_{i \ge 0} X_i, \\
\gr(\calC) \to \calC,  &\quad (X_i)_{i \ge 0} \mapsto \bigsqcup_{i \ge 0} X_i.
\end{align*}
Whenever $\calC$ is symmetric monoidal, such that $- \otimes -$ commutes with countable colimits on each component, we equip $\Fil(\calC)$ and $\gr(\calC)$ with symmetric monoidal structures given by the Day convolution (see \cite[Section 2.2.6]{lurie2017higher}). The operations are (informally) given by
\begin{align*}
\Fil(\calC) \times \Fil(\calC) \to \Fil(\calC), & \quad ((X_i),(Y_i)) \mapsto \Bigl(\colim_{j+k \le i} X_j \otimes Y_k\Bigr), \\[3pt]
\gr(\calC) \times \gr(\calC) \to \gr(\calC), & \quad ((X_i),(Y_i)) \mapsto \left(\bigoplus_{j+k=i} X_i \otimes Y_k\right)
\end{align*}
We have the functor of taking graded pieces (where $X_{-1} :=0$):
\[ \gr_*: \Fil(\calC) \to \gr(\calC), \quad (X_i) \mapsto (\cofib(X_{i-1} \to X_i)), \]
and it can be upgraded to a symmetric monoidal functor\footnote{Let $\mathbf 1$ be the unit object of $\calC$. Let $\mathbf 1^\otimes: \Fin_* \to \calC^\otimes$ be the trivial algebra over $\calC$, and take the \emph{right} Kan extension relative to $\Fin_* \xrightarrow 0 \mathbb Z_{\ge 0}^\otimes \to \Fin_*$ to get an algebra $\mathbf 1_0: \mathbb Z_{\ge 0}^\otimes \to \calC^\otimes$. On object we have $\mathbf 1_0 \cong (\mathbf 1 \to 0 \to 0 \to \cdots)$. Now, $\gr_*^\otimes$ is the composition of the tensor product $\Fil(\calC)^\otimes \to \Mod_{\mathbf 1_0} (\Fil(\calC)^\otimes)$ and the forgetful functors $\Mod_{\mathbf 1_0} (\Fil(\calC)^\otimes) \to \Fil(\calC)^\otimes \to \gr(\calC)^\otimes$. It is lax monoidal, and can be verified to be symmetric monoidal by hand.}
\[ \gr_*^\otimes: \Fil(\calC)^\otimes \to \gr(\calC)^\otimes. \]
On objects, it coincides with $\gr_*$ as above, and its symmetric monoidal structure is given by
\[ \bigoplus_{j+k=i} \cofib(X_{i-1} \to X_i) \otimes \cofib(Y_{k-1} \to Y_k) \to \cofib\Bigl(\colim_{j+k\le i-1} X_j \otimes Y_k \to \colim_{j+k \le i} X_j \otimes Y_k\Bigr). \]
Also, it is easy to see that $\gr_*$ is conservative (i.e. detects isomorphisms), preserves and detects colimits. We will frequently use the three facts above without comments --- that $\gr_*$ is symmetric monoidal, is conservative, and preserves and detects colimits.

We would also like to mention that every graded object $\gr(\calC)$ can be viewed as a filtered object in $\Fil(\calC)$ via splitting:
\[ F: (X_i) \mapsto \left( \bigoplus_{j=0}^i X_j \right)_i, \]
more formally speaking is the left Kan extension of $\mathbb Z_{\ge 0,\discrete} \to \calC$ along $\mathbb Z_{\ge 0,\discrete} \to \mathbb Z_{\ge 0}$. If $\calC$ is the underlying category of a symmetric monoidal category $\calC^\otimes$ as above ($\calC$ admits countable colimits and $- \otimes -$ commutes with countable colimits on each component), then the map above can be upgraded to a symmetric monoidal functor
\[ F^\otimes: \gr(\calC)^\otimes \to \Fil(\calC)^\otimes \]
as the left adjoint of the forgetful map $\Fil(\calC)^\otimes \to \gr(\calC)^\otimes$. (One can use the formalism of free algebras \cite[Section 3.1.3]{lurie2017higher} to explicitly describe the left adjoint.) We have a natural isomorphism $\id \to \gr \circ F$, which can be upgraded to an isomorphism of $\infty$-operad maps\footnote{Let $U: \Fil(\calC)^\otimes \to \gr(\calC)^\otimes$ denote the forgetful functor, and we have a canonical isomorphism of algebras $\mathbf 1_{\gr(\calC)} \to U(\mathbf 1_0)$ (since $\mathbf 1_{\gr(\calC)}$ is the initial object in the algebras of $\gr(\calC)$), which extends to a map $X \to UF(X) \xrightarrow\cong U(\mathbf 1_0) \otimes UF(X) \to U(\mathbf 1_0 \otimes F(X)) = \gr_*(F(X))$ (note that all maps here can be upgraded to a map of $\infty$-operad maps).

Also, consider the map $V:\gr(\calC)^\otimes \to \Mod_{\mathbf 1_0}(\Fil(\calC)^\otimes), X \mapsto \mathbf 1_0 \otimes F(X)$ and the map $V': \Mod_{\mathbf 1_0}(\Fil(\calC)^\otimes) \to \gr(\calC)^\otimes, Y \mapsto U(Y)$. We already constructed an isomorphism $\id \xrightarrow\cong V' V$. An isomorphism $VV' \xrightarrow\cong \id$ may be given as $VV'(Y) = \mathbf 1_0 \otimes FU(Y) \to \mathbf 1_0 \otimes Y \to Y$, so we get an equivalence $\gr(\calC)^\otimes \cong \Mod_{\mathbf 1_0}(\Fil(\calC)^\otimes)$.}.

\begin{ex}[Relative tensor products]\label{ex: relative tensor products}
We will sometimes use the relative tensor products in symmetric monoidal $\infty$-categories. For its definition, see \cite[Section 4.4]{lurie2017higher} and \cite[Section 5.1.3]{lurie2017higher}.

We will frequently run into the following fact when using relative tensor products. Let $\calC^\otimes$ be a symmetric monoidal $\infty$-category such that the underlying category $\calC$ admits geometric realisations (i.e. colimits over $\simp^\op$), and $- \otimes -$ commutes with geometric realisations on each variable. Then, \cite[Theorem 4.4.2.8]{lurie2017higher} and \cite[Theorem 5.1.3.2]{lurie2017higher} shows that for any commutative algebra $A \in \CAlg(\calC)$ and $A$-modules $M,N$, the relative tensor product $M \otimes_A N$ is calculated as the geometric realisation of the diagram
\[ \begin{tikzcd} \cdots \ar[r] \ar[r,shift left=2] \ar[r,shift right=2] & M \otimes A \otimes N \ar[r,shift left=1] \ar[r,shift right=1] & M \otimes N. \end{tikzcd} \]
Therefore, for $\calC^\otimes$ and $\calD^\otimes$ as above, if $F: \calC^\otimes \to \calD^\otimes$ is symmetric monoidal and commutes with geometric realisations (on the underlying category), then $F$ preserves relative tensor products.

For example, if $\calC^\otimes$ is a symmetric monoidal category whose underlying category is a presentable stable $\infty$-category, and $- \otimes -$ commutes with colimits on each factors, then the functors
\[ \gr_*: \Fil(\calC)^\otimes \to \gr(\calC)^\otimes, \quad U: \Fil(\calC)^\otimes \to \calC^\otimes, \quad U': \gr(\calC)^\otimes \to \calC^\otimes \]
(where $U$ and $U'$ are the functors of taking underlying objects) all commute with relative tensor products. We will use this fact without further explanation, especially in case of $U$ and $U'$ --- the readers may view this as naturally equipping $M \otimes_A N$ with a filtered structure (resp. graded structure).
\end{ex}

\begin{ex}\label{ex: shearing}
Let $\calC^\otimes$ be $\mathbb Z$-linear (i.e. we are given a symmetric monoidal map $D^\perf(\mathbb Z)^\otimes \to \calC^\otimes$) and presentable. For any integer $n$, consider the map $\gr(\calC) \to \gr(\calC)$ given by
\[ (X_i) \mapsto (X_i[2ni]). \]
It can be upgraded to a symmetric monoidal equivalence in the following way: consider the symmetric monoidal functor $A_m: \mathbb Z_{\ge 0,\discrete} \to D(\mathbb Z)$ given by
\[ i \mapsto \mathbb Z[2ni]. \]
For any $\infty$-operad $\calO^\otimes$, consider the map $\Alg_{\calO \times \mathbb Z_{\ge 0, \discrete}}(\calC) \to \Alg_{\calO \times \mathbb Z_{\ge 0, \discrete}}(\calC)$ given by tensor product with $p_2^* A_m$ over $\mathbb Z$. This and the universal property of Day convolutions gives the desired map $F_n: (M_i) \mapsto (M_i [2ni])$. Take the universal case $\calO^\otimes = \gr(\calC)^\otimes$, using explicit descriptions of objects and morphisms in $\gr(\calC)^\otimes$ (see \cite[Propositions 2.2.6.4 and 2.2.6.6]{lurie2017higher}) and we see that the tensor product is given just as expected
\[ F_n((M_i) \otimes (N_i)) = \left( \bigoplus_{i_1+i_2=i} M_{i_1} \otimes N_{i_2} [2ni] \right)_i \cong F_n(M_i) \otimes F_n(N_i) \]
For $m,n \in \mathbb Z$, we have natural identifications $A_m \otimes A_n \cong A_{m+n}$. If we redefine $F_n$ as the $n$th power of $F_1$, then we get a $\dE_1$-monoidal functor $\mathbb Z^\otimes_\discrete \to \Aut_{\mathsf{Op}(\infty)}(\gr(\calC)^\otimes)^\otimes$ given by $n \mapsto F_n = F_1^n$ (induced by the functor $B\mathbb Z \to \mathsf{Op}(\infty)$ sending the only vertex to $\gr(\calC)^\otimes$ and the generating automorphism to $F_1$ --- note that the Kan complex $B\mathbb Z \cong S^1$ classifies automorphisms).

Similarly, if $\calC^\otimes$ is $A$-linear for some classical commtative ring $A$, and $J$ be an invertible $A$-module (concentrated on degree 0), then we have the symmetric monoidal functor $J^*$ on $\gr(D(A))^\otimes$ given by
\[ (M_n) \mapsto (J^n \otimes_A^\bL M_n). \]
Observe that $J^*$ is an automorphism of $\gr(D(A))^\otimes$ with inverse given by $J_* := (J^{-1})^*$, so it preserves colimits, and preserves relative tensor products by Example \ref{ex: relative tensor products}.
\end{ex}

\begin{ex}\label{ex: symmetric monoidal of truncations}
Let $A$ be a connective $\dE_\infty$-ring, then the $\infty$-category of $A$-module spectra $D(A)$ is equipped with a $t$-structure (see Proposition \ref{prop: t-structure on D(R)}). Using the explicit formula of Day convolutions, the full subcategory $\calC \subseteq \Fun(\mathbb Z, D(A))$ consisting of objects $(X_n)$ where $X_n$ is $(-n)$-connective is stable under colimits and tensor products. The inclusion functor $\calC \to \Fun(\mathbb Z, D(A))$ admits a right adjoint $(X_n) \mapsto (\tau^{\le n}X_n)$, and is automatically lax monoidal. We can also restrict ourselves to $\Fil(D^{\ge 0}(A)) = \Fun(\mathbb Z_{\ge 0}, D^{\ge 0}(A))$ with $\calC_0$ consisting of objects $(X_n)$ for which each $X_n$ is $(-n)$-connective. In particular, the composition
\[ D^{\ge 0}(A) \xrightarrow{\text{diagonal}} \Fun(\mathbb Z_{\ge 0}, D^{\ge 0}(A)) \to \calC_0 \]
is lax monoidal, and for each $X \in \CAlg(D^{\ge 0}(A))$, we get a filtered algebra $(\tau^{\le n} X)$. The graded pieces of $(\tau^{\le n} X)$ is exactly the shifting of the cohomology ring of $X$.

For example, whenever $(A,I)$ is a bounded prism and $R$ is $p$-complete and $p$-completely smooth over $\overline A = A/I$, then the Hodge--Tate filtration on $\overline\Prism_{R/A} = R\Gamma((R/A)_\Prism, \overline \calO)$ is given by $\bigl(\tau^{\le n} R\Gamma((R/A)_\Prism, \overline \calO)\bigr)$. See \cite[Theorem 4.11]{Bhatt_2022} and \cite[Construction 7.6]{Bhatt_2022}.
\end{ex}
\vspace{0.3em}
\noindent\textbf{\scshape{Simplicial rings and $\dE_\infty$-rings}}

Our notations will take a classical style. For any simplicial ring $A$, we use $D(A)$ to denote the symmetric monoidal $\infty$-category of $A$-modules, and $\cD(A)$ the full subcategory of complete modules (see Definition \ref{defi: derived complete}). If $A$ is a classical ring (i.e. $0$-truncated), we denote $\Mod_A = D(A)^\heartsuit$ the category of classical modules over $A$. For a $p$-complete classical ring $A$ of bounded $p^\infty$-torsion, we denote $\widehat\Mod_{A,\Flat}$ the category of $p$-complete and $p$-completely flat modules over $A$, and $\widehat\Mod_{A,\Flat,\fil}$ the full subcategory of $\Fil(\cD(A))$ all of whose graded pieces lie in $\widehat\Mod_{A,\Flat}$. The category $\widehat\Mod_{A,\Flat,\fil}$ admits a symmetric monoidal structure given by
\[ (X_i) \cotimes_A^\bL (Y_j) = \Bigl( \ccolim_{j+k \le i} X_j \cotimes_A^\bL Y_k \Bigr) = \Bigl( \colim_{j+k \le i} X_j \cotimes_A^\bL Y_k \Bigr), \]
inherited from $\cD(A)$ (this follows from Corollary \ref{cor: tensor conc on degree 0}).

We also discuss flatness conditions of modules over a connective $\dE_\infty$-ring. We will use the following propositions without further comments:

\begin{prop}[{\cite[Proposition 7.1.1.13]{lurie2017higher}}]\label{prop: t-structure on D(R)}
Let $R$ be a connective $\dE_1$-ring, then on the category $D(R)$ of left $R$-module spectra:
\begin{enumerate}[label=(\arabic*)]
\item the full subcategories $D^{\le 0}(R)$ and $D^{\ge 0}(R)$ of connective and coconnective module spectra give an accessible $t$-structure on the stable $\infty$-category $D(R)$;
\item the category $D^{\le 0}(R)$ is generated by $R$ under colimits and extensions;
\item the $t$-structure is both left complete and right complete, and taking $\pi_0$ induces an equivalence $D^\heartsuit(R) \to \Mod_{\pi_0(R)}$ from the category of left $R$-module spectra to classical left $\pi_0(R)$-modules;
\item the full subcategories $D^{\le 0}(R),D^{\ge 0}(R) \subseteq D(R)$ are stable under small products and small filtered colimits.
\end{enumerate}
\end{prop}

\begin{defi}[{\cite[Definition 7.2.2.10]{lurie2017higher} and \cite[Definition B.6.1.1]{lurie2018spectral}}]\label{defi: flatness}
Let $R$ be a connective $\dE_\infty$-ring and $M$ be an $R$-module. We say that $M$ is flat (resp. faithfully flat) over $R$ if and only if
\begin{enumerate}[label=(\arabic*)]
\item $\pi_0(M)$ is flat (resp. faithfully flat) over $\pi_0(R)$,
\item for all $n$, the map $\pi_0(M) \otimes_{\pi_0(R)} \pi_n(R) \to \pi_n(M)$ are isomorphisms.
\end{enumerate}

It is clear that if $M$ is flat over $R$, then $M$ is also connective.
\end{defi}

\begin{prop}\label{prop: flatness as t-exactness}
For any connective $\dE_\infty$-ring $R$ and $R$-module $M$,
\begin{enumerate}[label=(\arabic*)]
\item $M$ is flat over $R$ if and only if $M \otimes_R^\bL -$ is $t$-exact, if and only if $M$ is connective and $M \otimes_R^\bL \pi_0(R)$ is flat over $\pi_0(R)$,
\item $M$ is faithfully flat over $R$ if and only if $M \otimes_R^\bL -$ is $t$-exact and conservative, if and only if $M$ is connective and $M \otimes_R^\bL \pi_0(R)$ is faithfully flat over $\pi_0(R)$.
\end{enumerate}
\end{prop}

\begin{proof}
It is clear that $M \otimes_R^\bL -$ is right $t$-exact (i.e. sends $D^{\le 0} (R)$ to $D^{\le 0} (R)$) if and only if $M$ is connective. Therefore, we may assume that $M$ is connective in the entire process. Now, (1) is a direct consequence of \cite[Theorem 7.2.2.15]{lurie2017higher}, and (2) is a direct consequence of (1) and \cite[Theorem 7.2.2.13]{lurie2017higher}.
\end{proof}

Therefore, we define $\Tor$-amplitude in a similar manner:

\begin{defi}\label{defi: Tor amplitude}
Let $R$ be a connective $\dE_\infty$-ring and $M$ be a $R$-module. For $-\infty \le a \le b \le \infty$ (we do not allow $a=b=\infty$ or $a=b=-\infty$), we say that \emph{$M$ is of $\Tor$-amplitude in $[a,b]$ over $R$} if and only if $M \otimes_R^\bL -$ sends $D^{\le 0}(R)$ into $D^{\le b}(R)$ and $D^{\ge 0}(R)$ into $D^{\ge a}(R)$. This holds if and only if $M \in D^{\le b} (R)$ and $M \otimes_R^\bL \pi_0(R)$ is of $\Tor$-amplitude in $[a,b]$ over $\pi_0(R)$ (Recall that $N \in D^{\ge 0} (R)$ is the filtered colimit of $\tau^{\le n} N$).
\end{defi}

It is clear that for $M \in D^{\le b}(R)$, $M$ is of $\Tor$-amplitude in $[a,b]$ if and only if for any $-\infty \le c \le d \le \infty$ and $N \in D^{[c,d]}(R)$, we have $M \otimes_R^\bL N \in D^{[c+a,d+b]}(R)$.

\begin{rmk}
In the definition above, if $R$ is $n$-truncated for some $n \in \mathbb Z_+$, then $D(R)$ is generated by perfect complexes under filtered colimits, hence $D(R)$ is generated by bounded objects under filtered colimits. Taking the truncations, we see that $D^{\le 0}(R)$ and $D^{\ge 0}(R)$ are generated by bounded objects under filtered colimits. Therefore, we see that $M$ is flat (resp. faithfully flat, of $\Tor$-amplitude in $[a,b]$) over $R$ if and only if $M \otimes_R^\bL \pi_0(R)$ is flat (resp. faithfully flat, of $\Tor$-amplitude in $[a,b]$) over $\pi_0(R)$.
\end{rmk}

We will also make use of the following homological algebra result:
\begin{prop}\label{prop: totalisation commute with filtered colimits for coconnective spectra}
In the category of coconnective spectra $\Sptr_{\le 0}$, filtered colimits commute with totalisations (i.e. limits over $\simp$).

As a consequence, for any simplicial ring $A$ and finitely generated ideal $I \subseteq \pi_0(A)$, if we are given a diagram $F: \mathcal I \times \simp \to \cD(A)$ (where $\mathcal I$ is a filtered $\infty$-category) whose image lies in some $\cD^{\ge b}(A)$ (where $b>-\infty$), then the natural morphism in $\cD(A)$
\[ \ccolim_{\mathcal I} \lim_\simp F\to \lim_\simp \ccolim_{\mathcal I} F \]
is an isomorphism. (Note that $\ccolim_{\mathcal I}$ may not be left exact, so the colimits and limits must be calculated in $\cD(A)$ instead of $\cD^{\ge b}(A)$.)
\end{prop}

\begin{proof}
Denote $\calS = \Sptr^\op$. We use the $\infty$-categorical Dold-Kan correspondence for $\calS$ established in \cite[Sect. 1.2.4]{lurie2017higher}. We have a Dold-Kan correspondence
\[ \Fun(\simp,\Sptr) \cong \Fun(\mathbb Z_{\ge 0}^\op, \Sptr). \]
For $X^\bullet \in \Fun(\simp, \Sptr)$, by \cite[Proposition 1.2.4.5]{lurie2017higher}, for $X^\bullet \in \Fun(\simp,\Sptr_{\le 0})$ and the corresponding object
\[ \cdots \to D(3) \to D(2) \to D(1) \to D(0) \in \Fun(\mathbb Z_{\ge 0}^\op, \Sptr) \]
we have a natural isomorphism $\pi_{-n}(\lim_\simp X^\bullet) \cong \pi_{-n} (D(n))$. Note that $D(n)$ is the limit of the $n$th skeleton of $X^\bullet$, so $D(n)$ is a finite limit. Therefore the functors $D(n): \Fun(\simp,\Sptr_{\le 0}) \to \Sptr$ commute with filtered colimits. It follows that the functors $X^\bullet \mapsto \pi_{-n}(\lim_\simp X^\bullet)$ commute with filtered colimits. Since $\Sptr$ is right complete, we conclude that filtered colimits commute with totalisations in $\Sptr_{\le 0}$ (note that the full subcategory $\Sptr_{\le 0} \subseteq \Sptr$ is stable under arbitrary limits and filtered colimits).

For the second part, we have an isomorphism in $D(A)$
\[ \colim_{\mathcal I} \lim_\simp F\to \lim_\simp \colim_{\mathcal I} F. \]
Take the derived completion (note that derived completion commutes with all colimits and limits by Lemma \ref{lem: derived completion presentability}), and we obtain the desired isomorphism.
\end{proof}

\begin{defi}\label{defi: compatible with totalisation}
In a stable $\infty$-category $\calC$ admitting totalisations and sequential colimits (i.e. colimits over $\mathbb Z_{\ge 0}$), we say that a cosimplicial object $X^\bullet \in \gr(\calC)$ is compatible with totalisation, if the map
\[ F(\lim_\simp X_i) \to \lim_\simp F(X_i) \]
is an isomorphism, where $F$ is the functor of taking the underlying object. We say that $X^\bullet \in \Fil(\calC)$ is compatible with totalisation, if $\gr_* X^\bullet$ is compatible with totalisation, and
\[ G(\lim_\simp X_i) \to \lim_\simp G(X_i) \]
is an isomorphism, where $G$ is the functor of taking the underlying objects.

Therefore, for $X^\bullet \in \gr(\calC)$ compatible with totalisation, we may safely view $\lim_{\simp,\gr(\calC)} X^\bullet$ as giving a grading of $\lim_{\simp,\calC} X^\bullet$; for $X^\bullet \in \Fil(\calC)$ compatible with totalisation, we may safely view $\lim_{\simp,\Fil(\calC)} X^\bullet$ as giving a filtration of $\lim_{\simp,\calC} X^\bullet$, and its graded pieces are $\lim_{\simp,\gr(\calC)} \gr_* X^\bullet$ with underlying module $\lim_{\simp,\calC} \gr_* X^\bullet$. If $X^\bullet$ is compatible with totalisation, then we may safely talk about $\lim_\simp X^\bullet$ and $\lim_\simp \gr_* X^\bullet$ without ambiguity of the category in which the limit is taken.

Therefore, by Proposition \ref{prop: totalisation commute with filtered colimits for coconnective spectra}, in the case $\calC = D(A), \cD(A)$, we see that if there exists an integer $b$ with all terms of $X^\bullet$ being $b$-truncated, then $X^\bullet$ is compatible with totalisation.
\end{defi}

\vspace{0.3em}
\noindent\textbf{\scshape{Symmetric powers}}

We now turn to discuss derived symmetric powers. We refer the readers to \cite[Sect. 25.2]{lurie2018spectral} for a complete discussion. Denote $\SCR$ the category of simplicial rings. The category $\SCRModcn$ consists of pairs $(A,M)$ where $A$ is a simplicial commutative ring and $M$ is a connective $A$-module. Denote $\calC \subseteq \SCRModcn$ be the full subcategory spanned by $(A,M)$ where $A \cong \mathbb Z[x_1,\ldots,x_d]$ for some $d \ge 0$ and $M$ is finite free over $A$, then $\calC$ is an ordinary category and $\SCRModcn \cong \calP_\Sigma(\calC)$. We have the functors $\Sym^n, \bigwedge\nolimits^n, \Gamma^n : \SCRModcn \to \SCRModcn$ given by the left Kan extensions of the functors on $\calC$ given as follows:
\begin{align*}
(A,M) & \mapsto (A,\Sym^n_A M), \\
(A,M) & \mapsto \left(A,\bigwedge\nolimits^n_A M \right), \\
(A,M) & \mapsto (A,\Gamma^n_A M).
\end{align*}%Moreover, let $\Mod(\gr(\Sptr))$ be the category consisting of pairs $(A,M)$ where $A$ is an algebra of $\gr(\Sptr)$ and $M$ is an $A$-module.
These functors commute with sifted colimits, and it is easy to see that they commutes with base change (e.g. $\Sym_{A'}^n (A' \otimes_A^\bL M) = A' \otimes_A^\bL \Sym_A^n(M)$). Therefore, When $R$ is discrete and $M$ is finite free, we see that $\Sym^n_R(M), \bigwedge^n_R(M), \Gamma^n_R(M)$ coincides with the classical definition. Take the filtered colimit, and we see that they coincide with the classical definition whenever $M$ is flat over (the discrete ring) $R$.

Consider the fibred product
\[ \grAlg = \SCR \times_{\Fun(\{0\},\CAlg(\gr(\Sptr)))} \Fun(\Delta^1, \CAlg(\gr(\Sptr))) \]
consisting of pairs $(A,M)$ where $A$ is a simplicial commutative ring and $M$ is a graded $\dE_\infty$-$A$-algebra. Since $\SCR \to \CAlg(\Sptr) \to \CAlg(\gr(\Sptr))$ commute with colimits, the projections $\grAlg \to \SCR$ and $\grAlg \to \CAlg(\gr(\Sptr))$ all commute with colimits. Then we have the functors $\Sym^*, \Gamma^*, \bigwedge\nolimits^*[*]: \SCRModcn \to \grAlg$ as above. It is easy to see that their restrictions on $\calC$ commutes with finite coproducts, so $\Sym^*$ and $\Gamma^*$ are functors commuting with arbitrary colimits. By \cite[Proposition 25.2.4.2]{lurie2018spectral}, we have the isomorphisms
\begin{align*}
\Sym^n_A (M[1]) & \cong \left( \bigwedge\nolimits_A^nM \right) [n], \\
\bigwedge\nolimits^n_A(M[1]) & \cong \Gamma_A^n(M) [n],
\end{align*}
such that the case for $n=0$ and $n=1$ agrees with the isomorphisms
\[ \Sym^0 \cong \Gamma^0 \cong \bigwedge\nolimits^0 \cong A, \quad \Sym^1 \cong \Gamma^1 \cong \bigwedge\nolimits^1 \cong M. \]
Moreover, we modify the argument of \cite[Proposition 25.2.4.2]{lurie2018spectral} as follows: for any discrete commutative ring $R$ and an exact sequence of finite free $R$-modules
\[ 0 \to M' \to M \to M'' \to 0, \]
we have an isomorphism of differential graded algebras over the category $\gr(\Mod_R)$ (i.e. each term of the complex admits a grading): (the Koszul complex)
\[ \Kos(\Sym_R^*(M);\Sym_R^*(M) \otimes_R M') \xrightarrow\cong \Sym^*_R(M''), \]
and the dualized version (we write it in the form of the de Rham complex)
\[ \DR(\Gamma_R^*(M);\Gamma_R^*(M) \otimes_R M'') \xleftarrow\cong \Gamma^*_R(M'). \]
Note that we also have comparison morphisms (ring homomorphisms, but not necessarily isomorphisms)
\begin{align*}
\Kos(\Sym_R^*(M);\Sym_R^*(M) \otimes_R M') & \to \Kos(R; M') \cong \bigwedge\nolimits^*(M')[*], \\
\DR(\Gamma_R^*(M);\Gamma_R^*(M) \otimes_R M'') & \leftarrow \DR(R;M'') \cong \bigwedge\nolimits^*(M'')[-*].
\end{align*}
Take the left Kan extension, we obtain natural morphisms in $\grAlg$ (where $F,G$ are derived functors of the Koszul complex and the de Rham complex respectively)
\begin{align*}
\Sym^*_R(M'') \xleftarrow\cong F & \to \bigwedge\nolimits^*(M')[*], \\
\Gamma^*_R(M') \xrightarrow\cong G & \leftarrow \bigwedge\nolimits^*(M'')[-*].
\end{align*}
By the argument the same as in \cite[Proposition 25.2.4.2]{lurie2018spectral}, the graded pieces of $F$ and $G$ admit filtrations (the stupid filtration in the discrete case, with graded pieces $\bigwedge\nolimits^i M' \otimes^\bL \Sym^{n-i} M$ and $\bigwedge\nolimits^i M \otimes^\bL \Gamma^{n-i} M''$ respectively) compatible with the stupid filtration on the right-hand side, and we see that the arrows on the right-hand side are isomorphism whenever $M=0$. Take $M=0$, and we obtain natural isomorphisms \emph{as functors $\SCRModcn \to \grAlg$ over $\SCR$} (it is crucial that the isomorphisms are isomorphisms of rings, because we will explicitly use the PD ring structure of $\Gamma^*(M) \cong \bigwedge\nolimits^*(M[1])[-*]$ in the future):
\[ \Sym^* (M'[1])\cong \bigwedge\nolimits^*(M')[*], \quad \Gamma^*(M') \cong \bigwedge\nolimits^*(M'[1])[-*]. \]
When we are working over a base $A \in \SCR$ and a finitely generated ideal $I \subseteq \pi_0(A)$, denote $\widehat\Sym^*, \cbigwedge^*[*], \cGamma^* : \SCRModcn \times_{\SCR} \SCR_{A/} \to \grAlg_A$ their completed versions, where
\[ \grAlg_A := \SCR_{A/} \times_{\Fun(\{0\},\CAlg(\gr(D(A))))} \Fun(\Delta^1, \CAlg(\gr(D(A)))) \]
consists of pairs $(B,B')$ where $B$ is a simplicial ring over $A$ and $B'$ is a graded $\dE_\infty$-algebra over $B$.

Now, $\Sym^*_A(M)$ admits a simplicial ring structure as the left adjoint of the forgetful functor $\Fun(\Delta^1,\SCR) \to \SCRModcn$, and therefore $\Sym_A^*(M) : \SCRModcn \to \Fun(\Delta^1,\SCR)$ preserves colimits. Compose with the colimit-preserving functor $\SCR \to \CAlg(\Sptr)$, and then note that taking underlying modules $\CAlg(\gr(\Sptr)) \to \CAlg(\Sptr)$ is conservative (i.e. detects isomorphisms), we see that the functor $\Sym^*: \SCRModcn \to \grAlg$ preserves colimits. By Example \ref{ex: shearing}, all functors
\[ \Sym^*, \bigwedge\nolimits^*[*], \Gamma^*: \SCRModcn \to \grAlg \]
and their completed versions
\[ \widehat\Sym^*, \cbigwedge^*[*], \cGamma^* : \SCRModcn \times_{\SCR} \SCR_{A/} \to \grAlg_A \]
preserve arbitrary colimits.

At last, the reader should always keep in their mind of Lemma \ref{lemma: concentrated on degree zero}. We are going to use it without further comment.

\vspace{0.3em}
\noindent\textbf{\scshape{Cotangent complexes}}

We give a quick overview of cotangent complexes. We have the functor $F$ of taking the trivial square-zero extension (\cite[Construction 25.3.1.1]{lurie2018spectral})
\[\begin{tikzcd}[row sep=large]
\SCRModcn \ar[rr,"F"] \ar[rd] & & \Fun(\Delta^1, \SCR) \ar[ld,"\text{evaluate at $0$}"] \\
& \SCR
\end{tikzcd}\]
given by $F(A,M) = (A \to A \oplus M)$, and its left adjoint is the construction of cotangent complexes
\[ (A \to B) \mapsto (B, \dL_{B/A}). \]
In particular, the map
\[ \dL_{-/-}: \Fun(\Delta^1, \SCR) \to \SCRModcn \]
commutes with arbitrary colimits. For example, we may compose it with the colimit preserving functor (by Example \ref{ex: shearing})
\[ \bigwedge\nolimits^*[-*]: \SCRModcn \to \grAlg, \]
and we get a colimit preserving functor $\Fun(\Delta^1, \SCR) \to \grAlg$ over $\SCR$:
\[ (A \to B) \mapsto \bigoplus_{n \ge 0} \bigwedge\nolimits^n_A(\dL_{B/A}). \]

\section{Preliminaries on $p$-complete rings}
\subsection{Completely flat modules and adic modules}
We refer the readers to \cite{Bhatt_2022} for $\delta$-rings and prisms.

\begin{defi}[{cf. \cite[Section 7.3]{lurie2018spectral}, \cite[Definition 2.31]{mathew2021faithfullyflatdescentperfect}}]\label{defi: derived complete}
Let $A$ be a simplicial ring and $I \subseteq \pi_0 \left( A \right)$ be a finitely generated ideal. Then, an object $M \in D \left( A \right)$ is called \emph{derived $I$-adically complete}, if and only if for each $x \in I$ we have
\[ \lim (\cdots \xrightarrow x M \xrightarrow x M) =0 \]
in $D(A)$. An object $M \in D(A)$ is called \emph{$I$-torsion} (resp. \emph{$I^\infty$-torsion}) if all cohomology groups $\hol^i(M)$ are $I$-torsion (resp. $I^\infty$-torsion). We use $\cD(A)$ to denote the full subcategory of derived $I$-adically complete modules over $A$, and $D^{\tors I}(A)$ (resp. $D^{\tors{I^\infty}} (A)$) the full subcategory of $I$-torsion (resp. $I^\infty$-torsion) modules over $A$. The inclusion functor $i: D \left( A \right) \to \cD \left( A \right)$ has a left adjoint, called the \emph{derived $I$-adic completion}. If $I = (x_1,\ldots,x_n)$, then the completion functor sends $M \in D \left( A \right)$ to
\[ \widehat M := \lim M \otimes_{\mathbb Z [x_1,\ldots,x_n]}^\bL \mathbb Z [x_1,\ldots,x_n] / \bigl( x_1^k, \ldots, x_n^k \bigr). \]
Note that if $M=B$ is a simplicial algebra over $A$, then $\widehat B$ is equipped with a natural structure of a simplicial ring, which we also call the \emph{derived $I$-adic completion} of $B$. Also, it is clear that the functor $M \mapsto \widehat M$ is right $t$-exact, i.e. it sends $D^{\le 0}$ to $D^{\le 0}$.

For $M,N \in D \left( A \right),$ we will denote the completion of $M \otimes_A^\bL N$ by $M \cotimes_A^\bL N.$ By \cite[Corollary 7.3.5.2]{lurie2018spectral}, $\cD(A)$ is equipped with a symmetric monoidal structure inherited from $D(A)$, and is given by $(M,N) \mapsto M \cotimes_A^\bL N$ on objects.
\end{defi}

\begin{lemma}\label{lem: derived completion presentability}
For a simplicial ring $A$ and a finitely generated ideal $I \subseteq \pi_0(A)$, the adjoint pair of functors (where $i$ denotes the inclusion functor)
\[ \widehat{(-)}: D(A) \rightleftarrows \cD(A) : i \]
realises $\cD(A)$ as an accessible localisation (see \cite[Definition 5.2.7.2]{lurie2009higher}) of $D(A)$. Moreover, the functor $\widehat{(-)}$ commutes with all colimits and limits, and $i$ commutes with all $\omega_1$-filtered colimits and limits.
\end{lemma}

\begin{proof}
Denote $L = i \circ \widehat{(-)} : D(A) \to D(A)$. We first prove that $L$ commutes with $\omega_1$-filtered colimits and all limits. Write
\[ L(X) = \lim_n X \otimes_{\mathbb Z[x_1,\ldots,x_d]}^\bL \mathbb Z[x_1,\ldots,x_d] / (x_1^n,\ldots,x_d^n). \]
Now, 
\[ \mathbb Z[x_1,\ldots,x_d] / (x_1^n,\ldots,x_d^n) \cong \Kos(\mathbb Z[x_1,\ldots,x_d]; x_1^n,\ldots,x_d^n) \]
is a perfect complex over $\mathbb Z[x_1,\ldots,x_d]$, so the functor $- \otimes_{\mathbb Z[x_1,\ldots,x_d]}^\bL \mathbb \mathbb Z[x_1,\ldots,x_d] / (x_1^n,\ldots,x_d^n)$ commutes with all colimits and all limits. Now $\lim_n$ commutes with all $\omega_1$-filtered colimits and all limits, and we see that $L$ commutes with $\omega_1$-filtered colimits and all limits. By \cite[Theorem 5.5.1.1 (5)]{lurie2009higher}, we see that $\cD(A)$ is a presentable $\infty$-category.

By adjunction, we see that $\widehat{(-)}$ preserves all colimits and $i$ preserves all limits. By \cite[Proposition 5.5.1.2]{lurie2017higher}, we see that $\cD(A)$ is presentable. Now, since the fully faithful functor $i$ detects limits, we see that $\widehat{(-)}$ commutes with all limits. To show that $i$ commutes with $\omega_1$-filtered colimits, let $F: \mathcal I \to \cD(A)$ be a diagram where $\mathcal I$ is $\omega_1$-filtered. Now since $\widehat{(-)} \circ i \cong \id$, we have $F \cong \widehat{(-)} \circ i \circ F$. We only need to show that the natural map
\[ \colim i \circ F \to i ( \colim F ) \]
is an isomorphism. Write $F \cong (-) \circ G$ (with $G = i \circ F$), since $(-)$ commutes with colimits, we only need to prove that the natural map
\[ \colim L \circ G \to L(\colim G) \]
is an isomorphism. This follows from that $L$ commutes with $\omega_1$-filtered colimits.
\end{proof}

\begin{lemma}\label{lem: t-structure on complete modules}
Let $A$ be a simplicial ring and $I \subseteq \pi_0(A)$ be a finitely generated ideal. Then,
\begin{enumerate}[label=(\arabic*)]
\item for each $x \in \pi_0(A)$ and $M \in D(A)$, denote $T_x(M)$ the limit
\[ \lim (\cdots \xrightarrow x M \xrightarrow x M). \]
Then, $T_x$ has $t$-amplitude in $[0,1]$, and there exists exact sequences
\[ 0 \to \hol^1(T_x(\hol^{n-1}(M))) \to \hol^0(T_x(M)) \to \hol^0(T_x(\hol^n(M))) \to 0. \]
In particular, $M$ is derived $I$-adically complete if and only if each $\hol^i(M)$ is derived $I$-adically complete;
\item for $M \in \cD(A)$ and $n \in \mathbb Z$, the complexes $\tau^{\le n}M$ and $\tau^{\ge n}M$ also lie in $\cD(A)$. In particular, the full subcategories $\cD^{\le 0}(A) := D^{\le 0}(A) \cap \cD(A)$ and $\cD^{\ge 0}(A) := D^{\ge 0}(A) \cap \cD(A)$ defines an accessible $t$-structure on $\cD(A)$.
\end{enumerate}
\end{lemma}

\begin{proof}
For (1), we may assume that $n=0$. Note that the limit over $\mathbb Z_{\ge0}^\op$ may be depicted as the following equaliser:
\[ \lim(\cdots \xrightarrow{f_3} X_2 \xrightarrow{f_2} X_1 \xrightarrow{f_1} X_0) = \lim \biggl( \prod_{i=0}^\infty X_i \rightrightarrows \prod_{i=0}^\infty X_i \biggr) \]
where the two maps are given by $\id$ and $(f_{i+1} \text{pr}_{i+1})$. Therefore, the exact functor
\[ M \mapsto T_x(M) = \lim (\cdots \xrightarrow x M \xrightarrow x M) \]
sends $D^{\le 0}$ to $D^{\le 1}$ and $D^{\ge 0}$ to $D^{\ge 0}$. Therefore, take the truncation $\tau^{\le 0} \tau^{\ge -1} M$, and we may assume that $M \in D^{[-1,0]}(A)$. Now, we have an exact triangle
\[ \hol^{-1}(M) [1] \to M \to \hol^0(M), \]
inducing an exact sequence
\[ 0 \to \hol^1(T_x(\hol^{-1}(M))) \to \hol^0(T_x(M)) \to \hol^0(T_x(\hol^0(M))) \to 0, \]
as desired. The other statements are direct consequences of (1).
\end{proof}

When $A$ is a classical ring and $I=dA$, we have the following proposition:
\begin{prop}\label{prop: derived completion is classical when bounded}
Let $A$ be a ring and $I=dA.$ Let $M$ be a (classical) module over $A$ with bounded $d^\infty$-torsion (i.e. there exists $n_0$ with $M \left[ d^{n_0} \right] = M \left[ d^\infty \right]$), then the derived completion of $M$ with respect to $dI$ agrees with the classical completion. Moreover, the map $M[d^n] \to \widehat M[d^n]$ is an isomorphism for any $n \in \mathbb Z_+$. In particular $\widehat M[d^{n_0}] = \widehat M[d^\infty]$.
\end{prop}

\begin{proof}
We only need to prove that there is an isomorphism between pro-objects
\[ \bigl\{ \cofib \bigl( M \xrightarrow{d^n} M \big) \bigr\} \to \{ M / d^n M \}, \]
and then taking the derived limit would give us the result. Now, the fibre of this morphism is the pro-object
\[ \{ M \left[ d^n \right] \left[ -1 \right] \} \]
with connecting morphisms $d: M \left[ d^n \right] \to M\left[ d^{n+1} \right].$ By our assumptions, the $n_0$-fold composition $d^{n_0}: M \left[ d^n \right] \to M \left[ d^{n+n_0} \right]$ is always zero, so the fibre is isomorphic to zero, as desired.

To prove the last part of the theorem, replace $d$ by $d^n$, and we may assume that $n=1$. For $m \in M[d]$, if its image in $\widehat M[d]$ is zero, then $m \in d^iM$ for any $i \in \mathbb Z_+$. In particular $m=d^{n_0}m'$ for some $m' \in M$. We have $d^{n_0+1}m' = dm =0$, so our assumption shows that $m' \in M[d^{n_0}]$ and $m=d^{n_0}m' =0$, so $M[d] \to \widehat M[d]$ is injective.

To show that it is surjective, assume that $\left( m_i \right) \in \clim 0 M / d^i M$ is $d$-torsion. It follows that $dm_i \in d^iM$ for all $i$. Take $m'_i \in M$ with $dm_i = d^i m'_i$, and we see that $d(m_i - d^{i-1}m'_i) =0$ for $i \ge 1$. Take $n_i = m_{i+1} - d^i m'_{i+1}$, and we see that $n_i - m_i \in d^iM$, and $n_i \in M[d]$. Replace $(m_i)$ by $(n_i)$, and we may assume that $dm_i =0$ for all $i$. Fix $i \ge n_0$, and we may write $m_i - m_{i+1} = d^i x$. We have $d^{i+1} x =0$, so $d^{n_0} x =0$. It follows that $d^i x =0$, and we get $m_i = m_{i+1}$. Therefore, $(m_i)$ is eventually constant, and we see that $M[d] \to \widehat M[d]$ is surjective.
\end{proof}

We will also introduce another class of modules that we would like to consider. For the sake of convenience, we will only consider $p$-adic completeness.

\begin{prop-defi}\label{prop-defi: adic complete}\mbox{}
\begin{enumerate}[label=(\arabic*)]
\item A \emph{Tate ring} over $\mathbb Q_p$ is a topological ring $A$ over $\mathbb Q_p$, such that there exists an open subring $A^+ \subseteq A$ such that $A = A^+ \bigl[ \frac 1 p \bigr]$ and the induced topology on $A^+$ is $p$-adic and complete. Such ring $A^+$ is called a ring of definition for $A^+$. For any $p$-complete ring $A_0$ of bounded $p^\infty$-torsion, it is easy to see that $A_0 \bigl[ \frac 1 p \bigr]$ is a Tate ring over $\mathbb Q_p$, and a ring of definition is given by the $p$-torsion-free quotient of $A_0$. For subsets $X,Y$ of $A$, denote $X \cdot Y$ the \emph{subgroup} of $A$ generated by $xy$ for $x \in X$ and $y\in Y$,

A subset $X \subseteq A$ is called \emph{bounded}, if for any open neighbourhood $U$ of $0$, there exists an open neighbourhood $V$ of $0$ such that $X \cdot V \subseteq U$. It is clear that for bounded subsets $X,Y \subseteq A$, the set $X \cdot Y$ is bounded. If $A^+$ is a ring of definition of $A$, then $X$ is bounded if and only if there exists $N \in \mathbb Z$ with $X \subseteq p^N A^+$. As a result, if subgroups $M \subseteq N \subseteq A$ are open and bounded, then $N/M$ is killed by a (finite) power of $p$.

An element $a \in A$ is called power-bounded (resp. topologically nilpotent), if the set $\{ a^n \}_{n \in \mathbb N}$ is bounded (resp. tends to zero). The set of power-bounded elements and topologically nilpotent elements are denoted $A^\circ$ and $A^{\circ\circ}$.

\item Fix a Tate ring $A$ over $\mathbb Q_p$. Then, a subring $A^+$ of $A$ is a ring of definition if and only if $A^+$ is open and bounded. If $A^+_1$ and $A^+_2$ are rings of definition, then $A^+_0:=A^+_1 \cdot A^+_2$ is a ring of definition. If $A^+ \subseteq A$ is a ring of definition and $a \in A^\circ$, then $A^+ [a]$ is a ring of definition. Therefore, $A^\circ$ is the union of all rings of definition of $A$, so $A^\circ$ is an open subring of $A$.

The set of topologically nilpotent elements $A^{\circ\circ}$ is an open ideal of $A^\circ$ contained in the Jacobson radical.

\item Let $A$ be a Tate ring over $\mathbb Q_p$ with a ring of definition $A^+$. A topological $A$-module $M$ is called \emph{adic}, if there exists an open $A^+$-submodule $M^+ \subseteq M$ such that $M = M^+ \bigl[ \frac 1 p \bigr]$, and that the induced topology on $M^+$ is $p$-adic and complete. Then, this definition does not depend on the choice of $A^+$. Such an $M^+$ is called \emph{an $A^+$-module of definition}. When $A^+$ is clear from context, we simply call it a module of definition of $M$.

A subset $X$ of $M$ is called \emph{bounded}, if for any open neighbourhood $U \subseteq M$ of $0$, there exists an open neighbourhood $V \subseteq A$ of $0$ such that $V \cdot X \subseteq U$. Then, for any ring of definition $A^+ \subseteq A$, $X$ is bounded if and only if it is contained in some $A^+$-module of definition of $M$. Also, a $A^+$-module of definition is the same as an open bounded $A^+$-submodule of $M^+$.

Assume that $A_0$ is a classical $p$-complete ring of bounded $p^\infty$-torsion and $M_0$ is a $p$-complete $A_0$-module with bounded $p^\infty$-torsion, then $M_0 \bigl[ \frac 1 p \bigr]$ is an adic module over $A_0 \bigl[ \frac 1 p \bigr]$, whose module of definition is the image of $M_0$.
\end{enumerate}
\end{prop-defi}

\begin{proof}
For (1), if $A_0$ is a $p$-complete ring with bounded $p^\infty$-torsion, then denote its $p$-torsion-free quotient by $A^+$, then $A^+$ is also $p$-adically complete, and we see that $A_0 \bigl[ \frac 1 p \bigr] \cong A^+ \bigl[ \frac 1 p \bigr]$ can be made into a Tate ring over $\mathbb Q_p$ with ring-of-definition $A^+$.

If $X \subseteq A$ is bounded, then take $U = A^+$, and we see that there exists $N \in \mathbb Z$ with $X \cdot (p^N A^+) \subseteq A^+$. This implies $X \subseteq p^{-N}A^+$. Conversely if $X \subseteq p^N A^+$, it is easy to see that $X$ is bounded.

For (2), it is clear that a ring of definition is an open and bounded subring. Conversely, if $A^+ \subseteq A$ is open and bounded, then $p^M A^+_1 \subseteq A^+ \subseteq p^N A^+_1$ for some ring of definition $A^+_1$ and $M,N \in \mathbb Z$. This implies that $p^N A^+$ also form a fundamental system of neighbourhood, so $A^+$ is a ring of definition.

If $A_1^+$ and $A_2^+$ are rings of definition, then we may find $N \in \mathbb Z_+$ with $A_2^+ \subseteq p^{-N} A_1^+$. It follows that $A_1^+ \cdot A_2^+ \subseteq A_1^+ \cdot p^{-N} A_1^+ = p^{-N} A_1^+$ is bounded. It is automatically open, so it is a ring of definition.

If $a \in A^\circ$ and $A^+$ is a ring of definition, then $X:=\{ a^n \}$ is bounded, so $A^+[a] = A^+ \cdot X$ is also bounded. It is an open subring, so it is a ring of definition.

It is clear that any ring of definition $A^+$ is contained in $A^\circ$, so our previous discussions shows that $A^\circ$ is the union of all rings of definitions of $A$.

It is clear that $A^{\circ\circ} \subseteq A^\circ$. Since the topology is non-archimedean, one easily shows that $A^{\circ\circ}$ is an ideal of $A^\circ$ We have $pA^\circ \subseteq A^{\circ\circ}$, so $A^{\circ\circ}$ is an open ideal. For $x \in A^{\circ\circ}$ we have
\[ y:=\sum_{n \in \mathbb Z} x^n \in A^\circ \]
and $(1-x)y =1$, so $A^{\circ\circ}$ is contained in the Jacobson radical of $A^\circ$.

For (3), let $A^+_1$ and $A^+_2$ be rings of definitions of $A$ with $A^+_1 \subseteq A^+_2$. Let $M$ be a topological $A$-module. It is clear that any $A^+_2$-module of definition is an $A^+_1$-module of definition. Conversely assume that $M^+_1$ is an $A^+_1$-module of definition, then define $M^+_2 = A^+_2 \cdot M^+_1$. There exists $N \in \mathbb Z$ with $A^+_2 \subseteq p^{-N} A^+_1$, and we see that $M^+_1 \subseteq M^+_2 \subseteq p^{-N}M^+_1$, so $p^n M^+_2$ ($n \in \mathbb Z$) is a fundamental system of neighbourhoods, and we see that $M^+_2$ is a $A^+_2$-module of definition. This proves that the definition does not depend on the choice of $A^+$. The discussion of boundedness is the same as in (1).

Assume that $M_0/A_0$ is as described, then denote $A^+$ and $M^+$ the $p$-torsion-free quotients of $A_0$ and $M_0$ respectively, so $M^+$ is a $p$-complete $p$-torsion-free $A^+$-module, and we see that $M_0 \bigl[ \frac 1 p \bigr] = M^+ \bigl[ \frac 1 p \bigr]$ is an adic module of $A_0 \bigl[ \frac 1 p \bigr] = A^+ \bigl[ \frac 1 p \bigr]$.
\end{proof}

\begin{lemma}\label{lem: closed submodule}\mbox{}
\begin{enumerate}[label=(\arabic*)]
\item Let $M$ be a classically $p$-complete abelian group and $N \subseteq M$ be a closed subgroup, then $N$ and $M/N$ are also classically $p$-complete. If $M/N$ is of bounded $p^\infty$-torsion, then the induced topology on $N$ agrees with the $p$-adic topology, and the converse holds whenever $M$ has bounded $p^\infty$-torsion.

\item Let $M$ be an adic module over a Tate ring $A$ over $\mathbb Q_p$ and $N \subseteq M$ be a closed $A$-submodule. Then $N$ and $M/N$ are also adic over $A$ with the induced and quotient topology. Let $A^+ \subseteq A$ be a ring of definition and $M^+ \subseteq M$ be a module of definition, then $M^+ \cap N$ and $(M^++N)/N$ are modules of definition for $N$ and $M/N$.
\end{enumerate}
\end{lemma}

\begin{proof}
For (1), it is clear that $\bigcap p^n N =0$. For $x \in M$, if the image of $x$ lies in $\bigcap p^n(M/N)$, then for each $n \in \mathbb Z_+$, the intersection $(x+p^n M) \cap N$ is non-empty. Since $N$ is closed, we see that $x \in N$, and $\bigcap p^n(M/N) =0$.

For any sequence $(x_n)_{n \ge 0}$ in $N$, denote
\[ S_n = \sum_{m<n} p^m x_m, \quad R_n = \sum_{m \ge n} p^{m-n} x_m, \]
so we have
\[ R_0 = S_n + p^nR_n. \]
Since $N$ is closed, we see that all $R_n$ lies in $N$, and $S_n \to R_0$. Therefore $N$ is $p$-adically complete. The completeness of $M/N$ is straightforward, and we omit it.

If $M/N$ is of bounded $p^\infty$-torsion, say $(M/N)[p^\infty] = (M/N)[p^d]$, then we see that for any $n$ we have $p^{n+d}M \cap N \subseteq p^n N$, so the induced topology agrees with the adic topology. Assume that $M$ has bounded $p^\infty$-torsion and the induced topology on $M/N$ agrees with the $p$-adic topology. Say $M[p^\infty] = M[p^d]$. Take $d' \in \mathbb Z_+, d' \ge d+1$ such that $p^{d'} M \cap N \subseteq p^{d+1} N$. For $\overline x \in (M/N)[p^{d'}]$, take a lifting $x \in M$ of $\overline x$, and we see that $p^{d'}x \in p^{d'}M \cap N \subseteq p^{d+1} N$. Write $p^{d'} x = p^{d+1} y$ for $y \in N$, and we see that $p^{d+1}(p^{d'-d-1}x-y) =0$. Since $M[p^\infty] = M[p^d]$, we get $p^d(p^{d'-d-1}x-y) =0$, and $p^{d'-1} x = p^d y \in N$. Therefore $\overline x \in (M/N)[p^{d'-1}]$, and
\[ (M/N)[p^{d'}] = (M/N)[p^{d'-1}]. \]
It follows that $(M/N)[p^\infty] = (M/N)[p^{d'-1}]$, so $M/N$ is of bounded $p^\infty$-torsion, as desired.

For (2), take $A^+$ and $M^+$ as in the description. Denote $N^+ = M^+ \cap N$ and $(M/N)^+ = (M^++N)/N$. A simple calculation shows that a fundamental neighbourhood system for $N^+$ and $(M/N)^+$ are given by $\{p^nN^+\}$ and $\{ p^n(M/N)^+ \}$, and we are done.
\end{proof}

\begin{lemma}[Maps between adic modules are continuous if and only if bounded]\label{lem: continuous equiv to bounded}
Let $A$ be a ring with bounded $p^\infty$-torsion, and $M,N$ be adic modules over $A\bigl[\frac 1 p \bigr]$. Then, for any $A$-module map $f: M \to N$, $f$ is continuous if and only if there exists modules of definition $M_0 \subseteq M, N_0 \subseteq N$ with $f(M_0) \subseteq N_0$.
\end{lemma}

\begin{proof}
Assume first that $f$ is continuous. Let $N_0$ be any module of definition of $N$, and $f^{-1}(N_0)$ is open in $M$. By definition, modules of definition of $M$ form a neighbourhood basis at $0$, hence we may take a module of definition $M_0 \subseteq f^{-1}(N_0).$

Assume that there exists modules of definition $M_0 \subseteq M, N_0 \subseteq N$ with $f(M_0) \subseteq N_0$. It follows that $f^{-1}(p^n N_0) \supseteq p^n M_0$. As $f$ is linear, and $\left\{ p^n N_0 \right\}$ form a neighbourhood basis at $0$, we see that $f$ is continuous.
\end{proof}

We also have the Banach open mapping theorem.

\begin{theo}[Banach open mapping theorem]\label{theo: banach open mapping}
Let $A$ be a ring with bounded $p^\infty$-torsion, and $M,N$ be adic modules over $A \bigl[ \frac 1 p \bigr]$. If a homomorphism $f: M \to N$ is continuous and surjective, then it is an open mapping.
\end{theo}

\begin{proof}
Pick a module of definition $M_0 \subseteq M$. Since $f$ is surjective, we see that
\[ N = \bigcup_{n \in \mathbb Z} \overline{p^n f(M_0)}. \]
Therefore, by \cite[\href{https://stacks.math.columbia.edu/tag/0CQV}{Tag 0CQV}]{stacks-project}, we see that $\overline{p^n f(M_0)}$ is open for \emph{some} $n \in \mathbb Z$. Since $p: N \to N$ is a homeomorphism, we see that for \emph{any} $n \in \mathbb Z$, the group
\[ \overline{p^n f(M_0)} = \overline{f(p^nM_0)} \]
is open in $N$. Since $p^n M_0$ form a fundamental system of neighbourhoods of $0$, we see that for any open subgroup $M' \subseteq M$, the group $\overline{f(M')}$ is open in $N$. Now the result follows from \cite[\href{https://stacks.math.columbia.edu/tag/0CQW}{Tag 0CQW}]{stacks-project}.
\end{proof}

\begin{cor}[Uniform boundedness theorem]\label{cor: uniform boundedness}
Let $M,N$ be adic modules over $\mathbb Q_p$ and $\left\{ T_i: M \to N \right\}_{i \in I}$ be a family of continuous homomorphisms. Assume that for each $m \in M$, the set $\left\{ T_i(m) \right\}$ is bounded in $N$ (i.e. contained in a module of of definition of $N$), then $\left\{ T_i \right\}$ is uniformly bounded (i.e. there exists modules of definitions $M_0 \subseteq M$ and $N_0 \subseteq N$ with $T_i(M_0) \subseteq N_0$ for all $i \in I$).
\end{cor}

\begin{proof}
Fix modules of definitions $M_0 \subseteq M$ and $N_0 \subseteq N$. Take
\[ M'_0 = M_0 \cap \bigcap_{i \in I} T_i^{-1}(N_0), \]
then $M'_0$ is closed in $M_0$. Also, by our condition, for each $m$, there exists $N \in \mathbb Z$ with $p^N m \in M_0$ and $p^N T_i(m) \in N_0$ for each $i \in I$, so $M'_0 \bigl[ \frac 1 p \bigr] = M$. By Lemma \ref{lem: closed submodule}, $M'_0$ is $p$-adically complete. Denote $M'_0$ equipped with the $p$-adic topology by $\widehat{M'_0}$, then the map $\widehat{M'_0} \bigl[ \frac 1 p \bigr] \to M_0 \bigl[ \frac 1 p \bigr]$ is an isomorphism, hence is a homeomorphism by Theorem \ref{theo: banach open mapping}. Therefore, $M'_0$ as the image of $\widehat{M'_0}$ in $M$, is open, and we conclude that $M'_0$ is a module of definition of $M$. It is clear that $T_i(M'_0) \subseteq N_0$ for all $i \in I$, and we are done.
\end{proof}

We now turn to investigate completed tensor product and flatness.

\begin{prop-defi}[{cf. \cite[Subsection 1.1]{Bhatt_2022}}]\label{prop-defi: complete flatness}
Let $A$ be a simplicial ring and $I \subseteq \pi_0(A)$ be a finitely generated ideal. Let $M$ be an $A$-module, and $\{x_1,\ldots,x_d\}$ be a generating set of $I$. Fix $-\infty \le a \le b \le \infty$ (we do not allow $a=b=\infty$ or $a=b=-\infty$). Then, the following conditions are equivalent:
\begin{enumerate}[label=(\arabic*)]
\item for any $N \in D^{\tors{I^\infty}}(A)$, if $N \in D^{\le 0}(A)$ then $M \otimes_A^\bL N \in D^{\le b}(A)$; if $N \in D^{\ge 0}(A)$ then $M \otimes_A^\bL N \in D^{\ge a}(A)$;
\item for any $N \in D^{\tors I}(A)$, if $N \in D^{\le 0}(A)$ then $M \otimes_A^\bL N \in D^{\le b}(A)$; if $N \in D^{\ge 0}(A)$ then $M \otimes_A^\bL N \in D^{\ge a}(A)$;
\item $M \otimes_{\mathbb Z[x_1,\ldots,x_n]}^\bL \mathbb Z$ is of $\Tor$-amplitude $[a,b]$ over $A \otimes_{\mathbb Z[x_1,\ldots,x_n]}^\bL \mathbb Z$;
\item the derived completion $\widehat M \in D^{\le b}(A)$, and $M \otimes_A^\bL \pi_0(A)/I$ is of $\Tor$-amplitude $[a,b]$ over $\pi_0(A)/I$.
\end{enumerate}
If any of the above conditions are satisfied, we say that $M$ is $I$-completely of $\Tor$-amplitude $[a,b]$ over $A$. If $a=b=0$, we say that $M$ is $I$-completely flat over $A$.

If $M$ is $I$-completely flat over $A$, then the following conditions are equivalent:
\begin{enumerate}[label=(\alph*)]
\item the map $D^{\tors{I^\infty}}(A) \to D^{\tors{I^\infty}}(A)$, $N \mapsto M \otimes_A^\bL N$ is conservative;
\item the map $D^{\tors I}(A) \to D^{\tors I}(A)$, $N \mapsto M \otimes_A^\bL N$ is conservative;
\item $M \otimes_{\mathbb Z[x_1,\ldots,x_n]}^\bL \mathbb Z$ is faithfully flat over $A \otimes_{\mathbb Z[x_1,\ldots,x_n]}^\bL \mathbb Z$;
\item $M \otimes_A^\bL \pi_0(A)/I$ is faithfully flat over $\pi_0(A)/I$.
\end{enumerate}
If any of the above conditions are satisfied, we say that $M$ is $I$-completely faithfully flat over $A$.

It is clear that the above conditions only depends on $\widehat M$, so we will usually talk about $I$-complete $\Tor$-amplitude only for $I$-complete modules.
\end{prop-defi}

\begin{proof}
The implications (1) $\Rightarrow$ (2) $\Rightarrow$ (3) is evident. Assume that (3) holds. Observe that the module
\[ \mathbb Z[x_1,\ldots,x_n] / (x_1^r,\ldots,x_n^r) \]
admits a finite filtration with graded pieces are isomorphic to $\mathbb Z = \mathbb Z[x_1,\ldots,x_n]/(x_1,\ldots,x_n)$. Therefore,
\[ A \otimes_{\mathbb Z[x_1,\ldots,x_n]}^\bL \mathbb Z[x_1,\ldots,x_n] / (x_1^r,\ldots,x_n^r) \]
admits a finite filtration whose graded pieces are isomorphic to $A \otimes_{\mathbb Z[x_1,\ldots,x_n]}^\bL \mathbb Z$. It follows that the modules
\[ M \otimes_{\mathbb Z[x_1,\ldots,x_n]}^\bL \mathbb Z[x_1,\ldots,x_n] / (x_1^r,\ldots,x_n^r) \in D^{\le b}(A). \]
Take the limit, note that the transition maps are surjective on $\hol^b$, and we see that $\widehat M \in D^{\le b}(A)$. That $M \otimes_A^\bL \pi_0(A)/I$ is of $\Tor$-amplitude $[a,b]$ is evident.

Now we prove (4) $\Rightarrow$ (1). Note that $M \to \widehat M$ is $I$-local (i.e. is a finite colimit of modules where at least one of $x_i$ is invertible), so $M \otimes_A^\bL N \to \widehat M \otimes_A^\bL N$ is an isomorphism. Therefore we may replace $M$ by $\widehat M$. The part `if $N \in D^{\le 0}(A)$' now becomes evident. Also, for $N \in D^{\ge 0} (A)$, we may write $N = \colim_i \tau^{\le i} N$, and all cofibres $\tau^{\le (i-1)} N \to \tau^{\le i} N$ are concentrated on degree $i$. Therefore, after a shift we may assume that $N$ is concentrated on degree 0. Now $N$ is a module over $\pi_0(A)$ which is $I^\infty$-torsion, so we may write $N = \colim_i N[I^i]$. Each cofibre $N[I^{i-1}] \to N[I^i]$ are concentrated on degree 0, and are modules over $\pi_0(A)/I$. Therefore, we may assume that $N \in D(\pi_0(A)/I)^\heartsuit$, and the conclusion follows from the assumption that $M \otimes_A^\bL \pi_0(A)/I$ is of $\Tor$-amplitude $[a,b]$ over $\pi_0(A)/I$.

Now assume that $M$ is $I$-completely flat over $A$, we prove the faithfully flat part. The implications (a) $\Rightarrow$ (b) $\Rightarrow$ (c) is evident, (c) $\Leftrightarrow$ (d) is a direct consequence of Proposition \ref{prop: flatness as t-exactness}, so we only need to show (c) $\Rightarrow$ (a).

Taking the cofibre, and we only need to prove that for $N \in D^{\tors{I^\infty}} (A)$, we have $M \otimes_A^\bL N =0 \Rightarrow N=0$. Consider the Koszul complexes
\[ N_i = \Kos(N; x_1^i,\ldots,x_n^i) [-n]. \]
View $N_i$ as `the annihilator of $I^i$ in $N$', and we have transition maps
\[ N_1 \to N_2 \to \cdots \to N. \]
Denote $P = \mathbb Z[x_1,\ldots,x_n]$, then the cofibre of $\colim N_i \to N$ is isomorphic to the complex
\[ N \otimes_P^\bL \left[ P[x_1^{-1}] \oplus \cdots \oplus P[x_n^{-1}] \to \bigoplus_{i<j} P[x_i^{-1}x_j^{-1}] \to \cdots \to P[x_1^{-1}\cdots x_n^{-1}] \right]. \]
As $N$ is $I^\infty$-torsion, we see that the complex above is zero, hence $N \cong \colim N_i$. Therefore, we only need to deal with the case $N=N_i$, hence assuming that $N$ is a module over
\[ A \otimes_P P/(x_1^i,\ldots,x_n^i). \]
As above, we may take a filtration
\[ 0 = Q_0 \subseteq Q_1 \subseteq \cdots \subseteq Q_s = P/(x_1^i,\ldots,x_n^i), \]
such that all $Q_j/Q_{j-1}$ are isomorphic to $\mathbb Z = P / (x_1,\ldots,x_n)$. Now, $N \otimes_P^\bL (Q_j/Q_{j-1})$ are modules over $A \otimes_P^\bL \mathbb Z$, and $M \otimes_A^\bL (N \otimes_P^\bL (Q_j/Q_{j-1})) =0$. By assumption, we conclude that $N \otimes_P^\bL (Q_j/Q_{j-1}) =0$. Induction now shows that $N \otimes_P^\bL Q_j =0,$ so $N \cong N \otimes_P^\bL Q_s =0$, as desired.
\end{proof}

\begin{rmk}
In \cite[Subsection 1.1]{Bhatt_2022}, Bhatt and Scholze directly define $M \in D(A)$ as $I$-completely flat over $A$ if and only if $M \otimes_A^\bL \pi_0(A)/I$ is flat over $\pi_0(A)/I$. However, when $A$ is not truncated, omitting the connectiveness assumption may cause some undesirable trouble. For example (this is the example in the remark before \cite[Corollary 5.6.6.5]{lurie2018spectral}), take the simplicial ring $A = \Sym^*_{\mathbb Q} (\mathbb Q[2])$ with $I = 0 \subseteq \pi_0(A)$. Since we are working over $\dQ$, the underlying $\dE_\infty$-ring of $A$ may be identified with the differential graded algebra $\dQ[x]$ with $x$ on (homological) degree 2. Define the $A$-module $M = \mathbb Q[x,x^{-1}]$ (as a differential graded module over $A$). By \cite[Proposition 7.2.1.19]{lurie2017higher}, we see that $\pi_*(M \otimes_A^\bL \pi_0(A)) \cong \mathbb Q[x,x^{-1}] \otimes_{\mathbb Q[x]} \mathbb Q =0$. However, $M$ is not connective, so it cannot be flat over $A$.

The good news is, if $A$ is a $n$-truncated simplicial ring for some $n \in \mathbb Z_+$, then their definition causes no trouble. By the remark following Definition \ref{defi: Tor amplitude}, $M \otimes_{\mathbb Z[x_1,\ldots,x_n]}^\bL \mathbb Z$ is of $\Tor$-amplitude $[a,b]$ over $A \otimes_{\mathbb Z[x_1,\ldots,x_n]}^\bL \mathbb Z$ if and only if $M \otimes_A^\bL \pi_0(A)/I$ is of $\Tor$-amplitude $[a,b]$ over $\pi_0(A)/I$, so we may use (3) to see that the two definitions agree. In particular, if $A$ is a classical commutative ring, then $M$ is of $I$-complete $\Tor$-amplitude $[a,b]$ over $A$ if and only if $M \otimes_A^\bL A/I$ is of $I$-complete $\Tor$-amplitude over $A$.
\end{rmk}

\begin{prop}[{\cite[Lemma 4.7]{bhatt2019topologicalhochschildhomologyintegral}}]\label{prop: p-completely flat}
Fix a (classical) ring $A$ with bounded $p^\infty$-torsion.

\begin{enumerate}[label=(\arabic*)]
\item If $M \in D \left( A \right)$ is derived $p$-complete and $p$-completely flat, then $M$ is a classically $p$-complete $A$-module concentrated in degree 0, with bounded $p^\infty$-torsion, such that $M/p^nM$ is flat over $A/p^nA$ for all $n \ge 1.$ Moreover, for all $n \ge 1,$ the map
\[ M \otimes_A (A \left[ p^n \right]) \to M \left[ p^n \right] \]
is an isomorphism.
\item Conversely, if $N$ is a classically $p$-complete $A$-module with bounded $p^\infty$-torsion such that $N/p^nN$ is flat over $A/p^nA$ for all $n \ge 1,$ then $N \in D \left( A \right)$ is $p$-completely flat.
\end{enumerate}
\end{prop}

\begin{prop}[{\cite[Lemma 3.7 (2)]{Bhatt_2022}}]\label{prop: (p,I)-completely flat}
Let $\left( A,I \right)$ be a bounded prism and $M \in D \left( A \right)$ be derived $\left( p,I \right)$-complete and $\left( p,I \right)$-completely flat. Then, $M$ is a classically $\left( p,I \right)$-complete $A$-module concentrated on degree 0, and is $I$-torsion free.
\end{prop}

\begin{defi}\label{defi: completely syntomic smooth etale}
Let $A$ be a $p$-complete classical ring of bounded $p^\infty$-torsion (resp. $(A,I)$ be a bounded prism), and $B/A$ be a derived $p$-adically complete (resp. derived $(p,I)$-adically complete) simplicial ring. We say that $B$ is $p$-completely (resp. $(p,I)$-completely) syntomic, smooth, \'etale over $A$, if $B \otimes_A^\bL A/p$ (resp. $B \otimes_A^\bL A/(p,I)$) is syntomic, smooth, \'etale over $A/p$ (resp. $A/(p,I)$).

By the propositions above, any such algebra $B$ is $p$-completely flat (resp. $(p,I)$-completely flat) over $A$, hence is concentrated on degree 0.
\end{defi}

\begin{cor}\label{cor: tensor conc on degree 0}
Let $A$ be a classically $p$-complete ring with bounded $p^\infty$-torsion, $M,N$ be classically $p$-complete $A$-modules with bounded $p^\infty$-torsion, such that $M$ is $p$-completely flat over $A$. Then, $M \cotimes_A^\bL N$ is classically $p$-complete over $A$ with bounded $p^\infty$-torsion, and we have the following isomorphism
\[ \bigl( M \cotimes_A^\bL N \bigr) \left[ p^n \right] \cong M \otimes_A^\bL (N \left[ p^n \right]) \cong M \otimes_A (N[p^n]). \]

We will use the simpler symbol $M \cotimes_A N$ to denote $M \cotimes_A^\bL N$ in this case.
\end{cor}

\begin{proof}
Take $A' = A \oplus N.$ We can make $A'$ a ring by making $N$ an ideal with $N^2=0.$ Thus by Proposition \ref{prop: p-completely flat}, since $M \cotimes_A^\bL A'$ is $p$-complete and $p$-completely flat over $A',$ is it classically $p$-adically complete, with an isomorphism
\[ \bigl( M \cotimes_A^\bL A' \bigr) \left[ p^n \right] \cong (M \cotimes_A^\bL A') \cotimes_{A'}^\bL A' \left[ p^n \right] \cong M \cotimes_A^\bL (A'[p^n]) \cong M \otimes_A^\bL (A'[p^n]). \]
Thus, as $A' \cong A \oplus N,$ $M \cotimes_A^\bL N$ is also classically $p$-complete with bounded $p^\infty$-torsion, and that we have an isomorphism of the corresponding direct sum
\[ \bigl( M \cotimes_A^\bL N \bigr) \left[ p^n \right] \cong M \otimes_A^\bL N \left[ p^n \right]. \qedhere \]
\end{proof}

\begin{lemma}\label{lem: p-complete module as classical completion}
Let $M \in D^{\le 0}(\mathbb Z)$ such that the derived $p$-adic completion $\widehat M$ is concentrated on degree 0 and classically $p$-adically complete (e.g. of bounded $p^\infty$-torsion). Then, the natural map
\[ \widehat M \to \widehat{\hol^0(M)} \to \clims 0 n \hol^0(M) / p^n \hol^0(M) \]
identifies $\widehat M$ as the \emph{classical} $p$-adic completion of $\hol^0(M)$.
\end{lemma}

\begin{proof}
We have
\[ \widehat M / p^n \cong \hol^0(\widehat M \otimes_\mathbb Z^\bL \mathbb Z/p^n) \cong \hol^0(M \otimes_\mathbb Z^\bL \mathbb Z/p^n) \cong \hol^0(M)/p^n. \]
Take the inverse limit (both $\lim$ and $\clim 0$ work) gives us the result.
\end{proof}

\begin{rmk}
Usually we will only use this lemma for $M$ $p$-completely flat over some ring $A$ of bounded $p^\infty$-torsion.

Since $\widehat{(-)}$ is right $t$-exact, we see that the natural map $\widehat M \cong \hol^0(\widehat M) \to \hol^0(\widehat{\hol^0(M)})$ is an isomorphism. However, it is not a priori correct that $\widehat{\hol^0(M)}$ is concentrated on degree 0, so we cannot say that $\widehat M$ is isomorphic to the derived completion of $\hol^0(M)$ --- instead we only have an isomorphism $\widehat M \cong \hol^0(\widehat{\hol^0(M)})$.
\end{rmk}

\begin{prop-defi}\label{prop-defi: base change of adic modules}
Let $A$ be a Tate ring over $\mathbb Q_p$. For adic modules $M_1,\ldots,M_n,N$ over $A$, we define
\[ \Hom_A^\ad((M_1,\ldots,M_n),N) \]
to be the $A$-module consisting of \emph{continuous} multilinear functions $M_1 \times \cdots \times M_n \to N$. This is the same as bounded multilinear functions $M_1 \times \cdots \times M_n \to N$.

\begin{enumerate}[label=(\arabic*)]
\item The $A$-module $\Hom_A^\ad((M_1,\ldots,M_n),N)$ has an adic module structure
\[ \underline{\Hom}^\ad_A((M_1,\ldots,M_n),N) \]
such that for any adic modules $M'_1,\ldots,M'_m$, the natural map
\[ \Hom_A^\ad((M'_1,\ldots,M'_m), \underline{\Hom}^\ad_A((M_1,\ldots,M_n),N)) \to \Hom_A^\ad((M'_1,\ldots,M'_m,M_1,\ldots,M_n),N) \]
is an isomorphism. Moreover, if $A^+$ is a ring of definition of $A$ and $M_i^+$ are $A^+$-modules of definitions of $M_i$, then a module of definition of $\underline{\Hom}^\ad_A((M_1,\ldots,M_n),N)$ is given the set of multilinear functions
\[ \Hom_{A^+}((M_1^+,\ldots,M_n^+), N^+). \]

In particular, for the functor $F(N) = \Hom_A^\ad((M_1,\ldots,M_n),N)$, the module
\[ \underline{\Hom}^\ad_A((M_1,\ldots,M_n),N) \]
is represented by the functor
\[ M \mapsto F(\underline{\Hom}^\ad_A(M,N)). \]

\item The functor
\[ \Hom_A^\ad((M_1,\ldots,M_n),-) \]
is corepresentable by an adic module, which we denote by
\[ M_1 \cotimes_A \cdots \cotimes_A M_n. \]
The last paragraph shows that the modules
\[ \underline{\Hom}^\ad_A((M_1,\ldots,M_n),N), \quad \underline{\Hom}^\ad_A(M_1 \cotimes_A \cdots \cotimes_A M_n,N) \]
are naturally isomorphic.

Moreover, the obvious natural maps
\[ M_1 \cotimes_A \cdots \cotimes_A M_n \cotimes_A M'_1 \cotimes_A \cdots \cotimes_A M'_m \to (M_1 \cotimes_A \cdots \cotimes_A M_n) \cotimes_A (M'_1 \cotimes_A \cdots \cotimes_A M'_m) \]
are isomorphisms, so $- \cotimes_A -$ gives a symmetric monoidal structure on the category of adic $A$-modules.

\item If $A_0$ is a classical $p$-complete ring of bounded $p^\infty$-torsion, $M_0/A_0$ is $p$-complete and $p$-completely flat. For any adic module $N/A_0 \bigl[ \frac 1 p \bigr]$ denote
\[ M_0 \cotimes_{A_0} N = \bigl( M_0 \bigl[ \frac 1 p \bigr] \bigr) \cotimes_{A_0 [\frac 1 p ]} N. \]
Then, for any $p$-complete $A_0$-module $N_0$ of bounded $p^\infty$-torsion, $M_0 \cotimes_{A_0} N_0$ is $p$-complete and of bounded $p^\infty$-torsion, and the natural map
\[ M_0 \cotimes_{A_0} \bigl( N_0 \bigl[ \frac 1 p \bigr] \bigr) \to (M_0 \cotimes_{A_0} N_0) \bigl[ \frac 1 p \bigr] \]
is an isomorphism.
\end{enumerate}
\end{prop-defi}

\begin{proof}
For (1), define the adic module
\[ \underline{\Hom}^\ad_A((M_1,\ldots,M_n),N) = \Hom_{A^+}((M_1^+,\ldots,M_n^+), N^+) \bigl [\frac 1 p \bigr], \]
and it is easy to verify the universal properties.

For (2), let $A^+$ be a ring of definition of $A$ and $M_i^+$ be a $A^+$-module of definition of $M_i$. Denote $M$ to be the $p$-torsion-free quotient of $M_1^+ \otimes_{A^+} \cdots \otimes_{A^+} M_n^+$, and $M'$ the classical $p$-adic completion of $M$. It is easy to verify that $M' \bigl[ \frac 1 p \bigr]$ is the desired adic module.

The the functorial perspective gives the isomorphism
\[ \underline{\Hom}^\ad_A((M_1,\ldots,M_n),N) \cong \underline{\Hom}^\ad_A(M_1 \cotimes_A \cdots \cotimes_A M_n,N), \]
whose underlying map is the given isomorphism of functors
\[ \Hom^\ad_A((M_1,\ldots,M_n),N) \cong \Hom^\ad_A(M_1 \cotimes_A \cdots \cotimes_A M_n,N). \]
Therefore we have
\begin{align*}
& \Hom^\ad_A ((M_1 \cotimes_A \cdots \cotimes_A M_n) \cotimes_A (M'_1 \cotimes_A \cdots \cotimes_A M'_m),N ) \\
\cong{} & \Hom^\ad_A (M_1 \cotimes_A \cdots \cotimes_A M_n, \underline{\Hom}^\ad_A (M'_1 \cotimes_A \cdots \cotimes_A M'_m,N)) \\
\cong{} & \Hom^\ad_A ((M_1,\ldots,M_n), \underline{\Hom}^\ad_A ((M'_1,\ldots,M'_m),N)) \\
\cong{} & \Hom^\ad_A ((M_1,\ldots,M_n,M'_1,\ldots,M'_m),N) \\
\cong{} & \Hom^\ad_A (M_1 \cotimes_A \cdots \cotimes_A M_n \cotimes_A M'_1 \cotimes_A \cdots \cotimes_A M'_m,N),
\end{align*}
giving the isomorphism as desired.

For (3), it is easy to observe that $M_0 \cotimes_{A_0} N_0$ is $p$-complete and of bounded $p^\infty$-torsion by Corollary \ref{cor: tensor conc on degree 0}, and the isomorphism is a direct consequence of the universal property.
\end{proof}

\begin{lemma}\label{lemma: direct sum with bounded torsion}
Let $A$ be a ring and $\left\{ M_\lambda \right\}_{\lambda \in \Lambda}$ be a family of classically $p$-complete $A$-modules. Also, assume that the $p^\infty$-torsion of $\left\{ M_\lambda \right\}$ is uniformly bounded, i.e. there exists $n_0 \in \mathbb Z_+$ such that all $M_\lambda$ has
\[ M_\lambda \left[ p^\infty \right] = M_\lambda \left[ p^{n_0} \right]. \]
Then, the natural map
\[ \cbigoplus_{\lambda \in \Lambda} M_\lambda \to \prod_{\lambda \in \Lambda} M_\lambda, \]
is injective, with its image equal to
\[ \Bigl\{ \left( m_\lambda \right) \in \prod M_\lambda \mathbin{\Big|} \forall N \in \mathbb Z_+, \bigl\{ \lambda \in \Lambda \mid m_\lambda \notin p^N M_\lambda \bigr\} \text{\ is finite} \Bigr\}. \]
\end{lemma}

\begin{proof}
This follows immediately from the fact that
\[ \bigoplus_{\lambda \in \Lambda} M_\lambda \]
has bounded $p^\infty$-torsion, and that the derived completion of a bounded $p^\infty$-torsion module is identical with its classical completion.
\end{proof}

\begin{prop}\label{prop: conc on degree 0 by reduced}
Let $A$ be a ring and $I \subseteq A$ be an effective Cartier divisor. Assume that $M \in D \left( A \right)$ is derived $I$-complete, and that $M \otimes_A^\bL A/I$ is concentrated on degree 0, then $M$ is concentrated on degree 0, and is $I$-torsion-free.
\end{prop}

\begin{proof}
By induction, all $M \otimes_A^\bL A/I^n$ are concentrated on degree 0, and $M \otimes_A^\bL A/I^{n+1} \to M \otimes_A^\bL A/I^n$ are surjective. Since $I$ is an effective Cartier divisor, we see that
\[ M \cong \lim M \otimes_A^\bL A/I^n \]
is concentrated on degree 0. As $M \otimes_A^\bL A/I$ is concentrated on degree 0, we see that
\[ \ker \left( M \otimes_A I \to M \right) \cong \Tor_1^A \left( M,A/I \right) =0, \]
hence $M$ is $I$-torsion-free.
\end{proof}

We are also going to need some properties of `completed PD-powers'. Keep in mind of the following Lemma. We are going to use it without further mentioning it.

\begin{lemma}\label{lemma: concentrated on degree zero}\mbox{}
\begin{enumerate}[label=(\arabic*)]
\item Let $R$ be a classical ring and $M \in \Mod_R$ be an $R$-module concentrated on degree 0. Then, $\hol^0(\Sym^*(M)), \hol^0(\bigwedge^*(M)), \hol^0(\Gamma^*(M))$ agrees with the classical construction.

\item Let $R$ be a classical ring and $M$ be a flat $R$-module. Then, $\Sym^*(M), \bigwedge^*(M), \Gamma^*(M)$ are all flat $R$-modules, and they agree with the classical construction.

\item Let $R$ be a $p$-completele classical ring with bounded $p^\infty$-torsion and $M/R$ be $p$-complete and $p$-completely flat. Then, $\widehat{\Sym}^*(M), \cbigwedge^*(M), \cGamma^*(M)$ are all $p$-completely flat over $R$, and they agree with the classical completion of $\hol^0(\Sym^*(M)), \hol^0(\bigwedge^*(M)), \hol^0(\Gamma^*(M))$ respectively.
\end{enumerate}
\end{lemma}

\begin{proof}
For (1), take a presentation $F_1 \to F_0 \to M$ where $F_i$ are free, and we are reduced to (2).

For (2), one only need to write $M$ as a filtered colimit of finite free $R$-modules.

For (3), reduce $\otimes_R^\bL (R/p)$, and we prove the first part of the statement. The second part of statement is a corollary of Lemma \ref{lem: p-complete module as classical completion}.
\end{proof}

\begin{lemma}\label{lem: finite group closed embedding}
Let $M$ be a classically $p$-complete abelian group of bounded $p^\infty$-torsion, and $G$ be a finite group $G$ acting on $M$. Then, the map $M^G \to M$ is a closed embedding, with both side equipped with the $p$-adic topology. Furthermore, we have an isomorphism
\[ M^G \to \clims 0 n \left( M/p^n M \right)^G. \]
\end{lemma}

\begin{proof}
It is clear that with the induced topology, $M^G \subseteq M$ is closed, so we only need to prove that the induced topology coincides with the $p$-adic topology on $M^G.$ Take $e \in \mathbb Z_{\ge 0}$ such that $p^e \parallel \left|G \right|.$

Fix $n \in \mathbb Z_+,$ and we have to prove that $p^n M^G \supseteq p^{n'} M \cap M^G$ for some $n' \in \mathbb Z_+.$ As $M$ is of bounded $p^\infty$-torsion, we may take $n_0 \ge n$ such that $M \left[ p^\infty \right] = M \left[ p^{n_0} \right].$ Consider the map
\[ f: M \to p^{n'} M \]
given by $m \mapsto p^{n'} m.$ We get an exact sequence
\[ M^G \xrightarrow{p^{n'}} \left( p^n M \right)^G \to \hol^1 \left( G, M \left[ p^n \right] \right). \]
By the properties of cohomologies for finite groups, $\hol^1 \left( G, M \left[ p^n \right] \right)$ is annihilated by $\left| G \right|,$ hence by $p^e.$ Therefore, for any $m \in \bigl( p^{n'}M \bigr)^G,$ there exists $m' \in M^G$ such that $p^{n'} m'. = p^em.$ Take $n'=n_0+e,$ and we see that for any $m \in \bigl( p^{n_0+e} M \bigr)^G,$ there exists $m' \in M^G$ with
\[ p^{n_0} m' - m \in M \left[ p^e \right]. \]
Now, $p^{n_0} m' - m \in p^{n_0} M,$ and by assumption, it is easy to prove that $p^{n_0} M \cap M \left[ p^\infty \right] = \left\{ 0 \right\}.$ Therefore, we get $m = p^{n_0} m' \in p^n M^G.$ It follows that
\[ p^{n_0+e} M \cap M^G = \left( p^{n_0+e} M \right)^G \subseteq p^n M^G, \]
and we have found our $n$ as desired.

We now prove that we have the isomorphism as desired. Write
\[ M^G = \bigl( \clims 0 n M/p^n M \bigr)^G, \]
and the isomorphism follows from the fact that limits commute with limits in categories (here we use the ordinary category of abelian groups).
\end{proof}

\begin{prop}\label{prop: flat pd power}
Fix a ring $A$ which is classically $p$-adically complete, and of bounded $p^\infty$-torsion. Let $M$ be a derived $p$-complete and $p$-completely flat module over $A$. Then, we have a closed embedding of $p$-complete modules (with the $p$-adic topology)
\[ \cGamma^n_A \left( M \right) \to M^{\cotimes n/A} \]
given by
\[ \gamma^{d_1} \left( m_1 \right) \cdots \gamma^{d_k} \left( m_k \right) \mapsto \sum_{\sigma \in \Sigma_n / \Sigma_{d_1} \times \cdots \times \Sigma_{d_k}} \sigma \bigl( m_1^{\otimes d_1} \otimes \cdots \otimes m_k^{\otimes d_k} \bigr), \]
whose cokernel is $p$-completely flat over $A$. Also, for any $A$-module $N$ which is classically $p$-complete and of bounded $p^\infty$-torsion, it induces an isomorphism
\[ N \cotimes_A \cGamma^n_A \left( M \right) \to \bigl( N \cotimes_A M^{\cotimes n/A} \bigr)^{\Sigma_n}.\]

As a consequence, for any adic module $N$ over $A \bigl[ \frac 1 p \bigr]$, the above morphism is an isomorphism.
\end{prop}

\begin{proof}
The construction of the map $\cGamma^n_A \left( M \right) \to M^{\cotimes n/A}$ is nothing but the completion of its classical version, so we will omit it here. We also omit the verification that its cokernel is $p$-completely flat, since we may also reduce modulo $p$. We only need to verify that it is a closed embedding, and it induces the desired isomorphism.

We first verify that the map $N \cotimes_A \cGamma^n_A \left( M \right) \to \bigl( N \cotimes_A M^{\cotimes n/A} \bigr)^{\Sigma_n}$ is an isomorphism. By Lemma \ref{lem: finite group closed embedding}, we only need to verify that
\[ \left( N/p^m N \right) \otimes_{A/p^m A} \Gamma^n_{A/p^mA} \left( M/p^m M \right) \to \left( N/p^m N \right) \otimes_{A/p^m A} \bigl( \left( M/p^m M \right)^{\otimes n / (A/p^mA)} \bigr)^{\Sigma_n} \]
is an isomorphism, and then take $\clim 0$. This is the classical result for flat modules over arbitrary rings.

Now, to show that the designated map is a closed embedding, just apply Lemma \ref{lem: finite group closed embedding}.
\end{proof}

\begin{prop}\label{prop: topology of flat pd power}
Under the condition of Proposition \ref{prop: flat pd power}, we have
\[ p^n \cdot \bigl( N \cotimes_A M^{\cotimes n/A} \bigr)^{\Sigma_n} = \bigl( p^n ( N \cotimes_A M^{\cotimes n/A} ) \bigr)^{\Sigma_n}. \]
\end{prop}

\begin{proof}
Using the same technique as in Corollary \ref{cor: tensor conc on degree 0}, we may assume that $N$ is a ring over $A$. Now, we may pass to the case where the base-change has already been done, and assume that $N=A.$

Now, this is a direct result of that
\[ \bigl( M^{\cotimes n/A} \bigr)^{\Sigma_n} \to M^{\cotimes n/A} \]
has $p$-completely flat cokernel. More precisely, let $M_1 \subseteq M$ be an injection of $R$-modules, whose cokernel is derived $p$-complete, and $p$-completely flat over $R$. It follows that
\[ \Tor_1^R \left( M / M_1, R/p^n R \right) =0, \]
so the map
\[ M_1 / p^n M_1 \to M / p^n M \]
is injective, as desired.
\end{proof}

We will also need the theory of descent.

\begin{theo}\label{thm: descent theory}
Assume that $A,B$ are $p$-complete rings of bounded $p$-torsion, and $A \to B$ be $p$-completely faithfully flat. Consider the cosimplicial ring $B(n) = B^{\cotimes(n+1)/A}$. Then, we have:
\begin{enumerate}[label=(\arabic*)]
\item For any module $M$ over $A$ such that $M$ is $p$-complete of bounded $p^\infty$-torsion or $M$ is adic over $A\bigl[ \frac 1 p \bigr]$, we have a quasi-isomorphism of cosimplicial $A$-modules:
\[ M \to M \cotimes_A B(\bullet). \]
\item We have descent theory of finite projective modules for $A \to B(\bullet)$ and $A\bigl[ \frac 1 p \bigr] \to B(\bullet) \bigl[ \frac 1 p \bigr]$. In other words, we have equivalences of categories
\begin{align*}
\Vect(A) & \to \lim_\simp \Vect(B(\bullet)), \\
\Vect\bigl(A\bigl[\frac 1 p \bigr]\bigr) & \to \lim_\simp \Vect\bigl(B(\bullet) \bigl[ \frac 1 p \bigr]\bigr).
\end{align*}
\end{enumerate}
\end{theo}

\begin{proof}
For (1), see Lemma \ref{lem: acyclic cosimplicial}. We remark that it is not a priori correct that $\lim_\simp M \cotimes_A B(\bullet) \cong M \cotimes_A^\bL \lim_\simp B(\bullet)$ since the complex $B(\bullet)$ is not bounded.

For (2), the isomorphism on top follows from taking limit of the isomorphisms
\[ \Vect(A/p^n) \to \lim_\simp \Vect(B(\bullet)/p^n), \]
and the isomorphism at the bottom is \cite[Theorem 7.8]{mathew2021faithfullyflatdescentperfect} for $\text{Perf}_{[0,0]}$. We will also give an elementary proof (i.e. without using too much $\infty$-category) for this particular case in Appendix A.
\end{proof}

We now discuss (complete) flatness for filtered algebras.
\begin{prop}\label{prop: flatness by graded}
Let $B$ be a strictly connective filtered $\dE_\infty$-ring (i.e. all $\gr_*(B)$ are connective), and $M$ be a filtered $B$-module. Then, if the underlying module of $\gr_* M$ is (faithfully) flat over $\gr_* B$ (see \cite[Definition 7.2.2.10]{lurie2017higher} and \cite[Definition B.6.1.1]{lurie2018spectral}), then the underlying module of $M$ is (faithfully) flat over $B$.
\end{prop}

\begin{proof}
Since $\gr_*(M)$ is connective, we conclude that $M$ is strictly connective.

Let $N$ be any $B$-module concentrated on degree 0, and we may equip $N$ with the trivial filtration $\Fil_n(N) = N, \forall n \ge 0$, so $\gr_*(N)$ is concentrated on degree 0. We have (see Example \ref{ex: relative tensor products})
\[ \gr_* (M \otimes_B^\bL N) \cong \gr_*(M) \otimes_{\gr_*(B)}^\bL N, \]
which is concentrated on degree 0 by assumption. Therefore, $M \otimes_B^\bL N$ are strictly concentrated on degree 0, and we see that $M$ is flat over $B$ (see \cite[Theorem 7.2.2.15 (5)]{lurie2017higher}).

As for the faithfully flat part, by \cite[Proposition 7.2.2.13]{lurie2017higher}, we only need to prove that for any $B$-module $N$ concentrated on degree 0, we have $M \otimes_B^\bL N =0 \Rightarrow N=0$. Assume that $M \otimes_B^\bL N =0$. Since $M \otimes_B^\bL N$ is strictly concentrated on degree 0, all its graded pieces are zero, and
\[ 0 = \gr_* (M \otimes_B^\bL N) \cong \gr_*(M) \otimes_{\gr_*(B)}^\bL \gr_*(N). \]
Since $\gr_*(M)$ is faithfully flat over $\gr_*(B)$, we see that $\gr_*(N) =0,$ and it follows that $N = \gr_0(N) =0$, as desired.
\end{proof}

\begin{cor}\label{cor: flatness by graded complete case}
Let $B$ be a (classical)\footnote{We avoid discussing simplicial filtered rings here for simplicity.} filtered ring over $\mathbb Z$, and $M$ be a filtered $B$-module. Then, if the underlying module of $\gr_* M$ is $p$-completely (faithfully) flat over $\gr_* B$, then the underlying module of $M$ is $p$-completely (faithfully) flat over $B$.
\end{cor}

\begin{proof}
Take the tensor product $- \otimes_\mathbb Z^\bL \mathbb Z/p\mathbb Z$, and we are reduced to Proposition \ref{prop: flatness by graded}.
\end{proof}

\begin{prop}\label{prop: pd strict flatness}
For any ring $A$ and an exact sequence of flat $A$-modules
\[ 0 \to M \to N \to Q \to 0, \]
the map $\Gamma^\bullet_A(M) \to \Gamma^\bullet_A(N)$ is faithfully flat.
\end{prop}

\begin{proof}
Write $Q$ as the filtered colimit of finite free modules, say
\[ Q = \colim_i Q_i \]
and take the pullbacks $N_i = N \times_Q Q_i$. We have exact sequences
\[ 0 \to M \to N_i \to Q_i \to 0, \]
and $\Gamma^\bullet_A(N) = \colim_i \Gamma_A^\bullet(N_i)$. Therefore, we only need to prove this for $(M,N_i,Q_i)$, hence we may assume that $Q$ is finite free.

In this case, the exact sequence is split, and we see that
\[ \Gamma^\bullet_A(N) \cong \Gamma^\bullet_A(M) \otimes_A \Gamma_A^\bullet(Q), \]
hence is faithfully flat over $\Gamma^\bullet_A(M)$.
\end{proof}

\begin{cor}\label{cor: pd strict flatness complete case}
For any $p$-complete ring $A$ of bounded $p^\infty$-torsion and an exact sequence of $p$-complete and $p$-completely flat $A$-modules
\[ 0 \to M \to N \to Q \to 0, \]
the map $\cGamma^\bullet_A(M) \to \cGamma^\bullet_A(N)$ is $p$-completely faithfully flat.
\end{cor}

\begin{proof}
Take the tensor product $-\otimes_A^\bL A/p$, and we are reduced to Proposition \ref{prop: pd strict flatness}.
\end{proof}

\subsection{Complete cotangent complex}

In this subsection, we will define complete cotangent complexes, and verify that they satisfy some desirable properties.

\begin{defi}\label{defi: complete cotangent complex}
Let $A \to B$ be simplicial rings and $I \subseteq \pi_0 \left( B \right)$ be a finitely generated ideal. We define the completed cotangent complex $\cdL_{B/A}$ to be the (derived) $I$-adic completion of the (uncompleted) cotangent complex $\dL_{B/A}.$
\end{defi}

\begin{prop}\label{prop: complete cotangent completion}
In the above context, let $\widehat B$ be the (derived) $I$-adic completion of $B$. Then $\cdL_{\widehat B/B} =0$. Moreover, the canonical map $\cdL_{\widehat B/A} \to \cdL_{B/A}$ is an isomorphism.
\end{prop}

\begin{proof}
Take a set of generators $a_1,\ldots,a_n$ of $I$ (and their liftings in the zeroth-complex of $B$). Recall that for any $X \in D \left( B \right),$ we have (see Definition \ref{defi: derived complete})
\[ \widehat X = \lim X \otimes_{\mathbb Z \left[ x_1,\ldots,x_n\right]}^\bL \mathbb Z \left[ x_1,\ldots,x_n \right]/ \bigl( x_1^k,\ldots,x_n^k \bigr), \]
where $\mathbb Z \left[ x_1,\ldots,x_n\right] \to B$ sends $x_i$ to $a_i.$ For simplicity, denote
\[ B_k = B \otimes_{\mathbb Z \left[ x_1,\ldots,x_n\right]}^\bL \mathbb Z \left[ x_1,\ldots,x_n \right]/ \bigl( x_1^k,\ldots,x_n^k \bigr). \]
Consider the exact triangle
\[ \widehat B \otimes_B^\bL \dL_{B/A} \to \dL_{\widehat B/A} \to \dL_{\widehat B/B}. \]
Take the $I$-adic completion, and we see that (since $\widehat B \otimes_B^\bL B_k = B_k$) there exists an exact triangle
\[ \cdL_{B/A} \to \cdL_{\widehat B/A} \to \cdL_{\widehat B/B}. \]
Therefore, we only need to prove that $\cdL_{\widehat B/B} =0.$

Now, since $\widehat B \otimes_B^\bL B_k = B_k,$ we have
\[ \dL_{\widehat B/B} \otimes_B^\bL B_k = \dL_{\widehat B \otimes_B^\bL B_k / B_k} =0, \]
so
\[ \cdL_{\widehat B/B} = \lim_k \dL_{\widehat B/B} \otimes_B^\bL B_k =0, \]
as desired.
\end{proof}

\begin{cor}
Let $A$ be a ring and $P$ be an $A$-algebra of bounded $p^\infty$-torsion for which $\dL_{P/A}$ is flat over $P$. Then,
\[ \cdL_{\widehat P/A} \cong \cdL_{P/A} \]
is $p$-completely flat over $\widehat P$.
\end{cor}

\begin{notn}\label{notn: cOmega}
When $A,B$ are classical rings of bounded $p^\infty$-torsion, and $\cdL_{B/A}$ is $p$-completely flat over $B$, we also denote it by $\cOmega_{B/A}$.

If $\dL_{B/A}$ is flat over $B$, by Lemma \ref{lem: p-complete module as classical completion} and Proposition \ref{prop: derived completion is classical when bounded}, $\cOmega_{B/A}$ is always the derived (=classical) $p$-adic completion of $\Omega_{B/A}$, so our notations would be of no ambiguity.
\end{notn}

\begin{cor}
Let $A$ be a $p$-complete classical ring of bounded $p^\infty$-torsion. Denote $\SCR_{A/}$ ($\widehat\SCR_{A/}$) be the $\infty$-category of ($p$-complete) simplicial commutative rings over $A$, and $\FinPoly_A$ ($\widehat\FinPoly_A$) be the category of finite ($p$-complete) polynomial rings over $A$. Then, the functor $\cdL_{-/A}: \SCR_{A/} \to \cD \left( A \right)$ ($\cdL_{-/A}: \widehat\SCR_{A/} \to \cD \left( A \right)$) is the left Kan extension of its restriction on $\FinPoly_A$ ($\widehat\FinPoly_A$) given by
\[ P \mapsto \cOmega_{P/A}. \]
\end{cor}

\begin{proof}
Observe that $\SCR_{A/} = \calP_\Sigma \left( \FinPoly_A \right)$ and $\widehat\SCR_{A/}$ is generated by $\widehat\FinPoly_A$ under sifted colimits. Therefore, we only need to prove that the functors commute with sifted colimits. This is just a direct consequence of Proposition \ref{prop: complete cotangent completion} and the fact $\dL_{-/A}$ and $\widehat{(-)}$ commute with sifted colimits.
\end{proof}

\begin{lemma}\label{lemma: completed conormal}
Let $R$ be an $\overline{A}$-algebra satisfying Condition \ref{cond: tor -1,0} such that the structure map $\overline{A}\to R$ is surjective with kernel $J$. Then $R$ satisfies Condition \ref{cond: tor -1}, and the $p$-completely flat module $\cdL_{R/\overline{A}}[-1]$ agrees with the classical $p$-adic completion of $J/J^2$. As in Sect. \ref{sect: notations}, we will always use $(J/J^2)^\land$ to denote the \emph{classical} $p$-adic completion of $J/J^2$.
\end{lemma}

\begin{proof}
For (1), observe that under such condition, $\cdL_{R/A} \in \cD^{\le -1} \left( R \right),$ and we see that its Tor amplitude must be $\left\{ -1 \right\},$ i.e. Condition \ref{cond: tor -1} holds. The last part of statement follow from the observation that $J/J^2 = \hol^0(\dL_{R/\overline A}[-1])$ and Lemma \ref{lem: p-complete module as classical completion}.
\end{proof}

\begin{construction}\label{cons: dHodge}
    Let $A\to B$ be a map of derived $p$-complete simplicial rings such that $A$ is a classical ring, and $A \left\{ 1 \right\}$ be an invertible $A$-module (concentrated on degree 0). We define the graded (complete) $B$-algebra
    \[ \dHodge_{B/A} = \cbigoplus_{k \ge 0} \left( \cbigwedge^k(\cdL_{B/A}) \right) \left[ -k \right] \left\{ -k \right\}, \]
    where $\left\{ -k \right\}$ means ${}\otimes_A A\left\{ 1 \right\}^{\otimes(-k)}.$ We remark that $\dHodge_{B/A}$ is basically a twist of the functor $\cbigwedge^*[-*] \circ \cdL_{-/-}$ (see Sect. \ref{sect: remarks on infty-cat}), and it is well-defined by the last paragraph of Example \ref{ex: shearing}. By Example \ref{ex: shearing} and the last subsection of Section \ref{sect: remarks on infty-cat}, the functors $\cdL_{-/-}, \cbigwedge^*[-*], (M_n) \mapsto (M_n\{-n\})$ all commute with colimits (in a suitable sense), so over a fixed base $A, A\{1 \}$, the functor over $\widehat\SCR_{A/}$
    \[ \dHodge_{-/-}: \Fun(\Delta^1, \widehat\SCR_{A/}) \to \widehat\grAlg_A \]
    commutes with arbitrary colimits, where 
    \[ \widehat\grAlg_A := \widehat\SCR_{A/} \times_{\Fun(\{0\},\CAlg(\gr(\cD(A))))} \Fun(\Delta^1, \CAlg(\gr(\cD(A)))) \]
    consists of pairs $(B,B')$ where $B$ is a derived $p$-complete simplicial ring over $A$ and $B'$ is a derived $p$-complete graded $\dE_\infty$-algebra over $B$. Since $\widehat\grAlg_A \to \gr(\cD(A))$ commutes with colimits, we see that the induced functor
    \[ \dHodge_{-/-}: \Fun(\Delta^1, \widehat\SCR_{A/}) \to \gr(\cD(A)) \]
    also commute with colimits.
\end{construction}

\subsection{Prismatic cohomology}

As in \cite{bhatt2022absolute}, we will use $\Prism_{R/A}$ to denote the derived prismatic cohomology.

In this paper, we will frequently use the following theorem:

\begin{theo}[{\cite[Theorem 4.3.6]{bhatt2022absolute}, affine case}]\label{theo: comparison}
Let $R$ be a ring over $\overline A$ satisfying Condition \ref{cond: tor -1,0}. Then, the natural map
\[ \Prism_{R/A} \to R\Gamma \bigl( ( R/A )_\Prism, \calO \bigr) \]
is an isomorphism.
\end{theo}

As a corollary,
\begin{cor}[{\cite[Remark 4.3.9]{bhatt2022absolute}}]\label{cor: universal property tor -1}
If $R/\overline A$ satisfies Condition \ref{cond: tor -1}, then $\Prism_{R/A}$ is a bounded prism, and is the final object of $(R/A)_\Prism$ given by the natural maps of derived prismatic cohomology
\[\xymatrix{
R \ar[r] & \overline\Prism_{R/A} \\
A \ar[r] \ar[u] & \Prism_{R/A} \ar[u]
}\]
\end{cor}

\begin{proof}
This is essentially the argument as in \cite[Remark 4.3.9]{bhatt2022absolute} (which uses the argument in \cite[Lemma 7.8]{Bhatt_2022}), and we repeat it here for the sake of clarity.

By Hodge--Tate comparison, $\overline\Prism_{R/A}$ is $p$-complete and $p$-completely flat over $R$. Therefore, by Proposition \ref{prop: p-completely flat}, $\overline \Prism_{R/A}$ is a classically $p$-complete $R$-module concentrated on degree 0. and is of bounded $p^\infty$-torsion. By Proposition \ref{prop: conc on degree 0 by reduced}, we see that $\Prism_{R/A}$ is concentrated on degree 0, and is $I$-torsion-free. For each $(B,IB) \in (R/A)_\Prism$, denote $\psi_B: \Prism_{R/A} \to B$ the map given by the composition of the isomorphism
\[ \psi: \Prism_{R/A} \to R\Gamma \bigl( (R/A)_\Prism, \calO \bigr) \]
in Theorem \ref{theo: comparison} and the projection $R\Gamma \bigl( (R/A)_\Prism, \calO \bigr) \to B$. By the argument as in \cite[Lemma 7.7]{Bhatt_2022}, $\psi_B$ is a map in $(R/A)_\Prism$. For any $f:B \to C$ in $(R/A)_\Prism$ we have $f \psi_B = \psi_C$ (the commutation is in the ordinary-categorical sense since the involved objects are discrete). Take $f = \psi_{\Prism_{R/A}}: \Prism_{R/A} \to \Prism_{R/A}$, and we see that $e:=\psi_{\Prism_{R/A}}$ is an idempotent morphism. Let $B_0$ be the object given by the retract of $e$, say, $r: \Prism_{R/A} \to B_0$. Therefore, the map
\[ \psi_{B_0}: \Prism_{R/A} \xrightarrow{\psi} R\Gamma \bigl( (R/A)_\Prism, \calO \bigr) \to B_0 \]
is an isomorphism. We have $r = re = r\psi_{\Prism_{R/A}} = \psi_{B_0}$, so $r$ is an isomorphism, and it follows that $e$ is an isomorphism. As $e$ is idempotent, we see that $e=\id$. For any morphism $f: \Prism_{R/A} \to B$ in $(R/A)_\Prism$, we have $f=fe = f\psi_{\Prism_{R/A}} = \psi_B$, so $\Prism_{R/A}$ is the initial object of $(R/A)_\Prism$. (Note that $e$ is idempotent follows crucially from the categorical argument, instead of direct calculations.)
\end{proof}

\begin{rmk}
Given $(A \to R) \to (B \to \overline B),$ we can describe the map $\Prism_{R/A} \to B$ explicitly. Compose it with the natural isomorphism $B \to \Prism_{\overline B/B},$ and we get a map $\varphi: \Prism_{R/A} \to \Prism_{\overline B/B}.$ Using the universal property as above, we see that the map $\varphi$ agrees with the map induced by functoriality from $(A \to R) \to (B/\overline B)$.
\end{rmk}

\begin{cor}\label{cor: identification of prismatic maps}
Let $A \to S$ be a map of bounded prisms satisfying Condition \ref{cond: flat + tor 0}, and that $\overline S \to R$ be a map satisfying Condition \ref{cond: tor -1}. We can define the map $\Prism_{R/A} \to \Prism_{R/S}$ in two ways: one is by functoriality $(R,A) \to (R,S),$ and the other is by projection
\[ \Prism_{R/A} \cong \lim_{(B,IB) \in (R/A)_\Prism} B \xrightarrow{B=\Prism_{R/S}} \Prism_{R/S}. \]
Then, the two maps are naturally identified with each other.
\end{cor}

\begin{proof}
For prisms $A \to B$ and $R/\overline A, R'/\overline B,$ we denote the natural map $p_{(A\to R) \to (B\to R')}: \Prism_{R/A} \to \Prism_{R'/B}.$

For any bounded $p$-complete ring $R/\overline A,$ we have the map
\[ \xi_R: \Prism_{R/A} \to \lim_{(B,IB) \in (R/A)_\Prism} B \]
given in \cite[Proposition 4.3.4]{bhatt2022absolute}. Denote its projections by $\varphi_{R/A,B}: \Prism_{R/A} \to B$. For any bounded prism $B$, denote the natural isomorphism $\psi_B: B \to \Prism_{\overline B/B},$ so $\psi_B$ is naturally identified with the inverse of $\varphi_{\overline B/B,B}.$ Functoriality gives a natural isomorphism
\[ \varphi_{\overline B/B,B} p_{(A\to R) \to (B \to \overline B)} \cong \varphi_{R/A,B}, \]
hence giving a natural isomorphism
\[ \psi_B \varphi_{R/A,B} \cong p_{(A \to R) \to (B \to \overline B)}. \]

Now, given the claim, for the map
\[ \varphi_{R/A,\Prism_{R/S}}: \Prism_{R/A} \to \lim_{(B,IB) \in (R/A)_\Prism} B \to \Prism_{R/S} \]
we see that
\[ \psi_{\Prism_{R/S}} \varphi_{R/A,\Prism_{R/S}} \cong p_{(A \to R) \to (\Prism_{R/S} \to \overline \Prism_{R/S})} \cong p_{(S \to R) \to (\Prism_{R/S} \to \overline\Prism_{R/S})} p_{(A \to R) \to (S \to R)}. \]
Also, by the remark following Corollary \ref{cor: universal property tor -1}, we see that $\psi_{\Prism_{R/S}}$ is identified with $p_{(S \to R) \to (\Prism_{R/S} \to \overline\Prism_{R/S})}.$ Therefore, since $\psi_{\Prism_{R/S}}$ is an isomorphism, we see that $\varphi_{R/A,\Prism_{R/S}}$ is identified with $p_{(A \to R) \to (S \to R)},$ as desired.
\end{proof}

\begin{cor}\label{cor: animated universal property}
Let $(A,I)$ be a bounded prism and $R/\overline A$ be a ring satisfying Condition \ref{cond: tor -1}. Define the category $(R/A)_\Prism^\an$ consisting of animated $(p,I)$-adically complete $\delta$-rings $B/A$ equipped with a map $R \to \overline B := \overline A \otimes_A^\bL B$. Then, $(R/A)_\Prism^\an$ admits an initial object $\Prism_{R/A}$.
\end{cor}

\begin{proof}
For $B \in (R/A)_\Prism^\an$ define $\eta_B \in \Map_{(R/A)_\Prism^\an}( \Prism_{R/A}, B )$ by
\[ \eta_B: \Prism_{R/A} \to \Prism_{\overline B/B} \cong B. \]
Denote $B_0 = \Prism_{R/A}$, then by the remark following Corollary \ref{cor: universal property tor -1}, $\eta_{B_0} = \id_{B_0}.$ For any other map $f: B_0 \to B$, we have a natural equivalence
\[ f \cong f \circ \eta_{B_0} = \eta_{B}. \]
Denote $\calC_0 = \Map(B_0,B)$. We have maps $\eta_B: \{* \} \to \calC_0$ and the unique projection $\pi: \calC_0 \to \{*\}$. It is clear that $\pi \circ \eta_B = \id_{\{*\}}$. The above argument provides an equivalence $\id_{\calC_0} \cong \underline{\eta_B} = \eta_B \circ \pi$, so $\pi$ is an equivalence of categories, as desired.
\end{proof}

\begin{prop}\label{prop: weakly final}
Assume that $S/A$ satisfies Condition \ref{cond: flat + tor 0}, and that $R/\overline S$ satisfies Condition \ref{cond: tor -1}. Then, $\Prism_{R/S} \to *$ is a covering in the topos $\Shv \left( (R/A)_\Prism \right).$ Moreover, it is represented by a $(p,I)$-completely faithfully flat affine cover.
\end{prop}

\begin{proof}
Let $\left( B,IB \right)$ be any object of $(R/A)_\Prism.$ There exists a natural morphism $S \cotimes_A^\bL B \to \overline B$, where $S \to \overline B$ is given by $S \to R \to \overline B.$ By conditions, $S \cotimes_A^\bL B$ is $\left( p,I \right)$-completely flat over $B$, hence is concentrated on degree 0 by Proposition \ref{prop: (p,I)-completely flat}. We verify that this morphism satisfies Condition \ref{cond: tor -1}.

By definition, $B$ is a bounded prism, so we need to prove that $\cdL_{\overline B/\overline S \cotimes_{\overline A}^\bL \overline B} \left[ -1 \right]$ is $p$-completely flat over $\overline B.$ Consider the map of rings $\overline B \to \overline S \cotimes_{\overline A}^\bL \overline B \to \overline B,$ and we get an exact triangle of cotangent complexes, inducing an isomorphism
\[ \cdL_{\overline B/\overline S \cotimes_{\overline A}^\bL \overline B} \left[ -1 \right] \cong \overline B \cotimes^\bL_{\overline S \cotimes_{\overline A}^\bL \overline B} \cdL_{\overline S \cotimes_{\overline A}^\bL \overline B / \overline B}. \]
As $\cdL_{\overline S \cotimes_{\overline A}^\bL \overline B / \overline B} \cong \cdL_{\overline S/\overline A} \cotimes_A^\bL \overline B$ is $p$-completely flat over $\overline S \cotimes_{\overline A}^\bL \overline B,$ we see that $\cdL_{\overline B/\overline S \cotimes_{\overline A}^\bL \overline B} \left[ -1 \right]$ is $p$-completely flat over $\overline B,$ as desired.

Now define $C = \Prism_{\overline B/S \cotimes_A^\bL B}.$ By verifying the universal property (Corollary \ref{cor: universal property tor -1}), we see that $\left( C,IC \right)$ is the product of $\left( B,IB \right)$ and $\left( \Prism_{R/S},I\Prism_{R/S} \right)$ in $( R/A )_\Prism.$ Since $\overline B \to \overline C$ is $p$-completely faithfully flat by Hodge--Tate comparison (faithfulness follows from $\Fil_0$), we see that $B \to C$ is also $\left( p,I \right)$-completely faithfully flat, so $\left( C,IC \right) = \left( \Prism_{R/S},I\Prism_{R/S} \right) \times \left( B,IB \right)$ is a $(p,I)$-completely faithfully flat affine cover of $\left( B,IB \right)$, as desired.
\end{proof}

\begin{defi}\label{defi: crystal}
Let $\calC$ be a category and $\calO$ be a functor from $\calC$ to commutative rings (we will use the covariant convention unless otherwise stated). A $\calO$-crystal on $\calC$ is a functor $\calE$ of $\calO$-modules, such that each $\calE \left( x \right)$ is finite projective over $\calO \left( x \right)$, and for each morphism $x \to y$, the induced map
\[ \calO \left( y \right) \otimes_{\calO \left( x \right)} \calE \left( x \right) \to \calE \left( y \right) \]
is an isomorphism. The category of $\calO$-crystals will be denoted as
\[ \Vect \left( \calC, \calO \right). \]
\end{defi}

\begin{lemma}\label{lemma: equivalence of crystals}
Let $F: \calC' \to \calC$ be a functor between small categories, and $\calO$ be a functor from $\calC$ to commutative rings. Assume that $F$ is initial in the sense of infinite categories. Then, the natural map
\[ F^*: \Vect(\calC,\mathcal O) \to \Vect(\calC',F^* \mathcal O) \]
is an equivalence of symmetric monoidal categories.
\end{lemma}

\begin{proof}
Since $F^*$ is a symmetric monoidal functor by definition, we only need to verify that the underlying functor of $F^*$ is an equivalence.

Its inverse functor is given as follows: for any vector bundle $\cal E'$ of $F^*\calO$ on $\cal C'$, define $\calE$ by
\[ \calE(x) = \colim_{x' \in \calC'_{/x}} \calO(x) \otimes_{\calO(Fx')}\calE'(x'). \]
Since $F$ is initial, $\calC'_{/x}$ is weakly contractible. Also, in the colimit above, the maps are quasi-isomorphisms. It follows that the maps
\[ \calO(x) \otimes_{\calO(Fx')}\calE'(x') \to \calE(x) \]
are all isomorphisms. We omit the verification that $\calE(x)$ is indeed a crystal, and that the construction above is an inverse of $F^*$.
\end{proof}

\begin{lemma}\label{lem: weakly final computation}
Assume that $R/\overline A$ is $p$-complete with bounded torsion, and $(B,IB) \in (R/A)_\Prism$ is an element such that $B \to *$ is represented by a faithfully flat cover. Denote its $(n+1)$-fold product by $(B^n,IB^n)$, so $(B^\bullet,IB^\bullet)$ is a (co)simplicial object in $(R/A)_\Prism$. Then,
\begin{enumerate}[label=(\arabic*)]
\item for any vector bundle $\calE$ of $\overline\calO$-modules or $\overline\calO\bigl[ \frac 1 p \bigr]$-modules on $(R/A)_\Prism$, $\calE$ is a sheaf of spectra, and the natural maps
\[ \lim_{(B',IB') \in (R/A)_\Prism} \calE(B',IB') \to R\Gamma((R/A)_\Prism,\calE) \]
and
\[ \lim_{(B',IB') \in (R/A)_\Prism} \calE(B',IB') \to \lim_\simp \calE(B^\bullet,IB^\bullet) \]
are isomorphisms.

\item we have equivalences of symmetric monoidal categories given by restriction:
\begin{align*}
\Vect\bigl((R/A)_\Prism,\overline\calO\bigr) & \to \Vect\bigl(\simp, \overline B^\bullet\bigr), \\
\Vect\bigl((R/A)_\Prism,\overline\calO\bigl[ \frac 1 p \bigr]\bigr) & \to \Vect\bigl(\simp, \overline B^\bullet\bigl[\frac 1 p \bigr]\bigr).
\end{align*}
\end{enumerate}
\end{lemma}

\begin{proof}
These are direct consequences of Theorem \ref{thm: descent theory}. More precisely, Theorem \ref{thm: descent theory} shows that $\calE$ is a sheaf of spectra, and the following conclusions follow from this and that $B \to *$ is an effective epimorphism\footnote{In the classical language, see\cite[\href{https://stacks.math.columbia.edu/tag/03F9}{Tag 03F9}]{stacks-project} and \cite[\href{https://stacks.math.columbia.edu/tag/03AZ}{Tag 03AZ}]{stacks-project}, note that taking limit over all $(B',IB')$ is nothing but computing $R\Gamma$ as presheaves.}. The item (2) follows from that $B' \mapsto \Vect(B'), \Vect\bigl(B' \bigl[ \frac 1 p \bigr]\bigr)$ are stacks (= sheaves of 1-groupoids), and the equivalence is the stacky condition for the cover $B \to *$.
\end{proof}

We will also need the Hodge--Tate comparison:
\begin{prop}[{\cite[Construction 7.6]{Bhatt_2022}}]\label{prop: HT comparison}
Let $(A,I)$ be a bounded prism. For any animated ring $R$ over $A$, $\overline\Prism_{R/A}$ admits an (exhaustive in the sense of derived $p$-complete modules) increasing filtration, with a natural isomorphism $\gr_* \overline\Prism_{R/A} \cong \dHodge_{R/A}$ as functors $\widehat\SCR_{A/} \to \widehat\grAlg_A$ over $\widehat\SCR_{A/}$, where $\widehat\grAlg_A$ consists of pairs $(B,B')$ where $B$ is a derived $p$-complete simplicial ring over $A$ and $B'$ is a derived $p$-complete graded $\dE_\infty$-algebra over $B$.
\end{prop}

\section{De Rham Realization of rational Hodge--Tate prismatic crystals of quasi-l.c.i algebras}

In this section, we prove Theorem \ref{theo: main dR realization}.

\subsection{Cosimplicial rings and stratifications}

This section mainly consists of technical results. Recall that for any classically $p$-complete ring $R$ and vector bundle $E/R$, we have the gerbe $B(E^\vee)^\sharp$ given as the colimit of the simplicial object
\[ [n] \mapsto \Spf \bigl(\cGamma^\bullet_R (E^{\oplus n}) \bigr). \]
Therefore, a vector bundle $\calE$ over $B(E^\vee)^\sharp$ is canonically identified with descent data $M^\bullet$ over the cosimplicial ring $\cGamma^\bullet_R (E^{\oplus n})$, which is further identified with an integrable Higgs bundle
\[ \nabla: M \to M \otimes_R E \]
over $R$. We also have a symmetric monoidal comparison of cohomologies
\[ R\Gamma(\calE) \cong \Tot( M^\bullet ) \cong  \DR(M, \dif: M \to M \otimes_R E). \]

In this section, we generalise the above result in two aspects: we are going to replace $E$ by a \emph{$p$-complete and $p$-completely flat module} $\Omega$ over $R$, and we replace the gerbe $B(E^\vee)^\sharp$ by an affine $p$-completely flat formal stack $\mathfrak X$ over $B(\Omega^\vee)^\sharp$. Each of such $\mathfrak X$ is identified with a \emph{split PD ring} $S^\bullet$ over $R$ (see Condition \ref{cond: split pd}). As above, vector bundles over $\mathfrak X$ are identified with a stratification of vector bundles $M^\bullet/S^\bullet$, which is further identified with an integrable Higgs bundle
\[ \nabla: M \to M \cotimes_R \Omega \]
compatible with a differential $\dif: S \to S \cotimes_R \Omega$ on $S$. Again we have a comparison of cohomologies. These statements are made precise in Proposition \ref{prop: strat to connection}, Theorem \ref{theo: comp strat connection} and Theorem \ref{theo: comp strat connection cohomologies}.

Also, we won't be directly constructing these stacks and gerbes, but rather take an algebraic approach. Interested readers can do the translation job by themselves.

Consider the condition as below:
\begin{cond}\label{cond: very conc deg 0}
    Let $R$ be a classically $p$-completely ring with bounded $p^\infty$-torsion. We say that an $R$-module $M$ satisfies this condition, if either of the following two conditions holds:
    \begin{enumerate}[label=(\arabic*)]
        \item $M$ is classically $p$-complete with bounded $p^\infty$-torsion,
        \item $M$ is an adic module over $R \bigl[ \frac 1 p \bigr].$
    \end{enumerate}
\end{cond}

We will often consider cosimplicial objects. We denote $\left[ n \right] = \left\{ 0,\ldots,n \right\},$ and denote
\begin{itemize}
\item $\delta_i^n: \left[ n-1 \right] \to \left[ n \right]$ the injection missing $i$ (with $0 \le i \le n$).
\item $\sigma_i^n: \left[ n+1 \right] \to \left[ n \right]$ the surjection hitting $i$ twice (with $0 \le i \le n$).
\item $q_i^n: \left[ 0 \right] \to \left[ n \right]$ sending $0$ to $n$ (with $0 \le i \le n$).
\item $\alpha_i^n: \left[ 1 \right] \to \left[ n \right]$ sending $t$ to $t+i$ (with $0 \le i \le n-1$).
\item $\beta_{i,j}^n: \left[ 1 \right] \to \left[ n \right]$ sending $0$ to $i$ and $1$ to $j$ (with $0 \le i \le j \le n$).
\end{itemize}

We impose the following condition on the cosimplicial ring:

\begin{cond}\label{cond: split pd}
Let $R$ be classically $p$-complete with bounded $p^\infty$-torsion, and $S^\bullet$ be a cosimplicial ring over $R$, such that each $S^n$ is $p$-adically complete and $p$-completely flat over $R$. We say that the cosimplicial ring is \emph{split PD} over $R$, if
\begin{enumerate}[label=(\arabic*)]
\item each $\ker \left( S^n \to S^0 \right)$ is equipped with a PD structure compatible with the cosimplicial ring structure,
\item we are given a $p$-complete and $p$-completely flat $R$-module $\Omega$ and a $R$-linear map $\varphi: \Omega \to \ker \left( S^1 \to S^0 \right),$ such that $\bigl( \delta_0^2-\delta_1^2+\delta_2^2 \bigr) \varphi =0,$ and the natural map
\[ q_0^n \otimes \left( \alpha_0^n,\ldots,\alpha_{n-1}^n \right): S^0 \cotimes_R \Omega^{\oplus n} \to S^n \]
(where $\alpha_i^n: [1] \to [n]$ is given by $\alpha_i^n \left( t \right) = i+t$) induces an isomorphism
\[ S^0 \cotimes_R \cGamma^\bullet_R \bigl( \Omega^{\oplus n} \bigr) \to S^n, \]
where $S^n$ is viewed as an $S^0$-algebra via the inclusion $i^n: [0] \to [n]$. A simple induction would show that
\[ \beta^n_{i,j} \left( \varphi( \omega) \right) = \sum_{k=i}^{j-1} \alpha_i^n \left( \varphi( \omega) \right), \forall \omega \in \Omega. \]
\end{enumerate}

Although not explicitly stated, it should be noticed that $\Omega \to \ker (S^1 \to S^0)$ is also part of the data of a split PD structure. When we are going to emphasise on the role of this map, we shall write $(S^\bullet, \varphi: \Omega \to S^1)$.
\end{cond}

\begin{lemma}\label{lem: commutative constraint}
    There exists a unique PD ring map $\tau: S^1 \to S^1$ switching $q_0^1 (s)$ and $q_1^1(s),$ and takes $\varphi \left( \omega \right)$ to $\varphi \left( -\omega \right).$
\end{lemma}

\begin{proof}
Uniqueness is evident. As for existence, by our definition, there exists a unique PD ring map $\tau: S^1 \to S^1$ \emph{taking} $q_0^1(s)$ to $q_1^1(s)$ and $\varphi \left( \omega \right)$ to $\varphi \left( -\omega \right).$ We only need to verify that $\tau \circ q_1^1 = q_0^1$.

Now, there exists a unique PD ring map $\epsilon: S^2 \to S^1$ taking $q_0^2(s)$ to $q_0^1(s)$, $\alpha_0^2 \left( \varphi(\omega) \right)$ to $\varphi(\omega),$ and $\alpha_1^2 \left( \varphi(\omega) \right)$ to $-\varphi(\omega).$ By verifying on the basis, it is clear that $\epsilon \circ \delta_2^2 = \id_{S^2}$. We have
\begin{align*}
\epsilon \circ \delta_1^2 (q_0^1(s)) & = \epsilon \circ q_0^2 (s) = q_0^1 (s), \\
\epsilon \circ \delta_1^2(\varphi(\omega)) & = \epsilon \circ (\alpha_0^2+\alpha_1^2)(\varphi(\omega)) = \varphi(\omega) -\varphi(\omega) =0, \\[10pt]
\epsilon \circ \delta_0^2 (q_0^1(s)) & = \epsilon \circ q_1^2 (s) = \epsilon \circ \delta_2^2 \circ q_1^1 (s) = q_1^1(s), \\
\epsilon \circ \delta_0^2(\varphi(\omega)) & = \epsilon \circ \alpha_1^2(\varphi(\omega)) = -\varphi(\omega).
\end{align*}
It follows that $\epsilon \circ \delta_1^2 = \delta_1^1 \sigma_0^0,$ and $\epsilon \circ \delta_0^2 = \tau,$ so we have
\begin{align*}
\epsilon ( q_2^2 (s) ) & = \epsilon(\delta_0^2 q_1^1(s)) = \tau (q_1^1 (s)), \\
\epsilon ( q_2^2 (s) ) & = \epsilon(\delta_1^2 q_1^1(s)) = \delta_1^1 \sigma_0^0 q_1^1(s) = q_0^1(s).
\end{align*}
Therefore, for any $s\in S^0$ we have
\[ \tau(q_1^1(s)) = q_0^1(s), \]
as desired.
\end{proof}

\begin{prop}\label{prop: commutative constraint}
Let $S^\bullet$ be split PD over $R$. Let $\mathscr F_+$ be the category whose objects are $\left[ n \right],$ but whose morphisms are all morphisms as sets, so $\simp$ is a non-full subcategory of $\mathscr F_+$. Then, there exists a unique extension of the functor $S^\bullet: \simp \to \cAlg_{R,\Flat,\PD}$ (where $\cAlg_{R,\Flat,\PD}$ denotes the category of PD algebras $(S,I,\gamma)$ over $R$ where $S$ and $S/I$ are $p$-complete and $p$-completely flat over $R$) to a functor $S(\bullet): \mathscr F_+ \to \cAlg_{R,\Flat,\PD}$, such that the transposition $\left[ 1 \right] \to \left[ 1 \right]$ is sent to $\tau$ as in Lemma \ref{lem: commutative constraint}.

In particular, for all injective maps $\varphi: \left[ i \right] \to \left[ j \right]$, it induces an isomorphism
\[ S^j \cong S^i \cotimes_R \cGamma_R^\bullet \bigl( \Omega^{\oplus ( j-i )} \bigr), \]
hence the induced map $S^i \to S^j$ is also $p$-completely flat.
\end{prop}

\begin{proof}
We first prove uniqueness. Observe that any map $f: \left[ 1 \right] \to \left[ n \right]$ is either in $\simp$, or that $f\tau$ is in $\simp.$ It follows that $S (f)$ is uniquely determined. Now, any ring $S^m$ is generated by all $\beta_{i,j}^n (S^1),$ so $S^m \to S^n$ is uniquely determined by looking at its composition with $\beta_{i,j}^n S^1 \to S^m.$

To show existence, we first denote $\beta_{j,i}^n = \beta_{i,j}^n \tau: S^1 \to S^n$ (where $0 \le i \le j \le n$). Note that $\beta_{i,i}^n = q_i^n \sigma_0^0 = q_i^n \sigma_0^0 \tau = \beta_{i,i}^n \tau,$ so the above map $\beta_{j,i}^n$ is well-defined. We claim that for any map $f: \left[ m \right] \to \left[ n \right]$ in $\mathscr F_+$, there exists a unique PD ring homomorphism $\widetilde f: S^m \to S^n$ such that $\widetilde f \beta_{i,j}^m = \beta_{f(i),f(j)}^n$ for all $0 \le i,j \le m.$ This would then complete the proof.

Uniqueness is still evident. As for existence, we define $\widetilde f: S^m \to S^n$ as the unique ring map taking $q_0^m(s)$ to $q_{f(0)}^n(s),$ and taking $\alpha_i^m(\varphi(\omega))$ to $\beta_{f(i),f(i+1)}^n(\varphi(\omega)).$ Using the additive property of $\beta,$ we see that $\widetilde f \circ \beta_{i,j}^m (\varphi(\omega)) = \beta^n_{f(i),f(j)} (\varphi(\omega)).$ Verify on a basis, and we see that
\[ \widetilde f \circ \beta_{0,i}^m = \beta_{f(0),f(i)}^n \]
for all $0 \le i \le m.$ Compose $q_1^1$ on the right-hand side, and we see that $\widetilde f \circ q_i^m = q_{f(i)}^n$ for all $0 \le i \le m.$ Verify again on a basis, and we see that $\widetilde f \circ \beta_{i,j}^m = \beta_{f(i),f(j)}^m,$ as desired.
\end{proof}

\begin{rmk}
Lemma \ref{lem: commutative constraint} and Proposition \ref{prop: commutative constraint} can be proved using the abstract nonsense of \emph{groupoid objects} (see for example \cite[Definition 6.1.2.7]{lurie2009higher} or \cite[\href{https://stacks.math.columbia.edu/tag/0230}{Tag 0230}]{stacks-project}) --- using the definitions, one can show that $S^\bullet$ is a cogroupoid object in $\cAlg_{R,\Flat,\PD}$. In fact, even the (image) of the transposition itself can be uniquely determined by the $\simp$-structure. This follows from the simple observation as follows: given a cogroupoid $\simp \to \cAlg_{R,\Flat,\PD}$, any extension functor $\mathscr F_+ \to \cAlg_{R,\Flat,\PD}$ is the left Kan extension of the original functor.

A concrete way to show that the extension is uniquely determined is that given any extension to a functor $\mathscr F_+ \to \cAlg_{R,\Flat,\PD}$, (the image of) the morphism $\tau$ fits into the following commutative diagram which will uniquely determine $\tau$
\[\begin{tikzcd}
S^2 \ar[rd,dashed,"\varepsilon"] & S^1 \ar[d,dotted,"\tau"] \ar[l,"\delta_0^2"'] \\
S^1 \cotimes_{S^0} S^1 \ar[u,"{(\delta_2^2,\delta_1^2)}","\cong"'] \ar[r,"{(\id,\delta_1^1\sigma_0^0)}"'] & S^1
\end{tikzcd}\]
with $\varepsilon:[2] \to [1]$ given by $(0,1,2) \mapsto (0,1,0)$. However, we will not be needing this property in the future.
\end{rmk}

For future reference, we will also need a filtered version of split-PD-ness.

\begin{defi}[Filtration on split PD rings]\label{defi: filtered split PD}
Let $S^\bullet$ be split PD over $R$. We say that $S^\bullet$ is \emph{filtered split PD} over $R$, if $S^\bullet$ is equipped with an \emph{increasing} filtration $\Fil_d (S^n),$ such that
\begin{itemize}
\item each graded piece $\gr_d (S^n)$ is $p$-complete and $p$-completely flat over $R$,
\item $S^n = \colim_{d \in \mathbb Z_{\ge 0}} \Fil_d( S^n )$ in the category $\cD(R)$,
\item the filtrations are compatible with the cosimplicial structure and the PD structure (i.e. for $x \in \ker (S^n \to S^0) \cap \Fil_d (S^n)$ we have $\gamma_r(x) \in \Fil_{dr} (S^n)$),
\item the isomorphisms
\[ S^0 \cotimes_R \cGamma_R^\bullet (\Omega^{\oplus n}) \to S^n \]
are strict isomorphisms of filtered PD rings over $R$ (where $\cGamma^d(\Omega)$ is of degree $d$).
\end{itemize}

Under the above assumptions, it is easy to verify that the map $\tau$ as in Lemma \ref{lem: commutative constraint} is compatible with filtrations (i.e. $\tau(\Fil_d) \subseteq \Fil_d$), therefore the unique $\mathscr F_+$-structure as in Proposition \ref{prop: commutative constraint} will also be compatible with the filtrations.
\end{defi}

\begin{defi}\label{defi: graded split PD}
Similarly, we can define \emph{graded split PD} cosimplicial rings over $R.$ Clearly, if $S^\bullet$ is a filtered split PD cosimplicial ring over $R$, then $\gr_* (S^\bullet)$ is a graded split PD cosimplicial ring over $R$.

For example, let $A \to P \to R$ be a sequence of maps of $p$-complete rings, such that $A \to P$ is $p$-completely flat, $\cOmega_{P/A} := \cdL_{P/A}$ is $p$-completely flat over $P$, and that $R/P$ satisfies Condition \ref{cond: tor -1}. We also fix an invertible module $A \left\{ 1 \right\}$ over $A$. We now take $P(n) = P^{\cotimes(n+1)/A}$ and
\[ S(n) = \dHodge_{R/P(n)} \cong \cGamma^\bullet_R \bigl( \cdL_{R/P(n)}[-1]\{-1\} \bigr). \]
Take $\Omega = R \cotimes_P \cOmega_{P/A} \left\{ -1 \right\},$ and $\Omega \to S(1)$ defined by
\[ \Omega \cong R \cotimes_P \cdL_{P/P(1)}[-1]\{-1\} \to \cdL_{R/P(1)} [-1] \{-1\}. \]
Using some arguments of commutative algebra, it is easy to see that the above data gives a graded split PD structure on $S(\bullet) = \dHodge_{R/P(\bullet)}.$
\end{defi}

\begin{defi}[Stratifications]\label{defi:strat}
    For any split PD cosimplicial ring $S^\bullet$ (where each $S^n$ is classically $p$-complete with bounded $p^\infty$-torsion) over $R$, a \emph{stratification} along $S^\bullet$ is a classically $p$-complete module $M/S^0$ with bounded $p^\infty$-torsion, as well as an $S^1$-linear map
    \[ \ep:S^1\cotimes_{\delta_0^1,S^0}M\to S^1\cotimes_{\delta_1^1,S^0}M, \]
    such that:

    (1) The following diagram commutes:
    \[ \xymatrixcolsep{0.2in}
    \xymatrix@C-2.3pc{ & S^2\cotimes_{\delta_0^2,S^1}S^1\cotimes_{\delta_0^1,S^0}M\ar[rr]^-{\id_{S^2}\otimes_{\delta_0^2,S^1}\ep}\ar@{=}[d]&&S^2\cotimes_{\delta_0^2,S^1}S^1\cotimes_{\delta_1^1,S^0}M\ar@{=}[d]&\\
    S^2\cotimes_{\delta_1^2,S^1}S^1\cotimes_{\delta_0^1,S^0}M\ar[ddr]_-{\id_{S^2}\otimes_{\delta_1^2,S^1}\ep}\ar@{=}[r]&S^2\cotimes_{q_2^2,S^0}M&&S^2\cotimes_{q_1^2,S^0}M\ar@{=}[r]&S^2\cotimes_{\delta_2^2,S^1}S^1\cotimes_{\delta_0^1,S^0}M\ar[ldd]^-{\id_{S^2}\otimes_{\delta_2^2,S^1}\ep}\\
    &&S^2\cotimes_{q_0^2,S^0}M\ar@{=}[dl]\ar@{=}[dr]&&\\
    &S^2\cotimes_{\delta_1^2,S^1}S^1\cotimes_{\delta_1^1,S^0}M&&S^2\cotimes_{\delta_2^2,S^1}S^1\cotimes_{\delta_1^1,S^0}M&} \]

    (2) We have a commutative diagram
    \[\xymatrix{
     S^1 \cotimes_{\delta_0^1,S^0} M \ar[rr]^\epsilon \ar[rd]_{\sigma_0^0 \otimes \id} & &  S^1 \cotimes_{\delta_1^1, S^0} M \ar[ld]^{\sigma_0^0 \otimes \id} \\
     & M
    }\]
\end{defi}

In fact, we can deal with a slightly more general class of stratifications.

We will also employ a `rational' version of the stratification set-up as above:

\begin{defi}\label{defi:rational strat}
    For any split PD cosimplicial ring $S^\bullet$ (where each $S^n$ is classically $p$-complete with bounded $p^\infty$-torsion), a \emph{stratification} along $S^\bullet \bigl[ \frac 1 p \bigr]$ is an adic module $M/S^0 \bigl[ \frac 1 p \bigr]$ as well as a continuous $S^1$-linear map
    \[ \ep:S^1\cotimes_{\delta_0^1,S^0}M\to S^1\cotimes_{\delta_1^1,S^0}M, \]
    such that:

    (1) The following diagram commutes:
    \[ \xymatrixcolsep{0.2in}
    \xymatrix@C-2.3pc{ & S^2\cotimes_{\delta_0^2,S^1}S^1\cotimes_{\delta_0^1,S^0}M\ar[rr]^-{\id_{S^2}\otimes_{\delta_0^2,S^1}\ep}\ar@{=}[d]&&S^2\cotimes_{\delta_0^2,S^1}S^1\cotimes_{\delta_1^1,S^0}M\ar@{=}[d]&\\
    S^2\cotimes_{\delta_1^2,S^1}S^1\cotimes_{\delta_0^1,S^0}M\ar[ddr]_-{\id_{S^2}\otimes_{\delta_1^2,S^1}\ep}\ar@{=}[r]&S^2\cotimes_{q_2^2,S^0}M&&S^2\cotimes_{q_1^2,S^0}M\ar@{=}[r]&S^2\cotimes_{\delta_2^2,S^1}S^1\cotimes_{\delta_0^1,S^0}M\ar[ldd]^-{\id_{S^2}\otimes_{\delta_2^2,S^1}\ep}\\
    &&S^2\cotimes_{q_0^2,S^0}M\ar@{=}[dl]\ar@{=}[dr]&&\\
    &S^2\cotimes_{\delta_1^2,S^1}S^1\cotimes_{\delta_1^1,S^0}M&&S^2\cotimes_{\delta_2^2,S^1}S^1\cotimes_{\delta_1^1,S^0}M&} \]

    (2) We have a commutative diagram
    \[\xymatrix{
     S^1 \cotimes_{\delta_0^1,S^0} M \ar[rr]^\epsilon \ar[rd]_{\sigma_0^0 \otimes \id} & &  S^1 \cotimes_{\delta_1^1, S^0} M \ar[ld]^{\sigma_0^0 \otimes \id} \\
     & M
    }\]
\end{defi}

\begin{prop}\label{prop: commutative constraint strat}
Let $S(\bullet): \mathscr F^+ \to \cAlg_{R,\Flat,\PD}$ be as in Proposition \ref{prop: commutative constraint}. Consider the category of $S(\bullet)$-modules $M(\bullet)$ satisfying Condition \ref{cond: very conc deg 0}, such that for each map $f: [m] \to [n],$ the map $S(n) \cotimes_{S(m)} M(m) \to M(n)$ are isomorphisms. Then, the functor $M(\bullet) \mapsto \left( M(0), \epsilon: S^1 \cotimes_{\delta_0^1, S^0} M(0) \to M(1) \to S^1 \cotimes_{\delta_1^1, S^0} M(0) \right)$ is an equivalence between the category of $M(\bullet)$ and the category of stratifications over $S^\bullet$ and $S^\bullet \bigl[ \frac 1 p \bigr].$
\end{prop}

\begin{proof}
This is some standard argument. The inverse functor is given by taking $M(n) = S^n \cotimes_{q_0^n, S^0} M$ and maps $S(n) \cotimes_{f, S(m)} M(m) \to M(n)$ given by
\[ S(n) \cotimes_{f,S(m)} M(m) \cong S(n) \cotimes_{\beta^n_{0,f(0)},S(1)} S(1) \cotimes_{\delta_0^1,S(0)} M \xrightarrow{\id \otimes \ep} S(n) \cotimes_{\beta^n_{0,f(0)},S(1)} S(1) \cotimes_{\delta_1^1,S(0)} M \cong M(n). \]
We omit the verifications.
\end{proof}

We now turn to differentials and connections.

\begin{defi}\label{defi: Higgs field}
Let $R$ be a classically $p$-complete ring with bounded $p^\infty$-torsion, and $\Omega$ be a $p$-complete and $p$-completely flat $R$-module.

Let $M$ be an $R$-module satisfying Condition \ref{cond: very conc deg 0}. Consider a \emph{continuous} (automatic when $M$ is $p$-adically complete of bounded $p^\infty$-torsion) $R$-linear map
\[ \nabla: M \to M \cotimes_R \Omega. \]
This map in turn induces maps
\[ \nabla: M \cotimes_R \cbigwedge^k(\Omega) \to M \cotimes_R \cbigwedge^{k+1}(\Omega) \]
via
\[ m \otimes \omega \mapsto \nabla(m) \wedge \omega. \]
We say that $\nabla$ is \emph{a Higgs field on $M$} if $\nabla \circ \nabla : M \to M \cotimes_R \cbigwedge^2(\Omega)$ is zero. It is easy to see that we may form the \emph{de Rham complex} $\DR(M,\nabla)$ of $\nabla$ as a differential graded module, with terms $M \cotimes_R \cbigwedge^k(\Omega)$ and differentials $\nabla: M \cotimes_R \cbigwedge^k(\Omega) \to M \cotimes_R \cbigwedge^{k+1}(\Omega)$. Fix $\Omega$. The category $(M,\nabla)$ of $p$-adically complete and $p$-completely flat modules $M/R$ equipped with a Higgs field $\nabla: M \to M \cotimes_R \Omega$ is a symmetric monoidal category: for $(M_1,\nabla_1)$ and $(M_2,\nabla_2)$, their tensor product is $(M_1 \cotimes_R M_2, \nabla_1 \otimes \id + \id \otimes \nabla_2)$. The unit object is $(R,0)$.

Assume that the algebra $S/R$ is $p$-complete and $p$-completely flat and $\dif: S \to S \cotimes_R \Omega$ is a Higgs field. If $\dif$ is a differential, then $(S,\dif)$ is an algebra in the symmetric monoidal category above, hence $\DR(S,\dif)$ is a differential graded algebra. For an $S$-module $M$ satisfying Condition \ref{cond: very conc deg 0} and a Higgs field $\nabla: M \to M \cotimes_R \Omega$, we say that $\nabla$ is \emph{an integrable $\dif$-connection}, if
\[ \nabla(sm) = m \otimes \dif(s) + s\nabla(m) \]
for each $s \in S$ and $m \in M$. In this case, $\DR(M,\nabla)$ (as a differential graded $R$-module) is canonically equipped with a structure of a $\DR(S,\dif)$-module. The category of integrable $\dif$-connections $(M,\nabla)$ where $M$ is $p$-complete and $p$-completely flat is symmetric monoidal by
\[ (M_1,\nabla_1) \otimes (M_2,\nabla_2) = (M_1 \cotimes_S M_2, \nabla(m_1 \otimes m_2) = \nabla_1(m_1) \otimes m_2 + m_1 \otimes \nabla_2(m_2)) \]
with unit object $(S,\dif)$. Similarly, the category of integrable $\dif$-connections $(M,\nabla)$ with $M/S\bigl[ \frac 1 p \bigr]$ finite projective is a symmetric monoidal category.
\end{defi}

\begin{lemma}\label{lem: connection to stratification}
Let $R$ be a classically $p$-complete ring with bounded $p^\infty$-torsion, and $\Omega$ be a $p$-complete and $p$-completely flat $R$-module.

Let $M$ be an $R$-module satisfying Condition \ref{cond: very conc deg 0}, equipped with a Higgs field
\[ \nabla: M \to M \cotimes_R \Omega. \]
Then, there exists a unique sequence of maps
\[ \nabla^n: M \to M \cotimes_R \cGamma^n \left( \Omega \right), \]
such that $\nabla^0=\id_M, \nabla^1=\nabla,$ and the following diagram commutes for all $n$:
\begin{equation}\label{eq: strat connection}
\xymatrix{
M \ar[r]^(.32){\nabla^n} \ar@{=}[d] & M \cotimes_R \cGamma^n \left( \Omega \right) \ar[r]^(.46){\id \otimes \cGamma^n \left( \Delta \right)} & M \cotimes_R \cGamma^n \left( \Omega \oplus \Omega \right) \ar[d]^(.37)\cong \\
M \ar[rr]_(.33){( \nabla^i \otimes \id ) \circ \nabla^{n-i}} & & \displaystyle\bigoplus_{i=0}^n M \cotimes_R \cGamma^i \left( \Omega \right) \cotimes_R \cGamma^{n-i} \left( \Omega \right)
}
\end{equation}
where $\Delta = \left( \id,\id \right): \Omega \to \Omega.$ Moreover, these $\nabla^n$ fits into the commutative diagram as follows:
\[\xymatrix{
& M \cotimes_R \cGamma^n \left( \Omega \right) \ar[rd] \\
M \ar[rr]_{\nabla^n_0} \ar[ru]^{\nabla^n} & & M \cotimes \Omega^{\cotimes n/R}
}\]
where $\cGamma^n \left( \Omega \right) \to \Omega^{\cotimes n/R}$ is given in Proposition \ref{prop: flat pd power}, and $\nabla_0^n$ is the composition of the chain below:
\[ M \xrightarrow{\nabla} M \cotimes_R \Omega \xrightarrow{\nabla \otimes \id} M \cotimes_R \Omega \cotimes_R \Omega \to \cdots \to M \cotimes_R \Omega^{\cotimes n/R}. \]
\end{lemma}

\begin{proof}
By definition $\nabla \circ \nabla =0$, and we see that the image $\nabla_0^n$ is invariant under the action of $\left( 1\ 2 \right).$ By induction, it is easy to see that the image of $\nabla_0^n$ is invariant under the action of $\Sigma_1 \times \Sigma_{n-1}.$ Combine these two, we see that the image of $\nabla_0^n$ is invariant under the action of $\Sigma_n.$ Therefore, by Proposition \ref{prop: flat pd power}, we see that $\nabla_0^n$ factors through the injection $\cGamma^n \left( \Omega \right) \to M \cotimes_R \Omega^{\cotimes n/R}.$

We now prove the commutativity of the first diagram. Using Proposition \ref{prop: flat pd power}, we may identify, for example, $M \cotimes_R \cGamma^n \left( M \right)$ with $M \cotimes_R \bigl( \Omega^{\cotimes n} \bigr)^{\Sigma_n}$. Therefore, we only need to prove that the following diagram commutes:
\[\xymatrix{
M \ar[r]^(.32){\nabla^n_0} \ar@{=}[d] & M \cotimes_R \bigl( \Omega^{\cotimes n} \bigr)^{\Sigma_n} \ar[r]^(.46){\id \otimes \Delta^{\otimes n}} & M \cotimes_R \left( \Omega \oplus \Omega \right)^{\cotimes n} \\
M \ar[rr]_(.33){( \nabla^i_0 \otimes \id ) \circ \nabla^{n-i}_0} & & \displaystyle\bigoplus_{i=0}^n M \cotimes_R \bigl( \Omega^{\cotimes i} \bigr)^{\Sigma_i} \cotimes_R \bigl( \Omega^{\cotimes ( n-i )} \bigr)^{\Sigma_{n-i}} \ar[u]_(.6){\id \otimes \psi}
}\]
where $\psi$ takes $\omega_i \otimes \omega_{n-i}$ (with $\omega_i \in \bigl( \Omega^{\cotimes i} \bigr)^{\Sigma_i}, \omega_{n-i} \in \bigl( \Omega^{\cotimes \left( n-i \right)} \bigr)^{\Sigma_{n-i}}$) to
\[ \sum_{\sigma \in \Sigma_n / \Sigma_i \times \Sigma_{n-i}} \sigma \bigl( i_1^{\otimes i} \left( \omega_i \right) \otimes i_2^{\otimes (n-i)} \left( \omega_{n-i} \right) \bigr), \]
where $i_1,i_2: \Omega \to \Omega \oplus \Omega$ are injections to the first and second components. Note that $\left( \nabla_0^i \otimes \id \right) \circ \nabla_0^{n-i} = \nabla_0^n,$ so we essentially need to prove that for any $\omega \in \bigl( \Omega^{\cotimes n} \bigr)^{\Sigma_n}$ (hence $\omega$ is fixed by $\Sigma_i \times \Sigma_{n-i} \subseteq \Sigma_n$), we have
\[ \Delta^{\otimes n} \left( \omega \right) = \sum_{i=1}^n \sum_{\sigma \in \Sigma_n / \Sigma_i \times \Sigma_{n-i}} \sigma \circ \bigl( i_1^{\otimes i} \otimes i_2^{\otimes (n-i)} \bigr) \left( \omega \right). \]
This is easy to verify. For any subset $I \subseteq = \left\{ 1,\ldots,n \right\},$ consider the map
\[ j_I: \Omega^{\cotimes n} \to \left( \Omega \oplus \Omega \right)^{\cotimes n}, \]
which acts on factors in $I$ by $i_1,$ and factors outside $I$ by $i_2$. It is easy to prove that
\[ \sum_I j_I = \Delta^{\otimes n}. \]
Now, for any $I$ with $\left| I \right| = i,$ we may take $\sigma_I \in \Sigma_n$ with $\sigma_I \left( \left\{ 1,\ldots,i \right\} \right) = I.$ It is easy to see that
\[ j_I = \sigma_I \circ \bigl( i_1^{\otimes i} \otimes i_2^{\otimes (n-i)} \bigr) \circ \sigma_I^{-1}. \]
(Recall that $\sigma \left( \omega_1 \otimes \cdots \otimes \omega_n \right) = \omega_{\sigma^{-1} 1} \otimes \cdots \otimes \omega_{\sigma^{-1}n}.$) Now, all $\sigma_I$ (where $\left| I \right| =i$) form a system of representative of $\Sigma_n / \Sigma_i \times \Sigma_{n-i},$ so (note that $\omega$ is invariant under $\Sigma_n$)
\[ \sum_{\sigma \in \Sigma_n / \Sigma_i \times \Sigma_{n-i}} \sigma \circ \bigl( i_1^{\otimes i} \otimes i_2^{\otimes (n-i)} \bigr) \left( \omega \right) = \sum_{\left| I \right| = i} \sigma_I \circ \bigl( i_1^{\otimes i} \otimes i_2^{\otimes (n-i)} \bigr) \circ \sigma_I^{-1} \left( \omega \right). \]
Take the sum over all $i$, we get
\[ \sum_{i=1}^n \sum_{\sigma \in \Sigma_n / \Sigma_i \times \Sigma_{n-i}} \sigma \circ \bigl( i_1^{\otimes i} \otimes i_2^{\otimes (n-i)} \bigr) \left( \omega \right) = \Delta^{\otimes n} \left( \omega \right). \]

At last, uniqueness is easy to prove. By induction, we only need to prove that for each $n \ge 1,$ the map
\[ M \cotimes_R \cGamma^n \left( \Omega \right) \xrightarrow{\id \otimes \cGamma^n (\Delta)} M \cotimes_R \cGamma^n \left( \Omega \oplus \Omega \right) \to M \cotimes_R \Omega \cotimes_R \cGamma^{n-1} \left( \Omega \right) \]
is injective. Take the composition, we only need to prove that the following map
\[ M \cotimes_R \cGamma^n \left( \Omega \right) \xrightarrow{\id \otimes \cGamma^n (\Delta_n)} M \cotimes_R \cGamma^n \left( \Omega^{\oplus n} \right) \to M \cotimes_R \Omega^{\cotimes n}\]
is injective (where $\Delta_n$ is the diagonal, and $\cGamma^n \left( \Omega^{\oplus n} \right) \to \Omega^{\cotimes n}$ is given by
\[ \cGamma^n \left( \Omega^{\oplus n} \right) \cong \bigoplus_{i_1+\cdots+i_n=n} \cGamma^{i_1} \left( \Omega \right) \cotimes_R \cdots \cotimes_R \cGamma^{i_n} \left( \Omega \right) \]
taking the projection to the coordinate $\left( 1,1,\ldots,1 \right)$). However, it is easy to verify that this map is exactly the map given in Proposition \ref{prop: flat pd power}, and is injective as desired.
\end{proof}

We also have a converse for Lemma \ref{lem: connection to stratification} as follows:
\begin{lemma}
Let $R$ be a classically $p$-complete ring with bounded $p^\infty$-torsion, and $\Omega$ be a $p$-complete and $p$-completely flat $R$-module. Let $M$ be an $R$-module satisfying Condition \ref{cond: very conc deg 0}. Assume that there exists a sequence of maps $\nabla^n: M \to M \cotimes_R \cGamma^n \left( \Omega \right)$ with $\nabla^0 = \id_M$, and that the diagram \eqref{eq: strat connection} commutes, then $\nabla^1: M \to M \cotimes_R \Omega$ is a Higgs field.
\end{lemma}

\begin{proof}
Consider the diagram \eqref{eq: strat connection} for $n=2$. The image of
\[ M \cotimes_R \Gamma^2 \left( \Omega \right) \to M \cotimes_R \Gamma^2 \left( \Omega \oplus \Omega \right) \to M \cotimes_R \Omega \cotimes_R \Omega \]
is invariant under $\Sigma_2,$ hence the image of $\left( \nabla^1 \otimes \id \right) \circ \nabla^1$ is invariant under $\Sigma_2,$ and this is exactly the condition that $\nabla^1$ is integrable.
\end{proof}

\begin{defi}\label{defi: tn}
Let $R,\Omega,M$ be as in Lemma \ref{lem: connection to stratification}. Then, a Higgs field
\[ \nabla: M \to M \cotimes_R \Omega \]
is called \emph{topolocially nilpotent}, if the sequence of maps $\nabla^n: M \to M \cotimes_R \cGamma^n \left( \Omega \right)$ given in Lemma \ref{lem: connection to stratification} fits into a map
\[ \alpha: M \to \cbigoplus_{n \ge 0} M \cotimes_R \cGamma^n \left( \Omega \right). \]
\end{defi}

We have the following criterion for topological nilpotency:

\begin{lemma}\label{lem: topologically nilpotent}
Let $R,\Omega,M$ and $\nabla: M \to M \cotimes_R \Omega$ be as in Lemma \ref{lem: connection to stratification}. Also, assume that we have the integral case (i.e. $M$ itself is classically $p$-complete and of bounded $p^\infty$-torsion). Then, $\nabla$ is topologically nilpotent, if and only if for any $m \in M$ and any $N \in \mathbb Z_+,$ there exists $n \in \mathbb Z_+$ with $\nabla^n \left( m \right) \in p \bigl( M \cotimes_R \cGamma^n \left( \Omega \right) \bigr).$
\end{lemma}

\begin{proof}
We only need to prove the `only if' part.

By Proposition \ref{prop: topology of flat pd power}, $\nabla^n \left( m \right) \in p^N \bigl( M \cotimes_R \cGamma^n \left( \Omega \right) \bigr)$ if and only if $\nabla^n_0 \left( m \right) \in p^N \bigl( M \cotimes_R \Omega^{\cotimes n} \bigr).$ Since $\nabla_0^n$ is defined by iterated composition, we see that $p^N \mid \nabla_0^n \left( m \right)$ implies that for all $n' \ge n$ we have $p^N \mid \nabla_0^{n'} \left( m \right).$

Now, fix $u \in M \cotimes_R \Omega^{\cotimes n_0}.$ We may write
\[ u = \sum_{i=1}^k m_i \otimes \omega_i + pv, \]
where $m_i \in M, \omega_i \in \Omega^{\cotimes n_0},$ and $v \in M \cotimes_R \Omega^{\cotimes n_0}.$ By our condition and the previous paragraph, there exists $n \in \mathbb Z_+$ such that for all $n' \ge n$ we have $p \mid \nabla^{n'}_0 \left( m_i \right).$ This implies that for all $n' \ge n$ we have
\[ \left( \nabla \circ \id \right)^{\circ n'} \left( u \right) = \sum_{i=1}^k \nabla_0^{n'} \left( m_i \right) \otimes \omega_i + p\left( \nabla \circ \id \right)^{n'} \left( v \right)  \in p \bigl( M \cotimes_R \Omega^{\cotimes (n_0+n')} \bigr). \]
By induction, we see that for all $N \in \mathbb Z_+$, there exists $n \in \mathbb Z_+$ such that for all $n' \ge n$ we have
\[ \left( \nabla \circ \id \right)^{\circ n'} \left( u \right) \in p^N \bigl( M \cotimes_R \Omega^{\cotimes (n_0+n')} \bigr). \]
In particular, for all $m \in M$ and $N \in \mathbb Z_+$, there exists $n \in \mathbb Z_+,$ such that
\[ \nabla_0^{n'} \left( m \right) \in p^N \bigl( M \cotimes_R \cGamma^{n'} \left( \Omega \right) \bigr). \]

Combine the first two paragraphs, and we see that for any $m \in M$ and $N \in \mathbb Z_+,$ there exists $n_0 \in \mathbb Z_+,$ such that
\[ p^N \mid \nabla^n \left( m \right) \]
for all $n \ge n_0.$ Now, this is equivalent to that for all $m \in M,$ the element
\[ \left( \nabla^n \left( m \right) \right)_{n \ge 0} \in \prod_{n \ge 0} M \cotimes_R \cGamma^n \left( \Omega \right) \]
lies in the submodule
\[ \cbigoplus_{n \ge 0} M \cotimes_R \cGamma^n \left( \Omega \right). \]
Now, our conclusion becomes evident.
\end{proof}

We now focus on the side of stratifications.

\begin{lemma}\label{lem: strat diagram chasing}
Assume that $R$ is a classically $p$-complete ring with bounded $p^\infty$-torsion, and $S^\bullet$ is split PD over $R$. Let $M/S^0$ be a module satisfying Condition \ref{cond: very conc deg 0}, and $\ep: S^1 \cotimes_{\delta_0^1, S^0} M \to S^1 \cotimes_{\delta_1^1, S^0} M$ be a continuous $S^1$-linear map. Define $\varphi^n: S^0 \cotimes_R \cGamma_R^n(\Omega) \to S^1$ by
\[ \varphi^n(s \otimes \gamma_{i_1}(\omega_1) \cdots \gamma_{i_n}(\omega_n)) = \delta_1^1(s) \otimes (\gamma_{i_1}(\varphi(\omega_1)) \cdots \gamma_{i_n}(\varphi(\omega_n)) \]
and its base change $\varphi^n_M: M \cotimes_R \cGamma_R^n(\Omega) \to S^1 \cotimes_{\delta_1^1, S^0} M$, so we have an isomorphism
\[ M \cotimes_R \cbigoplus_n \cGamma_R^n(\Omega) \to S^1 \cotimes_{\delta_1^1, S^0} M \]
induces by $\varphi^n_M$. Write
\[ \ep(1 \otimes m) = \sum_{n=0}^\infty \varphi_M^n(\nabla^n(m)). \]
Then, $\ep$ is a stratification if and only if $\nabla^0 = \id_M$, and $\{ \nabla^n \}$ fits into the commutative diagram \eqref{eq: strat connection}.
\end{lemma}

\begin{proof}
This is a routine diagram chasing, and is easier done on a scratch paper than written down.

The condition (2) of stratifications boils down to $\nabla^0 = \id_M$. We will check condition (1).

Denote $\varphi^n_0$ the restriction of $\varphi^n$ on $\cGamma_R^n(\Omega)$, so
\[ \varphi^n_0 (\gamma_{i_1}(\omega_1) \cdots \gamma_{i_n}(\omega_n)) = \gamma_{i_1}(\varphi(\omega_1)) \cdots \gamma_{i_n}(\varphi(\omega_n)), \]
then $\varphi^n_0$ gives a PD ring homomorphism $\cGamma_R^\bullet(\Omega) \to S^1$.

For $0 \le i, j \le n$ and $r \ge 0$, denote $\varphi^{n,r}_{i,j}$ the composition
$S_0 \cotimes_R \cGamma_R^r(\Omega) \xrightarrow{\varphi^r} S^1 \xrightarrow{\beta_{i,j}^n} S^n$. Define
\[ \psi^{i,j}: S_0 \cotimes_R \cGamma^i_R(\Omega) \cotimes_R \cGamma^j_R(\Omega) \to S^2 \]
by
\[ \psi^{i,j}(s \otimes \omega_1 \otimes \omega_2) = q_0^1(s) \delta_2^2(\varphi_0^i(\omega_1)) \delta_0^2(\varphi_0^j(\omega_2)), \]
Then $\psi^{i,j}$ induces an isomorphism
\[ \cbigoplus_{i,j} S^0 \cotimes_R \cGamma^i_R(\Omega) \cotimes_R \cGamma^j_R(\Omega) \to S^2, \]
where $q^2_0: S^0 \to S^2$. Also denote the sum of $\psi^{i,j}$ by
\[ \psi^n: S_0 \cotimes_R \cGamma^n_R(\Omega \oplus \Omega) \to S^2. \]
Denote their base change to $M$ by
\begin{gather*}
\varphi^{n,r}_{i,j,M}: M \cotimes_R \cGamma_R^r(\Omega) \to S^n \cotimes_{q_i^n, S^0} M, \\
\psi^{i,j}_M: M \cotimes_R \cGamma^i(\Omega) \cotimes_R \cGamma^j(\Omega) \to S^2 \cotimes_{q_0^2,S^2} M, \\
\psi^n_M: M \cotimes_R \cGamma^n_R (\Omega \oplus \Omega) \to S^2 \cotimes_{q_0^2,S^2} M.
\end{gather*}
Denote $\Delta: \Omega \to \Omega \oplus \Omega$ the diagonal map, and it is clear that
\[ \psi^n_M \circ (\id \otimes \cGamma^n(\Delta)) = \varphi^{2,n}_{0,2,M}. \]
Also, it is easy to verify that for $x \in M \cotimes_R \cGamma^i(R)$ and $y \in \cGamma^j(R)$ we have
\[ \psi_M^{i,j}(x \otimes y) = \delta_0^2(\varphi^j(y)) \varphi^{2,i}_{0,1,M}(x). \]

We first verify that
\[ \id_{S^n} \otimes_{\beta^n_{i,j}, S^0} \ep: S^n \cotimes_{q_j^n, S^0} M \cong S^n \cotimes_{\beta^n_{i,j}, S^1} (S^1 \cotimes_{\delta_0^1, S^0} M) \to S^n \cotimes_{\beta^n_{i,j}, S^1} (S^1 \cotimes_{\delta_1^1, S^0} M) \cong S^n \cotimes_{q_i^n, S^0} M \]
sends $1 \otimes m$ to
\[ \sum_{r=0}^\infty \varphi^{n,r}_{i,j,M} (\nabla^r(m)). \]
In fact, this is a direct consequence of the following commutative diagram
\[ \xymatrix{
S^n \cotimes_{q_i^n, S^0} M \ar[rr]^\cong & & S^n \cotimes_{\beta^n_{i,j}, S^1} (S^1 \cotimes_{\delta_1^1, S^0} M) \\
& M \cotimes_R \cGamma_R^r(\Omega) \ar[lu]^{\varphi^{n,r}_{i,j,M}} \ar[ru]_{1 \otimes \varphi^r_M}
}\]
and we are done.

We next verify that the composition
\[ M \cotimes_R \cGamma_R^r(\Omega) \xrightarrow{\varphi^{2,r}_{1,2,M}} S^2 \cotimes_{q_1^2,S^0} M \xrightarrow{\id_{S^2} \otimes_{\delta_2^2,S^1} \ep} S^2 \cotimes_{q_0^2, S^0} M \]
is given by
\[ x \mapsto \sum_{n=0}^\infty \psi^{n,r}_M \circ (\nabla^n \otimes \id) (x). \]
This is again straightforward: by the last paragraph, $\id_{S^2} \otimes_{\delta_2^2,S^1} \ep$ sends $1 \otimes m$ to for $\omega \in \cGamma^r(\Omega)$, so the composition acts by (by linearity)
\[ m \otimes \omega \mapsto \delta_0^2(\varphi^r(\omega))\otimes m \mapsto \delta_0^2(\varphi^r(\omega)) \sum_{n=0}^\infty \varphi^{2,n}_{0,1,M} (\nabla^n(m)) = \sum_{n=0}^\infty \psi^{n,r}_M (\nabla^n(m) \otimes \omega), \]
as desired.

Now, we can calculate the composition
\[ S^2 \cotimes_{q_2^2, S^0} M \xrightarrow{\id_{S^2} \otimes_{\delta_0^2,S^1},\ep} S^2 \cotimes_{q_1^2, S^0} M \xrightarrow{\id_{S^2} \otimes_{\delta_2^2,S^1} \ep} S^2 \cotimes_{q_0^2, S^0} M. \]
The first map sends $1 \otimes m$ to
\[ \sum_{r=0}^\infty \varphi^{2,r}_{1,2,M}(\nabla^r(m)). \]
Write the image of $1 \otimes m$ under the composition by
\[ \sum_{n,r \ge 0} \psi^{n,r}_M(c_{n,r}). \]
Note that the projection maps $S^2 \cotimes_{q_0^2,S^0} M \to M \cotimes_R \cGamma_R^i(\Omega) \cotimes_R \cGamma_R^j(\Omega)$ (projection to the direct sum decomposition induced by $\psi_M^{n,r}$) are continuous, so the results from the last paragraph shows that the $c_{n,r} = (\nabla^n \otimes \id) \circ \nabla^r (m)$\footnote{One can directly check that the elements $(\nabla^n \otimes \id) \circ \nabla^r (m)$ tends to zero, but the continuity of projection maps saves us from this discussion.}.

Now we know that
\[ \sum_{n,r \ge 0} \psi^{n,r}_M(c_{n,r}) = (\id_{S^2} \otimes_{\delta_1^2, S^1} \ep) (1 \otimes m) = \sum_{n=0}^\infty \varphi_{0,2,M}^{2,n}(\nabla^n(m)). \]
Since $\varphi_{0,2,M}^{2,n} = \psi_M^n \circ (\id \otimes \cGamma^n(\Delta))$, we conclude that
\[ (c_{i,j})_{i+j=n} \in \bigoplus_{i+j=n} M \cotimes_R \cGamma^i(\Omega) \cotimes \cGamma^j(\Omega) \]
is the image of $(\id \otimes \cGamma^n(\Delta))(m) \in M \cotimes_R \cGamma^n(\Omega \oplus \Omega)$. Recall that $c_{i,j} = (\nabla^i \otimes \id) \circ \nabla^j (m)$ in the last paragraph, and this is exactly the content of the diagram \eqref{eq: strat connection}, as desired.
\end{proof}

\begin{prop}\label{prop: strat to connection}
Let $R$ be a classically $p$-complete ring with bounded $p^\infty$-torsion, and $S^\bullet$ be split PD over $R$ (see Condition \ref{cond: split pd}). Define
\[ \dif^n: S^0 \to S^0 \cotimes_R \cGamma^n \left( \Omega \right) \]
by
\[ \delta_0^1 \left( s \right) = \sum_{n=0}^\infty \varphi^n \left( \dif^n \left( s \right) \right), \]
where $\varphi^n: S^0 \cotimes_R \cGamma^n \left( \Omega \right) \to S^1$ is induced by $q_0^1 \otimes \varphi: S^0 \cotimes_R \Omega \to S^1.$ Similarly, let $M$ be a stratification over $S^\bullet$ or $S^\bullet \bigl[ \frac 1 p \bigr].$ Define
\[ \nabla^n: M \to M \cotimes_R \cGamma^n \left( \Omega \right) \]
by
\[ \epsilon \left( 1 \otimes m \right) = \sum_{n=0}^\infty \varphi_M^n \left( \nabla^n \left( m \right) \right), \]
where $\varphi^n_M$ is the base change $M \cotimes_{S^0} \varphi^n$. Then,
\begin{enumerate}[label=(\arabic*)]
\item The operators $\dif^0$ and $\nabla^0$ are identity, and the systems $\left\{ \dif^n \right\}$ and $\left\{ \nabla^n \right\}$ fits into the commutative diagram \eqref{eq: strat connection}, so $\dif:=\dif^1$ and $\nabla:=\nabla^1$ are topologically nilpotent Higgs fields on $S^0$ and $M$.
\item $\dif$ is, in fact, a differential operator $S^0 \to S^0 \cotimes_R \Omega,$ and $\nabla$ is an integrable $\dif$-connection.
\item if $S^\bullet$ is filtered split PD (see Definition \ref{defi: filtered split PD}), then $\dif$ satisfies Griffith transversality. Furthermore, $\dif^k$ sends $\Fil_d (S^0)$ to $\Fil_{d-k} (S^0) \cotimes_R \cGamma^k(\Omega);$ if $S^\bullet$ is graded split PD, then $\dif^k$ sends $\gr_d (S^0)$ to $\gr_{d-k} (S^0) \cotimes_R \cGamma^k(\Omega).$
\end{enumerate}
\end{prop}

\begin{proof}
For (1), note that $\left\{ \dif^n \right\}$ is nothing but a special case of $\nabla^n$ for $\ep = \id_{S^1}$. We can then apply Lemma \ref{lem: strat diagram chasing}.

As for (2), for $s \in S^0$ and $m \in M,$ we calculate as follows:
\[ \epsilon \left( 1 \otimes sm \right) = \delta_0^1 \left( s \right) \epsilon \left( 1 \otimes m \right) = \sum_{k,\ell} \varphi^k \left( \dif^n \left( s \right) \right) \varphi^\ell_M \bigl( \nabla^\ell \left( m \right) \bigr) = \sum_{n=0}^\infty \sum_{k+\ell = n} \varphi_M^n \bigl( \dif^k \left( s \right) \nabla^\ell \left( m \right) \bigr). \]
It follows that
\[ \nabla^n \left( sm \right) = \sum_{k+\ell = n} \dif^k \left( s \right) \nabla^\ell \left( m \right). \]
Take $n=1$ and we get (recall $\nabla^0 \left( m \right) = m \otimes 1$)
\[ \nabla \left( sm \right) = s\nabla \left( m \right) + m \otimes \dif \left( s \right). \]
Take $M=S^0,$ and we see that
\[ \dif ( ss' ) = s\dif ( s' ) + s' \dif \left( s \right), \]
as desired.

As for (3), we will only prove the filtered case. Note that for $s \in \Fil_d (S^0)$ we have
\[ \Fil_d(S^1) \ni \delta_0^1(s) = \sum_{k=0}^\infty \varphi^k( \dif^k (s)). \]
Now, since the map $S^0 \cotimes_R \cGamma^\bullet(\Omega) \to S^1$ is strict, we see that $\sum_k \dif^k(s) \in S^0 \cotimes_R \cGamma^\bullet(\Omega)$ lies in
\[ \Fil_d \bigl( S^0 \cotimes_R \cGamma^\bullet(\Omega) \bigr) = \bigoplus_{k=0}^d \Fil_{d-k} (S^0) \cotimes_R \cGamma^k(\Omega), \]
therefore $\dif^k(s) \in \Fil_{d-k} (S^0) \cotimes_R \cGamma^k(\Omega),$ as desired.
\end{proof}

\begin{ex}\label{ex: differential dHodge}
Let $A \to P \to R$ be as in Definition \ref{defi: graded split PD}, then the differential is given by, for $x \in (J/J^2)^\land \left\{ -1 \right\}$ (where $J = \ker(P \to R)$),
\[ \dif^k(\gamma^r(x)) = \gamma^{r-k}(x) \otimes \gamma^k(\dif x), \]
with $\dif: (J/J^2)^\land \left\{ -1 \right\} \to R \cotimes_P \cOmega_{P/A} \left\{ -1 \right\}$ is the map in the conormal sequence.
\end{ex}

\begin{lemma}\label{lem: Leibniz rule}
Let $R$ be a classically $p$-complete ring with bounded $p$-torsion, $\Omega$ a $p$-complete and $p$-completely flat $R$-module, and $S$ be a $p$-complete and $p$-completely flat $R$-algebra equipped with a $R$-linear derivation
\[ \dif: S \to S \cotimes_R \Omega. \]
such that $\dif$ is a Higgs field over $S/R$.

Let $M$ be an $S$-module satisfying Condition \ref{cond: very conc deg 0} equipped with an integrable $\dif$-connection
\[ \nabla: M \to M \cotimes_R \Omega. \]
Then, for all $s \in S$ and $m \in M$ we have
\[ \nabla^n \left( sm \right) = \sum_{k+\ell = n} \dif^k \left( s \right) \cdot \nabla^\ell \left( m \right), \]
where the multiplication is taken in the $S \cotimes_R \cGamma^\bullet \left( \Omega \right)$-module $M \cotimes_R \cGamma^\bullet \left( \Omega \right).$
\end{lemma}

\begin{proof}
Denote
\[ S' = S \cotimes_R \cGamma^\bullet \left( \Omega \right). \]
We have two differentials on $S'$ (both of them are Higgs fields over $R$). One is the base change of $\dif: S \to S \cotimes_R \Omega,$ and we will denote it by $\dif_1.$ The other is the base change of the differential $\cGamma^\bullet \left( \Omega \right) \to \cGamma^\bullet \left( \Omega \right) \cotimes_R \Omega,$ and we will denote it by $\dif_2.$ Note that $\dif_2$ is injective on the positive degrees, whose cokernel is $p$-complete and $p$-completely flat over $S$ (by reducing ${}\otimes_R^\bL R/p$, applying the colimit argument, and the case for free modules is simply PD polynomials).

Now, $\dif_2$ extends to an integrable $\dif_2$-connection
\[ \nabla_2: M \cotimes_S S' \to M \cotimes_S S' \cotimes_R \Omega, \]
i.e.
\[ \nabla_2: M \cotimes_R \cGamma^\bullet \left( \Omega \right) \to M \cotimes_R \cGamma^\bullet \left( \Omega \right) \cotimes_R \Omega. \]
It follows that $\nabla_2$ is injective on the positive degrees. Also, $\nabla$ base-changes to an integrable $\dif_1$-connection
\[ \nabla_1: M \cotimes_R \cGamma^\bullet \left( \Omega \right) \to M \cotimes_R \cGamma^\bullet \left( \Omega \right) \cotimes_R \Omega. \]
It is easy to verify that $\nabla_2$ on $M \cotimes_R \cGamma^n \left( \Omega \right)$ coincides with the map
\[ M \cotimes_R \cGamma^n \left( \Omega \right) \xrightarrow{\id \otimes \cGamma^n (\Delta)} M \cotimes_R \cGamma^n \left( \Omega \oplus \Omega \right) \to M \cotimes_R \Omega \cotimes_R \cGamma^{n-1} \left( \Omega \right) \to M \cotimes_R \cGamma^{n-1} \left( \Omega \right) \cotimes_R \Omega, \]
therefore by diagram \eqref{eq: strat connection} we have
\[ \nabla_2 \left( \nabla^n \left( m \right) \right) = \nabla_1 \bigl( \nabla^{n-1} \left( m \right) \bigr). \]
Similarly we have
\[ \dif_2 \left( \dif^n \left( s \right) \right) = \dif_1 \bigl( \dif^{n-1} \left( s \right) \bigr). \]
We must also remark that $\nabla_2 \left( m \otimes 1 \right) =0$ and $\dif_2 \left( s \right) =0.$

Now, we prove by induction. The case $n=0$ is trivial. Assume case $n-1,$ we will prove this for case $n$. As $\nabla_2$ is injective on positive degrees, we only need to prove that
\[ \nabla_2 \left( \nabla^n \left( sm \right) \right) = \sum_{k+\ell = n} \nabla_2 \bigl( \dif^k \left( s \right) \cdot \nabla^\ell \left( m \right) \bigr). \]
For the left-hand side, by the inductive hypothesis, we have
\begin{align*}
\nabla_2 \left( \nabla^n \left( sm \right) \right) & = \nabla_1 \left( \nabla^{n-1} \left( sm \right) \right) \\
& = \sum_{k+\ell = n-1} \dif_1 \bigl( \dif^k (s) \bigr) \cdot \nabla^\ell \left( m \right) + \sum_{k+\ell = n-1} \dif^k (s) \cdot \nabla_1 \bigl( \nabla^\ell (m) \bigr). 
\end{align*}
On the right-hand side, we have
\begin{align*}
& \sum_{k+\ell = n} \nabla_2 \bigl( \dif^k \left( s \right) \cdot \nabla^\ell \left( m \right) \bigr) \\
={} & \sum_{k+\ell = n} \dif_2 \bigl( \dif^k (s) \bigr) \cdot \nabla^\ell \left( m \right) + \sum_{k+\ell = n} \dif^k \left( s \right) \cdot \nabla_2 \bigl( \nabla^\ell (m) \bigr).
\end{align*}
Using $\dif_2 \circ \dif^k = \dif_1 \circ \dif^{k-1}$ and $\nabla_2 \circ \nabla^\ell = \nabla_1 \circ \nabla^{\ell-1},$ it is easy to see that these values are the equal to each other, as desired.
\end{proof}

\begin{prop}\label{prop: connection to strat}
Let $R$ be a classically $p$-complete ring with bounded $p$-torsion, and $S^\bullet$ be split PD over $R$. Let $M$ be an $S^0$-module satisfying Condition \ref{cond: very conc deg 0}, equipped with an integrable $\dif$-connection
\[ \nabla: M \to M \cotimes_R \Omega, \]
and is topologically nilpotent in the sense of Definition \ref{defi: tn}. Then, there exists a unique stratification over $S^\bullet$ of $S^\bullet \bigl[ \frac 1 p \bigr]$
\[ \epsilon: S^1 \cotimes_{\delta_0^1, S^0} M \to S^1 \cotimes_{\delta_1^1, S^0} M \]
such that $\nabla$ is the connection determined by $\epsilon$ in Proposition \ref{prop: strat to connection}.
\end{prop}

\begin{proof}
Uniqueness follows from the formula
\[ \epsilon \left( 1 \otimes m \right) = \sum_{n=0}^\infty \varphi_M^n \left( \nabla^n \left( m \right) \right). \]
As for existence, consider the map
\[ \alpha \left( m \right) = \sum_{n=0}^\infty \varphi_M^n \left( \nabla^n \left( m \right) \right). \]
We now show that $\alpha: M \to S^1 \cotimes_{\delta_1^1, S^0} M$ is $( S^0,\delta_0^1 )$-linear, i.e. we have to prove that
\[ \alpha \left( sm \right) = \delta_0^1 \left( s \right) \cdot \sum_{n=0}^\infty \varphi_M^n \left( \nabla^n \left( m \right) \right). \]
Consider the expansion
\[ \delta_0^1 \left( s \right) = \sum_{n=0}^\infty \varphi^n \left( \dif^n \left( s \right) \right). \]
Since $\dif^n(s) \to 0$ and $\nabla^n(m) \to 0$ (in the sense that for any $N\in \mathbb Z_+$, $\dif^n(s)$ is eventually divisible by $p^N$, and for any $M_0 \subseteq M$ open $S^0$-submodule, $\nabla^n(m)$ eventually lies in $M_0 \cotimes_R \cGamma_R^n(\Omega)$), it is easy to verify that $\dif^k(s) \nabla^\ell(m) \to 0, |k| + |\ell| \to \infty.$ Therefore, by Fubini's Theorem, we only need to prove the identity
\[ \nabla^n \left( sm \right) = \sum_{k+\ell = n} \dif^k \left( s \right) \nabla^\ell \left( m \right). \]
This is exactly the content of Lemma \ref{lem: Leibniz rule}.

Now, $\alpha$ extends to an $S^1$-linear map
\[ \epsilon: S^1 \cotimes_{\delta_0^1, S^0} M \to S^1 \cotimes_{\delta_1^1, S^0} M. \]
We have to verify both conditions (1) and (2) from Definition \ref{defi:strat}. This is a direct consequence of Lemma \ref{lem: strat diagram chasing}.
\end{proof}

We summarise the results of this section as below:

\begin{theo}\label{theo: comp strat connection}
Let $S^\bullet$ be split PD over $R$. Let $\dif: S^0 \to S^0 \cotimes_R \Omega$ be the differential operator in Proposition \ref{prop: strat to connection}. Then, for any $M/S^0$ satisfying Condition \ref{cond: very conc deg 0}, the map defined in Proposition \ref{prop: strat to connection} from the set of stratifications $\left( M,\ep \right)$ on $M$ over $S^\bullet$ or $S^\bullet \bigl[ \frac 1 p \bigr]$, to the set of topologically nilpotent (continuous) integrable connections $\left( M,\nabla \right)$ on $M$ over $\left( S^0,\dif \right)$, is bijective. 

Moreover, the above map gives an equvalence between the category of stratifications $(M,\ep)$ over $S^\bullet$ (resp. $S^\bullet \bigl[ \frac 1 p \bigr]$) and the category of topologically nilpotent connections $(M,\nabla)$ over $(S^0,\dif)$ with $M$ $p$-complete of bounded $p^\infty$-torsion (resp. $M$ adic over $S^0\bigl[ \frac 1 p \bigr]$). It is symmetric monoidal when restricted to the full subcategory of $p$-complete and $p$-completely flat modules $M/S^0$ or finite projective modules $M/S^0\bigl[ \frac 1 p \bigr]$.
\end{theo}

\begin{proof}
This is merely a rephrasing of Proposition \ref{prop: connection to strat} --- one should also observe that for modules equipped with stratifications $(M,\ep)$ and $(N,\ep)$, a map $M \to N$ is compatible with the stratification if and only if it is compatible with the corresponding connection --- which is straightforward from the definition of the connection. We omit the verification that it is symmetric monoidal --- this is again some very easy calculations.
\end{proof}

We will also provide a ring version of the above theorem.

\begin{theo}\label{theo: split PD vs derivation}
Let $R$ be a classically $p$-adically complete ring and $\Omega$ be $p$-completely flat (and $p$-adically complete) over $R$. For any $p$-complete and $p$-completely flat ring $S^0/R$, the groupoid of split PD cosimplicial rings $(S^\bullet,\varphi: \Omega \to S^1 )$ is equivalent to the set of differentials $\dif: S^0 \to S^0 \cotimes_R \Omega$ which is a topologically nilpotent Higgs field over $R$. The equivalence is given by mapping $(S^\bullet,\varphi: \Omega \to S^1)$ to the differential $\dif: S^0 \to S^0 \cotimes_R \Omega$ associated to it.

Furthermore, if $S^0$ is filtered (resp. graded), then we have an equivalence of the groupoid of filtered (resp. graded) split PD cosimplicial rings $(S^\bullet, \varphi: \Omega \to S^1)$ to the set of differentials $\dif: S^0 \to S^0 \cotimes_R \Omega$ which is topologically nilpotent, integrable as a connection over $R$, and satisfies the Griffith transversality (resp. compatible with the grading).
\end{theo}

\begin{proof}
It is easy to verify that any split PD cosimplicial ring $(S^\bullet,\varphi: \Omega \to S^1 )$ has only trivial automorphisms fixing $S^0$, so we only need to verify that the map is essentially surjective.

Consider the ring $E^n = \cGamma^\bullet_R(\Omega^{\oplus n})$ with transition maps induced by
\[ f_*: \Omega^{\oplus n} \to \Omega^{\oplus m}, (x_i) \mapsto \left( \sum_{f(i-1) < j \le f(i)} x_i \right)_j.\]
Then, $(E^\bullet, \varphi: \Omega \to E^1)$ is split PD over $R$ by definition. Now we are given a topologically nilpotent Higgs field $\dif: S^0 \to S^0 \cotimes_R \Omega$, so it will correspond to a stratification $M^\bullet / E^\bullet$ by Theorem \ref{theo: comp strat connection}. Now, since $\dif$ is a derivation, one sees that $(S^0, \dif)$ is an algebra in the category of topologically nilpotent continuous integrable connections $(M,\nabla)$ over $R$ with $M$ $p$-complete and $p$-completely flat over $R$. Since the equivalence in Theorem \ref{theo: comp strat connection} is symmetric monoidal whenever $M$ is $p$-complete and $p$-completely flat over $S^0$, we see that $M^\bullet / E^\bullet$ is again an algebra, hence it is a ring. Take $S^\bullet = M^\bullet$, its PD structure extended from the structure map $E^\bullet \to S^\bullet$, and we are done.

Now assume that $S^0$ is filtered. The filtration on $S^i$ is characterised by the property as follows: for any $q^i_j: S^0 \to S^i$, the map $S^0 \cotimes_{E^0} E^i \to S^i$ is an isomorphism of filtered PD rings. As for existence, let the filtration given by $q^i_j$ be denoted $\Fil^{(j)}$. By definition, $\Fil^{(j)}_k$ is generated by elements of form $s_{k_1} \otimes e_{k_2}$, where $k_1+k_2 = k$, $s_{k_1} \in \Fil_{k_1}(S^0)$, $e_{k_2} \in \Fil_{k_2} (E^i)$. Therefore, it would suffice to show that for any $0 \le j,j' \le i$ and $s \in \Fil_k( S^0 )$, we have $q^i_{j'}(s) \in \Fil_k^{(j)} (s)$. This is evident for $j=0,j'=1,i=1$, and it true in general by Proposition \ref{prop: commutative constraint}. The proof for the graded case is similar.
\end{proof}

\begin{lemma}\label{lem: cosimplicial connection}
Let $S^\bullet$ be split PD over $R.$ Denote $\Omega \left( n \right) = \Omega^{\oplus (n+1)}.$ Then, there exists a unique $R$-linear differential compatible with the $\mathscr F_+$-structure (see Proposition \ref{prop: commutative constraint}) and the PD structure
\[ \dif^{(\bullet)}: S(\bullet) \to S(\bullet) \cotimes_R \Omega \left( \bullet \right) \]
with $\dif^{(0)} = \dif$ and $\dif^{(1)} \left( \varphi \left( \omega \right) \right) = 1 \otimes \left( -\omega,\omega \right)$ for any $\omega \in \Omega$. Moreover, $\dif^{(\bullet)}$ is a Higgs field over $R$.

Let $\left( M,\ep \right)$ be a stratification over $S^\bullet$ or $S^\bullet \bigl[ \frac 1 p \bigr].$ Denote $M(\bullet)$ the $S(\bullet)$-module given in Proposition \ref{prop: commutative constraint strat}. Then, there exists a unique integrable $\dif^{(\bullet)}$-connection
\[ \nabla^{(\bullet)}: M(\bullet) \to M(\bullet) \cotimes_R \Omega \left( \bullet \right) \]
with $\nabla^{(0)}=\nabla.$
\end{lemma}

\begin{proof}
We omit the proof of uniqueness. As for existence, observe that $\mathscr F^+$, equivalent to the category of nonempty finite sets, has products. We define
\[ S(i,j) = S(\left[ i \right] \times \left[ j \right]), \quad M(i,j) = M(\left[ i \right] \times \left[ j \right])\]
and we may identify $S(i,0)$ with $S(i)$ in a natural way. For $f: \left[ 1 \right] \to \left[ i \right]$ and $g: \left[ 1 \right] \to \left[ j \right],$ we have the map $\gamma_{f,g}: S(1) \to S(i,j).$

Now, $S (i,\bullet)$ has a split PD structure given by
\[ \varphi^{(i)} = \bigl( \gamma_{0,\id}, \gamma_{1,\id},\ldots,\gamma_{i,\id} \bigr) \circ \varphi: \Omega \left( i \right) \to S \left( i,1 \right) \]
(i.e. the $j$th component is mapped to $S(i,1)$ by $\Omega \xrightarrow{\varphi} S(1) \xrightarrow{\gamma_{j,\id}} S(i,1)$). For $f: [i] \to [j],$ we will denote the induced map $S(n,i) \to S(n,j)$ by $f^{(n)}.$

It is easy to verify that this split PD structure is natural in $i$, so we get a natural differential by Proposition \ref{prop: strat to connection}
\[ \dif^{(\bullet)}: S(\bullet) \to S(\bullet) \cotimes_R \Omega(\bullet). \]
It is easy to see that $\dif^{(0)}=\dif.$ By the additive relation for $\varphi$, we have
\[ \gamma_{\id,1} \left( \varphi \left( \omega \right) \right) - \gamma_{\id,0} \left( \varphi \left( \omega \right) \right) = \gamma_{1,\id} \left( \varphi \left( \omega \right) \right) - \gamma_{0,\id} \left( \varphi \left( \omega \right) \right). \]
The left-hand side is identified with
\[ \bigl(q_1^1\bigr)^{(1)} \left( \varphi \left( \omega \right) \right) - \bigl(q_0^1\bigr)^{(1)} \left( \varphi \left( \omega \right) \right), \]
and the right-hand side is identified with $\varphi^{(1)} \left( -\omega, \omega \right),$ so $\dif^{(1)} (\varphi(\omega)) = 1 \otimes \left( -\omega,\omega \right),$ as desired.

View $M(i,\bullet)$ as a stratification on $M(i),$ and by Proposition \ref{prop: strat to connection}, we get a continuous $\dif^{(\bullet)}$-integrable connection $\nabla^{(\bullet)}$ with $\nabla^{(0)}=\nabla.$
\end{proof}

\begin{theo}\label{theo: comp strat connection cohomologies}
Under the bijection as in Theorem \ref{theo: comp strat connection}, there exists a natural (quasi-)isomorphism in $D \left( R \right)$:
\[ S^\bullet \cotimes_{S^0} M \cong \DR \left( M,\nabla \right), \]
where $S^\bullet \cotimes_{S^0} M$ is the cosimplicial module determined by $\left( M,\ep \right)$, and $\DR \left( M,\nabla \right)$ is the de Rham complex of $\nabla$:
\[ M \to M \cotimes_R \Omega \to M \cotimes_R \cbigwedge\nolimits^2_R \left( \Omega \right) \to \cdots. \]

Moreover, if we restrict ourselves to $p$-complete and $p$-completely flat modules over $S^0$ or finite projective modules over $S^0 \bigl[ \frac 1 p \bigr]$, then the above map is an isomorphism of symmetric monoidal functors. In particular, the isomorphism
\[ S^\bullet \cong \DR(S,\dif) \]
is an isomorphism of $\dE_\infty$-algebras. Whenever $S^\bullet$ is filtered split PD (resp. graded split PD), then $S^\bullet$ is compatible with totalisations (see Definition \ref{defi: compatible with totalisation}), and the above isomorphism is an isomorphism of filtered (resp. graded) complete $\dE_\infty$-$R$-algebras.
\end{theo}

\begin{proof}
This method is essentially a rewriting of that in \cite{Tian_2023}, but with much more details.

For the sake of simplicity, denote $M^\bullet = S^\bullet \cotimes_{S^0} M$ and $\Omega^r = \cbigwedge^r_R \left( \Omega \right).$ Define
\[ \Omega^r \left( n \right) = \cbigwedge^r_R \bigl( \Omega^{\oplus (n+1)} \bigr), \]
so $\Omega^r \left( \bullet \right)$ becomes a cosimplicial object in an obvious way.

We first prove that $\Omega^r \left( \bullet \right)$ are homotopic to zero when $r \ge 1.$ Since taking wedges preserves being homotopic to zero, we only need to prove that $\Omega^{\oplus \left( \bullet +1 \right)}$ is homotopic to zero. Now observe that
\[ \Omega^{\oplus \left( \bullet +1 \right)} \cong \Omega \otimes_R R^{\oplus \left( \bullet +1 \right)}, \]
and $R^{\oplus \left( \bullet+1 \right)}$ is homotopic to zero by \cite[Example 2.16]{bhatt2011crystallinecohomologyrhamcohomology}.

Now, consider the differential $\dif^{(\bullet)}$ and the connection $\nabla^{(\bullet)}$ as in Lemma \ref{lem: cosimplicial connection}. For each $i$, we define the de Rham complex
\[ \DR \bigl( M^i, \nabla^{(i)} \bigr) = \bigl( M^i \to M^i \cotimes_R \Omega^1 \left( i \right) \to M^i \cotimes_R \Omega^2 \left( i \right) \to \cdots \bigr). \]
Using the definition of $\nabla^{(\bullet)}$ and $\dif^{(\bullet)},$ we see that
\[ \DR \left( M,\nabla \right) \cotimes_R \DR \bigl( \cGamma_R^\bullet(\Omega^{\oplus i}), \dif' \bigr) \to \DR \bigl( M^i, \nabla^{(i)} \bigr) \]
sending $M$ to $M^i$ by $q_0^i$ and $\Omega^{\oplus i}$ to $S^i$ by $\left( \alpha_i^n \right),$ is an isomorphism of complexes (where $\dif'$ is the differential of the PD polynomial ring).

We claim that for any $R$-module $M$ satisfying Condition \ref{cond: very conc deg 0}, the map
\[ M \to M \cotimes_R \DR \bigl( \cGamma_R^\bullet (\Omega^{\oplus i}), \dif' \bigr) \]
is a quasi-isomorphism. Taking modules of definition, we may assume that $M$ is $p$-complete and of bounded $p^\infty$-torsion. Taking the graded pieces, we only need to prove that the complexes
\[ M \cotimes_R \gr_r \DR \bigl( \cGamma_R^\bullet (\Omega^{\oplus i}), \dif' \bigr) \]
are acyclic for $r \ge 1$. Now the complex
\[ \gr_r \DR \bigl( \cGamma_R^\bullet (\Omega^{\oplus i}), \dif' \bigr) \cong \Bigl[\cdots \to \cGamma_R^r(\Omega^{\oplus i}) \to \cGamma_R^{r-1}(\Omega^{\oplus i}) \cotimes_R \Omega^{\oplus i} \to \cdots \to \cGamma_R^0(\Omega^{\oplus i}) \cotimes_R \cbigwedge^r_R(\Omega^{\oplus i}) \to 0 \to \cdots \Bigr] \]
are bounded, and each term is $p$-complete and $p$-completely flat over $R$, so
\[ M \cotimes_R \gr_r \DR \bigl( \cGamma_R^\bullet (\Omega^{\oplus i}), \dif' \bigr) \cong M \cotimes_R^\bL \gr_r \DR \bigl( \cGamma_R^\bullet (\Omega^{\oplus i}), \dif' \bigr). \]
Therefore we only need to prove that $\gr_r \DR \bigl( \cGamma_R^\bullet (\Omega^{\oplus i}), \dif' \bigr)$ are acyclic. Using derived Nakayama's lemma, we only need to prove that for any flat module $F/(R/p)$ and $r \ge 1$, the complex
\[ \gr_r \DR \bigl( \Gamma_{R/p}^\bullet (F), \dif' \bigr) \]
is acyclic. Taking the filtered colimit, we only need to prove this for $F/(R/p)$ free, and the result becomes well-known. Write down the double-complex, and we see that the map
\[ \DR(M,\nabla) \to \DR \left( M,\nabla \right) \cotimes_R \DR \bigl( \cGamma_R^\bullet(\Omega^{\oplus i}), \dif' \bigr) \]
is a quasi-isomorphism, so the transition maps
\[ q_0^i: \DR \left( M,\nabla \right) \to \DR \bigl( M^i, \nabla^{(i)} \bigr) \]
are quasi-isomorphisms. It follows that (since $\simp$ is weakly contractible) all maps $\left[ i \right] \to \left[ j \right]$ induces a quasi-isomorphism
\[ \DR \bigl( M^i, \nabla^{(i)} \bigr) \to \DR \bigl( M^j, \nabla^{(j)} \bigr), \]
and that the natural projection
\[ \lim_\simp \DR \bigl( M^i, \nabla^{(i)} \bigr) \to \DR \bigl( M,\nabla \bigr) \]
is a quasi-isomorphism.

Now, we show that
\[ \lim_\simp \DR \bigl( M^i, \nabla^{(i)} \bigr) \to \lim_\simp M^i \]
is a quasi-isomorphism. Take the fibre, and we only need to prove that
\[ \lim_\simp \sigma^{\ge 1} \DR \bigl( M^i,\nabla^{(i)} \bigr) \cong 0. \]
It is easy to see that the left-hand side admits a complete \emph{decreasing} filtration $\sigma^{\ge i} \DR \bigl( M^i,\nabla^{(i)} \bigr)$ with graded pieces
\[ \gr^r = \lim_\simp M^\bullet \cotimes_R \Omega^r( \bullet ). \]
By our previous discussions, $\Omega^r(\bullet)$ is homotopic to zero (as a cosimplicial object), so these graded pieces are homotopic to zero, and we see that
\[ M^\bullet \cotimes_R \Omega^r \left( \bullet \right) \]
is also homotopic to zero.

At last we give a remark on the symmetric monoidal structure. In fact, for all $i$, the map $\DR(M^i,\nabla^{(i)}) \to M^i$ are maps of lax monoidal functors (essentially by the proof of Lemma \ref{lem: cosimplicial connection}), so their limit
\[ \lim_\simp \DR(M^i,\nabla^{(i)}) \to \lim_\simp M^i \]
and the projection
\[ \lim_\simp \DR(M^i,\nabla^{(i)}) \to \DR(M,\nabla) \]
must also be maps (isomorphism by previous discussions) of lax monoidal functors, as desired.

In the filtered or graded case, since all terms of $S^\bullet \otimes_{S^0} M$ are coconnective, Proposition \ref{prop: totalisation commute with filtered colimits for coconnective spectra} shows that $S^\bullet \otimes_{S^0} M$ is compatible with totalisation. The definition easily shows that our construction respects filtered structures and graded structures. (In the case of filtered split PD ring, we may take the graded pieces to see that the quasi-isomorphisms are strict.)
\end{proof}

We now proceed to give a simpler way to calculate the isomorphism given in Theorem \ref{theo: comp strat connection cohomologies} in the case $\dHodge_{R/P(\bullet)}$ in Definition \ref{defi: graded split PD}.

\begin{prop}\label{prop: calculation of wedge}
Let $R$ be a ring. Denote $\mathcal T_R$ the category of exact triangles
\[ M \to N \to Q \]
in $D^{\le 0}(R)$ (i.e. an exact sequence in $D(R)$ all of whose terms are connective). Then there exists a functor $D_R: \mathcal T_R \to \CAlg(\gr(D(R)))$ such that whenever $M$ and $N$ are flat over $R$,
\[ D_R(M \to N \to Q) = \DR(\Gamma^*_R(M), \dif: \Gamma^*_R(M) \to \Gamma^*_R(M) \otimes_R N).\footnote{Here $\DR$ denotes the uncompleted de Rham complex analogous to Definition \ref{defi: Higgs field}.} \]
Also, there exists a natural isomorphism
\[ D_R (M \to N \to Q) \cong \bigwedge\nolimits^*_R(Q)[-*] \]
such that
\begin{enumerate}[label=(\alph*)]
\item whenever $M=0$ and $N$ is flat, the map is the tautological isomorphism
\[ \DR(R, 0: R \to N) \cong \bigwedge\nolimits^*_R(Q)[-*], \]
\item whenever $M$ and $N$ are flat and $M \to N$ is surjective, then the map is given by the inverse of the map
\[ \bigwedge\nolimits^*_R(Q)[-*] \cong \Gamma^*_R(Q[-1]) \to \DR(\Gamma^*_R(M), \dif: \Gamma^*_R(M) \to \Gamma^*_R(M) \otimes_R N). \]
\end{enumerate}
\end{prop}

\begin{proof}
Note that $\calT_R$ can be identified with the category of morphisms $M \to N$ in $D^{\le 0} (R)$. Therefore $\calT_R$ is identified with the Ind-completion of $\calT_R^0$ consisting of exact sequences
\[ 0 \to M \to N \to Q \to 0 \]
of \emph{finite free modules} over $R$. Denote $\calT'_R$ the category of $M \to N \to Q$ where $M$ and $N$ are flat, then one may check the graded pieces
\begin{align*}
& \gr_r \DR(\Gamma^*_R(M), \dif: \Gamma^*_R(M) \to \Gamma^*_R(M) \otimes_R N) \\
\cong{} & \Bigl[\Gamma^r_R(M) \to \Gamma^{r-1}_R(M) \otimes_R N \to \cdots \to M \otimes \bigwedge\nolimits^{r-1}_R(N) \to \bigwedge\nolimits^r_R(N) \to 0 \to \cdots\Bigr]
\end{align*}
it admits a finite filtration by functors of form $\Gamma^i_R(M) \otimes_R \bigwedge\nolimits^j_R(N) \cong \Gamma^i_R(M) \otimes_R^\bL \bigwedge\nolimits^j_R(N)$. Since $\otimes_R^\bL$ preserves colimits, these functors are left Kan extended from $\calT_R^0$, and we see that the functor
\[ (M \to N \to Q) \mapsto \DR(\Gamma^*_R(M), \dif: \Gamma^*_R(M) \to \Gamma^*_R(M) \otimes_R N) \]
on $\calT'_R$ is left Kan extended from $\calT_R^0$. We may define $D_R$ to be its left Kan extension to the entire $\calT_R$.

We define the natural map
\[ D_R (M \to N \to Q) \to \bigwedge\nolimits^*_R(Q)[-*] \]
to be the left Kan extension of the functor
\[ D_R(M \to N \to Q) \to D_R(0 \to Q \to Q) \cong \bigwedge\nolimits^*_R(Q)[-*] \]
on $\calT_R^0$. To check that it is an isomorphism of graded objects, we have to prove that for any $(M \to N \to Q) \in \calT_R^0$, the map
\[ \Bigl[\Gamma^r_R(M) \to \Gamma^{r-1}_R(M) \otimes_R N \to \cdots \to M \otimes \bigwedge\nolimits^{r-1}_R(N) \to \bigwedge\nolimits^r_R(N) \to 0 \to \cdots\Bigr] \to \Bigl(\bigwedge\nolimits_R^r(Q) \Bigr)[-r] \]
is a quasi-isomorphism. Take the duality, and we only need to prove that there exists a quasi-isomorphism of graded $R$-complexes
\[ \Kos(R;0: Q^\vee \to R) \to \Kos(\Sym_R^*(M^\vee), N^\vee \to M^\vee \to \Sym_R^*(M^\vee)). \]
In other words, we have to prove that there exists quasi-isomorphism of graded $\dE_\infty$-$R$-algebras
\[ R \otimes_{\Sym_R^*(Q^\vee)}^\bL R \to \Sym_R^*(M^\vee) \otimes_{\Sym_R^*(N^\vee)}^\bL R. \]
This follows from the formula
\[ \Sym_R^*(M^\vee) \cong \Sym_R^*(N^\vee) \otimes_{\Sym_R^*(Q^\vee)}^\bL R. \]

At last we check (b). Recall that by definition (see Sect. \ref{sect: notations}), the isomorphism
\[ \Gamma^*_R(Q[-1]) \to \bigwedge\nolimits^*_R(Q)[-*] \]
is given by the composition
\[ \Gamma^*_R(Q[-1]) \to D_R(Q[-1] \to 0 \to Q) \to \bigwedge\nolimits^*_R(Q)[-*]. \]
And now our result follows from the commutative diagram as below:
\[\xymatrix{
\Gamma^*_R(Q[-1]) \ar[r] \ar[rd] & D_R(M \to N \to Q) \ar[d]^\cong \ar[rd] \\
& D_R(Q[-1] \to 0 \to Q) \ar[r] & \bigwedge\nolimits^*_R(Q)[-*]
}\]
This completes the proof.
\end{proof}

\begin{prop}\label{prop: calculation of wedge completed}
Let $R$ be a $p$-complete ring of bounded $p^\infty$-torsion. Denote $\widehat \calT_R$ the category of exact triangles
\[ M \to N \to Q \]
in $\cD^{\le 0}(R)$ (i.e. an exact sequence in $\cD(R)$ all of whose terms are connective). Then there exists a functor $\cD_R: \widehat \calT_R \to \CAlg(\gr(\cD(R)))$ such that whenever $M$ and $N$ are $p$-complete and $p$-completely flat over $R$,
\[ \cD_R(M \to N \to Q) = \widehat\DR(\cGamma^*_R(M), \dif: \cGamma^*_R(M) \to \cGamma^*_R(M) \cotimes_R N).\footnote{Here we use $\widehat\DR$ to denote the completed de Rham complex in contrast with Proposition \ref{prop: calculation of wedge}.} \]
Also, there exists a natural isomorphism
\[ \cD_R (M \to N \to Q) \cong \cbigwedge^*_R(Q)[-*] \]
such that
\begin{enumerate}[label=(\alph*)]
\item whenever $M=0$ and $N$ is $p$-complete and $p$-completely flat, the map is the tautological isomorphism
\[ \widehat \DR(R, 0: R \to N) \cong \cbigwedge^*_R(Q)[-*], \]
\item whenever $M$ and $N$ are $p$-complete and $p$-completely flat, and $M \to N$ is surjective, then the map is given by the inverse of the map
\[ \cbigwedge^*_R(Q)[-*] \cong \cGamma^*_R(Q[-1]) \to \widehat \DR(\cGamma^*_R(M), \dif: \cGamma^*_R(M) \to \cGamma^*_R(M) \otimes_R N). \]
\end{enumerate}
\end{prop}

\begin{proof}
We continue to use the notations as in Proposition \ref{prop: calculation of wedge}. We define $\cD_R$ as the composition of (the restriction of) $D_R$ with the completion functor $D(R) \to \cD(R)$. Therefore $\cD_R$ is also left Kan extended from $\calT_R^0$. Denote $\widehat \calT'_R$ to be the full subcategory of $\widehat \calT_R$ consisting of sequences $M \to N \to Q$ where $M$ and $N$ are $p$-complete and $p$-completely flat over $R$, and define 
\[ \cD'_R(M \to N \to Q) = \widehat\DR(\cGamma^*_R(M), \dif: \cGamma^*_R(M) \to \cGamma^*_R(M) \cotimes_R N) \]
on $\widehat \calT'_R$. The same argument of graded pieces and an artificial filtration shows that $\cD'_R$ is left Kan extended from $\calT_R^0$, so $\cD'_R$ is identified with $\cD_R |_{\widehat \calT'_R}$.

Now, the property (a) holds on the category of finite free $R$-modules, and the functors on both sides are left Kan extended from $\calT_R^0$, so it holds on the category of $p$-complete and $p$-completely flat modules. The property (b) holds on the category of sequence $M \to N \to Q$ where $M$ and $Q[-1]$ are finite free $R$-modules, so we may left Kan extend it to show that it holds on the category for $M \to N \to Q$ where $M,N$ are $p$-complete and $p$-completely flat and $M \to N$ is surjective (in which case $Q[-1]$ is $p$-complete and $p$-completely flat).
\end{proof}

\begin{cor}\label{cor: natural isomorphism of dR and dHodge}
In the situation of Definition \ref{defi: graded split PD}, the map $\cD_R$ in Proposition \ref{prop: calculation of wedge completed} induces a natural isomorphism
\[ \dHodge_{R/A} \cong \DR(\dHodge_{R/P}, \dif: \dHodge_{R/P} \to \dHodge_{R/P} \cotimes_R \Omega). \]
\end{cor}

\begin{lemma}\label{lem: split cosimplicial}
In the situation of Definition \ref{defi: graded split PD}, the augmented cosimplicial object $\cdL_{R/A} \to \cdL_{R/P(\bullet)}$ is a split cosimplicial object (see \cite[Section 4.7.2]{lurie2017higher}), hence for any functor $F: \cD(R) \to \calD$ (where $\calD$ is any $\infty$-category), the limit of $F(\cdL_{R/P(\bullet)})$ over $\simp$ exists, and the map
\[ \cdL_{R/A} \to \lim_\simp F(\cdL_{R/P(\bullet)}) \]
is an isomorphism.

In particular, the map
\[ \dHodge_{R/A} \to \lim_\simp \dHodge_{R/P(\bullet)} \]
is an isomorphism in $\CAlg(\gr(\cD(R)))$ and in $\CAlg(\cD(R))$.
\end{lemma}

\begin{proof}
For any derived $p$-adically complete simplicial ring $B/A$ and map $B \to R$, we have $\cdL_{R/B} \cong \cofib(R \cotimes_B^\bL \cdL_{B/A} \to \cdL_{R/A})$. Therefore, we only need to show that the augmented cosimplicial object $0 \to R \cotimes_{P(\bullet)}^\bL \cdL_{P(\bullet)/A}$ is a split cosimplicial object. Denote $\Omega_0 = R \cotimes_P^\bL \cdL_{P/A}$, then $\Omega_0$ is $p$-complete and $p$-completely flat over $R$, and the object above is isomorphic to
\[ 0 \to \Omega_0^{\oplus(\bullet + 1)}. \]
For $n \ge -1$, denote $e_0,\ldots,e_n$ the injection maps $\Omega_0 \to \Omega_0^{\oplus(n+1)}$, and denote $e_{-\infty} =0$. Then, it is easy to check that for any map $f: \{-\infty\}\cup[n] \to \{-\infty\} \cup[m]$ in $\simp_{-\infty}$ (see \cite[Definition 4.7.2.1]{lurie2017higher}), we have a unique map $\Omega_0^{\oplus(n+1)} \to \Omega_0^{\oplus(m+1)}$ sending each $e_i$ to $e_{f(i)}$. This gives the desired split cosimplicial structure.
\end{proof}

\begin{rmk}
Lurie's original proof (\cite[Lemma 6.1.3.16]{lurie2009higher}) of split cosimplicial diagrams being limit diagrams is quite complicated. A simple proof can be done by observing that $\simp \to \simp_{-\infty}$ is initial (in the $\infty$-categorical sense) since it has a right adjoint.
\end{rmk}

\begin{prop}\label{prop: comp strat connection compatible dHodge}
In the situation of Definition \ref{defi: graded split PD}, the map of graded complete $R$-algebras provided by Theorem \ref{theo: comp strat connection cohomologies}
\[ \DR(\dHodge_{R/P}, \dif) \cong \lim_\simp \dHodge_{R/P(\bullet)} \]
coincides with the map
\[ \DR(\dHodge_{R/P},\dif) \to \dHodge_{R/A} \to \lim_\simp \dHodge_{R/P(\bullet)}, \]
where $\DR(\dHodge_{R/P},\dif) \to \dHodge_{R/A}$ is the isomorphism provided in Corollary \ref{cor: natural isomorphism of dR and dHodge}.
\end{prop}

\begin{proof}
Denote $S^\bullet = \dHodge_{R/P(\bullet)}$. The isomorphism in Theorem \ref{theo: comp strat connection cohomologies} is given by the arrow as below:
\begin{equation}\label{eq: diagram comp strat connection cohomologies}
\xymatrix{
& \displaystyle \lim_\simp \DR (S^i, \dif^{(i)}) \ar[ld] \ar[rd] & \\
\DR(S^0, \dif) \ar@{-->}[rr] & & \displaystyle \lim_\simp S^i
}
\end{equation}
Consider the natural identification $\varphi_i: \DR( S^i, d^{(i)}) \cong \dHodge_{R/A}$ given by Corollary \ref{cor: natural isomorphism of dR and dHodge} applied to $R/P(i)/A$. It is easy to see that $\varphi_i$ is compatible with the $\mathscr F_+$-structure. Therefore, we have the commutative diagram as below:
\[\xymatrix{
\displaystyle \lim_\simp \DR (S^i, \dif^{(i)}) \ar[r] \ar[d] & \displaystyle \lim_\simp \dHodge_{R/A} \ar[d] \ar[rd] & \\
\DR(S^0, \dif) \ar[r] & \dHodge_{R/A} \ar[r] & \displaystyle \lim_\simp S^i
}\]
where $\dHodge_{R/A} \to S^i \cong \dHodge_{R/P(i)}$ is given by the natural map $A \to P(i)$. By functoriality of the map in Corollary \ref{cor: natural isomorphism of dR and dHodge} applied to $R/P(i)/A \to R/P(i)/P(i)$, the composition (as maps between graded objects)
\[ \DR( S^i, \dif^{(i)} ) \to \dHodge_{R/A} \to S^i \]
agrees with the map
\[ \DR ( S^i, \dif^{(i)}) \to \DR(S^i,0: S^i \to 0) = S^i. \]
It follows that the composition
\[ \lim_\simp \DR (S^i, \dif^{(i)}) \to \lim_\simp \dHodge_{R/A} \to \lim_\simp S^i \]
agrees with the map $\lim_\simp \DR (S^i, \dif^{(i)}) \to \lim_\simp S^i$ as in \eqref{eq: diagram comp strat connection cohomologies}. Therefore, $\DR( S^0,\dif) \to \lim_\simp S^i$ agrees with $\DR( S^0,\dif ) \to \dHodge_{R/A} \to \lim_\simp S^i$, as desired.
\end{proof}

\subsection{Prismatic envelopes and PD-envelopes}
\begin{lemma}\label{lem: cohomology of Esharp}
Let $S$ be a $p$-complete ring of bounded $p^\infty$-torsion and $E$ be a countably generated projective module over $W(S)$. Also, fix a limit cardinal $\kappa$ with cofinality $\lambda>\omega$.

\begin{enumerate}[label=(\arabic*)]
\item Consider the essentially small site $\calC$ of $\Spec R$, where $R$ is a $p$-nilpotent (i.e. $p^n=0$ for some $n$) $S$-algebra, and $\# R < \kappa$. Equipped $\calC$ with the fpqc topology, and denote $\calO$ its structure sheaf. Then, $\lambda$-small products in $\Shv_\calC (\Sptr)$ are $t$-exact.
\item Denote $W$ the functor of Witt vectors on $\calC$, then for any $\Spec R \in \calC$, we have $R\Gamma(R,W) = W(R)$.
\item Define the $\calO$-module scheme $\mathbb G_a^\sharp = \Spf(S [ x ]_\PD^\land)$, then the natural isomorphism $\mathbb G_a^\sharp \to \ker F: W \to F^*W$ induces a short exact sequence of $W$-modules in the fpqc topology:
\[ 0 \to \mathbb G_a^\sharp \to W \xrightarrow F F_* W \to 0. \]
\item Assume that $E \in \cD(S)$ and $E \otimes_S^\bL S/p$ is $\lambda$-small projective\footnote{projective and admits a generating set of cardinality $<\lambda$} over $S/p$, so $E$ determines a sheaf $\calE$ of projective $\calO$-modules\footnote{See \cite[\href{https://stacks.math.columbia.edu/tag/05CG}{Tag 05CG}]{stacks-project}.}. Then, $\calE$ is an $\calO$-module in $\Shv_\calC(\Sptr)$, and the functor $\underline{\RHom}(\calE,-)$ on the category of $\calO$-modules in $\Shv_\calC(\Sptr)$ is $t$-exact.

In particular, if we take $(E^\vee)^\sharp := \underline{\Hom}_\calO(\calE,\mathbb G_a^\sharp)$ then
\[ R\Gamma(R,(E^\vee)^\sharp) = \RHom_R(\calE(R), R\Gamma(R,\mathbb G_a^\sharp)). \]
\item The sheaf $(E^\vee)^\sharp$ above is represented by the $p$-complete and $p$-completely flat $S$-algebra $\cGamma^\bullet_S(E)$.
\end{enumerate}
\end{lemma}

\begin{proof}
As for (1), we only need to prove right $t$-exactness. By definition (see \cite[Sect. 1.3.2]{lurie2018spectral}), we only need to prove that for if $\mathscr F_i$ with $i \in I$ and $\# I < \lambda$ are presheaves of abelian groups with trivial sheafifications, then $\prod_{i \in I} \mathscr F_i$ has trivial sheafification. For any $(x_i) \in \prod \mathscr F_i(R)$, we may find affine fpqc covers $\Spec(R_i) \to \Spec(R)$ in $\calC$ for each $i$ such that $\left. x_i \right|_{R_i} =0$. Take
\[ R_* = \bigotimes_{i \in I} R_i, \]
so $R \to R_*$ is faithfully flat. Note that $p$ is nilpotent in $R$, so $p$ is also nilpotent in $R_*$. By definition, since $\kappa$ has cofinality $\lambda > \# I$, we have $\# R_* < \sum \# R_i < \kappa$, hence $\Spec(R_*) \in \calC$. It is easy to see that $\left. (x_i) \right|_{R_*} =0$. Therefore, every section of $\prod_{i \in I} \mathscr F_i$ is locally zero, this proves the claim.

For the item (2), observe that $\calO$ is a sheaf of spectra on $\calC$. Now the presheaf of spectra $W_n$ admits a finite filtration with graded pieces $\ker(W_n \to W_{n-1}) \cong F_*^{n-1} \calO$, and they are also sheaves of spectra. Therefore, $W_n$ is a sheaf of spectra, and $W_n \to W_{n-1}$ are surjective, so $W = \clim 0 = \lim W_n$ is also a sheaf of spectra.

The item (3) is nothing but \cite[Lemma 5.7]{bhatt2022prismatization}. A direct proof is seen as follows: for each $\Spec (R) \in \calC$ we may find a quasisyntomic cover $R \to R'$ with $\Spec(R')\in \calC$, such that $W(R)$ lies in the image of $F: W(R') \to W(R')$. Repeat this process and take the direct limit (using $\lambda > \omega$), we obtain a quasisyntomic cover $R \to S$ with $F: W(S) \to W(S)$ surjective.

As for (4), it is clear that $\calE$ is a sheaf of spectra. Since the $t$-structure can be checked locally, we only need to prove that $\underline{\RHom}(\calE,-)$ is left $t$-exact on $\calC_{/\Spec R}$. Denote $E = \calE(R)$, then $E$ is $\lambda$-small projective over $R$, so it is a retract of a $\lambda$-small free module $F$ over $R$. Now, we see that $\underline{\RHom}_{\calC_{/\Spec R}}(\calE,-) \cong \RHom_R(E,-)$ is a retract of $\RHom_R(F,-)$. By (1), $\RHom_R(F,-)$ is $t$-exact, so $\underline{\RHom}_{\calC_{/\Spec R}}(\calE,-)$ is also $t$-exact as its retraction.

For the second part of statement, denote $\mathscr F = \fib(F: W \to F_* W)$, so $\mathscr F$ is the sheafification of $\mathbb G_a^\sharp$ in $\Shv_\calC(\Sptr)$, hence $\mathscr F$ is connective. The natural map $\mathbb G_a^\sharp \to \mathscr F$ induces a natural isomorphism of sheaves of connective spectra
\[ \mathbb G_a^\sharp \to \tau^{\le 0} \mathscr F. \]
Now $\calE$ is connective. Take $\mathscr F' = \underline{\RHom}_{\calO} (\calE, \mathscr F)$, so by adjunction, the natural map
\[ \underline{\taulezRHom}_{\calO} (\calE, \mathbb G_a^\sharp) \to \tau^{\le 0} \mathscr F' \]
(recall that $\taulezRHom$ is the inner $\Hom$ in the category of connective spectra) is an isomorphism of sheaves of connective spectra. Observe that $\mathscr F'$ is connective since $\underline{\RHom}_{\calO} (\calE, -)$ is $t$-exact, and we conclude that the map above identifies $\mathscr F'$ as the sheafification of
\[ \underline{\taulezRHom}_{\calO} (\calE, \mathbb G_a^\sharp) \cong \underline{\Hom}_{\calO} (\calE, \mathbb G_a^\sharp). \]
In other words, the natural map
\[ R\Gamma(R,(E^\vee)^\sharp) \to R\Gamma(R,\mathscr F') = \RHom_R(\calE(R), R\Gamma(R,\mathbb G_a^\sharp)) \]
is an isomorphism.

As for (5), this is a direct calculation (one may resort to Lemma \ref{lem: p-complete module as classical completion}). A universal object in $(E^\vee)^\sharp (\cGamma_S^\bullet(E))$ is given by mapping each $x \in E$ to the sequence $(\gamma_i(x)) \in (E^\vee)^\sharp$.
\end{proof}

\begin{prop}\label{prop: trivial prismatization}
    Let $(A,I)$ be a bounded prism, and $S$ be a derived $(p,I)$-complete and $(p,I)$-completely flat $\delta$-algebra over $A$, such that $\overline S/p \otimes_{\overline S}^\bL \dL_{\overline S/\overline A}$ is \emph{projective} over $\overline S/p$. Denote $\overline S = S/IS$, then there exists a canonical isomorphism of filtered $\overline S$-algebras
    \[\dHodge_{\overline{S}/\overline{A}} \cong \overline\Prism_{\overline{S}/A}\]
    which fits into the diagram below:
    \[\xymatrix{
    \overline S \oplus \cdL_{\overline S/\overline A} \left[ -1 \right] \left\{ -1 \right\} \ar@{=}[r] \ar[d] & \cdL_{\overline S/A} \left[ -1 \right] \left\{ -1 \right\} \ar[d] \\
    \dHodge_{\overline S/\overline A} \ar[r] & \overline\Prism_{\overline S/A}
    }\]
    (the vertical arrow on the right-hand side is the isomorphism constructed in \cite[Proposition 4.15]{Bhatt_2022}).
\end{prop}

\begin{proof}
    Assume that $\overline S/p \otimes_{\overline S}^\bL \dL_{\overline S/\overline A}$ is generated by a set of cardinality $\lambda$. Pick a cardinal $\kappa>\#S$ of cofinality $>\lambda$, and we apply Lemma \ref{lem: cohomology of Esharp}. Denote $\calC$ the category in Lemma \ref{lem: cohomology of Esharp}.

    We use the prismatisations introduced in \cite{bhatt2022prismatization}. Take $X=\Spf(\overline S)$ and \[X^\HT_{A}\to X\] be its Hodge--Tate stack relative to $(A,I)$. The $\delta$-structure of $S$ provides a section $X\to X_A^\HT$.

    We now repeat the arguments in \cite[Proposition 5.12]{bhatt2022prismatization} to show that $X_A^\HT$ is a $BT_X^\sharp$-torsor, where $T_X^\sharp := (\cOmega_X^\vee)^\sharp$ as in Lemma \ref{lem: cohomology of Esharp}. By derived deformation theory (see for example \cite[Section 5.1]{cesnavicius2023purityflatcohomology} for a discussion on derived deformation theory), for the square zero extension in $\calC$ (see \cite[Construction 5.10]{bhatt2022prismatization})
    \[ B\mathbb G_a^\sharp\{1\} \to \overline{W(R)} \to R \]
    the obstruction to lifting $\overline S \to R$ to $\overline S \to \overline{W(R)}$ is given by $\Ext^1_{\overline S}(\dL_{\overline S/\overline A}, B\mathbb G_a^\sharp (R)\{1\})$ and the groupoid of liftings is given by $\taulezRHom_{\overline S}(\dL_{\overline S/\overline A}, B\mathbb G_a^\sharp (R)\{1\})$. Denote $E' = R \otimes_{\overline S}^\bL \dL_{\overline S/\overline A}$, and we see that $E'$ is projective over $R$, so the obstruction vanishes, and the groupoid of liftings is
    \[ \taulezRHom_R(E', B\mathbb G_a^\sharp (R)\{1\}) = (\tau^{\le 1} \RHom_R(E', \tau^{\le 1} R\Gamma(R, \mathbb G_a^\sharp))) [1] \cong R\Gamma(R, T_X^\sharp)[1] = \Gamma(R, BT_X^\sharp), \]
    where the last two equalities follows from Lemma \ref{lem: cohomology of Esharp} (4).
    
    Now, the formal scheme map $X^\HT_A \to X$ admits a section given by $S$. By \cite[Theorem 7.20]{bhatt2022prismatization} (\cite[Proposition 7.22]{bhatt2022prismatization} guarantees that we can use the classical prismatisation, so we will not emphasise on $X_A^\HT$ being derived or classical) and \cite[Construction 7.6]{bhatt2022prismatization}, there exists isomorphisms
    \[R\Gamma(BT^\sharp_{X/\overline{A}},\calO)\cong R\Gamma(X^{HT}_A,\calO)\cong \overline\Prism_{\overline{S}/A}. \]
    The comparison map $R\Gamma(BT^\sharp_{X/\overline{A}},\calO) \cong \dHodge_{\overline S / \overline A}$ is provided by Theorem \ref{theo: comp strat connection cohomologies} applied to the graded split PD cosimplicial ring $\dHodge_{\overline S/\overline S(\bullet)}$ (in place of \cite[Lemma 7.8]{bhatt2022prismatization}). 
    
    We repeat the arguments in \cite[Remark 7.7]{bhatt2022prismatization} to show that the above map is compatible with the Hodge--Tate comparison. Since on each side, the $i$th graded piece is concentrated on degree $i$, the compatibility statement is purely a statement of cohomology rings. Therefore, we only need to show that they agree on the generating set $\cOmega_{\overline S/\overline A} \{ -1\}$ on degree 1. Argue locally, we may assume that $A$ is oriented with generator $d \in I$. Denote $E(\bullet) = \Prism_{\overline S/S(\bullet)}$, where $S(n) = S^{\cotimes(n+1)}$, and we identify $\overline\Prism_{\overline S/A}$ with $\overline{E(\bullet)}$ (since $E(n)$ is the $(n+1)$-fold product of $S$ over $X_A^\Prism$ by Corollary \ref{cor: animated universal property}). By definition, the map $\cOmega_{\overline S/\overline A} \{-1\} \to \overline{E(1)}$ is given by $d^{-1} \otimes \dif s \mapsto d^{-1} (\delta_0^1(s) - \delta_1^1(s))$ (where $s \in S$). Inspecting the proof of Theorem \ref{theo: comp strat connection cohomologies}\footnote{Recall that by Definition \ref{defi: graded split PD}, we have $d^{-1} \otimes \dif_0(s) = \varphi(d^{-1} \otimes \dif s)$ and $\dif^{(1)} (d^{-1} \otimes \dif_0 s) = \partial ( 1 \otimes (d^{-1} \otimes \dif s))$ by the definitions in Lemma \ref{lem: cosimplicial connection}.}, one sees that the map $\cOmega_{\overline S/\overline A} \{-1\} \to \dHodge_{\overline S/\overline {S(1)}}$ is given by $d^{-1} \dif s \mapsto d^{-1} \otimes \dif_0 s$, where $\dif_0: \overline S \to \cOmega_{S/A} \cong \cdL_{\overline S/\overline{S(1)}} \cong (J/J^2)^\land$ denotes the canonical map $s \mapsto \delta_0^1(s) - \delta_1^1(s)$ (where $J := \ker(\overline{S(1)} \to \overline S)$). Therefore, we only need to show that for each $s \in S$, the induced map $\dHodge_{\overline S/\overline{S(1)}} \to \overline{E(1)}$ sends $d^{-1} \otimes \dif_0 s$ to $d^{-1} (\delta_0^1(s) - \delta_1^1(s))$. This burns down to the following statement: for any $p$-nilpotent ring $R$ equipped with a map $\alpha: \overline{E(1)} \to R$, the induced path object $p: x \to x$ (where $x: \overline{W(S)} \to \overline{W(R)} \to R$ is an element in $X^\HT_A(R)$) gives a differential $\dif': \overline S \to \mathbb G_a^\sharp(R) \{ 1 \}$. Then, the image of $\dif'(s)$ under $\mathbb G_a^\sharp \to \mathbb G_a$ is $d \otimes \alpha(d^{-1}(\delta_0^1(s) - \delta_1^1(s)))$ for any $s \in S$. By derived deformation theory (this involves some concrete calculations, note that one can always adjust the signs of the deformation theory, so there would be no sign issues), the map $\dif'$ can be explicitly written down as
    \[ \dif'(s) = \beta\left(\frac{\delta_0^1(s) - \delta_1^1(s)} d \right) \in \ker(d: W(R) \to W(R)), \]
    where $\beta: E(1) \to W(R)$ is induced from $\alpha$. Now, the desired identity can be verified directly using the definitions in \cite[Construction 5.10]{bhatt2022prismatization} (note that the projection $W[F] \cong \mathbb G_a^\sharp \to \mathbb G_a$ is the same thing as taking the projection $W[F] \to W_1 \cong \mathbb G_a$).

    To show that we have the final commutative diagram, note that $\overline S \oplus \cdL_{\overline S/\overline A} \left[ -1 \right] \left\{ -1 \right\}$ is identified with $\Fil_1(\dHodge_{\overline S/\overline A})$, and $\cdL_{\overline S/A} \left[ -1 \right] \left\{ -1 \right\}$ is identified with $\Fil_1(\overline\Prism_{\overline S/A})$, so the isomorphism gives us a natural identification of the two functors as above. Consider the filtered object $F(S)$
    \[ 0 \subseteq \overline S \subseteq \overline S \oplus \cdL_{\overline S/\overline A} \left[ -1 \right] \left\{ -1 \right\}. \]
    We only need to show that $F(S)$ has a unique automorphism compatible with its filtration and Hodge--Tate comparison. Such an automorphism is always given by a compatible family of $\overline S$-linear maps
    \[ \eta_{S/A}: \cdL_{\overline S/\overline A} [-1]\{-1\} \to \overline S. \]
    
    Now view this as a functor in $(A \to S)$. We have natural commutative diagrams for each $A \to S$:
    \[\xymatrix{
    \cdL_{\overline S/\overline A} [-1]\{-1\} \ar[rr] \ar[rd]_{\eta_{S/A}}& & \cdL_{\overline S/\overline S} [-1]\{-1\} =0 \ar[ld]^{\eta_{S/S}} \\
    & \overline S
    }\]
    giving a natural isomorphism $\eta_{S/A} \cong 0$ (natural in $\eta$), as desired\footnote{This involves the following argument as in Corollary \ref{cor: animated universal property}: for any $\infty$-category $\calC$ and object $\eta_0 \in \calC$, if we are given an isomorphism $\id_\calC \cong \underline{\eta_0}$, then for $\eta_0: \{*\} \to \calC$ and the projection $p: \calC \to \{*\}$, we see that $\eta_0 \circ p = \id_{\{*\}}$ and $p \circ \eta_0 = \underline{\eta_0} \cong \id_\calC$, so $p$ is an equivalence of categories, hence $\calC$ is contractible.}.
\end{proof}

\begin{rmk}
    When $S$ is $(p,I)$-completely smooth (i.e. $S \otimes_A^\bL A/(p,I)$ is smooth over $A/(p,I)$), then the argument agrees with the isomorphism in \cite{Tian_2023} by our proof of Theorem \ref{theo: comp strat connection cohomologies}.

    Also, if we only assume that $\cdL_{\overline S/\overline A}$ is $p$-completely flat, then the conclusion still holds by the following corollary, but the deformation theory becomes harder to handle, so we will not discuss the details of this case.
\end{rmk}

\begin{cor}\label{cor: trivial prismatization}
Let $\left( A,I \right)$ be a bounded prism. Denote $\Alg^\Delta_{A,\delta,\cplt}$ be the $\infty$-category of derived $\left( p,I \right)$-complete animated $\delta$-algebras over $A$. Then, we have an isomorphism of functors $\Alg^\Delta_{A,\delta,\cplt} \to \cD \left( A \right)$ natural in $A$
\[ \dHodge_{(- \otimes^\bL_A \overline A)/\overline A} \to \overline\Prism_{(- \otimes^\bL_A \overline A)/A} \]
which is compatible with the Hodge--Tate comparison, fits into the diagram below
\[\xymatrix{
\overline{(-)} \oplus \cdL_{(- \otimes^\bL_A \overline A)/\overline A} \left[ -1 \right] \left\{ -1 \right\} \ar@{=}[r] \ar[d] & \cdL_{(- \otimes^\bL_A \overline A)/A} \left[ -1 \right] \left\{ -1 \right\} \ar[d] \\
\dHodge_{(- \otimes^\bL_A \overline A)/\overline A} \ar[r] & \overline\Prism_{(- \otimes^\bL_A \overline A)/A}
}\]
and agrees with that given in Proposition \ref{prop: trivial prismatization} when restricted to the subcategory of $S$ satisfying its conditions.
\end{cor}

\begin{proof}
Denote $\calA$ the category of $(p,I)$-completely $\delta$-algebras $S/A$ satisfying the conditions for $S$ in Proposition \ref{prop: trivial prismatization}. Now, since $\calA$ contains all free $\delta$-algebras, $\calA$ generates $\Alg^\Delta_{A,\delta,\cplt}$ under sifted colimits. Now observe that all functors involved commute with sifted colimits, so they are the left Kan extensions of their restrictions to $\calA.$ We define our isomorphism $\tau'$ as the left Kan extension of $\tau$ restricted to $\calA$.

Note that we used free $\delta$-algebras, so it is necessary to consider possibly infinitely generated algebras $S$ in Proposition \ref{prop: trivial prismatization}.

At last, we give a remark that the isomorphism functor is natural in $A$: we may consider the category $\calC_0$ of pairs $((A,I),S)$ where $(A,I)$ is a bounded prism and $S$ is a derived $(p,I)$-complete animated $\delta$-algebra over $A$ --- one may easily construct this category out of the category of morphisms $A \to S$ of animated $\delta$-algebras. Now consider the full subcategory $\calA'$ consisting of $S/A$ satisfying the conditions in Proposition \ref{prop: trivial prismatization}. Denote $\calC$ the category of bounded prisms, then it is easy to see that $\calC_0 \to \calC$ is a coCartesian fibration and $\calA' \subseteq \calC_0$ is closed under coCartesian arrows. Therefore we may take the $(\calC_0 \to \calC)$-relative left Kan extension to complete our construction (the target category is pairs $((A,I),X)$ where $(A,I) \in \calC$ and $X \in \CAlg(\Fil(\cD(A)))$ ).
\end{proof}

\begin{theo}\label{theo: prism=pd derived construction}
    Let $(A,I)$ be a bounded prism. Let $(R,IR)$ be a bounded prism over $(A,I)$ such that $A \to R$ is surjective, and $\overline R/\overline A$ satisfies Condition \ref{cond: tor -1}. Denote $J = \fib(A \to R) = \ker(A \to R)$ and $\overline J = J \otimes_A^\bL \overline A$. Then $\overline J = \fib(\overline A \to \overline R)$ is concentrated on degree 0, and there is a canonical isomorphism compatible with the Hodge--Tate comparison
    \[\overline\Prism_{\overline{R}/A} \cong \cGamma_{\overline R}^*\bigl((\overline J/\overline J^2)^\land \{-1\}\bigr),\]
    such that the map $\bigl(\overline J/\overline J^2\bigr)^\land \left\{ -1 \right\} \to \overline \Prism_{\overline R/A}$ is given by the twist of the reduction of
    \[ J \to I\Prism_{\overline R/A} \]
    (note that by the universal property of $\Prism_{\overline R/A},$ the image of $J$ lies in $I\Prism_{\overline R/A}$).

    By naturality, the map $\overline\Prism_{\overline R/A} \to \overline\Prism_{\overline R/R} = \overline R$ is identified with $\cGamma_{\overline R}^*\bigl((\overline J/\overline J^2)^\land \{-1\}\bigr) \to \overline R$ (sending the positive degree part to zero). Therefore, the positive degree part of $\cGamma_{\overline R}^*\bigl((\overline J/\overline J^2)^\land \{-1\}\bigr) \cong \overline\Prism_{\overline R/A}$ agrees with the reduction of the kernel of $\Prism_{\overline R/A} \to R.$
\end{theo}

\begin{proof}
    {\Res}trict the isomorphism of \ref{cor: trivial prismatization} to the category of all $R$ as above. By Lemma \ref{lemma: completed conormal}, we have an identification $\dHodge_{\overline R/\overline A} \cong \cGamma_{\overline R}^*\bigl((\overline J/\overline J^2)^\land \{-1\}\bigr)$. To verify that $J/J^2 \left\{ -1 \right\} \to \overline \Prism_{\overline R/A}$ is indeed the desired map, we check the definition in \cite[Proposition 4.15]{Bhatt_2022}. It is given by the composition
    \[ J/J^2 \left\{ -1 \right\} \cong \hol^0(\dL_{R/A} \left[ -1 \right] \left\{ -1 \right\}) \to \dL_{\overline \Prism_{\overline R/A} / \Prism_{\overline R/A}} \left[ -1 \right] \left\{ -1 \right\} \cong \overline \Prism_{\overline R/A}. \]
    Using the formalities of conormal modules, it is not hard to verify the map as above is indeed given by $J \to I\Prism_{\overline R/A}$.
\end{proof}

A similar observation has been discovered in \cite[Subsection 4.6]{Tian_2023}. We review (and generalize) it here and explain the relationship between these two constructions.

\begin{construction}\label{cons: prismatic envelope=PD explicit construction}
    Keep the notations in Theorem \ref{theo: prism=pd derived construction}.

    As the ring $A/J$ as well as the quotient $\delta$-structure defines a prism in $\bigl((\overline{A/J})/A\bigr)_\Prism$, hence there is a canonical (by the universal property of $\Prism_{\overline{R}/A}$) surjective morphism of prisms $\Prism_{\overline{R}/A}\to A/J$. Let $K$ be its kernel.

    In addition, we define a canonical homomorphism $\cGamma_{\overline R}^*\bigl((\overline J/\overline J^2)^\land \{-1\}\bigr) \to \overline\Prism_{\overline{R}/A}$ sending $\overline{J}/\overline{J}^2$ to $K/IK$ as follows. Let $f$ be the canonical map from $A$ to $\Prism_{\overline{R}/A}$. As the induced homomorphism $f:\overline{A}\to \overline\Prism_{\overline{R}/A}$ factors through the projection $\overline{A}\to \overline{R}$ by definition, $f$ sends $J$ to $I\Prism_{\overline{R}/A}$. On the other hand, the definition of $K$ implies that $f$ sends $J$ to $K$. Hence 
    \[f(J)\subset K\cap I\Prism_{\overline{R}/A}=IK,\]
    the last equality holds, since $R=\Prism_{\overline{R}/A}$ is $I$-torsion free. Thus, by tensoring with the projection $A\to \overline{R}$
    we get a $\overline{R}$-linear homomorphism
    \[g:\overline J/\overline J^2\to IK/I^2K=K/IK\{1\}.\]
    Define the canonical homomorphism from $\cGamma_{\overline R}^*\bigl((\overline J/\overline J^2)^\land \{-1\}\bigr)$ to $\overline\Prism_{\overline{R}/A}$ as the unique algebra homomorphism extending the $\overline{R}$-linear homomorphism $g\{-1\}$.
    
    We claim that there is a PD-struture on $K/IK$. It is enough to define $\gamma_p$.
    Let
    \[\gamma_p=-\frac{\delta}{(p-1)!}:K\to K.\]
    Obviously, $K$ is $I$-torsion free, so $K/IK$ is canonically embedded in $\overline\Prism_{\overline{R}/A}$.
\end{construction}

The key observation in \cite{Tian_2023} is the following lemma.

\begin{lemma}\label{lemma: tian_ideal}
    Let $A$ be a bounded prism such that $I=(d)$ for a distinguished element $d\in A$. Then in the $(p,I)$-completion of the $\delta$-polynomial ring $A\{x\}^\land$, for any integer $l\geq 0$, we have
    \[\phi(\delta^l(x))\in d A\{x\}^\land+\sum_{j\geq 1}\delta^j(dx)A\{x\}^\land.\]
\end{lemma}

\begin{proof}
    We prove by induction that for each $l\geq 0$,
    \[\phi(\delta^l(x))\in J_l:=\sum_{j=1}^{l+1}d^{p^j}\delta^j(x)A\{x\}^\land+\sum_{j\geq 1}\delta^j(dx)A\{x\}^\land.\]

    For $l=0$, as 
    \[\delta(dx)=\delta(d)\phi(x)+d^p\delta(x),\]
    the claim follows from the fact that $\delta(d)$ is invertible. Suppose that we have proved for $l\leq k$. Since for all $j\geq 0$,\[\delta(d^{p^j}\delta^j(x))=\delta(d^{p^j})\phi(\delta^j(x))+d^{p^{j+1}}\delta^{j+1}(x),\]
    by inductive hypothesis, for any $j\leq k$, $\delta(d^{p^j}\delta^j(x))\in J_k$. It follows that \[\phi(\delta^{k+1}(x))\in \delta(d^{p^{k+1}}\delta^{k+1}(x))A\{x\}^\land+J_k,\]
    and we can assume $\phi(\delta^{k+1}(x))=\delta(d^{p^{k+1}}\delta^{k+1}(x))f+g$ for some $f\in A\{x\}^\land$ and $g\in J_k$.
    
    Denote $u=\phi(\delta^{k}(x))$, then
    \[ \delta \left( u \right) = \phi ( \delta^{k+1} (x) ) = \delta(d^{p^{k+1}}\delta^{k+1}(x))f+g, \]
    and
    \[\delta(d^{p^{k+1}}\delta^{k+1}(x))=\delta(d^{p^{k+1}})\phi(\delta^{k+1}(x))+d^{p^{k+2}}\delta^{k+2}(x)=\delta(d^{p^{k+1}})\delta(u)+d^{p^{k+2}}\delta^{k+2}(x).\]
    This implies that
    \[\delta(u)(1-\delta(d^{p^{k+1}})f) = \delta(u) - \delta (d^{p^{k+1}} \delta^{k+1} (x)) f + fd^{p^{k+2}} \delta^{k+2} (x) = g+ fd^{p^{k+2}} \delta^{k+2} (x) \in J_{k+1}.\] 
    Observe that
    \[ \delta(d^{p^{k+1}}) = \sum_{i=1}^{p^{k+1}} p^{i-1}\binom{p^{k+1}}i d^{p(p^{k+1}-i)} \delta \left( d \right)^i \in pA\{\frac{x_i}{d}\}^\land+dA\{\frac{x_i}{d}\}^\land,\]
    so $(1-\delta(d^{p^{k+1}})f)$ is invertible in $A\{x\}^\land$, and $\phi (\delta^{k+1} (x)) = \delta (u) \in J_{k+1},$ as desired.
\end{proof}

\begin{cor}\label{cor: Explicit PD structure}
    Keep the notations in Construction \ref{cons: prismatic envelope=PD explicit construction}. The map $\gamma_p:K
    \to K$ descends to a map $\overline{\gamma}_p:K/IK\to K/IK$. Moreover, $\overline{\gamma}_p$ defines a PD-structure on $K/IK$.
\end{cor}

\begin{proof}
%    Choose a set of generators $\{f_\lambda:\lambda\in \Lambda\}$ of $J$. Let
%    \[P:=A\{x_\lambda:\lambda\in \Lambda\}\]
%    and $f:P\to A$ be the $\delta$-homomorphism sending $x_\lambda$ to $f_\lambda$. Consider the following diagram of $\delta$-rings
%    \[\begin{tikzcd}[column sep=large]
%        P^\land\arrow[r,"x_\lambda\mapsto 0"]\ar[d,"f"] & A\ar[d]\\ A\ar[r] &R=A/J
%    \end{tikzcd}.\]
%    Taking prismatic envelopes of rows, there is a homomorphism of $\delta$-rings
%    \[\Prism_{\overline{A}/P^\land}\to \Prism_{\overline R/A}.\] We claim that this maps is surjective. It suffices to show that the induced homomorphism $\fil_n \overline \Prism_{\overline{A}/P^\land}\to \fil_n \overline\Prism_{\overline{R}/A}$ is surjective for all $n\geq 0$. By an induction on $n$, we only need to check the surjectivity of graded pieces. By derived Hodge--Tate comparison theorem, this is equivalent to check that the natural homomorphism
%    \[ \cbigwedge^n \cdL_{\overline{A}/\overline{P}}[-n] \to \cbigwedge^n \cdL_{\overline{R}/\overline{A}}[-n]\]
%    is surjective, which is obvious by Lemma \ref{lemma: completed conormal}.

%   Thus, we reduce to the case where $J$ is generated by a (transfinite) regular sequence $\left( a_\lambda: \lambda \in \Lambda \right)$. 
    Using a faithfully flat base change, we may assume that $A$ is an orientable prism, and $I = \left( d \right).$ Using the universal property of $\overline \Prism_{\overline R/A},$ it is easy to check that $K$ is generated (as an ideal) by $\delta^k \left( a/d \right)$ for $k \ge 0$ and $a \in J.$
    
    We first show that $\phi \left( K \right) \subseteq dK.$ As $K$ is a $\delta$-ideal, we have $\phi \left( K \right) \subseteq K$, so we only need to prove that $\phi \left( K \right) \subseteq d \overline \Prism_{\overline R/A}.$ For this, we only need to check on a set of generators of $K$. By Lemma \ref{lemma: tian_ideal}, we see that
    \[ \phi \bigl( \delta^l ( a/d ) \bigr) \in d \overline \Prism_{\overline R/A} + \sum_{j \ge 1} \delta^j \left( a \right) \overline \Prism_{\overline R/A}. \]
    Now, $\delta^j \left( a \right)$ lies in $J \subseteq d\overline \Prism_{\overline R/A}$, as desired.
    
    Now, we have
    \[ \gamma_p \left( x+y \right) - \gamma_p \left( x \right) - \gamma_p \left( y \right) = \frac{\delta\left( x \right) + \delta \left( y \right) - \delta \left( x+y \right)}{\left( p-1 \right)!} = \frac{\left( x+y \right)^p-x^p-y^p}{p!}, \]
    which is exactly
    \[ \sum_{k=1}^{p-1} \frac{x^k y^{p-k}}{k! \left( p-k \right)!}.\]
    Also, for $a \in \overline \Prism_{\overline R/A}$ and $x \in K,$ we have
    \[ \gamma_p \left( ax \right) = -\frac{\delta \left( ax \right)}{\left( p-1 \right)!} = -\frac{\delta \left( a \right) \phi \left( x \right) + \delta \left( x \right) a^p}{\left( p-1 \right)!} \]
    and
    \[ x^p - p!\gamma_p \left( x \right) = x^p + p\delta \left( x \right) = \phi \left( x \right). \]
    Since $\phi \left( x \right) \in dK,$ we see that $x^p-p!\gamma_p \left( x \right) \in dK$ and $\gamma_p \left( ax \right) \equiv a^p \gamma_p \left( x \right) \pmod{dK}.$ It remains to prove that $\gamma_p$ is well-defined after modulo $dK.$ By the identity of addition, we only need to prove that $\gamma_p \left( dK \right) \subseteq dK,$ which is a direct consequence of the identity of multiplication.
\end{proof}

\begin{theo}
    Keep the notations, the homomorphism
    \[\cGamma_{\overline R}^*\bigl((\overline J/\overline J^2)^\land \{-1\}\bigr) \to \overline\Prism_{\overline{R}/A}\]
    is equal to the homomorphism in Theorem \ref{theo: prism=pd derived construction}.
\end{theo}

\begin{proof}
    Let $P=A\{x_\lambda:\lambda\in\Lambda\}^\land$ be a completed (here the derived completion agrees with the classical completion by Proposition \ref{prop: (p,I)-completely flat}) $\delta$-polynomial ring over $A$ and view $A$ as a $\delta$-$P$-algebra via the $\delta$-homomorphism \[P\to A:\ x_\lambda\mapsto 0.\]
    Define $\overline{P}=P/IP$ and $J_0\subset P$ as the kernel of $P\to A$. Let $\overline{J_0}$ be the kernel of the induced map $\overline{P} \to \overline{A}$. Consider the following diagram of $\delta$-rings
    \[\begin{tikzcd}[column sep=large]
        P \arrow[r,"x_\lambda\mapsto 0"]\ar[d,"f"] & A\ar[d]\\ A\ar[r] &R=A/J
    \end{tikzcd}.\]
    Taking prismatic envelopes of rows, there is a homomorphism of $\delta$-rings
    \[\Prism_{\overline{A}/P}\to \Prism_{\overline R/A}.\] We claim that this maps is surjective. Since derived completion is right $t$-exact, and for $M \in \cD(A)$, $M \in D^{\le 0}$ if and only if $M \otimes_A^\bL \overline A \in D^{\le 0}$ (because $M = \lim M \otimes_A^\bL A/I^n$ and $M \otimes_A^\bL A/I^n$ admits a filtration with graded pieces $M \otimes_A^\bL \overline A \{i\}$), it suffices to show that the induced homomorphism $\Fil_n \overline \Prism_{\overline{A}/P}\to \Fil_n \overline\Prism_{\overline{R}/A}$ is surjective for all $n\geq 0$. By an induction on $n$, we only need to check the surjectivity of graded pieces. By Hodge--Tate comparison (Proposition \ref{prop: HT comparison}), this is equivalent to check that the natural homomorphism
    \[ \cbigwedge^n \cdL_{\overline{A}/\overline{P}}[-n] \to \cbigwedge^n \cdL_{\overline{R}/\overline{A}}[-n]\]
    is surjective, which is obvious by Lemma \ref{lemma: completed conormal}.
    
    Therefore, we only need to prove that the two homomorphisms
    \[\cGamma_{\overline P}^*\bigl((\overline J_0/\overline J_0^2)^\land \{-1\}\bigr) \to \overline\Prism_{\overline{A}/P}\]
    coincide.

    We may check this Zariski locally, hence assuming that $I$ is generated by a distinguished element $d$. Hence, we have a $\delta$-homomorphism
    \[\dZ_p\bigl\{t,\frac{1}{\delta(t)}\bigr\}^\land\to A\]
    sending $t$ to the element $d$ in $A$. As both homomorphisms from $\cGamma_{\overline P}^*\bigl((\overline J_0/\overline J_0^2)^\land \{-1\}\bigr)$ to $\overline\Prism_{\overline{A}/P}$ are functorial in $A$ and $P$, we only need to deal with the universal case $A=\dZ_p\{t,\frac{1}{\delta(t)}\}^\land$. In particular, we may assume that $A$ is transversal.

    Assuming the transversality of $A$, as both $\cGamma_{\overline P}^*\bigl((\overline J_0/\overline J_0^2)^\land \{-1\}\bigr)$ and $\overline \Prism_{\overline{R}/A}$ are $p$-complete and $p$-completely flat over $\overline{A}$, they are both $p$-torsion free. Hence, for any homomorphism $\psi:\widehat{\Sym}_{\overline P}^*\bigl((\overline J_0/\overline J_0^2)^\land \{-1\}\bigr)\to \overline\Prism_{\overline{A}/P}$, it admits at most one extension to $\cGamma_{\overline P}^*\bigl((\overline J_0/\overline J_0^2)^\land \{-1\}\bigr)$. Thus, we only need to check that two homomorphisms coincide on $\overline J_0/\overline{J}_0^2\{-1\}$, which is exactly the last sentence of Theorem \ref{theo: prism=pd derived construction}.
\end{proof}

\subsection{De Rham Realization}
In this section we state and prove Theorem \ref{theo: main dR realization}. For this we first introduce some notations.

\begin{notn}\label{notn: Strat and MIC}
Let $R$ be a classically $p$-adically complete ring with bounded $p^\infty$-torsion, and $S^\bullet$ be split PD over $R$ (see Condition \ref{cond: split pd}). Then, we will use
\[ \Strat \left( S^\bullet \right), \Strat \bigl( S^\bullet \bigl[\frac 1 p \bigr] \bigr) \]
to denote categories of finite projective stratifications over $S^\bullet$ and $S^\bullet \bigl[ \frac 1 p \bigr],$ and use
\[ \MIC^\tn \left( S^0,\dif \right), \MIC^\tn \bigl( S^0 \bigl[ \frac 1 p \bigr], \dif \bigr) \]
to denote the categories of finite projective modules over $S^0$ and $S^0 \bigl[ \frac 1 p \bigr]$ equipped with a topologically nilpotent (continuous) integrable connection (see Definition \ref{defi: tn}) over $\left( S^0,\dif \right)$ (where $\dif$ is the differential operator defined in Proposition \ref{prop: strat to connection}). In the case $S^0=R$, we will instead use the following notations for $\MIC^\tn \left( S^0,\dif \right)$ and $\MIC^\tn \bigl( S^0 \bigl[ \frac 1 p \bigr], \dif \bigr)$:
\[ \Higgs^\tn(R, \Omega), \quad \Higgs^\tn\bigl(R\bigl[ \frac 1 p \bigr],\Omega). \]
Here we do not write $\Omega \bigl[ \frac 1 p \bigr]$, because topology nilpotency really depends on the integral structure. Either by combining Theorem \ref{theo: comp strat connection} and Theorem \ref{theo: split PD vs derivation}, or by a calculation by hand, one sees that $\Strat$ is functorial in $S^\bullet$, and $\MIC^\tn$ is functorial in $(S^0,\dif)$ (these two functorialities are the same thing by Theorem \ref{theo: split PD vs derivation}).
\end{notn}

In this section, we will always assume that $A$ is a bounded prism, $S/A$ satisfies Condition \ref{cond: flat + tor 0}, and $R/\overline S$ satisfies Condition \ref{cond: tor -1}.

\begin{lemma}\label{lemma: E(n) split PD}
Let $(S,I)$ be any bounded prism. Define $S \left( n \right) = S^{\cotimes (n+1)}$ and $E \left( n \right) = \Prism_{\overline S / S \left( n \right)}$ (hence $E \left( 0 \right) = S$). Then, the cosimplicial ring $\overline{E \left( n \right)}$, equipped with the Hodge--Tate filtration, is filtered split PD (see Definition \ref{defi: filtered split PD}) over $\overline S,$ with $\Omega = \cOmega_{\overline S/\overline A} \left\{ -1 \right\},$ and $\varphi: \Omega \to \ker \bigl( \overline{E \left( 1 \right)} \to \overline{E \left( 0 \right)} \bigr)$ is given by $\Omega \cong (\overline J/\overline J^2) \left\{ -1 \right\} \to \overline\Prism_{\overline S/S \left( 1 \right)}$ as in Theorem \ref{theo: prism=pd derived construction} (where $J = \ker \left( S \left( 1 \right) \to S \right)$). Moreover, $\gr_* (S(\bullet))$ is identified with $\dHodge_{\overline S/\overline S(\bullet)}$ (see as in Definition \ref{defi: graded split PD}).
\end{lemma}

\begin{proof}
Denote $J_n = \ker \left( S \left( n \right) \to S \right).$ By Theorem \ref{theo: prism=pd derived construction}, we have natural (in $S$) PD structures on $\ker \bigl( \overline{E(n)} \to \overline S \bigr)$, and maps
\[ \varphi_n: \bigl( \overline J_n/\overline J_n^2 \bigr)^\land \left\{ -1 \right\} \to \ker \bigl( \overline{E(n)} \to \overline{E(0)} = \overline S \bigr)  \]
inducing natural (strict) isomorphisms of filtered rings
\[ \overline{E(n)} = \cGamma_{\overline S}^* \bigl( \bigl( \overline J_n/\overline J_n^2 \bigr)^\land \left\{ -1 \right\} \bigr). \]
Take $\varphi = \varphi_1.$

We first verify the additivity law, i.e. $\left(\delta_1^2 - \delta_2^2-\delta_0^2 \right) \varphi =0.$ It is equal to the composition
\[ \bigl( \overline J/\overline J^2 \bigr)^\land \left\{ -1 \right\} \xrightarrow{\delta_1^2 - \delta_2^2-\delta_0^2} \bigl( \overline J_2/\overline J_2^2 \bigr)^\land \left\{ -1 \right\} \xrightarrow{\varphi_2 \left\{ 1 \right\}} \ker \bigl( \overline{E(2)} \to \overline S \bigr), \]
where $\varphi_2$ is again the map in Theorem \ref{theo: prism=pd derived construction} for $S(2) \to S.$ Since $\bigl( \overline J/\overline J^2 \bigr)^\land$ is topologically generated by elements of form $1 \otimes s - s \otimes 1,$ it is easy to verify that
\[ \delta_1^2 - \delta_2^2 - \delta_0^2: \bigl( \overline J/\overline J^2 \bigr)^\land \to \bigl( \overline J_2/\overline J_2^2 \bigr)^\land \]
is zero, hence $\left(\delta_1^2 - \delta_2^2-\delta_0^2 \right) \varphi =0.$

We now verify the isomorphisms. However, strict isomorphisms can be verified on gradeds, and we already know that $\gr_*(\overline{E(\bullet)}) \cong \dHodge_{\overline S / \overline{S(\bullet)}}$ is graded split PD, so we get our isomorphisms as desired.
\end{proof}

\begin{prop}\label{prop: D(n) split PD}
Define $D \left( n \right) = \Prism_{R / S \left( n \right)}.$ Then, the cosimplicial ring $\overline{D \left( n \right)}$ is filtered split PD over $R$ with respect to the Hodge--Tate filtration, with $\Omega = R \cotimes_{\overline S} \cOmega_{\overline S/\overline A} \left\{ -1 \right\}$ (classical completion, by Lemma \ref{lem: p-complete module as classical completion}). Moreover, we have $\bigl( \overline{D(n)} \cdot \ker \bigl( \overline{E(n)} \to \overline{E(1)} \bigr) \bigr)^\land = \ker \bigl( \overline{D(n)} \to \overline{D(1)} \bigr)$, the PD structure on this ideal is uniquely extended from $\overline{E(n)}$, and $\varphi: \Omega \to \overline{D(1)}$ is induced by $\cOmega_{\overline S/\overline A} \left\{ -1 \right\} \to \overline{E(1)} \to \overline{D(1)}$ as in Lemma \ref{lemma: E(n) split PD}. Moreover, the split PD structure on $\gr_* ( \overline{D(\bullet)} )$ is the same as $\dHodge_{R/\overline{S(\bullet)}}$ in Definition \ref{defi: graded split PD}.
\end{prop}

\begin{proof}
We have natural maps $D(0) \cotimes_{E(0)} E(n) \to D(n)$, where $q_i^n: E(0) \to E(n)$ and $q_i^n: D(0) \to D(n)$. We now verify that after reducing modulo $I$, they are strict isomorphisms of filtered rings. This reduces to verifying that
\[ \dHodge_{R/\overline S} \cotimes_{\overline S} \dHodge_{\overline S/\overline{S(n)}} \to \dHodge_{R/\overline{S(n)}} \]
is an isomorphism, which is a direct consequence of the exact triangle of cotangent complexes.

As $\overline{D(0)} \cotimes_{\overline{E(0)}} \overline{E(n)} \to \overline{D(n)}$ are isomorphisms, the PD structure on $\ker \bigl( \overline{E(n)} \to \overline{E(1)} \bigr)$ indeed (uniquely, hence independent of $i$) extends to $\bigl( \overline{D(n)} \cdot \ker \bigl( \overline{E(n)} \to \overline{E(1)} \bigr) \bigr)^\land = \ker \bigl( \overline{D(n)} \to \overline{D(1)} \bigr)$ (equality follows from the tensor product). Again, the data give a split PD structure on $\overline{D(\bullet)}$ because the graded pieces $\gr_*(\overline{D(\bullet)})$ is split PD. (Note that the PD structure on $\gr_*(\overline{D(n)}) \cong \dHodge_{R/\overline{S(n)}}$ coincides with the extension of the PD structure on $\gr_*(\overline{E(n)}) \cong \dHodge_{\overline S/\overline{S(n)}}$)
\end{proof}

Now denote $D = D \left( 0 \right) = \Prism_{R/S},$ and $\dif: \overline D \to \overline D \cotimes_R ( R \cotimes_{\overline S} \cOmega_{\overline S/\overline A}) \left\{ -1 \right\}$ the differential as in Proposition \ref{prop: strat to connection}, so Theorem \ref{theo: comp strat connection} gives us the following proposition:

\begin{prop}
We have equivalences of symmetric monoidal categories:
\begin{align*}
\Strat \bigl( \overline{D \left( n \right)} \bigr) & \xrightarrow{\cong} \MIC^\tn \bigl( \overline D,\dif \bigr), \\
\Strat \bigl( \overline{D \left( n \right)} \bigl[ \frac 1 p \bigr] \bigr) & \xrightarrow{\cong} \MIC^\tn \bigl( \overline D \bigl[ \frac 1 p \bigr],\dif \bigr).
\end{align*}
\end{prop}

Also, by Proposition \ref{prop: weakly final}, $D \to *$ is a covering in the topos $\Shv \bigl( (R/A)_\Prism \bigr),$ so Lemma \ref{lem: weakly final computation} gives the equivalences of symmetric monoidal categories as below (along with comparison of cohomologies)
\begin{align*}
\Vect \bigl( (R/A)_\Prism, \overline\calO_{(R/A)_\Prism} \bigr) & \xrightarrow{\cong} \Strat \bigl( \overline{D \left( n \right)} \bigr), \\
\Vect \bigl( (R/A)_\Prism, \overline\calO_{(R/A)_\Prism} \bigl[ \frac 1 p \bigr] \bigr) & \xrightarrow{\cong} \Strat \bigl( \overline{D \left( n \right)} \bigl[ \frac 1 p \bigr] \bigr).
\end{align*}

We are now ready to prove one of our main theorems:

\begin{theo}[See Theorem \ref{theo: main dR realization}]\label{theo: main dR realization revisit}
    Let $R$ be an $\overline A$-algebra, and $S$ be a bounded prism over $A$, equipped with a morphism of $A$-algebras $S \to R.$ Assume that $S/A$ satisfies Condition \ref{cond: flat + tor 0} and $R/S$ satisfies Condition \ref{cond: tor -1}. Then, $D := \Prism_{R/S}$ is concentrated on degree 0, and a bounded prism over $S$.
    
    Define $\overline D = D/ID$. Then, we have an $R$-linear derivation on $\overline D$ (integrable as a Higgs field over $R$):
    \[ \dif: \overline D \to \overline D \cotimes_R \bigl( R \cotimes_{\overline S} \cOmega_{\overline S/\overline A} \bigr) \left\{ -1 \right\}, \]
    satisfying the Griffith transversality, with $\gr_*(\dif)$ as in Example \ref{ex: differential dHodge}, such that we have the symmetric monoidal correspondences below:
    \begin{align*}
        \Vect \bigl( \left( R/A \right)_\Prism, \overline\calO \bigr) & \cong \MIC^\tn \bigl( \overline D, \dif \bigr), \\
        \Vect \bigl( \left( R/A \right)_\Prism, \overline\calO \bigl[ \frac 1 p \bigr] \bigr) & \cong \MIC^\tn \bigl( \overline D \bigl[ \frac 1 p \bigr], \dif \bigr).
    \end{align*}
    Moreover, if $\mathcal E$ corresponds to $\left( M,\nabla \right)$ under the correspondence above, then we have a natural isomorphism of cohomologies (compatible with the symmetric monoidal structure)
    \[ R\Gamma \bigl( \left( R/A \right)_\Prism, \mathcal E \bigr) \cong \DR \left( M,\nabla \right). \]
    When $\calE=\overline\calO$ is trivial, then the above isomorphism can be made into a strict isomorphism of filtered complete $\dE_\infty$-$R$-algebras, and its graded pieces is identified with the map
    \[ \dHodge_{R/A} \to \DR(\dHodge_{R/S}, \dif) \]
    as in Corollary \ref{cor: natural isomorphism of dR and dHodge}.
\end{theo}

\begin{proof}
The part of categorical equivalences is just as what we have discussed. The part of cohomologies is nothing but Theorem \ref{theo: comp strat connection cohomologies}.

As for the filtered part, we know by Theorem \ref{theo: comp strat connection cohomologies} and Proposition \ref{prop: comp strat connection compatible dHodge}, that the cosimplicial object $\overline{D(\bullet)}$ is compatible with totalisation, and that the isomorphism of filtered objects
\[ \lim_\simp \overline{D(\bullet)} \cong \DR(\overline D, \dif) \]
is compatible with the Hodge--Tate filtration and the map $\dHodge_{R/A} \to \DR(\dHodge_{R/S}, \dif)$. Therefore, we only need to show that the map
\[ \lim_\simp\overline{D(\bullet)} \leftarrow R\Gamma((R/A)_\Prism, \overline\calO) \leftarrow \overline \Prism_{R/A} \]
can be made into an isomorphism of filtered objects compatible with the Hodge--Tate comparison (i.e. its graded pieces are $\dHodge_{R/A} \to \lim_\simp \dHodge_{R/S(\bullet)}$). Observe that the map is given by
\[ \lim_\simp \overline{D(\bullet)} = \lim_\simp \overline\Prism_{R/S(\bullet)} \leftarrow \lim_{(R/A)_\Prism} \overline\calO = R\Gamma((R/A)_\Prism, \overline\calO). \]
Recall that by Corollary \ref{cor: identification of prismatic maps},
\[ \overline\Prism_{R/A} \to R\Gamma((R/A)_\Prism, \overline\calO) \to \overline{D(\bullet)} \]
is identified with the map $\overline\Prism_{R/A} \to \overline\Prism_{R/S(\bullet)}$ by functoriality of $\overline\Prism.$ Therefore, the map of filtered objects $\overline\Prism_{R/A} \to \lim_\simp \overline\Prism_{R/S(\bullet)}$ reduces to the desired map after forgetting the filtration, as desired.
\end{proof}

\begin{ex}[Irrational disks]\label{ex: dR realisation for irrational disks}
Let $(A,I) = (\AAinf(\calO_C),\ker \theta)$ and $R = \calO_C\bigl\{ \frac T {p^r} \bigr\}$ (see Section \ref{sect: irrational disks}). Then, $R$ admits a $\delta$-lifting (with $\delta(T/[p^\flat]^s) =0$)
\[ S = \ccolim_{s \in \mathbb Q, s \le r} A \bigl\langle \frac T {[p^\flat]^s} \bigr\rangle, \]
satisfying Condition \ref{cond: flat + tor 0}, and we get symmetric monoidal correspondences as below:
\begin{align*}
    \Vect \bigl( \left( R/A \right)_\Prism, \overline\calO \bigr) & \cong \Higgs^\tn \bigl( R, \cOmega_{R/\overline A}\{-1\} \bigr), \\
    \Vect \bigl( \left( R/A \right)_\Prism, \overline\calO \bigl[ \frac 1 p \bigr] \bigr) & \cong \Higgs^\tn \bigl( R \bigl[ \frac 1 p \bigr], \cOmega_{R/\overline A}\{-1\} \bigr).
\end{align*}
In the correspondence as above, if $\calE$ corresponds to $(M,\nabla)$, then we have a natural isomorphism of the cohomologies
\[ R\Gamma \bigl( \left( R/A \right)_\Prism, \mathcal E \bigr) \cong \DR \left( M,\nabla \right). \]
\end{ex}

\section{Abstract coordinate models, abstract toric charts and Sen theory}

In preparation of various applications, we will construct some basic models. To avoid some boring calculation of Breuil--Kisin twists, we assume that the base prism has sufficiently many roots of unity. Explicitly, let
\[\dZ_p^{\cyc}=\widehat{\dZ_p[\mu_{p^{\infty}}]}\]
be the completion of the valuation ring of $\bigcup_{n\geq 1}\dQ_p(\mu_{p^n})$.

Let $\overline{A}/\dZ_p^{\cyc}$ be an integral perfectoid algebra and put $A=\AAinf(\overline{A})$. Let $R/\overline{A}$ be a $p$-complete and $p$-torsion free algebra that satisfies Condition \ref{cond: tor -1,0}.

\begin{defi}\label{defi: coord model}
    An abstract coordinate model of $R$ consists of the following:
    \begin{enumerate}
        \item A $\delta$-algebra $S$ over $A$ that satisfies Condition \ref{cond: flat + tor 0};
        \item A class of finitely many elements $T_1,T_2,\dots,T_d\in S$ such that $\delta(T_i)=0$ for any $i$;
        \item A homomorphism of $A$-algebras $f:S\to R$ satisfying Condition \ref{cond: tor -1} such that
        \[\cdL_{R/\overline A}[\frac{1}{p}]=\bigoplus_{i=1}^d R[\frac{1}{p}]\dif f(T_i)[0].\]
    \end{enumerate}
    The induced map $f:\overline{A}\langle T_1,T_2,\dots,T_d \rangle$ is called the chart of this abstract coordinate model. In addition, the abstract coordinate model is called toric if satisfies the following assumptions.
    \begin{enumerate}
        \item For any $i$, $f(T_i)$ is invertible in $R[\frac{1}{p}]$.
    
        \item The ring
        \[\bigl(\overline{A}\langle T_1^{\frac{1}{p^\infty}},T_2^{\frac{1}{p^\infty}},\dots, T_d^{\frac{1}{p^\infty}} \rangle\cotimes_{\overline{A}\langle T_1,T_2,\dots, T_d\rangle}R\bigr)[\frac{1}{p}]\]
        is a perfectoid algebra. 
    \end{enumerate}
    We define $R_\infty$ as the $p$-closure of \[\overline{A}\langle T_1^{\frac{1}{p^\infty}},T_2^{\frac{1}{p^\infty}},\dots, T_d^{\frac{1}{p^\infty}} \rangle\cotimes_{\overline{A}\langle T_1,T_2,\dots, T_d\rangle}R\]
    in
    \[\bigl(\overline{A}\langle T_1^{\frac{1}{p^\infty}},T_2^{\frac{1}{p^\infty}},\dots, T_d^{\frac{1}{p^\infty}} \rangle\cotimes_{\overline{A}\langle T_1,T_2,\dots, T_d\rangle}R\bigr)[\frac{1}{p}],\]
    which is an integral perfectoid algebra.
\end{defi}

\begin{rmk}
     We explain the idea of our framework so that the readers can understand our setting easily.

    The link between prismatic theory and the $p$-adic Simpson theory was introduced by Min-Wang in \cite{Min_2025}, which said that for any smooth formal scheme $X/\calO_{\mathbf{C}_p}$, the restriction provides an equivalence between Hodge--Tate crystals and $v$-vector bundles on the generic fiber of $X$ satisfying some nilpotency conditions of its Sen operator.

    The key technique of their proof was a calculation on the toric chart. Explicitly, they fixed an affine open $\Spf(R)\subseteq X$ and a $p$-completely \'etale homomorphism
    \[f:\calO_{\mathbf{C}_p}\langle T_1^{\pm1}, T_{2}^{\pm1},\dots, T_{d}^{\pm1} \rangle\to R.\]
    Under this assumption, we can lift $f$ to a $(p,I)$-completely \'etale homomorphism
    \[\widetilde f:\Ainf\langle T_1^{\pm1}, T_{2}^{\pm1},\dots, T_{d}^{\pm1} \rangle\to \widetilde R\]
    and there exists a $\delta$-structure on $\widetilde R$ such that $\delta(T_{i})=0$ for any $i$. Using this structure, they can translate Hodge--Tate crystals to Higgs bundles and can compare them to generalized representations by some exponentials.

    However, for general algebras over $\calO_{\mathbf{C}_p}$ satisfying Condition \ref{cond: tor -1,0}, it is hard to find such an \'etale homomorphism \[f:\calO_{\mathbf{C}_p}\langle T_1^{\pm1}, T_{2}^{\pm1},\dots, T_{d}^{\pm1} \rangle\to R.\]
    To solve this problem, we will not only choose a chart, but also choose a map $S\to R$ satisfying Condition \ref{cond: tor -1}. The reader may imagine that there exists a `ghost prism $\widetilde R$' so that $\delta(T_i)=0$ and $\overline{\widetilde R}=R$. 
\end{rmk}

\subsection{Locally perfectoid Tate algebras}

\begin{defi}\label{defi: locally perfectoid huber pairs}
    Let $(R,R^+)/(\mathbb{Q}_p,\mathbb{Z}_p)$ be a Tate--Huber pair. We say that $(R,R^+)$ is \textit{locally perfectoid} if it satisfies the following conditions:
    \begin{enumerate}
        \item $R$ is stably uniform;
        \item $\Spa(R,R^+)$ is a perfectoid space in the sense of \cite{scholze2012perfectoid}.
    \end{enumerate}
\end{defi}

Clearly, a perfectoid Tate--Huber pair $(R,R^+)$ over $\mathbb{Q}_p$ locally perfectoid. We provide another important example. Suppose $K/\mathbb{Q}_p$ is a perfectoid field and $\mathcal{O}_K$ is its ring of integers. Let
\[
\mathrm{D}_d := K\langle T_1,T_2,\dots,T_d \rangle
\]
and
\[
\mathrm{D}_d^+ := \mathcal{O}_K\langle T_1,T_2,\dots,T_d \rangle.
\]

\begin{theo}\label{theo: lci locally perfectoid}
    Under the above assumptions, let $R/K$ be a topologically finite-type Tate algebra and
    \[ f: \mathrm{D}_d \to R \]
    be an \'etale morphism in the sense of \cite[Definition 1.6.5]{Huber_1996}. Then for any ring of integers $R^+ \subseteq R$ containing all $f(T_i)$, the Tate--Huber pair $(R_\infty, R_\infty^+)$ is locally perfectoid, where
    \[
    R_\infty = R \cotimes_{\mathrm{D}_d} K\langle T_1^{\frac{1}{p^\infty}}, T_2^{\frac{1}{p^\infty}}, \dots, T_d^{\frac{1}{p^\infty}} \rangle,
    \]
    and $R_\infty^+$ is the integral closure of
    \[
    R^+ \cotimes_{\mathrm{D}_d^+} \mathcal{O}_K\langle T_1^{\frac{1}{p^\infty}}, T_2^{\frac{1}{p^\infty}}, \dots, T_d^{\frac{1}{p^\infty}} \rangle
    \]
    in $R_\infty$.
\end{theo}

\begin{proof}
    First, \cite[Proposition 8.1.2]{Fresnel_2004} shows that there exists a rational open cover
    \[ \Spa(R,R^+) = \bigcup_{i=1}^N \Spa(R_i,R_i^+) \]
    such that for all $i$, the induced morphism $f_i: \Spa(\mathrm{D}_d,\mathrm{D}_d^+) \to \Spa(R_i,R_i^+)$ is standard \'etale (i.e., a composition of finitely many rational localizations and finite \'etale morphisms). Therefore, by \cite[Theorem 1.8 \& Theorem 1.10]{scholze2012perfectoid}, for any $i$, 
    \[
    R_{i,\infty} := R_i \cotimes_{\mathrm{D}_d} K\langle T_1^{\frac{1}{p^\infty}}, T_2^{\frac{1}{p^\infty}}, \dots, T_d^{\frac{1}{p^\infty}} \rangle
    \]
    is a perfectoid Tate algebra.

    It remains to show that $R_\infty$ is stably uniform. Indeed, it suffices to show that $R_\infty$ is uniform. In fact, if this is correct, put
    \[ R_n=R\otimes_{\mathrm D_d}K\langle T_1^{\frac{1}{p^n}},T_2^{\frac{1}{p^n}}\dots,T_d^{\frac{1}{p^n}} \rangle,\ \forall n\geq 0\]
    and put $R_n^+$ the integral closure of
    \[ R^+ \otimes_{\mathrm D_{d}}\calO_K\langle  T_1^{\frac{1}{p^n}},T_2^{\frac{1}{p^n}}\dots,T_d^{\frac{1}{p^n}} \rangle \]
    in $R_n$. Then, as topological spaces, we have the inverse limit
    \[\Spa(R_\infty,R_\infty^+) = \lim_n \Spa(R_n,R_n^+).\]
    It follows that each rational localization of $R_{\infty}$ comes from a pull-back from a rational localization of some $R_n$. 
    Suppose $(S,S^+)$ is a rational localization of $(R_n,R^+_n)$ for some $n$. Define $f':\mathrm D_d\to S$ by sending $T_i$ to $f(T_i^{\frac{1}{p^n}})$. Then we may apply the same argument to $R=S$ and $f=f'$, deducing that $S\cotimes_{R_n}R_\infty$ is uniform.
    
    Recall from the first paragraph that we have chosen rational localizations $\{R\to R_i:i=1,2,\dots,N\}$ such that the induced morphism $\mathrm{D}_d \to R_i$ is standard \'etale.  Note that since $R$ is strongly noetherian, $(R,R^+)$ is sheafy. Hence the \v{C}ech sequence
    \[
    0 \to R^\circ \to \prod_{i=1}^N R_i^\circ \to \prod_{i,j=1}^N (R_i \cotimes_R R_j)^\circ
    \]
    is exact. Take the cokernel $Q$ of $R^\circ \to \prod R_i^\circ$, and we get an exact sequence
    \[ 0 \to R^\circ \to \prod_{i=1}^N R_i^\circ \to Q \to 0. \]
    Note that $Q$ is derived $p$-adically complete, and is $p$-torsion-free as a submodule of $\prod_{i,j=1}^N (R_i \cotimes_R R_j)^\circ$.
    
    Note that $\mathrm D_{d,\infty}^+:=\mathcal{O}_K\langle T_1^{\frac{1}{p^\infty}}, T_2^{\frac{1}{p^\infty}}, \dots, T_d^{\frac{1}{p^\infty}} \rangle$ is the $p$-adic completion of the free module
    \[\bigoplus_{\alpha_{1},\alpha_{2},\dots,\alpha_d\in{\dZ[\frac{1}{p}]\cap[0,1)}}\mathrm{D}_d^+T_1^{\alpha_1} \cdots T_d^{\alpha_d}\]
    over $\mathrm{D}_d^+$, so it is $p$-adically complete and $p$-completely flat. Therefore, $- \cotimes_{\mathrm{D}_d^+}\mathrm{D}_{d,\infty}^+$ calculates $- \cotimes^\bL_{\mathrm{D}_d^+}\mathrm{D}_{d,\infty}^+$. It follows that the induced sequence
    \[
    0 \to R^\circ \cotimes_{\mathrm{D}_d^+}\mathrm{D}_{d,\infty}^+ \to \prod_{i=1}^N R_i^\circ \cotimes_{\mathrm{D}_d^+} \mathrm{D}_{d,\infty}^+\to Q \cotimes_{\mathrm D_d^+} \mathrm{D}_{d,\infty}^+ \to 0
    \]
    is exact. Note that $Q \cotimes_{\mathrm D_d^+} \mathrm D_{d,\infty}^+$ is classically $p$-adically complete by Corollary \ref{cor: tensor conc on degree 0}. Inverting $p$, and we see that $R_\infty$ is a closed subalgebra of the perfectoid Tate ring $\prod_{i=1}^N R_{i,\infty}$, hence is uniform.
\end{proof}

\begin{defi}
    Let $(R,R^+)$ be a Tate Huber pair, and let $f_1, f_2, \dots, f_n \in R$. For any subset $I \subseteq \{1,\ldots,n\}$, define
    \[
    V_I(R,R^+) = \{ x \in \Spa(R,R^+) : |f_i|_x \leq 1, \forall i \in I; |f_j|_x \geq 1, \forall j \in \{1,\ldots,n\} \setminus I \}.
    \]
    Clearly, all the $V_I(R,R^+)$ form a rational open cover of $\Spa(R,R^+)$.
\end{defi}

\begin{lemma}\label{lemma: Laurent refine}
    Let $(R,R^+)$ be a Tate Huber pair, and let $X = \Spa(R,R^+)$. Then any open cover of $X$ can be refined by some Laurent open cover.
\end{lemma}

\begin{proof}
    \cite[Lemma 8]{Buzzard_2016}.
\end{proof}

\begin{prop}\label{prop: locally perfectoid independ integers}
    Suppose that $(R,R^+)$ is a Tate--Huber pair and $R$ is stably uniform. Then $(R,R^+)$ is locally perfectoid if and only if there exists a finite set $\{f_1, f_2, \dots, f_r\} \subseteq R$ such that for all $I \subseteq \{1,2,\dots,r\}$, the rational localization corresponding to $V_I(R,R^+)$ is perfectoid. In particular, if $(R,R^+)$ is locally perfectoid, then for any ring of integers $R' \subseteq R$, the Tate--Huber pair $(R,R')$ is locally perfectoid.
\end{prop}

\begin{proof}
    This is a direct corollary of Lemma \ref{lemma: Laurent refine}.
\end{proof}

\begin{defi}\label{defi:locally perfectoid algebras}
    For a Tate algebra $R$ over $\dQ_p$, call $R$ \emph{locally perfectoid} if there exists a ring of integers $R^+$ such that $(R,R^+)$ is locally perfectoid.
\end{defi}

By Proposition \ref{prop: locally perfectoid independ integers}, for any locally perfectoid algebra $R$ and any ring of integers $R^+$, the Tate--Huber pair $(R,R^+)$ is locally perfectoid (cf. Definition \ref{defi: locally perfectoid huber pairs}).

\begin{lemma}\label{lemma: v descent of locally perfectoid}
	Suppose that $(R,R^+)$ is a locally perfectoid Tate--Huber pair. Then the functor $\Gamma\big(\Spa(R,R^+),-\big)$ defines an equivalence from the category of $v$-vector bundles on $\Spa(R,R^+)$ to the category of finite-rank projective modules over $R$.
\end{lemma}

\begin{proof}
	Let $\mu$ be the projection from the $v$-topology to the analytic topology on $\Spa(R,R^+)$. Since $X=\Spa(R,R^+)$ is a perfectoid space, $\mu_*$ is an equivalence from the category of $v$-vector bundles on $X$ to the category of analytic vector bundles on $X$. Furthermore, since $(R,R^+)$ has the sheaf property, by \cite[Theorem 2.7.7]{KEDLAYA_2018}, taking global sections gives an equivalence between analytic vector bundles on $X$ and finite-rank projective modules over $R$. Composing these two equivalences yields the desired result.
\end{proof}

\begin{lemma}\label{lem: tensor product and continuous functions}
	Suppose that $A$ is a Tate ring and $M$ is an adic $A$-module (see Proposition-Definition \ref{prop-defi: adic complete}). For a quasi-compact topological space $X$, let $\calC(X,A)$ be the ring of $A$-valued continuous functions on $X$. Then there is a canonical isomorphism 
	\[\calC(X,A)\cotimes_{A}M \to \mathcal{C}(X,M).\]
\end{lemma}

\begin{proof}
We may replace $X$ by the profinite space $\pi_0(X)$, hence assuming that $X$ is profinite.

For any abelian group $M$, denote $\calC_0(X,M)$ the group of locally constant functions from $X$ to $M$. If we write $X = \lim_{i \in I^\op} X_i$ for some finite sets $X_i$ and a filtered category $I$, we get
\[ \calC_0(X,-) = \colim_{i \in I} \calC_0(X_i,-) \]
so $\calC_0(X,-)$ is exact, and the natural maps
\[ \calC_0(X,A) \otimes_A^\bL M \to \calC_0(X,A) \otimes_A M \to \calC_0(X,M) \]
are isomorphisms for any ring $A$ and $A$-module $M$. (We use here that $\calC_0(X,A) = \colim_{i \in I}\calC_0(X_i,A)$ is flat over $A$.)

Back to the original lemma, we may pick a ring of definition $A^+ \subseteq A$ and an $A^+$-module of definition $M^+ \subseteq M$. Now, $\calC_0(X,A^+)$ is flat over $A^+$, and we have an isomorphism
\[ \calC_0(X,A^+) \otimes_{A^+} M^+ \to \calC_0(X,M^+). \]
Now, take the derived $p$-adic completion on both sides. Since $\calC_0(X,-)$ is exact, we see that $\calC_0(X,M^+)/p^n = \calC_0(X,M^+/p^n)$, and we get an isomorphism $\calC_0(X,M^+)^\land \to \calC(X,M^+)$ (here $(-)^\land$ denotes the derived completion, which is the same as the classical completion). Therefore, $\calC(X,A^+)$ is $p$-complete and $p$-completely flat over $A^+$, and we get an isomorphism
\[ \calC(X,A^+) \cotimes_{A^+} M^+ \to \calC(X,M^+). \]
Inverting $p$, using Proposition-Definition \ref{prop-defi: base change of adic modules} (3), we get our result.
\end{proof}

\begin{prop}\label{prop: continuous funtion ring stably uniform}
	Suppose that $A$ is a stably uniform complete Tate ring. Then for every quasicompact topological space $X$, $\calC(X,A)$ is stably uniform.
\end{prop}

\begin{proof}
	We may again assume that $X$ is a profinite set. Let $\calC_0(X,A)$ be the ring of locally constant $A$-valued functions on $X$. Then $\calC_0(X,A)$ is dense in $\mathcal{C}(X,A)$, so every rational localization can be written as
	\[\mathcal{C}(X,A)\bigl\langle \frac{f_1, f_2, \ldots, f_r}{g} \bigr\rangle,\]
	such that all $f_i$ and $g$ belong to $\calC_0(X,A)$ (cf. \cite[Proposition 3.9]{Huber-continuous-valuations}). Take an open cover $X = \bigcup_{j=1}^N U_j$ where the $U_j$ are pairwise disjoint and such that all $f_i$ and $g$ are constant on each $U_j$.	For each $i, j$, let $f_{ij}$ denote the restriction of $f_i$ to $U_j$, and let $g_j$ denote the restriction of $g$ to $U_j$.
    
    By definition, we have
	\[\calC(X,A) = \prod_{j=1}^N \calC(U_j, A).\]
	Under this identification, $f_i = (f_{i1}, f_{i2}, \ldots, f_{iN})$. Therefore,
	\[\mathcal{C}(X,A)\langle \frac{f_1, f_2, \ldots, f_r}{g} \rangle = \prod_{j=1}^N \mathcal{C}(U_j, A)\langle \frac{f_{1j}, f_{2j}, \ldots, f_{rj}}{g_j} \rangle.\]
	For each $j \in \{ 1, \ldots, N\}$, by Lemma \ref{lem: tensor product and continuous functions}, we have
	\[\mathcal{C}(U_j, A)\bigl\langle \frac{f_{1j}, f_{2j}, \dots, f_{rj}}{g_j} \bigr\rangle = \mathcal{C}\bigl(U_j, A\bigl\langle \frac{f_{1j}, f_{2j}, \ldots, f_{rj}}{g_j} \bigr\rangle\bigr). \]
    Since $A\bigl\langle \frac{f_{1j}, f_{2j}, \ldots, f_{rj}}{g_j} \bigr\rangle$ is uniform by assumption, the ring above is also uniform. It follows that $\mathcal{C}(X,A)\langle \frac{f_1, f_2, \ldots, f_r}{g} \rangle$, as their product, is also uniform.
\end{proof}

\begin{cor}\label{cor: continuous function of locally perfectoid is locally perfectoid}
    Let $R$ be a locally perfectoid Tate algebra and $X$ be a quasi-compact space. Then the Tate ring
    \[\calC(X,R)\]
    is locally perfectoid.
\end{cor}

\begin{proof}
    Combine Proposition \ref{prop: continuous funtion ring stably uniform} and Lemma \ref{lem: tensor product and continuous functions}.
\end{proof}

\subsection{Generalized geometric Sen theory}

The geometric Sen theory was introduced by Camargo in \cite{camargo2023geometric}, which is a globalization of Faltings' local Simpson correspondence. For our use, we need to generalize the theory to a larger class of algebras, especially those algebras that are not topologically finitely generated.

Let $\overline{A}/\dZ_p^\cyc$ be an integral perfectoid algebra and put $A=\AAinf(\overline{A})$. Let $R/\overline{A}$ be a $p$-complete and $p$-torsion free algebra.

\begin{defi}\label{defi: toric chart}
    A toric chart of $R$ is a homomorphism of $\overline{A}$-algebras
    \[f:\overline{A}\langle T_1,T_2,\dots,T_d\rangle\to R\]
    satisfying the following conditions:
    \begin{enumerate}[label=(\arabic*)]
        \item $\cdL_{R/\overline{A}}[\frac{1}{p}]=\bigoplus_{j=1}^dR\bigl[\frac{1}{p}\bigr]\dif f(T_j)$;
        \item The ring 
        \[\bigl(\overline{A}\langle T_1^{\frac{1}{p^\infty}},T_2^{\frac{1}{p^\infty}},\dots, T_d^{\frac{1}{p^\infty}} \rangle\cotimes_{\overline{A}\langle T_1,T_2,\dots, T_d\rangle}R\bigr)\bigl[\frac{1}{p}\bigr]\]
        is locally perfectoid (Definition \ref{defi: locally perfectoid huber pairs});
        \item For any $i$, $f(T_i)$ is invertible in $R$.
    \end{enumerate}
    For such a toric chart, define $R_\infty$ to be the integral closure of
    \[ \overline{A}\langle T_1^{\frac{1}{p^\infty}},T_2^{\frac{1}{p^\infty}},\dots, T_d^{\frac{1}{p^\infty}} \rangle\cotimes_{\overline{A}\langle T_1,T_2,\dots, T_d\rangle}R\]
    inside
    \[\bigl(\overline{A}\langle T_1^{\frac{1}{p^\infty}},T_2^{\frac{1}{p^\infty}},\dots, T_d^{\frac{1}{p^\infty}} \rangle\cotimes_{\overline{A}\langle T_1,T_2,\dots, T_d\rangle}R\bigr)\bigl[\frac{1}{p}\bigr].\]
\end{defi}

\begin{rmk}
Whether $\overline A$ is $p$-torsion-free is unimportant in this definition, but one should note that $\overline A$ has bounded $p^\infty$-torsion by \cite[Lemma 2.34]{Bhatt_2022}. In fact, we have $\overline A[p^\infty]=\overline A[p]$.
\end{rmk}

\begin{defi}
    The algebra $R$ is called \emph{senable} if there exists a toric chart.
\end{defi}

\begin{theo}\label{theo: sheafy R}
    Let $R$ be a senable algebra over $\overline{A}$. Let $R^+$ be the integral closure of $R$ in $R\bigl[\frac{1}{p}\bigr]$. Then $R\bigl[\frac{1}{p}\bigr]$ is stably uniform (cf. \cite[The paragraph before Theorem 7]{Buzzard_2016}). In particular, the Huber pair $\bigl(R\bigl[\frac{1}{p}\bigr],R^+\bigr)$ is sheafy.
\end{theo}

\begin{proof}
    As $R$ is senable, we can choose a toric chart
    \[f:\overline{A}\langle T_1,T_2,\dots,T_d\rangle\to R.\]
    Note that
    \[\left(\overline{A}\langle T_1^{\frac{1}{p^\infty}},T_2^{\frac{1}{p^\infty}},\dots, T_d^{\frac{1}{p^\infty}} \rangle\cotimes_{\overline{A}\langle T_1,T_2,\dots, T_d\rangle}R\right)\bigl[\frac{1}{p}\bigr]\]
    is stably uniform by definition. Hence, the claim follows from the proof of \cite[Proposition 6.3.4]{Scholze_2020}.
\end{proof}

In the following, we will put $X=\Spa\bigl(R\bigl[\frac{1}{p}\bigr],R^+\bigr)$.

Recall that the construction of geometric Sen theory of rigid analytic varieties is divided into three steps. First, by passing to local, we can construct an equivalence between $v$-vector bundles and generalized representations. Second, we use the decompletion technique to construct the Sen operator relative to a toric chart. Finally, by an abstract calculation, we can glue our local constructions to the whole $v$-site.

In our generalized theory, we also follow this idea. Let $\Perf_{\overline{A}\bigl[\frac{1}{p}\bigr]}$ be the category of affinoid perfectoid spaces over $\overline{A}$.

Recall that we use $(R/A)_\Prism^\perf$ to denote the full subcategory of perfect prisms in $(R/A)_\Prism$. Let $\calE$ be an object in $\Vect\bigl( (R/A)_\Prism^\perf,\overline\calO\bigl[\frac{1}{p}\bigr] \bigr)$, we can define a $v$-vector bundle on $X$ by sending each $\Spa(S,S^+)$ and $f\in X(S,S^+)$ to $\calE(S^+)$. Here, we view the pair $(S^+,f:R\to S^+)$ as an object in $(R/A)_{\Prism}^\perf$. This defines a functor from $\Vect\bigl( (R/A)_\Prism^\perf,\overline\calO\bigl[\frac{1}{p}\bigr] \bigr)$ to the category of vector bundles on $X_v$, which we denote by $\mathrm{L}$. 

It is proved in \cite{Min_2025} that $\mathrm{L}$ is an equivalence. However, their construction of the inverse of $\mathrm{L}$ warrants further scrutiny. Specifically, we remark on a subtle point in \cite[Construction 5.4]{Min_2025} which appears to implicitly rely on the non-trivial fact that the quotient of a perfectoid ring by its $p$-power torsion remains perfectoid. This property is not automatic. We complete the proof of this property. Later, we learn that this has been already proved in \cite[Lecture \uppercase\expandafter{\romannumeral4}, Proposition 3.2]{bhatt_prismnote}. Our proof is slightly different from \cite{bhatt_prismnote}, and we want to write it down here.

\begin{prop}\label{prop: perfectoid divides out p-torsion}
Let $\overline A$ be integral perfectoid in the sense of \cite[Section 3]{bhatt2018integral}, and $J$ be the kernel of $\overline A \to \overline A \bigl[ \frac 1 p \bigr]$, then $J \subseteq \overline A$ is closed (in the $p$-adic topology), and $\overline A/J$ is also integral perfectoid.

In particular, $\overline A \bigl[ \frac 1 p \bigr] \cong (\overline A/J) \bigl[ \frac 1 p \bigr]$ is perfectoid as a Tate algebra (i.e. Fontaine's perfectoid).
\end{prop}

\begin{proof}
Write $A = \AAinf(\overline A)$ and $I = \ker \theta$, then $(A,I)$ is an orientable perfect prism by \cite[Theorem 3.11]{Bhatt_2022}. By \cite[Lemma 2.34]{Bhatt_2022}, we see that $J = \overline A[p]$ is closed. Write $A=W(S)$ for $S = \overline A^\flat$, and take a generator $d$ of $I$. Write $d = [a_0] + pu$ for some $u \in A^\times$ and $a_0 \in S$, and we take $J' \subseteq A$ to be the kernel of $A \to A \bigl[ \frac 1 {[a_0]} \bigr]$. We claim that $J'$ is a $\delta$-ideal for which $\phi^{-1}(J')=J'$.

First, fix $x \in J'$, so $[a_0]^Nx =0$ for some $N \in \mathbb Z_+$. Since $\delta([a_0])=0$, it follows that $[a_0]^{pN}\delta(x) = \delta([a_0]^Nx)=0$, and we see that $\delta(x) \in J'$. Now assume that $x \in A$ and that $\phi(x) \in J'$. Assume that $[a_0]^{pN} \phi(x) =0$, and we see that $\phi([a_0]^Nx) = [a_0]^{pN} \phi(x)=0$. Since $A$ is perfect, we conclude that $[a_0]^N x=0$, and therefore $x\in J'$.

Now take $J''$ to be the closure of $J'$ in $A$ (with respect to  the $(p,d)$-adic topology), and take $A' = A/J''$, so $A'$ is a perfect $\delta$-ring by the second paragraph, and is classically $(p,d)$-adically complete by the same argument as in Lemma \ref{lem: closed submodule}. Since $d$ is distinguished in $A'$, by \cite[Lemma 2.34]{Bhatt_2022}, we see that $(A',dA')$ is a perfect prism. Take $\overline{A'} = A'/dA'$ to be the integral perfectoid ring associated to $A'$, and we have a natural surjective map $\psi: \overline A \to \overline{A'}$. We claim that $\ker \psi = J$.

On the one hand, for $x \in J'$ we have $[a_0]^Nx=0$ for some $N$, and we see that $p^N \overline x = 0 \in \overline A$ for $\overline x \in \overline A$ the image of $x$ in $\overline A$. It follows that $\overline x \in J$, so $J' \subseteq A$ is contained in the inverse image of $J$. Since $J$ is closed, we see that $J''$ is also contained in the inverse image of $J$, and it follows that $\ker \psi = (J''+dA)/dA \subseteq J$. Now, to show that $J \subseteq \ker \psi$, we only need to show that $\overline{A'}$ is $p$-torsion-free.

Note that $[a_0] \in A'$ is $p$-torsion-free. For $\overline x \in \overline{A'}[p]$, take a lifting $x \in A'$, and write $px = dy$ for some $y \in A'$. It follows that $px = [a_0]y + puy$, so $[a_0]y \in pA'$. Now since $[a_0]$ is a non-zero divisor, we see that $[a_0][b]=0 \Rightarrow [b]=0$. Write down the Teichm\"uller expansion for $y$, and we see that $y \in pA'$, say $y=pz$. It follows that $px=pdz$. Since $A'$ is $p$-torsion-free (by \cite[Lemma 2.28]{Bhatt_2022}), we conclude that $x=dz \in A'$, and $\overline x=0$. Therefore, $\overline{A'}$ is $p$-torsion-free, and we see that $J = \ker \psi$.

Apply \cite[Lemma 3.21]{bhatt2018integral} to $\overline A/J \cong \overline{A'}$, we see that $(\overline A/J) \bigl[ \frac 1 p \bigr]$ is perfectoid in Fontaine's sense. By definition, $\overline A \bigl[ \frac 1 p \bigr] \cong (\overline A/J) \bigl[ \frac 1 p \bigr]$, and we are done.
\end{proof}

\begin{theo}\label{theo: v vector bundle=perfect crystals}
    The functor $\mathrm{L}$ is an equivalence.
\end{theo}

\begin{proof}
    In fact, this has already been proved in the proof of \cite[Theorem 5.1]{Min_2025}. We only construct the inverse of $\mathrm{L}$. Let $\calV$ be a vector bundle on $X_v$. For any $R\to T$ in $(R/A)_{\Prism}^\perf$, define $T^+$ as the integral closure of $T$ in $T\bigl[\frac{1}{p}\bigr]$. Now,
    $\bigl(T\bigl[\frac{1}{p}\bigr],T^+\bigr)$ forms a perfectoid Tate--Huber pair by Proposition \ref{prop: perfectoid divides out p-torsion}. We define a crystal on $(R/A)_\Prism^\perf$ as sending each $R\to T$ to $\calV\bigl(T\bigl[\frac{1}{p}\bigr],T^+\bigr)$. This defines the quasi-inverse of $\mathrm{L}$.
\end{proof}

Next, we will construct an equivalence between the category of generalized representations and the category of $v$-vector bundles. Indeed, this has already been proved several times (see, for example, \cite[Section 5]{Min_2025}). Here, we provide a slightly general form. We assume that all Tate--Huber pairs are over $\dQ_p$.

\begin{lemma}\label{lemma: preadic space morphisms}
    Let $(R,R^+)$ be a complete Tate--Huber pair and $X=\Spa^{\mathrm{ind}}(R,R^+)$. Then for any sheafy Tate--Huber pair $(S,S^+)$, the canonical map
    \[\Mor_{\CAfd}\left((R,R^+),(S,S^+)\right)\to \Mor\left(\Spa(S,S^+),X\right)\]
    is a bijection.
\end{lemma}

\begin{proof}
    By \cite[Proposition 3.4.3]{Scholze_2020}, any morphism from $\Spa(S,S^+)$ to $\Spa^{\mathrm{ind}}(R,R^+)$ is locally on $\Spa(S,S^+)$ induced by a map of Tate--Hubers from $(R,R^+)$ to $(S,S^+)$. Hence, the claim follows from the sheafiness of $(S,S^+)$.
\end{proof}

\begin{lemma}\label{lemma: diamondian adic space and coproduct}
    Let $R$ be a $p$-complete and $p$-torsion free algebra over $\overline{A}$. Let $\Perf_{R}$ be the category of triples $(S,S^+,f:R\to S^+)$, where $(S,S^+)$ be a perfectoid Tate--Huber pairs and $f$ be a homomorphism of $\overline{A}$ algebras.

    \begin{enumerate}
        \item The functor sending $(S,S^+)$ to $\Hom_{\overline{A}}(R,S^+)$ is represented by a diamond.
        \item The category $\Perf_{R}$ has arbitrary finite coproducts.
    \end{enumerate}
\end{lemma}

\begin{proof}
    Let $R^+$ be the integral closure of $R$ in $R[\frac{1}{p}]$, and $(S,S^+)$ be a Tate--Huber pair. Since $S^+$ is integrally closed in $S$, 
    \[\Hom(R,S^+)=\Hom(R^+,S^+).\]
    Hence, by Lemma \ref{lemma: preadic space morphisms}, the functor is equal to the associated diamond of $\Spa^{\mathrm{ind}}(R,R^+)$. See \cite[Section 10.1]{Scholze_2020} for details.

    Then we construct the coproduct of two objects $\Spa(S_1,S_1^+)$ and $\Spa(S_2,S_2^+)$ in $\Perf_R$. By \cite[Proposition 6.18]{scholze2012perfectoid}, the category of affinoid perfectoid spaces has arbitrary finite coproducts. Let $(T,T^+)/\overline{A}$ be the coproduct of $\Spa(S_1,S_1^+)$ and $\Spa(S_2,S_2^+)$ over $\overline A$. Let $I$ be the ideal generated by all elements of the form $r\otimes1-1\otimes r$ where $r\in R$. By the construction in \cite[Section \uppercase\expandafter{\romannumeral2}.1]{Scholze_2015torsion}, there exists a quotient $T/I$ in the category of perfectoid algebras, which is equal to the required coproduct.
\end{proof}

\begin{rmk}
    Keep the notation in Lemma \ref{lemma: diamondian adic space and coproduct} and its proof. Let $(S_1,S_1^+)$ and $(S_2,S_2^+)$ be two objects in $\Perf_{R}$. By Yoneda lemma, the adic spectrum of the coproduct of these two objects is equal to the diamondian fiber product
    \[\Spa(S_1,S_2^+)^\diamond\times_{\Spd(R,R^+)}\Spa(S_2,S_2^+)^\diamond.\]
\end{rmk}

\begin{defi}\label{defi:generalized rep}
    Let $G$ be a topologically group that acts on a topological ring $R$. A \emph{generalized representation} with respect to $(G,R)$ is a finite-rank projective module $M/R$ which is endowed with a semi-linear $G$-action. Use
    \[\Rep_G(R)\]
    to denote the category of all generalized representations.
\end{defi}

The relationship between $v$-vector bundles and generalized representations is given by the following theorem.

\begin{theo}\label{theo: v vector bundles=generalized representations}
    Let $(R,R^+)$ be a Tate--Huber pair and $\{(R_i,R_i^+):i\in I\}$ be an inductive system of finite \'etale algebras over $(R,R^+)$. Assume the following holds:
    \begin{enumerate}[label=(\arabic*)]
        \item $\colim_{i\in I} R_i$ has a locally perfectoid tilde-limit (cf. \cite[Definition 2.4.2]{Huber_1996}) $\widetilde R$. Let $\widetilde R^+$ be the closure of $\colim_{i\in I}R_i^+$ in $\widetilde R$.

        \item There exists an inverse system of finite groups $\{G_i:i\in I^{\op}\}$ such that each $(R_i,R^+_i)$ is a $G_i$-torsor over $(R,R^+)$, compatible with change in $i$. Let $G=\lim_{i\in I^{\op}}G_i$, where the limit is taken in the category of profinite group.
    \end{enumerate}
    Then, as diamonds, $\Spd(R,R^+)=\Spa(\widetilde
     R,\widetilde R^+)^\diamond/G$. Furthermore, the evaluation on $\Spa(\widetilde R,\widetilde R^+)$ defines an equivalence between the category of $v$-vector bundles on $\Spd(R,R^+)$ and $\Rep_G(\widetilde R)$.
\end{theo}

\begin{proof}
    The first claim is exactly \cite[Lemma 10.1.7]{Scholze_2020}. 
We only need to prove the second claim.
    
For a $(V,\rho) \in \Rep_{G}(\widetilde{R})$, define a $v$-vector bundle on $\Spa(R,R^+)$ as follows. For all affinoid perfectoid spaces $\Spa(S,S^+)$, a morphism $\Spa(S,S^+)^\diamond \to \Spd(R,R^+)$ and $i \in I$, define $(S_i,S_i^+)$ as the base change of $(S,S^+)$ along $(R,R^+) \to (R_i,R_i^+)$. Note that $(S_i,S_i^+)$ is a finite \'etale Tate--Huber pair over $(S,S^+)$, hence also perfectoid. Let $(\widetilde{S},\widetilde{S}^+)$ be the colimit of $(\widetilde{S}_i,\widetilde{S}_i^+)$, then there is a canonical map $(\widetilde{R},\widetilde{R}^+) \to (\widetilde{S},\widetilde{S}^+)$. Set $V(\Spa(\widetilde{S},\widetilde{S}^+)) = V \otimes_{\widetilde{R}} \widetilde{S}$ and construct a semilinear continuous $G$-action via base change. By the pro-finite \'etale-descent of perfectoid spaces (cf. \cite[Corollary 3.2.2]{Berger_2008}), we obtain a finite-rank projective module on $S$, denoted by $V(\Spa(S,S^+))$. By construction, the assignment $\Spa(S,S^+) \mapsto V(\Spa(S,S^+))$ is a $v$-vector bundle on $\Spa(R,R^+)$.

On the other hand, for a $v$-vector bundle ${\mathcal{E}}$ on $\Spd(R,R^+)$, consider the global sections of the pullback of ${\mathcal{E}}$ to $\Spa(\widetilde{R},\widetilde{R}^+)$, denoted by $V_{{\mathcal{E}}}$. By Lemma \ref{lemma: v descent of locally perfectoid}, this is a finite-rank projective $\widetilde{R}$-module. For any $g \in G$, the action of $g$ on $\widetilde{R}$ induces a semilinear automorphism of $V_{{\mathcal{E}}}$, which gives a semilinear action of $G$ on $V_{{\mathcal{E}}}$. Next, we show that this action is continuous. Indeed, consider the cosimplicial Tate--Huber pair 
\[
\{(\widetilde{R}^n,\widetilde{R}^{n,+}) : n \geq 0\} := \{(\calC(G^{n},R), \calC(G^n,R^+)) : n \geq 0\}.
\]
By Corollary \ref{cor: continuous function of locally perfectoid is locally perfectoid}, each term of this cosimplicial ring is a locally perfectoid ring. Therefore, the global sections of the pullback of ${\mathcal{E}}$ to $\Spa(\widetilde{R}^n,\widetilde{R}^{n,+})$ yield a finite-rank projective $\widetilde{R}^n$-module. This gives a stratification of the cosimplicial ring $\{\widetilde{R}^n : n \geq 0\}$. Using the proof in \cite[Example 5.7]{Min_2025}, we know that this stratification induces a continuous semilinear $G$-representation on $\widetilde{R}$, and this representation coincides with $V_{{\mathcal{E}}}$.

Finally, by construction, the above argument gives an equivalence between the category of $v$-vector bundles on $\Spd(R,R^+)$ and $\mathbf{Rep}_{G}(\widetilde{R})$.
\end{proof}

Next, we will use the famous decompletion theory to construct our Sen operator for a toric chart. Let \[f:\overline{A}\langle T_1,T_2,\dots,T_d\rangle\to R\] be a toric chart. For each $n\geq 1$, choose a primitive $p^{n}$-th root of unity $\zeta_{p^n}$ such that $\zeta_{p^n}^p=\zeta_{p^{n-1}}$. This defines a basis $t$ of the Tate twist. 
    Let $\Gamma=\dZ_p^d$ and denote by $\gamma_i$ the $i$-th canonical base $(0,0,\dots,0,1,0,\dots,0)\in \Gamma$. Define the action of $\Gamma$ on $R_{\infty}[\frac{1}{p}]$ (see Definition \ref{defi: toric chart}) by
    \[\gamma_i(T_j^{\frac{1}{p^n}})=\delta_{ij}\zeta_{p^n}T_j^{\frac{1}{p^n}}\]
    for any $i,j$.

    By definition, the set $\{ T_1^{\alpha_1}T_2^{\alpha_2}\dots T_d^{\alpha_d}: \alpha_i\in \frac{1}{p^\infty}\dZ_{\geq 0} \}$
    is an orthogonal basis of \[\overline{A}\langle T_1^{\frac{1}{p^\infty}},T_2^{\frac{1}{p^\infty}},\dots, T_d^{\frac{1}{p^\infty}} \rangle[\frac{1}{p}]\] as a topological $\overline{A}[\frac{1}{p}]$-module. For any $n\geq 0$ and $1\leq i\leq d$, define $R_{n}^i$ as the  $\overline{A}[\frac{1}{p}]$-linear continuous endomorphism of $\overline{A}\langle T_1^{\frac{1}{p^\infty}},T_2^{\frac{1}{p^\infty}},\dots, T_d^{\frac{1}{p^\infty}} \rangle[\frac{1}{p}]$ as sending
    $T_1^{\alpha_1}T_2^{\alpha_2}\dots T_d^{\alpha_d}$ to itself if $p^n\alpha_i\in \dZ$ and $0$ otherwise. It is easy to check that $R_n^i$ is $\overline{A}\langle T_1,T_2,\dots,T_d\rangle$-linear and thus, by the functorality of complete tensor products, can be extended to a $R$-linear continuous endomorphism of $R_{\infty}[\frac{1}{p}]$. We still denote the endomorphism by $R_{n}^i$.

\begin{prop}\label{prop: strong decom sen of charts}
    The datum $(R_{\infty}[\frac{1}{p}],\Gamma, R_{n}^i)$ forms a Camargo's $d$-dimensional Sen theory (cf. \cite[Definition 2.2.5]{camargo2023geometric}).
\end{prop}

\begin{proof}
    This comes from a direct checking of the definitions. See \cite[Example 2.2.7]{camargo2023geometric} and \cite[Proposition 2.2.14]{camargo2023geometric} for details.
\end{proof}

For any integer $n\geq 0$, define
\[R_N=\overline{A}\langle T_1^{\frac{1}{p^N}},T_{2}^{\frac{1}{p^N}},\dots,T_d^{\frac{1}{p^N}} \rangle\otimes_{\overline{A}\langle T_1,T_2\dots,T_d\rangle}R.\]
Completion is not necessary since $R_N$ is a finite free module of $R$. We have the following decompletion result as a corollary of Proposition \ref{prop: strong decom sen of charts}.

\begin{theo}\label{theo: decompletion}
    Keep the notation, for any $M\in \mathbf{Rep}_{\Gamma}(R_\infty)[\tfrac{1}{p}]$, there exists $N\geq 1$ and a $\Gamma$-invariant finite-rank projective sub-$R_{N}$-module $M_N$ such that the canonical homomorphism
    \[\iota:R_{\infty}\otimes_{R_N}M_N\to M\]
    is a $\Gamma$-equivariant isomorphism. Furthermore, $M_N$ is unique and is equal to the submodule of $p^N\Gamma$-analytic vectors of $M$.
\end{theo}

\begin{proof}
    If $M$ is free as a $R_\infty[\frac{1}{p}]$-module, the existence and uniqueness of $M_N$ is proved by \cite[Proposition 3.3.1]{Berger_2008} (for $1$-dimensional case) and \cite[Theorem 2.4.4 (1)]{camargo2023geometric} (for higher dimensional case). For general $M$, as $(R[\frac{1}{p}],R^+)$ is sheafy (cf. Theorem \ref{theo: sheafy R}), we can change $R[\frac{1}{p}]$ by a rational covering so that $M$ is free and construct $M_N$ by a gluing argument.

    For the rest, see \cite[Theorem 2.4.4 (2)]{camargo2023geometric}. Note that the argument in \cite[Theorem 3.2]{Laurent_Berger_2016} also works in our case and can show that $R_N[\tfrac{1}{p}]=R_\infty[\tfrac{1}{p}]^{p^N\Gamma-\mathrm{an}}$.
\end{proof}

As a corollary, for any $M\in\mathbf{Rep}_\Gamma(R_\infty[\tfrac{1}{p}])$, the canonical map
    \begin{equation}\label{eq: decompletion}
        M^{\la}\otimes_{R_{\infty}[\tfrac{1}{p}]^\la}R_\infty[\frac{1}{p}]\to M
    \end{equation}
is an isomorphism. Recall that $\{\gamma_i:i=1,2,\dots,d\}$ is the set of canonical basis of $\Gamma\cong\dZ_p$. The Lie algebra $\Lie(\Gamma)$ of $\Gamma$ is canonically isomorphic to $\dQ_p^d$ and we denote its canonical basis by $\{\log\gamma_i:i=1,2,\dots,d\}$. The derivation of $\Gamma$ acting on $M^\la$ provides a $R_\infty[\frac{1}{p}]^\la$-linear action of $\Lie(\Gamma)$ on $M^\la$. Let $\theta_i$ be the action of $\log\gamma_i$. We extend each $\theta_i$ linearly to $M$.

\begin{lemma}\label{lemma: properties of sen actions}
    Keep the notation and let $N$ be the integer in Theorem \ref{theo: decompletion}. The following holds.
    \begin{enumerate}
        \item For any $i$ and $n\geq N$, $\theta_i$ sends $M^{p^n\Gamma-\mathrm{an}}$ to itself.

        \item There exists $n\geq N$ such that for any $a\in p^n\dZ_p$ and $1\leq i\leq d$, the action of $a\gamma_i$ on $M_N$ is equal to the linear operator $\exp(a\theta_i)$.

        \item As actions of $M$, $\theta_i$ commutes with $\gamma_j$ for any $1\leq i,j\leq d$.
    \end{enumerate}
\end{lemma}

\begin{proof}
    By definition, the action of $p^N\Gamma$ on $M_N$ is $R_N$-linear and analytic. The first statement holds by \cite[Lemma 2.6]{Laurent_Berger_2016}. For the second, we fix a finitely generated open bounded sub-$R_N^\circ$-module $M_{N}^\circ$ of $M$ and enlarge $n$ so that $p^n\Gamma$ stabilizes $M_N^\circ$. The existence of $n$ comes from the fact that $p^N\Gamma$ acts on $M_N$ linearly. Keep enlarge $n$ so that the induced action of $p^n\Gamma$ on $M_N^\circ/p$ is trivial. Then the claim then follows from the fact that the exponential of a $p$-adic Lie group is an isomorphism on an open subgroup.

    For the third part, we only need to check on $M_n$ for some $n\geq N$ by Theorem \ref{theo: decompletion}. Hence, we can choose $n$ in the second claim. Then we have
    \[\gamma_j\exp(a\theta_i)=\exp(a\theta_i)\gamma_j\]
    for any $a\in p^n\dZ_p$. Take the derivation at $a=0$ and the claim follows.
\end{proof}

For our application, we also need to consider when the integer $N$ in Theorem \ref{theo: decompletion} can be chosen to be $0$.

\begin{theo}\label{theo: full decompletion}
    Assume more that for any $i$, $(1-\zeta_p)^{-1}\theta_i$ is topologically nilpotent, then we can choose the $N$ in Theorem \ref{theo: decompletion} to be $0$.
\end{theo}

\begin{proof}
    Choose $N\geq 0$ as in Theorem \ref{theo: decompletion}.

    Define the action of $\Gamma=\dZ_p^d$ on $M_N$ by the formula
    \[(a_1,a_2,\dots,a_d)*x=\exp(a_1\theta_1+a_2\theta_2+\dots+a_d\theta_d)x\]
    for any $a_i\in\dZ_p$ and $x\in M$. We note that this is well defined since $\theta_1$ commutes with each other.

    By Lemma \ref{lemma: properties of sen actions}, the for any $\gamma=(a_1,a_2,\dots,a_d)\in\Gamma$, the formula
    \[\gamma_{\text{Gal}}(x):=\gamma(\gamma^{-1}*x)=\gamma(\exp(-\sum_{j=1}^da_j\theta_j)x)\]
    defines a semi-linear locally constant action of $\Gamma$ on $M_N$. By Galois descent, if we denote
    \[M_0:=\{m\in M_N:\gamma_{\text{Gal}}(m)=m,\ \forall \gamma\in \Gamma\},\]
    then $M_0$ is finite-rank projective and $M_0\otimes_{R}R_N=M_N$. The action of $\Gamma$ on $M_0$ is $R$-linear by the definition of $M_0$. Thus, $M_0$ is the wanted sub-$R$-module of $M$.
\end{proof}

Finally, we provide a `chart-free' version of the decompletion, which is known as the geometric Sen theory (cf. \cite{camargo2023geometric}). 

\begin{prop-defi}\label{prop-defi: sen theory}
    Let $R/\overline A$ be a senable algebra, for any $\calE\in \Vect\bigl( (R/A)_\Prism^\perf,\overline\calO\bigl[\frac{1}{p}\bigr] \bigr)$, there exists a canonical $\overline{\calO}[\frac{1}{p}]$ linear homomorphism 
    \[\theta_{\calE}:\calE\to \calE\otimes_{R[\frac 1 p]}\cdL_{R/\overline{A}}[\tfrac{1}{p}]\{-1\},\]
    called the \emph{Sen operator} of $\calE$, satisfying the following assumptions.
    \begin{enumerate}
        \item The operator $\theta_\calE$ is a Higgs bundle, that is 
        \[\theta_{\calE}\wedge \theta_{\calE}=0.\]
        \item If $\calE$ is a base change of finite-rank projective module over $R[\frac{1}{p}]$,
        \[\theta_{\calE}=0.\]
    \end{enumerate}
\end{prop-defi}

\begin{proof}
    We first construct $\theta_\calE$ for a toric chart and then prove the independency of the choice of the charts.

    Let \[f:\overline{A}\langle T_1,T_2,\dots,T_d\rangle\to R\] be a toric chart. For each $n\geq 1$, choose a primitive $p^{n}$-th root of unity $\zeta_{p^n}$ such that $\zeta_{p^n}^p=\zeta_{p^{n-1}}$. This defines a basis $t$ of the Tate twist. Define $\Gamma$ and its action as above. By (\ref{eq: decompletion}),
    for any $(M,\rho)\in \mathbf{Rep}_\Gamma(R_\infty[\frac{1}{p}])$, the canonical map
    \begin{equation*}
        M^{\la}\otimes_{\bigcup_{n\geq 1}\overline{A}\langle T_1^{\frac{1}{p^n}},T_2^{\frac{1}{p^n}},\dots, T_d^{\frac{1}{p^n}} \rangle[\frac{1}{p}]}\overline{A}\langle T_1^{\frac{1}{p^\infty}},T_2^{\frac{1}{p^\infty}},\dots, T_d^{\frac{1}{p^\infty}} \rangle[\frac{1}{p}]\to M
    \end{equation*}
    is an isomorphism.

    As each $f(T_i)$ is invertible in $R[\frac{1}{p}]$, by Theorem \ref{theo: v vector bundle=perfect crystals} and Theorem \ref{theo: v vector bundles=generalized representations}, there exists a canonical equivalence between $\Vect\bigl( (R/A)_\Prism^\perf,\overline\calO\bigl[\frac{1}{p}\bigr] \bigr)$ and $\mathbf{Rep}_\Gamma(R_\infty[\frac{1}{p}])$. Let $\calE\in \Vect\bigl( (R/A)_\Prism^\perf,\overline\calO\bigl[\frac{1}{p}\bigr] \bigr)$, define $(M,\rho)$ as the corresponding generalized representation. Then there exists a $\dQ_p$-linear action of $\Lie(\Gamma)$ on $M^{\la}$ by taking derivations. Using the inclusion
    \[\log(\gamma_i)\mapsto tT_i\frac{\partial}{\partial T_i}\] and the equality (\ref{eq: decompletion}),
    we can extend the $\Lie(\Gamma)$-action to a linear $R_\infty\otimes_{R}\cdL_{R/\overline{A}}[\frac{1}{p}]^\vee$-action on $M$. This action is $\Gamma$-equivariant, again by the equivalence between $\Vect\bigl( (R/A)_\Prism^\perf,\overline\calO\bigl[\frac{1}{p}\bigr] \bigr)$ and $\mathbf{Rep}_\Gamma(R_\infty[\frac{1}{p}])$, which defines a linear homomorphism
    \[\theta_f:\calE\to \calE\otimes_{R[\frac 1 p]}\cdL_{R/\overline{A}}[\frac{1}{p}].\] 

    We to check the independency of the choice of $f$. This can be proved by the same argument in the proof of Theorem \cite[Theorem 4.8]{heuer2024padicsimpsoncorrespondencesmooth}. Also, we need to calculate the Sen operator of those $\calE$ which comes from a finite-rank projective module over $R[\frac{1}{p}]$. If this holds, the corresponding generalized representation $(M,\rho)$ is indeed trivial (i.e. $M=M^{\Gamma=id}\otimes_{R}R_{\infty}$), and the claim follows from the definition of $\theta$.
\end{proof}

\subsection{Local coordinates}
In this section, fix a bounded prism $(A,I)$, a ring $R/\overline A$ complete of bounded $p^\infty$-torsion, and a $\delta$-ring $P/A$ satisfying Condition \ref{cond: flat + tor 0} equipped with an $\overline A$-algebra map $\overline P \to R$. In the application, we are usually going to take $P = A \left\langle T_1,\ldots,T_d \right\rangle$ with $\delta(T_i)=0$.

\begin{prop-defi}\label{prop-defi: sheaf of local coordinates}
Consider the presheaf $\mathscr C_{R/P/A}$ on $(R/A)_\Prism$, such that $\mathscr C_{R/P/A}(B)$ consists of $\delta$-ring maps $P \to B$ lifting the composition $\overline P \to R \to \overline B$. Then,
\begin{enumerate}[label=(\arabic*)]
\item $\mathscr C_{R/P/A}$ is a sheaf;
\item we have a natural isomorphism of sheaves
\[ \prod_{i=1}^n \mathscr C_{R/P_i/A} \cong \mathscr C_{R/P/A}, \]
where $P = P_1 \cotimes_A \cdots \cotimes_A P_n$;
\item the map $\mathscr C_{R/P/A} \to *$ is represented by a $(p,I)$-completely faithfully flat affine cover.
\end{enumerate}
We call $\mathscr C_{R/P/A}$ the \emph{sheaf of local coordinates} with respect to $P$.
\end{prop-defi}

\begin{rmk}
When $P = A \left\langle T_1,\ldots,T_d \right\rangle$ with $\delta(T_i)=0$, the map $\overline P \to R$ is a `local coordinate' on $R$, and any element in $\mathscr C_{R/P/A}(B)$ represents a `local coordinate' on $B$ which is compatible with that on $R$. This is how $\mathscr C_{R/P/A}$ gets its name.
\end{rmk}

\begin{proof}
The item (1) is merely a straightforward discussion of ($p$-completely) faithfully flat descent, see Theorem \ref{thm: descent theory} (1). The item (2) is a direct consequence using the universal property of $P$ as a tensor product.

As for (3), for any $B$, we may calculate that
\[ h_B \times \mathscr C_{R/P/A} \cong h_{B'}, \]
where $B' = \Prism_{\overline B / B \cotimes_A P}$. By Hodge--Tate comparison (Proposition \ref{prop: HT comparison}), we see that $\overline B \to \overline B'$ is $p$-completely faithfully flat, so $B \to B'$ is $(p,I)$-completely faithfully flat, as desired.
\end{proof}

\begin{prop}\label{prop: natural algebra structure}
Let $R/P/A$ be as above. Define $P(n) = P^{\cotimes(n+1)/A}$ and $E(n) = \Prism_{\overline P/P(n)}.$ Then, there is a natural ring map
\[ R \cotimes_{\overline P} \overline{E(n)} \to \Gamma\bigl(\mathscr C_{R/P(n)/A}, \overline\calO_{(R/A)_\Prism} \bigr|_{\mathscr C_{R/P(n)/A}}\bigr) \]
compatible with the (co)simplicial structure.
\end{prop}

\begin{proof}
Any $h_B \to \mathscr C_{R/P(n)/A}$ in $\Shv((R/A)_\Prism)$ corresponds to a map $P(n) \to B$ lifting the composition $\overline{P(n)} \to \overline P \to R \to \overline B$. By the universal property as in Corollary \ref{cor: universal property tor -1}, this map uniquely extends to a map $E(n) \to B$ such that its reduction
\[ \overline P \to \overline{E(n)} \to \overline B \]
agrees with $\overline P \to R \to B$. Therefore, we have a natural map
\[ R \cotimes_{\overline P} \overline{E(n)} \to \overline B \]
compatible with all morphisms $h_B \to h_C$, hence we get a natural ring map
\[ R \cotimes_{\overline P} \overline{E(n)} \to \Gamma\bigl(\mathscr C_{R/P(n)/A}, \overline\calO_{(R/A)_\Prism} \bigr|_{\mathscr C_{R/P(n)/A}}\bigr). \]
It is easy to verify that the ring map commutes with the (co)simplicial structure.
\end{proof}

Combining Proposition-Definition \ref{prop-defi: sheaf of local coordinates} and Proposition \ref{prop: natural algebra structure}, since $\Vect(\mathscr F,\overline\calO)$ and $\Vect\bigl(\mathscr F,\overline\calO \bigl[ \frac 1 p \bigr]\bigr)$ is a stack on the topos $\Shv((R/A)_\Prism)$, we immediately get the following proposition:

\begin{prop}\label{prop: Higgs to crystal}
For any $R/P/A$ as above. Recall that $\bigl( R \cotimes_{\overline P} \overline{E(n)}, R \cotimes_{\overline P} \cOmega_{\overline P/\overline A} \{-1\} \bigr)$ has a split PD structure given by (the base-change of) Lemma \ref{lemma: E(n) split PD}. Then. there exists a functor of categories
\[
\Higgs^\tn \bigl(R, R \cotimes_{\overline P} \cOmega_{\overline P/\overline A} \{-1\} \bigr) \to \Vect((R/A)_\Prism, \overline\calO), 
\]
which fits into the commutative diagram as below:
\[\xymatrix{
\Higgs^\tn \bigl(R, R \cotimes_{\overline P} \cOmega_{\overline P/\overline A} \{-1\} \bigr) \ar[r] \ar[d]_\cong & \Vect \bigl((R/A)_\Prism, \overline\calO_{(R/A)_\Prism} \bigr) \ar[d]^\cong \\ 
\Strat \bigl( R \cotimes_{\overline P} \overline{E(\bullet)} \bigr) \ar[r] & \displaystyle \lim_\simp \Vect\bigl(\mathscr C_{R/P(\bullet)/A}, \overline\calO_{(R/A)_\Prism} \bigr|_{\mathscr C_{R/P(\bullet)/A}}\bigr)
}\]
Similarly for the invert-$p$ version.
\end{prop}

We now establish some ways to compute the functor as above.
\begin{prop}\label{prop: Higgs to crystal computation}
Let $R/P/A$ be as in Proposition \ref{prop: Higgs to crystal} and $(M,\nabla)$ be a Higgs field over $\bigl(R, R \cotimes_{\overline P} \cOmega_{\overline P/\overline A} \{-1\} \bigr)$ or $\bigl(R \bigl[ \frac 1 p \bigr], R \cotimes_{\overline P} \cOmega_{\overline P/\overline A} \{-1\} \bigl[ \frac 1 p \bigr] \bigr)$. Assume that it is mapped to the crystal $\calE$ on $(R/A)_\Prism$. Then, for any $B \in (R/A)_\Prism$,
\begin{enumerate}[label=(\arabic*)]
\item for any $f: P \to B$ in $\mathscr C_{R/P/A}(B)$, there exists an isomorphism $\eta_{B,f}:\overline B \cotimes_R M \to \calE(B)$ functorial in $(B,f) \in ((R/A)_\Prism)_{/\mathscr C_{R/P/A}}$,
\item for any pair of $f,g$ in $\mathscr C_{R/P/A}(B)$, they induce a map $h: E(1) \to B$ by $(f,g): P(1) \to B$. Then, the following diagram is commutative:
\[\xymatrix{
& \calE(B) & \\
\overline B \cotimes_R M \ar[ru]^{\eta_{B,f}}_\cong \ar@{=}[d] & & \overline B \cotimes_R M \ar[lu]_{\eta_{B,g}}^\cong \ar@{=}[d] \\
\overline B \cotimes_{h,R \cotimes_{\overline P} \overline{E(1)}} (\overline{E(1)} \cotimes_{\delta_1^1,\overline P} M) & & \overline B \cotimes_{h,R \cotimes_{\overline P} \overline{E(1)}} (\overline{E(1)} \cotimes_{\delta_0^1,\overline P}M) \ar[ll]^{\id \otimes \epsilon}_\cong
}\]
where $\epsilon: \overline{E(1)} \cotimes_{\delta_0^1,\overline P}M \to \overline{E(1)} \cotimes_{\delta_1^1,\overline P} M$ is the stratification corresponding to $\nabla$.
\end{enumerate}
\end{prop}

\begin{proof}
These are merely abstract nonsense. As for (1), this is the definition. As for (2), consider the stratification $M(\bullet)/\overline{E(\bullet)}$, and we may consider the natural map
\[ \overline B \cotimes_{\overline{E(1)}} M(1) \to \calE(B) \]
given by $(B,h)$. We have the commutative diagram by definition
\[\xymatrix{
& M(1) & \\
\overline{E(1)} \cotimes_{\delta_1^1, \overline P} M \ar[ru]^\cong & & \overline{E(1)} \cotimes_{\delta_0^1, \overline P} M \ar[lu]_\cong \ar[ll]^\epsilon_\cong
}\]
and taking the tensor product $\overline B \cotimes_{h,R \cotimes_{\overline P} \overline{E(1)}} -$ completes the proof.
\end{proof}

Also, we would like to compute the functor by means of the de Rham realisation as in Theorem \ref{theo: main dR realization revisit}.

\begin{prop}\label{prop: Higgs to MIC}
Let $R/P/A$ be as in Proposition \ref{prop: Higgs to crystal}, and let $B/A$ be a bounded prism satisfying Condition \ref{cond: flat + tor 0}, $P \to B$ be a map of prisms, $\overline B \to R$ be a ring map satisfying Condition \ref{cond: tor -1}, such that the composition $\overline P \to \overline B \to R$ agrees with the map $\overline P \to R$ in $R/P/A$. Denote $D = \Prism_{R/B}$ and $\dif: D \to D \cotimes_{\overline B} \cOmega_{\overline B/\overline A} \{-1\}$ the differential as in Theorem \ref{theo: main dR realization revisit}. Then, the composition
\[ \Higgs^\tn \bigl(R, R \cotimes_{\overline P} \cOmega_{\overline P/\overline A} \{-1\} \bigr) \to \Vect\bigl( (R/A)_\Prism, \overline\calO \bigr) \to \MIC^\tn \bigl( \overline D, \dif \bigr) \]
agrees with the the natural map
\[ \Higgs^\tn \bigl(R, R \cotimes_{\overline P} \cOmega_{\overline P/\overline A} \{-1\} \bigr) \to \MIC^\tn \bigl( \overline D, \dif \bigr) \]
obtained by taking tensor product
\[ (M,\nabla) \mapsto \bigl(\overline D \otimes_R M, \dif \otimes \id + \id \otimes (f \circ\nabla) \bigr) \]
where $f: R \cotimes_{\overline P} \cOmega_{\overline P/\overline A} \{-1\} \to R \cotimes_{\overline B} \cOmega_{\overline B/\overline A} \{-1\}$ is the natural map obtained from $P \to B$. Similarly for the case with $p$ inverted.
\end{prop}

\begin{proof}
Denote $B(n) = B^{\cotimes (n+1){/A}}$ and $D(n) = \Prism_{R/B(n)}$, and we have a natural map $h_{D(n)} \to \mathscr C_{R/P(n)/A}$ given by
\[ P(n) \to B(n) \to D(n). \]
Using Theorem \ref{theo: comp strat connection} and the construction in Theorem \ref{theo: main dR realization revisit} (namely, the stratification of a vector bundle is given by its restriction on $D(\bullet)$), the composition is identified with the limit of the composition
\[ \Vect\bigl(R \cotimes_{\overline P} \overline{E(n)} \bigr) \to \Vect\bigl(\mathscr C_{R/P(n)/A}, \overline\calO \bigr) \to \Vect\bigl(\overline{D(n)} \bigr) \]
over $\simp$, and the natural map below it is identified with the limit of the maps
\[ \Vect\bigl(R \cotimes_{\overline P} \overline{E(n)} \bigr) \to \Vect\bigl(\overline{D(n)} \bigr). \]
Therefore, we only need to verify that the ring maps
\[ u_n: R \cotimes_{\overline P} \overline{E(n)} \to \overline{D(n)}\]
as in Proposition \ref{prop: natural algebra structure} agrees with the ring maps
\[ v_n: R \cotimes_{\overline P} \overline{E(n)} \cong \cGamma_R^\bullet\bigl( R \cotimes_{\overline P} \cOmega_{\overline P/\overline A}^{\oplus n} \{-1\} \bigr) \to \overline{D(n)} \]
given by the split PD structure. Since $u_n,v_n$ are $R$-algebra maps, we only need to show that they agree on $\overline{E(n)}$. Consider the map
\[ g_n: \Prism_{\overline P/P(n)} \to \Prism_{R/B(n)} \]
given by the naturality of $\overline P/P/A \to R/B/A$. By Theorem \ref{theo: main dR realization revisit}, this map is induces a map of split PD rings, compatible with the map on differentials. Therefore, since $\overline{E(n)}$ is topologically generated (as a PD algebra over $\overline P$) purely out of $\alpha_i^n(\varphi(\omega))$ for $\omega \in \cOmega_{\overline P/\overline A}$, any map $\overline{E(n)} \to \overline{D(n)}$ is completely determined by its value on $\varphi(\Omega)$, and we see that the restriction
\[ v'_n: \overline{E(n)} \to \overline{D(n)} \]
agrees with the reduction $\overline g_n$ of $g_n$.

We still have to show that the restriction of $u_n$
\[ u'_n: \overline{E(n)} \to \overline{D(n)} \]
agrees with $\overline g_n$. By definition and the universal property (Corollary \ref{cor: universal property tor -1}), $u'_n$ is the reduction of the unique $\delta$-ring map
\[ h_n: E(n) \to D(n) \]
such that $P(n) \to E(n) \xrightarrow{h_n} D(n)$ agrees with $P(n) \to B(n) \to D(n)$. Observe that $g_n$ satisfies the condition, because we have the commutative diagram
\[\xymatrix{
P(n) \ar[r] \ar[d] & B(n) \ar[d] \\
\Prism_{\overline P/P(n)} \ar[r]_{g_n} & \Prism_{R/B(n)}
}\]
Therefore, the map $h_n$ agrees with $g_n$, hence $u'_n$ agrees with $\overline g_n$. We see that $u'_n = \overline g_n = v'_n$, so $u_n=v_n$, as desired.
\end{proof}

We now calculate the map in Proposition \ref{prop: Higgs to crystal} for a special case.

\begin{prop}\label{prop: Higgs to perfect crystals}
Let $\overline A / \mathbb Z_p^\cyc$ be an integral perfectoid algebra and put $A = \mathbb A_{\inf}( \overline A )$. Put
\[ P = A \left\langle T_1, \ldots, T_d \right\rangle \]
with $\delta$-structure given by $\delta(T_i)=0$. Assume that $R/\overline A$ is a senable algebra satisfying Condition \ref{cond: tor -1,0}, and
\[ f: \overline P \to R \]
is a toric chart. Then we have the a commutative diagram as follows
\[\xymatrix{
\Vect\bigl((R/A)_\Prism, \overline\calO\bigl[\frac 1 p \bigr]\bigr) \ar[r] & \Vect\bigl((R/A)_\Prism^\perf, \overline\calO\bigl[\frac 1 p \bigr]\bigr) \ar[d] \\
\Higgs^\tn(R,R \otimes_{\overline P}\cOmega_{\overline P/\overline A} \{-1\}) \ar[r] \ar[u] & \Rep_\Gamma \bigl( R_\infty \bigl[ \frac 1 p \bigr] \bigr)
}\]
where the arrow on the left-hand side is the map in Proposition \ref{prop: Higgs to crystal}, and the arrow at the bottom is given by $(M,\nabla) \mapsto M_\infty = R_\infty \otimes_R M$, where $\gamma_i$ acts semilinearly on $M_\infty$, and acts on $M$ by $\exp((\zeta_p-1) f(T_i)\nabla_i)$, in which
\[ \nabla(m) = \sum_{i=1}^d \nabla_i(m) \otimes \frac{\dif x_i}{\xi}, \]
and $\xi := \dfrac{[\ep]-1}{[\ep]^{1/p}-1} \in \AAinf(\dZ_p^\cyc)$.
\end{prop}

\begin{proof}
Fix $(M,\nabla) \in \Higgs^\tn(R,R \otimes_{\overline P}\cOmega_{\overline P/\overline A} \{-1\})$, and denote $M_\infty = R_\infty \otimes_R M$. Assume that it corresponds to the $\overline \calO\bigl[ \frac 1 p \bigr]$-crystal $\calE$.

Pick a finite covering of $\Spa\bigl(R_\infty\bigl[ \frac 1 p \bigr], R_\infty \bigr)$ by perfectoid rings $f_i: R_\infty \to \overline B_i$, and denote $B_i = \AAinf(\overline B_i)$. Denote the product (in the category of perfectoid rings) of $\overline B_i,\overline B_j$ over $R_\infty$ by $\overline B_{ij}$. Then, for every $\gamma \in \Gamma$, we have the following commutative diagram
\[\begin{tikzcd}[column sep=large, row sep=large]
0 \ar[r] & M_\infty \ar[r,"(f_i\gamma)"] \ar[d,"\gamma"'] & \prod_i \calE(B_i) \ar[r] \ar[d, equal] & \prod_{i,j} \calE(B_{ij}) \ar[d, equal] \\
0 \ar[r] & M_\infty \ar[r,"(f_i)"] & \prod_i \calE(B_i) \ar[r] & \prod_{i,j} \calE(B_{ij})
\end{tikzcd}\]
Therefore, to show that our defined action on $M_\infty$ agrees with the action given by the vector bundle, we only need to prove that for any perfectoid ring $\bigl( \overline B\bigl[ \frac 1 p \bigr], \overline B \bigr)$ over $\bigl(R_\infty\bigl[ \frac 1 p \bigr], R_\infty\bigr)$ with $B = \AAinf(\overline B)$, and we have the following commutative diagram
\[\begin{tikzcd}[column sep=large, row sep=large]
M_\infty \ar[rr, "\gamma"] \ar[rd, "f\circ\gamma"'] & & M_\infty \ar[ld, "f"] \\
& \calE(B) &
\end{tikzcd}\]
where $f$ denotes the map $R_\infty \to \overline B$.

Denote $f' = f \circ \gamma$, and they correspond to $g,g' \in \mathscr C_{R/P/A}(B)$ via $g(T_i) = [f(T_i^\flat)]$ and $g'(T_i) = [f'(T_i^\flat)]$. We only need to show that there exists a commutative diagram as below
\[\xymatrix{
& \calE(B) & \\
\overline B \otimes_{f',R_\infty} M_\infty \ar[rr]_{\gamma} \ar[ru]^{\eta_{B,g'}} & & \overline B \otimes_{f,R_\infty} M_\infty \ar[lu]_{\eta_{B,g}}
}\]
where $\gamma: \overline B \otimes_{f',R_\infty} M_\infty \to \overline B \otimes_{f,R_\infty} M_\infty$ is given by $b \otimes m \mapsto b \otimes \gamma(m)$.

For this, we have the commutative diagram in Proposition \ref{prop: Higgs to crystal computation}, where $h: E(1) \to B$ is induced by $(g,g')$ (note the directions of the arrows --- the arrow at the bottom is from right to left):
\[\xymatrix{
& \calE(B) & \\
\overline B \cotimes_R M \ar[ru]^{\eta_{B,g}}_\cong \ar@{=}[d] & & \overline B \cotimes_R M \ar[lu]_{\eta_{B,g'}}^\cong \ar@{=}[d] \\
\overline B \cotimes_{h,R \cotimes_{\overline P} \overline{E(1)}} (\overline{E(1)} \cotimes_{\delta_1^1,\overline P} M) & & \overline B \cotimes_{h,R \cotimes_{\overline P} \overline{E(1)}} (\overline{E(1)} \cotimes_{\delta_0^1,\overline P}M) \ar[ll]^{\id \otimes \epsilon}_\cong
}\]
Going from the arrow at the bottom, we have
\[ 1 \otimes m \mapsto 1 \otimes (1 \otimes m) \mapsto 1 \otimes \sum_{\underline k} h \left( \frac{\dif \underline T}{\xi} \right)^{[\underline k]} \otimes \nabla^{\underline k} (m), \]
where $\underline k$ runs through all multi-indices in $\mathbb N^d$. Assume that $\gamma = (n_1,\ldots,n_d)$ for $n_i \in \dZ_p$, and we only need to show that $\frac{\dif T_i}{\xi}$ is sent to $n_i ( \zeta_p-1 ) f(T_j)$ in $\overline B$. Indeed, in $B$, we have
\[ h \left( \frac{\dif T_i}{\xi} \right) = \frac{g'(T_i) - g(T_i)}\xi = \frac{[f(\gamma(T_i^\flat))]-[f(T_i^\flat)]}{\xi}. \]
Note that $f(\gamma(T_i^\flat)) = \ep^{n_i} T_i^\flat$, so the value above is equal to (note that by definition $B$ is $p$-torsion-free)
\[ \frac{[\ep]^{n_i}-1}{\xi} [f(T_i^\flat)] = \frac{[\ep]^{n_i}-1}{[\ep]-1} ([\ep]^{1/p}-1)[f(T_i^\flat)]. \]
We have
\[ \frac{[\ep]^{n_i}-1}{[\ep]- 1} = n_i + \binom{n_i} 2 ([\ep]-1) + \binom{n_i} 3 ([\ep]-1)^2 + \cdots, \]
so it reduces to $n_i$ in $\overline B$, and we see that $h \bigl( \frac{\dif T_i}{\xi} \bigr)$ reduces to $n_i (\zeta_p-1) T_i$ in $\overline B$, as desired.
\end{proof}

\section{Local Simpson correspondence}
The readers may check that whether all result of this sections may be applied for the integral case, but we only concern ourselves in the rational case.

\subsection{Small component, small connections and small crystals}
In this section, we will construct the concept `small component' of a filtered module over a ring $A$ and apply this idea to the categories $\MIC^\tn$. The readers may picture the `small component' $A^J$ as a neighbourhood of $\Spf A$, i.e. the unit disk contained in the disk of radius $r^{-1}$. Differential equations on $A^J$ are therefore differential equations defined on a neighbourhood of $A$. In the next subsection, we will see that under some appropriate `overconvergence' condition as above, we can always solve the differential equation, see (the proof) of Theorem \ref{theo: simpson correspondence}.

\begin{defi}[{\cite[Section 3.2]{khan2019virtualcartierdivisorsblowups}}]
Let $A$ be a simplicial (commutative) ring. A \emph{generalised Cartier divisor} is a pair $(J,\alpha: J \to A)$, where $J$ is a line bundle over $A$ and $\alpha$ is a map of $A$-modules. The $\infty$-groupoid of generalised Cartier divisors will be denoted $\GCart(A)$.
\end{defi}

We have the following proposition from \cite{khan2019virtualcartierdivisorsblowups}:
\begin{prop}[{\cite[Proposition 3.2.6]{khan2019virtualcartierdivisorsblowups}}]
There exists a canonical isomorphism of $\infty$-sheaves $\GCart \cong [\mathbb A^1 / \mathbb G_m]$, where $\mathbb G_m$ acts on $\mathbb A^1$ by scaling.
\end{prop}

For simplicity, we will only work with the case that we are interested in. Let $A$ be a classical ring, which is $p$-adically complete and of bounded $p^\infty$-torsion. For any generalised Cartier divisor $\alpha: J \to A$, we adopt the following definition.

\begin{prop-defi}[small component]\label{prop-defi: small component}
Denote $J^*: \gr(\widehat\Mod_{A,\Flat}) \to \gr(\widehat\Mod_{A,\Flat})$ by
\[ J^* (M_k) = (J^k \otimes_A M_k). \]
It is easy to see that $J^*$ is symmetric monoidal.

For $M \in \widehat\Mod_{A,\Flat,\fil}$, define the \emph{$J$-small component of $M$} by 
\[ \Fil_n(M^J) = \bigoplus_{k=0}^n J^k \otimes_A \Fil_k(M) / \bigoplus_{k=0}^{n-1} J^{k+1} \otimes_A \Fil_k(M), \]
where the map
\[ J^{k+1} \otimes_A \Fil_k(M) \to \bigoplus_{k=0}^n J^k \otimes_A \Fil_k(M) \]
is given by the difference of the maps
\[ J^{k+1} \otimes_A \Fil_k(M) \to J^k \otimes_A \Fil_k(M), \quad J^{k+1} \otimes_A \Fil_k(M) \to J^{k+1} \otimes_A \Fil_{k+1}(M) \]
induced by $J^{k+1} \to J^k$ and $\Fil_k(M) \to \Fil_{k+1}(M)$.

Then, $M^J$ is natural in $M$ and $J$. We have a natural isomorphism
\[ \gr_n(M^J) \xrightarrow{\cong} J^n \otimes_A \gr_n(M) \]
given by $J^k \otimes_A \Fil_k(M) \mapsto 0$ for $k <n$ and $J^n \otimes_A \Fil_n(M) \to J^n \otimes_A \gr_n(M)$ for $k=n$. We also have a natural map $M^J \to M$ given by
\[ J^n \otimes_A \Fil_n(M) \xrightarrow{\alpha^n} \Fil_n(M), \]
and it induce the map on graded pieces
\[ \gr_n(M^J) \cong J^n \otimes_A \gr_n(M) \xrightarrow{\alpha^n} \gr_n(M). \]
Moreover, for $M,N \in \widehat\Mod_{A,\Flat,\fil}$ and $i+j \le k$, the maps
\[ (J^i \otimes_A \Fil_i(M)) \cotimes_A (J^j \otimes_A \Fil_j(N)) \cong J^{i+j} \otimes_A \Fil_{i+j}(M) \to \Fil_k((M \cotimes_A N)^J) \]
makes $(-)^J$ into a symmetric monoidal functor, such that the isomorphism above $\gr_* \circ (-)^J \cong J^* \circ \gr_*$ is symmetric monoidal.
\end{prop-defi}

\begin{proof}
We prove that the transition maps $\Fil_{n-1}(M^J) \to \Fil_n(M^J)$ are injective with the desired cokernels. This is a straightforward consequence of the (classical) pushout square
\[\xymatrix{
\Fil_{n-1}(M^J) \ar[r] & \Fil_n(M^J) \\
J^n \otimes_A \Fil_{n-1}(M) \ar[r] \ar[u] & J^n \otimes_A \Fil_n(M) \ar[u]
}\]
and that the arrow on the bottom is injective. The other properties are easy to prove, and we omit them here.
\end{proof}

\begin{rmk}
The functor $(-)^J$ as above can be defined on the $\infty$-category $D(A)$, but the process is rather tricky. Possible ways are to define $(-)^J$ to be the right adjoint of the lax monoidal functor
\[ M = (M_0 \xrightarrow{f_0} M_1 \xrightarrow{f_1} M_2 \to \cdots) \mapsto (M_0 \xrightarrow{\alpha f_0} J^{-1} \otimes_A M_1 \xrightarrow{\alpha f_1} J^{-2} \otimes_A M^2 \to \cdots). \]
Another possible approach is to identify $\mathbb Z$-filtered modules with quasi-coherent modules on the stack $\Spec A \times [\mathbb A^1 / \mathbb G^m]$, and $(-)^J$ is the pullback via
\[ J: \Spec A \times [\mathbb A^1 / \mathbb G^m] \to \Spec A \times [\mathbb A^1 / \mathbb G^m] \]
(note that $\Spec A \times [\mathbb A^1 / \mathbb G^m]$ is a monoid over $\Spec A$, and the map is given by multiplication with $J: \Spec A \to \Spec A \times [\mathbb A^1 / \mathbb G^m]$). Since we won't use this construction, we will not make further discussions into it.
\end{rmk}

\begin{lemma}\label{lem: J-small differential}
Let $A$ be a $p$-complete ring with bounded torsion, $\alpha: J \to A$ be a generalised Cartier divisor, and $S/A$ be a strictly $p$-complete flat filtered ring. Assume that $\Omega$ is a $p$-complete and $p$-completely flat module over $A$, and
\[ \dif: S \to S \cotimes_A \Omega \]
is an $A$-linear differential operator satisfying the Griffith transversality. Denote $J^* \Omega = J \otimes_A \Omega$. Then, there exists a unique differential
\[ \dif': S^J \to S^J \cotimes_A J^*\Omega \]
such that $\dif'(a \otimes x) = a \otimes \dif(x)$ for any $x \in \Fil_k (S)$ and $a \in J^k$. Note that $\dif(x) \in \Fil_{k-1}(S)$ and
\[ a\otimes \dif(x) \in (J^k \otimes_A \Fil_{k-1}(S)) \cotimes_A \Omega \cong (J^{k-1} \otimes_A \Fil_{k-1}(S)) \cotimes_A J^*\Omega\]
is well-defined. We will denote this differential by $J^{-1} \dif$. We have a natural functor between categories:
\[ \MIC^\tn\bigl(S^J\bigl[\frac 1 p \bigr],J^{-1}\dif\bigr) \to \MIC^\tn\bigl(S\bigl[\frac 1p\bigr],\dif\bigr) \]
given by $(M,\nabla') \mapsto (S \otimes_{S^J} M, J\nabla')$, where $J\nabla'$ is defined by
\[ J\nabla'(s \otimes m) = \dif(s) \otimes m + s \alpha( \nabla'(x) ). \]
It is faithful whenever $\alpha \bigl[ \frac 1 p \bigr]$ is injective.
\end{lemma}

\begin{proof}
This is a direct result of Proposition-Definition \ref{prop-defi: small component}, as it describes the ring structure on $S^J$.
\end{proof}

\begin{lemma}\label{lem: J-small differential and split PD}
If $(S(\bullet), \varphi: \Omega \to \Fil_1(S(1)))$ is filtered split PD over $A$ (Definition \ref{defi: filtered split PD}), we may form the cosimplicial ring $S(\bullet)^J$. Then,
\begin{enumerate}[label=(\alph*)]
\item $S(\bullet)^J$ is filtered split PD with $\varphi^J: J \otimes_A \Omega \to S(1)^J$ given by
\[ \varphi^J( a \otimes \omega) = a \otimes \varphi(\omega), \]
\item the differential operator corresponding to $(S(\bullet)^J,\varphi^J)$ is $J^{-1}\dif$ as in Lemma \ref{lem: J-small differential},
\item taking gradings, then the natural isomorphism $J^*\gr_* S(\bullet) \to \gr_* S(\bullet)^J$ is commutative with the split PD structure, i.e. we have a commutative diagram (note that $\Omega$ is mapped into degree 1)
\[\xymatrix{
J^*\gr_* S(\bullet) \ar[rr] & & \gr_* S(\bullet)^J \\
& J^*\Omega \ar[lu]^{J^*\varphi} \ar[ru]_{\varphi^J}
}\]
\item the natural equivalences in Theorem \ref{theo: comp strat connection} fits into the commutative diagram (in the $(2,1)$-category of categories)
\[
\xymatrix{
\Strat\bigl(S(\bullet)^J\bigl[\frac 1 p \bigr]\bigr) \ar[r] \ar[d]_\cong & \Strat\bigl(S(\bullet)\bigl[\frac 1 p \bigr]\bigr) \ar[d]^\cong \\
\MIC^\tn\bigl(S^J\bigl[\frac 1 p\bigr],J^{-1}\dif\bigr) \ar[r] & \MIC^\tn\bigl(S\bigl[\frac 1p\bigr],\dif\bigr)
}\]
where the arrow at the bottom is as in Lemma \ref{lem: J-small differential}.
\end{enumerate}
\end{lemma}

\begin{proof}
These are trivial calculations.
\end{proof}

\begin{lemma}\label{lem: flatness of J-small component}
Let $A$ be a $p$-complete ring of bounded $p^\infty$-torsion, and $J$ be an invertible $A$-module. Let $B$ be a graded $p$-complete and $p$-completely flat $A$-algebra and $M$ be a graded derived $p$-complete $B$-module, then $M/B$ is of $p$-complete $\Tor$-amplitude in $[a,b]$ (resp. $p$-completely flat, $p$-completely faithfully flat) if and only if $J^*M$ is of $p$-complete $\Tor$-amplitude in $[a,b]$ (resp. $p$-completely flat, $p$-completely faithfully flat) over $J^* B$. Here, $J^*: \gr(D(A)) \to \gr(D(A))$ is defined by $(M_i) \mapsto (J^i \otimes_A M_i)$, which can be realised into a symmetric monoidal functor of $\infty$-categories by the same argument as in Example \ref{ex: shearing}.
\end{lemma}

\begin{proof}
Reduce $- \otimes_A^\bL A/p$, and we may assume that $pA=0$, so we only need to deal with the classical non-completed case. Recall that $\Tor$-amplitude (resp. flatness, faithfully flatness) can be checked Zariski locally, so we may reduce to the case where $J \cong A$, and in this case $J^* \cong \id$.
\end{proof}

\begin{lemma}\label{lem: flatness revisited}
Let $A$ be a $p$-complete ring of bounded $p^\infty$-torsion, and $B,B_1,\ldots,B_n,C$ be filtered $p$-complete $A$-algebras such that $\gr_* B, \gr_* B_i,\gr_* C$ are concentrated on degree 0 of bounded $p^\infty$-torsion. Also let $B \to B_i$ and $B \to C$ be $A$-algebra maps and $B_i \to C$ be $B$-algebra maps. Also, let $\alpha: J \to A$ be a generalised Cartier divisor.

\begin{enumerate}[label=(\arabic*)]
\item if $\gr_* B \to \gr_* B_i$ are $p$-completely faithfully flat, then the maps $B^J \to B_i^J$ are $p$-completely faithfully flat;
\item if the induced map
\[ \gr_* B_1 \cotimes_{\gr_* B}^\bL \cdots \cotimes_{\gr_* B}^\bL \gr_* B_n \to \gr_* C \]
is an isomorphism, then the induced map
\[ B_1^J \cotimes_{B^J}^\bL \cdots \cotimes_{B^J}^\bL B_n^J \to C^J \]
is a strict isomorphism of filtered complete $A$-algebras.
\end{enumerate}
\end{lemma}

\begin{proof}
For (1), by Lemma \ref{lem: flatness of J-small component}, the maps $J^* \gr_* B \to J^* \gr_* B_i$ are $p$-completely flat, i.e. the maps $\gr^* (B^J) \to \gr_* (B_i^J)$ are $p$-completely flat, and Corollary \ref{cor: flatness by graded complete case} shows that the maps $B^J \to B_i^J$ are $p$-completely faithfully flat.

For (2), to show that the induced map
\[ B_1^J \cotimes_{B^J}^\bL \cdots \cotimes_{B^J}^\bL B_n^J \to C^J \]
is a strict isomorphism of filtered complete $A$-algebras, we may take the graded pieces, and prove that
\[ \gr_* (B_1^J) \cotimes_{\gr_*(B^J)}^\bL \cdots \cotimes_{\gr_*(B^J)} \gr_* (B_n^J) \to \gr_*(C^J) \]
is an isomorphism of graded complete $A$-algebras. Using the natural isomorphism $\gr_*(X^J) \cong J^*(\gr_*(X))$, we only need to prove that the natural map
\[ J^* \gr_*B_1 \cotimes_{J^* \gr_* B}^\bL \cdots \cotimes_{J^* \gr_* B} J^* \gr_* B_n \to J^* \gr_* C \]
is an isomorphism. This follows from our condition and the fact that $J^*$ can be made into a symmetric monoidal equivalence $\gr(\cD(A)) \to \gr(\cD(A))$ preserving relative tensor products (see the last paragraph of Example \ref{ex: shearing}).

Here, we take advantage of $\cotimes^\bL$ being coproducts for algebras, so we do not need to construct $(-)^J$ being `relative symmetric monoidal'.
\end{proof}

\begin{theo}\label{theo: fake J-small connection}
Let $A$ be a bounded prism, $R/\overline A$ be a ring satisfying Condition \ref{cond: tor -1,0}, and $\alpha: J \to R$ be a generalised Cartier divisor. Then, there exists a category
\[ \Vect^{\fJsm J} \bigl( (R/A)_\Prism, \overline\calO \bigl[ \frac 1 p \bigr] \bigr) \]
functorial in $R,A$ and $J$, and a functor (faithful whenever $\alpha \bigl[ \frac 1 p \bigr]$ is injective)
\[ \psi_{R/A}: \Vect^{\fJsm J} \bigl( (R/A)_\Prism, \overline\calO \bigl[ \frac 1 p \bigr] \bigr) \to \Vect \bigl( (R/A)_\Prism, \overline\calO \bigr) \]
such that for each $A \to S \to R$ with $S/A$ satisfying Condition \ref{cond: flat + tor 0} and $R/\overline S$ satisfying Condition \ref{cond: tor -1}, we have a commutative diagram in the $(2,1)$-category of categories:
\[\xymatrix@C=5pc{
\Vect^{\fJsm J} \bigl( (R/A)_\Prism, \overline\calO \bigl[ \frac 1 p \bigr] \bigr) \ar[d]_{\psi_{R/A}} \ar[r]^(.58){\varphi_{R/S/A}}_(.58)\cong & \MIC^\tn \bigl( \overline\Prism_{R/S}^J \bigl[ \frac 1 p \bigr], J^{-1}\dif \bigr) \ar[d] \\
\Vect \bigl( (R/A)_\Prism, \overline\calO \bigl[ \frac 1 p \bigr] \bigr) \ar[r]_\cong & \MIC^\tn \bigl( \overline\Prism_{R/S} \bigl[ \frac 1 p \bigr], \dif \bigr) \\
}\]
\end{theo}

\begin{proof}
We will first ignore the set-theoretic difficulties in the following proof, and then address it in the remark below.

Fix $R/\overline A$ satisfying Condition \ref{cond: tor -1,0} and a generalised Cartier divisor $J \to R$. Consider the category $\mathcal A$ of diagrams $A \to S \to R$ with $S/A$ satisfying Condition \ref{cond: flat + tor 0} and $R/\overline S$ satisfying Condition \ref{cond: tor -1}. The category $\mathcal A$ is nonempty, and has pairwise coproduct given by
\[ ( A\to S_1 \to R) \sqcup (A \to S_2 \to R) = (A \to S_1 \cotimes_A S_2 \to R). \]
It follows that $\mathcal A$ is sifted, hence is weakly contractible (\cite[Proposition 5.5.8.7]{lurie2009higher}).

Denote $\Cat_1$ the (2,1)-category of small categories, and we may form the functor $F: \mathcal A \to \Cat_1$ by
\[ (A \to S \to R) \mapsto \MIC^\tn \bigl( \overline\Prism_{R/S}^N \bigl[ \frac 1 p \bigr], J^{-1}\dif \bigr). \]
If we proved that every morphism in $\mathcal A$ maps to an isomorphism, we will then be able to take $\Vect^{\fJsm J} \bigl( (R/A)_\Prism, \overline\calO \bigl[ \frac 1 p \bigr] \bigr)$ to be the limit of $F$. Let $x,y$ be objects in $\mathcal A$ and $f: x \to y$ be a morphism. We may form the coproduct $x\sqcup y$ and the insertion maps $i: x \to x \sqcup y, j : y \to x \sqcup y$. Form $\varphi: x \sqcup y \to y$ by $\varphi i = f$ and $\varphi j = \id_y$. If we proved that $Fi$ and $Fj$ are isomorphisms, then by the second identity, we see that $F\varphi$ is an isomorphism, and the first identity now shows that $f$ is an isomorphism. Therefore, we only need to prove that $Fi$ is an isomorphism for an insertion map
\[ i: x \to x\sqcup y.\]

Write $x=(A \to S_1 \to R)$ and $y=(A \to S_2 \to R),$ hence $x\sqcup y = (A \to S_1 \cotimes_A S_2 \to R)$, and $i$ is given by $i: S_1 \to S_1 \cotimes_A S_2$. Denote $S=S_1 \cotimes_A S_2$ and $\Omega_i = \overline S \cotimes_{\overline S_i} \cOmega_{\overline S_i/\overline A} \{-1\}$. Consider the filtered cosimplicial rings
\[ S_1(\bullet) = S_1^{\cotimes(\bullet+1)/A}, S_2(\bullet) = S_2^{\cotimes(\bullet+1)/A}, S(\bullet) = S^{\cotimes(\bullet+1)/A}. \]
Denote $D(\bullet) = \Prism_{R/S(\bullet)}$, and also the double cosimplicial ring
\[ D'(n,m) = \Prism_{R/S_1(n) \cotimes_A S_2(m)}. \]
Therefore, $D(n)$ is nothing else but $D'(n,n)$. %Now, as in Proposition \ref{prop: D(n) split PD}, $\ker(\overline{D'(n,m)} \to \overline{D'(0,0)})$ admits a PD structure by extension from
%\[ \overline{E'(n,m)} := \Prism_{\overline S_1 \cotimes_{\overline A} \overline S_2 / S_1 (n) \cotimes_A S_2(m)}. \]
%Also, as for the maps $\Omega_1 \to \overline\Prism_{R/S_1(1)}$ and $\Omega_2 \to \overline\Prism_{R/S_2(m)}$, they induce an isomorphism of rings
%\[ \overline{E'(n,m)} \cong \cGamma^\bullet_{\overline S}(\Omega_1^{\oplus n} \oplus \Omega_2^{\oplus m}). \]
%hence inducing an isomorphism
%\[ \overline{D'(n,m)} \cong \overline{D(0)} \cotimes_{\overline S} \cGamma^\bullet_{\overline S}(\Omega_1^{\oplus n} \oplus \Omega_2^{\oplus m}). \]

Take
\[ D_1(n) = \Prism_{R/S_1(n)}, \]
then we have a natural morphism of filtered double cosimplicial rings
\[ \overline{D_1(n)} \to \overline{D'(n,m)}, \quad\forall n,m.\]
Passing $\MIC$ to stratifications, and we only need to show that the map of cosimplicial rings
\[ \overline{D_1(n)}^J \to \overline{D'(n,n)}^J \]
induces an equivalence of categories of stratifications, i.e. we have to show that the map
\[ \Vect\bigl(\simp, \overline{D_1(n)}^J \bigl[ \frac 1 p \bigr] \bigr) \to \Vect\bigl(\simp,\overline{D'(n,n)}^J \bigl[ \frac 1 p \bigr] \bigr) \]
is an equivalence. Since $\simp^\op$ is sifted, the map $\simp \to \simp^2$ is initial (in the sense of $\infty$-categories), and by Lemma \ref{lemma: equivalence of crystals}, we only need to prove that the map
\[ \Vect\bigl( \simp^2, \overline{D_1(n)}^J \bigl[ \frac 1 p \bigr] \bigr) \to \Vect(\simp^2,\overline{D'(n,m)}^J \bigl[ \frac 1 p \bigr] \bigr) \]
is an equivalence of categories. Taking slices of the first factor, and we only need to prove that for each $n$, the map
\[ \Vect\bigl( \simp, \overline{D_1(n)}^J \bigl[ \frac 1 p \bigr] \bigr) \to \Vect(\simp,\overline{D'(n,\bullet)}^J \bigl[ \frac 1 p \bigr] \bigr) \]
is an equivalence of categories. By Theorem \ref{thm: descent theory} (2), we only need to show that $\overline{D_1(n)}^J \to \overline{D'(n,0)}^J$ are $p$-completely faithfully flat, and the natural maps
\[ \bigl(\overline{D'(n,0)}^J \bigr)^{\cotimes^\bL(m+1)/\overline{D_1(n)}^J} \to \overline{D'(n,m)}^J  \]
are isomorphisms. By Lemma \ref{lem: flatness revisited}, we only need to show that $\dHodge_{R/\overline{S_1(n)}} \to \dHodge_{R/\overline{S(n,0)}}$ is $p$-completely faithfully flat, and that the maps
\[ \dHodge_{R/\overline{S(n,0)}}^{\cotimes^\bL(m+1)/\dHodge_{R/\overline{S_1(n)}}} \to \dHodge_{R/\overline{S(n,m)}} \]
are isomorphisms. The faithful flatness part follows from the exact triangle (all of whose terms are concentrated on degree 0)
\[ \cdL_{R/\overline{S_1(n)}} [-1]\{-1\} \to \cdL_{R/\overline{S(n,0)}} [-1]\{-1\} \to R \cotimes_{\overline{S(n,0)}}^\bL \cdL_{\overline{S(n,0)} / \overline{S_1(n)}}\{-1\}, \]
and Corollary \ref{cor: pd strict flatness complete case}. The isomorphism part is now evident, once one observe that $\overline{S(n,m)} \cong \overline{S(n,0)}^{\cotimes^\bL(m+1)/{\overline{S_1(n)}}}$, and that $\dHodge_{-/-}$ commutes with colimits.

Now consider the category $\calC$ with objects $(A \to R,J \to R)$ with $A$ a prism, $R$ a bounded $p$-complete algebra over $\overline A$ satisfying Condition \ref{cond: tor -1,0}, and $J$ a generalised Cartier divisor over $A$. Denote $\calA'$ the category of all $(A \to S \to R, J \to R)$ with $S/A$ satisfying Condition \ref{cond: flat + tor 0}, $R/\overline S$ satisfying Condition \ref{cond: tor -1}, and $J \to R$ a generalised Cartier divisor. We have the natural projection
\[ \pi : \calA' \to \calC. \]
Note that given any particular $x  =(A \to R, J \to R) \in \calC$, the na\"ive fibre $\calA'_x$ of $\pi$ is given by the category $\calA$ as above, hence is weakly contractible. We have the functor
\[ F: \calA' \to \Cat_1, \quad (A \to S \to R, J\to R) \mapsto \MIC^\tn\bigl( \overline\Prism_{R/S} \bigl[ \frac 1 p \bigr]^J, J^{-1}\dif \bigr). \]
We claim that for any $x =(A \to R)\in \calC$, the functor $\calA'_{x} \to \calA'_{x/}$ is initial. Indeed, for $y = (A' \to S' \to R', A\to A', R\to R', R' \otimes_R J \to J') \in \calA'_{x/}$, the category $(\calA'_{x})_{/y}$ is the category of all commutative diagrams as below:
\[\xymatrix{
A \ar[r] \ar[d] & S \ar[r] \ar[d] & R \ar[d] \\
A' \ar[r] & S' \ar[r] & R'
}\]
where $(A \to S \to R) \in \calA'_x$. This category has pairwise coproduct given by $(S_1,S_2) \mapsto S_1 \cotimes_A S_2$, and we see that it is sifted, hence is weakly contractible (\cite[Proposition 5.5.8.7]{lurie2009higher}). Therefore, the right Kan extension of $F$ along $\pi$ is given by
\[ \RKE_\pi F(x) = \lim_{\calA'_{x/}} F = \lim_{\calA'_x} F. \]
We have already seen that every morphism $y \to y'$ in $\calA'_x$ induces an equivalence of categories $Fy \to Fy'$, hence every projection map
\[ \lim_{\calA'_x} F \to Fy \]
(where $y \in \calA'_x$) is an isomorphism. It follows that $F \to (\RKE_\pi F) \circ \pi$ is an isomorphism. We may define the functor $\Vect^{\fJsm J}$ to be the right Kan extension $\RKE_\pi F$.

At last, by our previous discussions, the natural map
\[ \Vect\bigl((R/A)_\Prism, \overline\calO \bigl[ \frac 1 p \bigr] \bigr)\to \Vect^{\fJsm \id} \bigl((R/A)_\Prism, \overline\calO \bigl[ \frac 1 p \bigr] \bigr) \]
is an isomorphism (where $\id$ denotes the trivial Cartier divisor $\id: A \to A$), and the commutative diagram is given by the naturality of $\varphi_{R/S/A}$ with respect to $J$ (and the natural transformation of functors $J \to \id$).
\end{proof}

\begin{rmk}
The main set-theoretic issue here is that the categories $\calA'_x$ are not small, let alone $\calA'$. There are two ways to resolve to this problem: one is to take a truncation by a uncountable strong limit cardinal $\kappa$, and then the categories are all (essentially) small; the other way is to construct a functorial resolution $(A \to P_{R/A} \to R)$ for each $R$. For example, take $P_{R/A}$ to be the completed polynomial ring generated by all elements of $R$, and its $\delta$ structure be given by $\delta(x)=0$ for each indeterminant $x$. Now we take $\Vect^{\fJsm J}\bigl((R/A)_\Prism,\overline\calO\bigl[\frac 1p\bigr]\bigr)$ to be $\MIC^\tn\bigl(\overline\Prism_{R/P_{R/A}}^J \bigl[ \frac 1 p \bigr], J^{-1}\dif \bigr)$, and define the transition maps by the functoriality of $P_{R/A}$. As for the map $\varphi_{R/S/A}$, one may use the arguments as in the proof, to write it as the composition of some functors given by the (inverse of) the insertion map $x \to x\sqcup y$, and one may fill in the squares by $2$-morphisms in a canonical way. We will omit the details here.
\end{rmk}

\begin{theo}
For any bounded prism $(A,I)$, a ring $A'/\overline A$ and a generalised Cartier divisor $\alpha: J \to A'$, the construction
\[ R \mapsto \Vect^{\fJsm {J_R}}\bigl((R/A)_\Prism, \overline\calO \bigl[ \frac 1 p \bigr] \bigr) \]
(where $J_R := R \otimes_{A'} J$) is a quasi-syntomic stack on the category of rings $R/A'$ such that $R/\overline A$ satisfies Condition \ref{cond: tor -1,0}.
\end{theo}

\begin{proof}
Note that for an algebra $R/A'$, the construction $(-)^{J_R}$ is the same as $(-)^J$, so we will always write $(-)^J$.

It is clear that the construction commutes with finite products. Let $R \to R'$ be a quasi-syntomic cover. Pick a completed polynomial ring $P/A$ and a surjection $P \to R$ of $A$-algebras, and then pick a completed polynomial ring $P'/P$ and a surjection $P' \to R'$ extending $P \to R$. We have the following commutative diagram
\[\xymatrix{
A \ar[r] \ar[d] & P \ar[r] \ar[d] & P' \ar[d] \\
\overline A \ar[r] & R \ar[r] & R'
}\]
where $P \to P'$ is faithfully flat and $\cdL_{P'/P}$ is $(p,I)$-completely flat over $P'$. Equip $P$ and $P'$ with the $\delta$-structure sending each indeterminate $T$ to $0$, so $P,P'$ are made into prisms over $A$, and $P \to P'$ is a map of prisms. Also, define
\[ P(m) = P^{\cotimes(m+1)/A}, \quad P'(n,m) = (P'(n))^{\cotimes (m+1)/A}. \]

Denote $R'(n) = (R')^{\cotimes(n+1)/R}$ and $P'(n) = (P')^{\cotimes(n+1)/P}$. By naturality, the map
\[ \MIC^{\tn} \bigl( \overline\Prism_{R/P}^J \bigl[ \frac 1 p \bigr], J^{-1} \dif \bigr) \to \lim_\simp \MIC^{\tn} \bigl( \overline\Prism_{R'(\bullet)/P'(\bullet)}^J \bigl[ \frac 1 p \bigr], J^{-1} \dif \bigr) \]
is identified with
\[ \Vect \bigl(\simp, \overline\Prism_{R/P(\bullet)}^J \bigl[ \frac 1 p \bigr]\bigr) \to \lim_{n \in \simp} \Vect \bigl(\simp, \overline\Prism_{R'/P'(n,\bullet)}^J \bigl[ \frac 1 p \bigr]\bigr) \]
Write $\Vect(\simp,-)$ as $\lim_{m \in \simp}$, and we only need to prove that for each $m \in \simp$, the map
\[ \Vect \bigl( \overline\Prism_{R/P(m)}^J \bigl[ \frac 1 p \bigr]\bigr) \to \lim_{n \in \simp} \Vect \bigl( \overline\Prism_{R'(n)/P'(n,m)}^J \bigl[ \frac 1 p \bigr]\bigr) \]
is an isomorphism. By Theorem \ref{thm: descent theory} (2), we only need to prove that $\overline\Prism_{R'/P'(m)}^J$ is $p$-completely faithfully flat over $\overline\Prism_{R/P(m)}^J$, and the natural maps
\[ \bigl( \overline\Prism_{R'/P'(0,m)}^J \bigr)^{\cotimes^\bL(n+1)/ \overline\Prism_{R/P(m)}^J} \to \overline\Prism_{R'(n)/P'(n,m)}^J \]
are isomorphisms. By Lemma \ref{lem: flatness revisited}, we only need to prove that $\dHodge_{R/\overline{P(m)}} \to \dHodge_{R'/\overline{P'(0,m)}}$ is $p$-completely faithfully flat, and the maps
\[ \bigl( \dHodge_{R'/\overline{P'(0,m)}} \bigr)^{\cotimes^\bL (n+1)/ \dHodge_{R/\overline{P(m)}}} \to \dHodge_{R'(n)/\overline{P'(n,m)}} \]
are isomorphisms. The isomorphisms are a direct consequence that $\dHodge_{-/-}$ commutes with colimits, so we only need to prove the first part. For the faithful flatness part, since $R \to R'$ is $p$-completely faithfully flat, we only need to prove that
\[ R' \cotimes_R \dHodge_{R/\overline{P(m)}} \to \dHodge_{R'/\overline{P'(0,m)}} \]
is $p$-completely faithfully flat. By Corollary \ref{cor: pd strict flatness complete case}, we only need to prove that the cofibre of
\[ R' \cotimes_R \cdL_{R/\overline{P(m)}}[-1] \to \cdL_{R'/\overline{P'(0,m)}}[-1] \]
is $p$-completely flat over $R'$.

Denote $B = \overline{P(m)}$ and $B' = \overline{P'(0,m)}$, so $B'$ is a $p$-completed polynomial ring over $B$. Define $R_1 = R \cotimes_B B'$, so there exists a map $R_1 \to R'$, and we have
\[ R' \cotimes_R \cdL_{R/B} \cong R' \cotimes_R^\bL \cdL_{R/B} \cong R' \cotimes_{R_1}^\bL \cdL_{R_1/B'}. \]
The map
\[ R' \cotimes_R \cdL_{R/\overline{P(m)}}[-1] \to \cdL_{R'/\overline{P'(0,m)}}[-1] \]
is identified with
\[ R' \cotimes_{R_1}^\bL \cdL_{R_1/B'}[-1] \to \cdL_{R'/B'}[-1], \]
and its cofibre is
\[ \cdL_{R'/R_1}[-1]. \]
Note that $R_1 \to R'$ is surjective, since $B' \to R'$ is surjective. It follows that $\cdL_{R'/R_1}[-1]$ is connective. Also, we have an exact triangle
\[ \cdL_{R'/R}[-1] \to \cdL_{R'/R_1}[-1] \to R' \cotimes_{R_1}^\bL \cdL_{R_1/R}. \]
The complexes $\cdL_{R'/R}[-1]$ and $R' \cotimes_{R_1}^\bL \cdL_{R_1/R} \cong R' \cotimes_{B'}^\bL \cOmega_{B'/B}$ are all of $p$-complete $\Tor$-amplitude on $[0,1]$, so $\cdL_{R'/R_1}[-1]$ is also of $p$-complete $\Tor$-amplitude on $[0,1]$. Since it is also connective, it must be of $p$-complete $\Tor$-amplitude on $[0,0]$, i.e. it is $p$-completely flat over $R'$, as desired.
\end{proof}

\begin{defi}\label{defi: fake J-small crystal globalised}
Let $(A,I)$ be a bounded prism. For any qcqs bounded $\overline A$-formal scheme $X$ and a generalised Cartier divisor $\alpha: \calJ \to \calO_X$, define the quasi-syntomic stack
\[ \Vect^{\fJsm \calJ} \bigl((X/A)_\Prism, \overline\calO\bigl[ \frac 1 p \bigr]\bigr) \]
on $X$ by
\[ \Spf(R) \mapsto \Vect^{\fJsm {\calJ(R)}}\bigl((R/A)_\Prism, \overline\calO \bigl[ \frac 1 p \bigr]\bigr). \]
This stack is natural in $(A,I,\calJ,X)$, and we have a natural map
\[ \Vect^{\fJsm \calJ}\bigl((X/A)_\Prism, \overline\calO \bigl[ \frac 1 p \bigr]\bigr) \to \Vect\bigl((X/A)_\Prism,\overline\calO \bigl[ \frac 1 p \bigr] \bigr) \]
which is an isomorphism when $\alpha = \id$, and is faithful whenever $\alpha \bigl[ \frac 1 p \bigr]$ is injective.

If $\alpha \bigl[ \frac 1 p \bigr]$ is invertible, we define $\Vect^{\sm \calJ} \bigl((X/A)_\Prism, \overline\calO \bigl[ \frac 1 p \bigr]\bigr)$ to be the essential image of
\[ \Vect^{\fJsm \calJ}\bigl((X/A)_\Prism, \overline\calO \bigl[ \frac 1 p \bigr]\bigr) \to \Vect\bigl((X/A)_\Prism,\overline\calO \bigl[ \frac 1 p \bigr] \bigr). \]
The \'etale local essential image is denoted by $\Vect^{\smet \calJ} \bigl((X/A)_\Prism, \overline\calO \bigl[ \frac 1 p \bigr]\bigr)$
\end{defi}

For future use, we will also need the following lemma:

\begin{lemma}\label{lem: Higgs is fpqc stack}
Let $R$ be a $p$-complete ring of bounded $p^\infty$-torsion, and $S^\bullet$ be split PD over $R$ (see Condition \ref{cond: split pd}). Then, the assignments
\begin{align*}
\Spf(R') & \mapsto \Strat (R' \cotimes_R S^\bullet), \\
\Spf(R') & \mapsto \Strat \bigl(R' \cotimes_R S^\bullet \bigl[ \frac 1 p \bigr] \bigr)
\end{align*}
are fpqc stacks on the category of bounded affine formal schemes over $R$.

In particular, if we put $S^\bullet = \cGamma^*_R(E^{\oplus \bullet})$ as in Theorem \ref{theo: split PD vs derivation}, by Theorem \ref{theo: comp strat connection}, we see that the assignments
\begin{align*}
\Spf(R') & \mapsto \Higgs^\tn(R',R' \cotimes_R E), \\
\Spf(R') & \mapsto \Higgs^\tn \bigl(R'\bigl[ \frac 1 p \bigr],R' \cotimes_R E\bigr)
\end{align*}
are fpqc stacks on the category of bounded affine formal schemes over $R$.
\end{lemma}

\begin{proof}
For any simplicial ring $X^\bullet$, we may write
\[ \Strat(X^\bullet) \cong \lim_{n \in \simp} \Vect(X^n). \]
Now, Theorem \ref{thm: descent theory} gives us the desired result.
\end{proof}

\subsection{Local Simpson correspondence}
\begin{prop}\label{prop: factor-through connection}
Let $(A,I)$ be a bounded prism and $R/\overline A$ be an algebra satisfying Condition \ref{cond: tor -1,0}, such that $\cdL_{R/\overline A}$ is concentrated on degree 0 and $p$-torsion-free.

Assume that the $\delta$-algebra $S/A$ satisfies Condition \ref{cond: flat + tor 0}, and an $A$-algebra map map $S \to R$ satisfies Condition \ref{cond: tor -1}. Denote $N = \cdL_{R/S}[-1]$. Then,
\begin{enumerate}[label=(\arabic*)]
\item for $k \in \mathbb Z_+$, the map
\[\cbigwedge\nolimits^k_R (N) \to \cbigwedge^k_R(R \cotimes_{\overline S} \cOmega_{\overline S/\overline A}) \]
is injective, whose cokernel $Q_k$ is of bounded $p^\infty$-torsion; moreover, $Q_k[p^\infty]$ is annihilated by $k!$;
\item the differential $\dif: \overline\Prism_{R/S} \to \overline\Prism_{R/S} \cotimes_R \bigl(R \cotimes_{\overline S} \cOmega_{\overline S/\overline A} \bigr)\{-1\}$ factors through $N \to R \cotimes_{\overline S} \cOmega_{\overline S/\overline A}$ uniquely to give a $R$-linear differential
\[ \dif_0: \overline\Prism_{R/S} \to \overline\Prism_{R/S} \cotimes_R N \{-1\}, \]
which is also a Higgs field over $R$.
\item the differential $\dif_0$ satisfies Griffith transversality, is topologically nilpotent, and for the corresponding filtered PD ring $D'(\bullet)$ given in Theorem \ref{theo: split PD vs derivation}, the ring $\gr_* D'(0)$ is faithfully flat over $R$, and the natural maps
\[ \gr_* D'(0)^{\cotimes^\bL (n+1) /R} \to \gr_* D'(n) \]
are isomorphisms. As a result, by Corollary \ref{cor: flatness by graded complete case} and Theorem \ref{thm: descent theory}, the functors below
\begin{align*}
\Vect(R) & \to \MIC^\tn(\overline\Prism_{R/\overline S}, \dif_0) \\
\Vect\bigl(R\bigl[ \frac 1 p \bigr]\bigr) & \to \MIC^\tn\bigl(\overline\Prism_{R/\overline S}\bigl[ \frac 1 p \bigr], \dif_0\bigr)
\end{align*}
are isomorphisms.
\item for $M/R$ satisfying Condition \ref{cond: very conc deg 0} and generalised Cartier divisor $\alpha: J \to R$, the map $M \to \DR(\overline \Prism_{R/S}^J \cotimes_R M, J^{-1} \dif_0 \otimes \id)$ is a quasi-isomorphism.
\end{enumerate}
\end{prop}

\begin{proof}
For (1), recall that we have a natural exact sequence (the cotangent sequence)
\[ 0 \to N \to \Omega \to \cdL_{R/\overline A} \to 0. \]
Observe that the map $N^{\cotimes k/R} \to \Omega^{\cotimes k/R}$ admits a factorisation
\[ N^{\cotimes k/R} \to \cdots \to N^{\cotimes i/R} \cotimes_R \Omega^{\cotimes(k-i)/R} \to \cdots \to \Omega^{\cotimes k/R}. \]
The cofibres of the maps in the factorisation are given by $N^{\cotimes i/R} \cotimes_R \Omega^{\cotimes(k-i-1)/R} \cotimes_R \cdL_{R/\overline A}$. By Corollary \ref{cor: tensor conc on degree 0}, these objects are concentrated on degree 0 and $p$-torsion-free. For any derived $p$-complete $R$-module $X$, consider the natural map
\[ \eta_X: X^{\cotimes^\bL k/R} \to \cbigwedge^k_R(X). \]
Define the map $\xi_X: \cbigwedge^k_R(X) \to X^{\cotimes^\bL k/R}$ as (the left Kan extension of) the map
\[ x_1 \wedge \cdots \wedge x_k \mapsto \sum_{\sigma \in \Sigma_k} \sgn(\sigma) \cdot x_{\sigma(1)} \otimes \cdots \otimes x_{\sigma(k)}. \]
We have a natural isomorphism $\eta_X \circ \xi_X \cong k!$. For $x \in Q_k[p^\infty]$, assume that $p^N x =0$. Take a lifting $y \in \cbigwedge^k_R(\Omega)$ of $x$, and we see that $p^N y$ is in the image of $\cbigwedge^k_R(N) \to \cbigwedge^k_R(\Omega)$. Therefore, $p^N \xi_\Omega(y)$ is in the image of $N^{\cotimes k/R} \to \Omega^{\cotimes k/R}$. Since the cokernel of the injection $N^{\cotimes k/R} \to \Omega^{\cotimes k/R}$ is $p$-torsion-free, we may take $u \in N^{\cotimes k/R}$ mapping to $\xi_\Omega(y)$. It follows that the image of $\eta_N(u)$ is $\eta_\Omega \xi_\Omega(y) = k! y$, and we see that $k!x =0$. Therefore, $Q_k[p^\infty]$ is annihilated by $k!$, as desired. Similarly, the kernel of $\cbigwedge\nolimits^k_R (N) \to \cbigwedge^k_R(R \cotimes_{\overline S} \cOmega_{\overline S/\overline A})$ is annihilated by $k!$, but the kernel is $p$-torsion-free as a submodule of the $p$-complete and $p$-completely flat $R$-module $\cbigwedge\nolimits^k_R (N)$, so the kernel is zero.

For (2), recall that by Hodge--Tate comparison,
\[ \cbigoplus_{k=0}^\infty \gr_k \bigl( \overline\Prism_{R/S} \bigr) \cong \cGamma_R^*(N), \]
so the ring
\[ \mathfrak A:= R\bigl[ \frac 1 p \bigr] [\Fil_1] \]
is dense in $\overline\Prism_{R/S} \bigl[ \frac 1 p \bigr]$ (for each $x \in p^k\Fil_k$ and $N \in \mathbb Z_+$, there exists $y \in R[\Fil_1]$ and $x' \in \Fil_k$ such that $x-x' \in p^N \Fil_k$ and $x'-y \in \Fil_{k-1}$, so induction shows that $\mathfrak A \cap \Fil_k \bigl[ \frac 1 p \bigr]$ is dense in $\Fil_k \bigl[ \frac 1 p \bigr]$, and hence $\mathfrak A$ is dense in $\overline\Prism_{R/S} \bigl[ \frac 1 p \bigr]$). Since the differential $\dif: \overline\Prism_{R/S} \to \overline\Prism_{R/S} \cotimes_R \bigl(R \cotimes_{\overline S} \cOmega_{\overline S/\overline A} \bigr)\{-1\}$ is compatible with the Hodge--Tate comparison, we see that the image of $\Fil_1$ is $\overline\Prism_{R/S} \cotimes_R N \{-1\}$. By Corollary \ref{cor: tensor conc on degree 0}, we have an exact sequence of $p$-torsion free modules
\[ 0 \to \overline\Prism_{R/S} \cotimes_R N\{-1 \} \to \overline\Prism_{R/S} \cotimes_R \bigl(R \cotimes_{\overline S} \cOmega_{\overline S/\overline A} \bigr)\{-1\} \to \overline\Prism_{R/S} \cotimes_R \cdL_{R/\overline A} \to 0. \]
Inverting $p$, and we see that $\dif(\mathfrak A) \subseteq \overline\Prism_{R/S} \cotimes_R N\{-1 \} \bigl[ \frac 1 p \bigr]$. Denote
\[ \mathfrak A^\circ = \mathfrak A \cap \overline\Prism_{R/S}, \]
so $\mathfrak A^\circ \bigl[ \frac 1 p \bigr] = \mathfrak A$. Since $\overline\Prism_{R/S} \cotimes_R \cdL_{R/\overline A}$ is $p$-torsion free by Corollary \ref{cor: tensor conc on degree 0}, we see that $\dif(\mathfrak A^\circ) \subseteq \overline\Prism_{R/S} \cotimes_R N\{-1 \}$. As $\mathfrak A^\circ$ is dense in $\overline\Prism_{R/S}$, we see that $\im \dif \subseteq \overline\Prism_{R/S} \cotimes_R N\{-1 \}$, and we get a desired $\dif_0$.

To prove that $\dif_0$ is a Higgs field, it suffices to prove that
\[ \overline\Prism_{R/S} \cotimes_R \bigwedge\nolimits_R^2(N) \{-2\} \to \overline\Prism_{R/S} \cotimes_R \bigwedge\nolimits_R^2\bigl( R \cotimes_{\overline S} \cOmega_{\overline S/\overline A} \bigr) \{-2 \} \]
is injective, which is follows from Corollary \ref{cor: tensor conc on degree 0} and the content of (1) for $k=2$.

Now we prove (3). We omit the verification that $\dif_0$ satisfies Griffith transversality. Observe that $\gr_*(\dif_0)$ is the canonical differential $\dif'$ on $\dHodge_{R/\overline S}$, which is topologically nilpotent. That $\dif_0$ is topologically nilpotent now follows from Lemma \ref{lem: topologically nilpotent} and a simple induction (note that we only need to verify the condition on $A^\circ$). At last, for the ring part, it is easy to see that $\gr_*(\overline\Prism_{R/S}, \dif_0) \cong (\cGamma^*_R(N\{-1\}),\dif': \cGamma^*_R(N\{-1\}) \cotimes_R N \{-1\})$ (where $\dif'$ is the natural derivation), and the graded PD cosimplicial ring corresponding to $(\cGamma^*_R(N\{-1\}),\dif': \cGamma^*_R(N\{-1\}) \cotimes_R N \{-1\})$ is $\cGamma^*_R(N\{-1\}^{\oplus(\bullet +1)})$. Therefore, $\gr_* D'(\bullet) \cong \cGamma^*_R(N\{-1\}^{\oplus(\bullet +1)})$, and our conclusion follows from Proposition \ref{prop: pd strict flatness}.

At last, (4) is a direct consequence of (3), Lemma \ref{lem: flatness revisited}, Theorem \ref{theo: comp strat connection cohomologies} and Theorem \ref{thm: descent theory}.
\end{proof}

\begin{rmk}
It is a priori hard to describe $D'(n)$ --- hard even for $D'(1)$.
\end{rmk}

\begin{theo}\label{theo: simpson correspondence}
Let $(A,I)$ be a bounded prism and $R/\overline A$ be a ring satisfying Condition \ref{cond: tor -1,0}, such that $R$ is $p$-torsion-free, $\cdL_{R/\overline A}$ is concentrated on degree 0 and $p$-torsion-free. Let $P/A$ be a $\delta$-algebra satisfying Condition \ref{cond: flat + tor 0}. Fix an element $a \in R$ which is invertible in $R \bigl[ \frac 1 p \bigr]$. Assume that we are given a map of $A$-algebras $P \to R$ such that the induced map $R \otimes_{\overline P}\cOmega_{\overline P/\overline A} \to \cdL_{R/\overline A}$ is an isomorphism after inverting $p$, and its cokernel is annihilated by $a$. Then, the functor
\[ \Higgs^\tn\bigl(R \bigl[ \frac 1 p \bigr], R \cotimes_{\overline P} \cOmega_{\overline P/\overline A} \{-1\}\bigr) \to \Vect\bigl( (R/A)_\Prism, \overline \calO \bigl[ \frac 1 p \bigr] \bigr) \]
given in Proposition \ref{prop: Higgs to crystal} is fully faithful, and its essential image contains
\[ \Vect^{\sm a}\bigl( (R/A)_\Prism, \overline \calO \bigl[ \frac 1 p \bigr] \bigr). \]

Moreover, for any element $b \in R \cap \bigl( R \bigl[ \frac 1 p \bigr] \bigr)^\times$, if we denote the essential image of the composition
\[ \Higgs^\tn\bigl(R \bigl[ \frac 1 p \bigr], bR \cotimes_{\overline P} \cOmega_{\overline P/\overline A} \{-1\}\bigr) \to \Higgs^\tn\bigl(R \bigl[ \frac 1 p \bigr], R \cotimes_{\overline P} \cOmega_{\overline P/\overline A} \{-1\}\bigr) \to \Vect\bigl( (R/A)_\Prism, \overline \calO \bigl[ \frac 1 p \bigr] \bigr) \]
by $\calC_b$, then
\[ \Vect^{\sm{ab}}\bigl( (R/A)_\Prism, \overline \calO \bigl[ \frac 1 p \bigr] \bigr) \subseteq \calC_b \subseteq \Vect^{\sm b}\bigl( (R/A)_\Prism, \overline \calO \bigl[ \frac 1 p \bigr] \bigr). \]
\end{theo}

\begin{proof}
Let $S/P$ by any completed polynomial ring equipped with a $\delta$-structure satisfying Condition \ref{cond: flat + tor 0}, such that there exists a surjective map of $P$-algebras $S \to R$. Denote $N = \widehat\dL_{R/\overline S}[-1]$, then $N$ is $p$-complete and $p$-completely flat over $R$.

Take $D = \Prism_{R/\overline S}$, then we have the topologically nilpotent differential $\dif_0: \overline D \to \overline D \cotimes_R N\{-1\}$ as in Proposition \ref{prop: factor-through connection}. Denote $\Omega = R \cotimes_{\overline S} \cOmega_{\overline S/\overline A}$, and $i: N \to \Omega, \pi: \Omega \to \cdL_{R/\overline A}$ the natural maps. Also denote $g: R \cotimes_{\overline P} \cOmega_{\overline P/\overline A} \to \cdL_{R/\overline A}$ the natural map. Note that after inverting $p$, $g$ becomes an isomorphism, and the domain of $g$ is $p$-torsion free, so $g$ is itself an injection. The natural map $P \to S$ induces the map $h: R \cotimes_{\overline P} \cOmega_{\overline P/\overline A} \to \Omega$, and we have the following commutative diagram
\[\xymatrix{
& & & R \cotimes_{\overline P} \cOmega_{\overline P/\overline A} \ar[d]^g \ar[ld]_h \\
0 \ar[r] & N \ar[r]_i & \Omega \ar[r]_\pi & \cdL_{R/\overline A} \ar[r] & 0
}\]
Since $S/P$ is a completed polynomial ring, we see that $h$ admits a left inverse. Invert the diagram above by $p$, and we get the following commutative diagram
\[\xymatrix{
& & & R \cotimes_{\overline P} \cOmega_{\overline P/\overline A} \bigl[ \frac 1 p \bigr] \ar[d]^\cong \ar[ld]_h \\
0 \ar[r] & N \bigl[ \frac 1 p \bigr] \ar[r]_i & \Omega \bigl[ \frac 1 p \bigr] \ar[r]_\pi & \cdL_{R/\overline A} \bigl[ \frac 1 p \bigr] \ar[r] & 0
}\]
Therefore, we may take a section $q: \Omega \to N \bigl[ \frac 1 p \bigr]$ via the formula $q(\omega) = i^{-1}(\omega - hg^{-1}p(\omega))$. Recall that $a(\coker g) =0$ by our condition. Therefore, for $\omega \in \Omega$, we have $a\pi(\omega) \in \im g$, and we see that $aq(\omega) \in N$. In other words, the section $q: \Omega \to N \bigl[ \frac 1 p \bigr]$ is in fact a section $q: \Omega \to a^{-1} N$. Definition of $\dif_0: \overline D \to \overline D \cotimes_R N \{-1\}$ shows that $\dif = i \circ \dif_0$, so $\dif_0 = q \circ \dif$.

Recall that $\overline D^a$ is the $a$-small component constructed in Proposition-Definition \ref{prop-defi: small component}. Denote $\dif'_0$ the composition $\overline D \to \overline D \cotimes_R N \to \overline D \cotimes_R a^{-1}N$. We first show that there exists a (2-)pullback diagram of categories
\begin{equation}\label{eq: Higgs by pullback}
\xymatrix{
\Higgs^\tn\bigl(R \bigl[ \frac 1 p \bigr], bR \otimes_{\overline P} \cOmega_{\overline P/\overline A} \{-1\}\bigr) \ar[r] \ar[d] & \MIC^\tn\bigl( \overline D^b \bigl[ \frac 1 p \bigr], b^{-1}\dif \bigr) \ar[d] \\
\Vect\bigl( R \bigl[ \frac 1 p \bigr] \bigr) \ar[r] & \MIC^\tn\bigl( \overline D^b \bigl[ \frac 1 p \bigr], b^{-1}\dif'_0 \bigr)
}
\end{equation}
where the arrow on the top is the $b$-small component of the morphism provided in Proposition \ref{prop: Higgs to MIC}, and the arrow on the right-hand side is induced by the map $q:\Omega \to a^{-1} N$. We have to prove the following two parts of statement:
\begin{enumerate}[label=(\arabic*)]
\item \emph{(fully faithful)} for any two topologically nilpotent Higgs fields $(M_1,\nabla_1), (M_2,\nabla_2) \in \Higgs^\tn\bigl(R \bigl[ \frac 1 p \bigr], bR \cotimes_{\overline P} \cOmega_{\overline P/\overline A} \{-1\}\bigr)$, a morphism $f: M_1 \to M_2$ is compatible with $\nabla$ if and only if its base change $M'_1 \to M'_2$ is compatible with the connection $\nabla'_i$ on $M'_i$ (where $M'_i = \overline D^b \otimes_R M_i$).
\item \emph{(essentially surjective)} for any finite projective module $M/ R \bigl[ \frac 1 p \bigr]$ with $M' = \overline D^b \otimes_R M$, if we are given a topologically nilpotent integrable $b^{-1}\dif$-connection (we require it to be a Higgs field over $R$)
\[ \nabla': M' \to M' \cotimes_R b\Omega\{-1\} \]
with $(\id \otimes q) \circ \nabla' = (b^{-1}\dif'_0) \otimes \id$, then there exists a topologically nilpotent Higgs field
\[ \nabla: M \to M \otimes_R (bR \cotimes_{\overline P} \cOmega_{\overline P/\overline A} \{-1\}) \]
inducing $\nabla'$.
\end{enumerate}

For (1), we only need to show that for any $(M,\nabla) \in \Higgs^\tn\bigl(R \bigl[ \frac 1 p \bigr], R \otimes_{\overline P} \cOmega_{\overline P/\overline A} \{-1\}\bigr)$ with $M' = \overline D^b \otimes_R M$, the map
\[ M \otimes_R (R \cotimes_{\overline P} \cOmega_{\overline P/\overline A}) \to M' \cotimes_R \Omega \]
is injective. Since $M \to M'$ is a closed embedding and $R \cotimes_{\overline P} \cOmega_{\overline P/\overline A}$ is $p$-complete and $p$-completely flat over $R$, by Lemma \ref{lem: closed submodule} and Corollary \ref{cor: tensor conc on degree 0} we see that the map
\[ M \otimes_R (R \cotimes_{\overline P} \cOmega_{\overline P/\overline A}) \to M' \otimes_R (R \cotimes_{\overline P} \cOmega_{\overline P/\overline A}) \]
is injective. To show that
\[ M' \cotimes_R (R \cotimes_{\overline P} \cOmega_{\overline P/\overline A}) \to M' \cotimes_R \Omega \]
is injective, one only need to observe that $h: R \cotimes_{\overline P} \cOmega_{\overline P/\overline A} \to \Omega$ has a left inverse.

For (2), we shall again resort to the injection above. We claim that the image of $M$ under the map $m \mapsto \nabla'(1 \otimes m)$ is contained in $M \otimes_R (bR \cotimes_{\overline P} \cOmega_{\overline P/\overline A}) \{-1\}$. Fix $m \in M$. Since $h$ is split injective, we have a split exact sequence
\[ 0 \to R \cotimes_{\overline P} \cOmega_{\overline P/\overline A} \to \Omega \to \coker h \to 0. \]
The map $q: (\coker h) \bigl[ \frac 1 p \bigr] \to N \bigl[ \frac 1 p \bigr]$ is an isomorphism, so we obtain a split exact sequence
\[ 0 \to M' \cotimes_R (bR \cotimes_{\overline P} \cOmega_{\overline P/\overline A}) \xrightarrow h M' \cotimes_R b\Omega \xrightarrow q M' \cotimes_R (ba^{-1}N) \to 0. \]
Therefore, by the condition $(\id \otimes q) \circ \nabla' = (b\dif'_0) \otimes \id$, we see that $\nabla'(1 \otimes m) \in M' \cotimes_R(bR \otimes_{\overline P} \cOmega_{\overline P/\overline A}) \{-1\}$. Also, by definition, the image of $\nabla'(1 \otimes m)$ under the composition
\begin{gather*}
M' \cotimes_R b\Omega\{-1\} \xrightarrow{\nabla' \otimes \id} M' \cotimes_R b\Omega \cotimes_R b\Omega \{-2\} \to M' \cotimes_R \cbigwedge_R^2(b\Omega) \{-2\} \\
\to M' \cotimes_R b\coker h \cotimes_R (bR \cotimes_{\overline P} \cOmega_{\overline P/\overline A}) \{-2\} \to M' \cotimes_R ba^{-1}N \cotimes_R (bR \cotimes_{\overline P} \cOmega_{\overline P/\overline A}) \{-2\}
\end{gather*}
is zero. In other words, the image of $\nabla'(1 \otimes m)$ under the map
\[ M' \cotimes_R(bR \cotimes_{\overline P} \cOmega_{\overline P/\overline A}) \{-1\} \xrightarrow{((\id \otimes q) \circ \nabla') \otimes \id} (M' \cotimes_R ba^{-1}N) \cotimes_R (bR \cotimes_{\overline P} \cOmega_{\overline P/\overline A}) \{-2\} \]
is zero. Note that by Proposition \ref{prop: factor-through connection} (4) applied to $M \cotimes_R (bR \cotimes_{\overline P} \cOmega_{\overline P/\overline A}) \{-1\}$, the kernel of the map above is $M \cotimes_R (bR \cotimes_{\overline P} \cOmega_{\overline P/\overline A}) \{-1\}$. This shows that the image of $M$ under the map $m \mapsto \nabla'(1 \otimes m)$ is contained in $M \cotimes_R (bR \cotimes_{\overline P} \cOmega_{\overline P/\overline A}) \{-1\}$. Since the map
\[ M \cotimes_R \cbigwedge_R^2(bR \cotimes_{\overline P} \cOmega_{\overline P/\overline A}) \to M \cotimes_R \cbigwedge_R^2(b\Omega) \]
has a left inverse, it is injective, and we see that the induced map $\nabla: M \to M \otimes_R (bR \cotimes_{\overline P} \cOmega_{\overline P/\overline A}) \{-1\}$ is a Higgs field. At last, because $h$ is split injective, we may take a section $q'$ of $h$, and the induced connection $(M', (\id \otimes q') \circ \nabla')$ is topologically nilpotent. The map $(M,\nabla) \to (M',(\id \otimes q') \circ \nabla')$ is a closed embedding, and a direct calculation (involving Lemma \ref{lem: closed submodule} (2)) shows that $\nabla$ is topologically nilpotent. Denote $\nabla'_1: M' \to M' \cotimes_R b\Omega \{-1\}$ the connection induced by $M$. Now, $\nabla'_1$ and $\nabla$ agree on $M$. By definition of connections, they must also agree on the $M'_1$, the $\overline D^b$-submodule of $M'$ generated by $M$. Since $M'_1$ is dense in $M'$, we conclude that $\nabla'_1=\nabla'$, and $\nabla'$ is induced by the topological nilpotent Higgs field $\nabla$. This proves our item (2).

Therefore, \eqref{eq: Higgs by pullback} is a 2-pullback diagram of categories.

We first prove that the functor
\[ \Higgs^\tn\bigl(R \bigl[ \frac 1 p \bigr] ,R \cotimes_{\overline P} \cOmega_{\overline P/\overline A} \{-1\}\bigr) \to \Vect\bigl( (R/A)_\Prism, \overline \calO \bigl[ \frac 1 p \bigr] \bigr) \]
is fully faithful. By \eqref{eq: Higgs by pullback}, we only need to prove that the map
\[ \Vect\bigl( R \bigl[ \frac 1 p \bigr] \bigr) \to \MIC^\tn\bigl( \overline D \bigl[ \frac 1 p \bigr], \dif'_0 \bigr) \]
is fully faithful. Fix finite projective modules $M_1,M_2$ over $R \bigl[ \frac 1 p \bigr]$. Denote $M'_i = \overline D \otimes_R M_i$ with connections $\nabla_i$. By Proposition \ref{prop: factor-through connection} (3), we see that $M_i = (M'_i)^{\nabla_i=0}$, therefore $\Hom(M_1,M_2) \to \Hom((M'_1,\nabla_1),(M'_2,\nabla_2))$ is injective. To show that it is surjective, for any map $\psi': (M'_1,\nabla_1) \to (M'_2,\nabla_2)$, it induces a map $\psi: M_1 \to M_2$ by taking $(-)^{\nabla=0}$. Now $M_i$ is a topological generating set of $M'_i$ as a $\overline D$-module, and $\psi'$ and $1\otimes \psi$ agrees on $M_1$, so they agree on the entire $M'_1$. In other words, $\psi'$ is induced by $\psi \in \Hom(M_1,M_2)$, proving that $\Hom(M_1,M_2) \to \Hom((M'_1,\nabla_1),(M'_2,\nabla_2))$ is surjective. This shows that the functor is fully faithful.

We now show that its essential image $\calC_b$ of the composition
\[ \Higgs^\tn\bigl(R \bigl[ \frac 1 p \bigr], bR \cotimes_{\overline P} \cOmega_{\overline P/\overline A} \{-1\}\bigr) \to \Higgs^\tn\bigl(R \bigl[ \frac 1 p \bigr], R \cotimes_{\overline P} \cOmega_{\overline P/\overline A} \{-1\}\bigr) \to \Vect\bigl( (R/A)_\Prism, \overline \calO \bigl[ \frac 1 p \bigr] \bigr) \]
satisfies
\[ \Vect^{\sm{ab}} \bigl( (R/A)_\Prism, \overline \calO \bigl[ \frac 1 p \bigr] \bigr) \subseteq \calC_b \subseteq \Vect^{\sm b} \bigl( (R/A)_\Prism, \overline \calO \bigl[ \frac 1 p \bigr] \bigr). \]
By the naturality for small components, this composition can also be given by
\[ \Higgs^\tn\bigl(R \bigl[ \frac 1 p \bigr], bR \cotimes_{\overline P} \cOmega_{\overline P/\overline A} \{-1\}\bigr) \to \MIC^\tn \bigl( \overline D^b \bigl[ \frac 1 p \bigr], b^{-1}\dif \bigr) \to  \MIC^\tn \bigl( \overline D \bigl[ \frac 1 p \bigr], \dif \bigr), \]
and it follows right away that $\calC_b \subseteq \Vect^{\sm b} \bigl( (R/A)_\Prism, \overline \calO \bigl[ \frac 1 p \bigr] \bigr)$. To show that
\[ \calC_b \supseteq \Vect^{\sm{ab}} \bigl( (R/A)_\Prism, \overline \calO \bigl[ \frac 1 p \bigr] \bigr), \]
we only need to show that the essential image of the map
\[ \MIC^\tn \bigl( \overline D^{ab} \bigl[ \frac 1 p \bigr], (ab)^{-1}\dif \bigr) \to \MIC^\tn \bigl( \overline D^b \bigl[ \frac 1 p \bigr], b^{-1}\dif \bigr) \]
is contained in the essential image of
\[ \Higgs^\tn\bigl(R \bigl[ \frac 1 p \bigr], bR \cotimes_{\overline P} \cOmega_{\overline P/\overline A} \{-1\}\bigr) \to \MIC^\tn \bigl( \overline D^b \bigl[ \frac 1 p \bigr], b^{-1}\dif \bigr). \]
By \eqref{eq: Higgs by pullback}, we only need to show that the essential image of the composition
\[ \MIC^\tn \bigl( \overline D^{ab} \bigl[ \frac 1 p \bigr], (ab)^{-1}\dif \bigr) \to \MIC^\tn \bigl( \overline D^b \bigl[ \frac 1 p \bigr], b^{-1}\dif \bigr) \to \MIC^\tn (\overline D^b \bigl[ \frac 1 p \bigr], b^{-1}\dif'_0) \]
is contained in the essential image of
\[ \Vect\bigl( R \bigl[ \frac 1 p \bigr] \bigr) \to \MIC^\tn\bigl( \overline D^b \bigl[ \frac 1 p \bigr], b^{-1} \dif'_0 \bigr). \]
It is clear that the essential image of the composition
\[ \MIC^\tn \bigl( \overline D^{ab} \bigl[ \frac 1 p \bigr], (ab)^{-1}\dif \bigr) \to \MIC^\tn \bigl( \overline D^b \bigl[ \frac 1 p \bigr], b^{-1}\dif \bigr) \to \MIC^\tn (\overline D^b \bigl[ \frac 1 p \bigr], b^{-1}\dif'_0) \]
is contained in the essential image of the composition
\[ \MIC^\tn \bigl( \overline D^{ab} \bigl[ \frac 1 p \bigr], (ab)^{-1}\dif'_0 \bigr) \to \MIC^\tn (\overline D^b \bigl[ \frac 1 p \bigr], b^{-1}\dif'_0). \]

By definition, the maps $(ab)^{-1} \dif'_0$ and $b^{-1} \dif'_0$ fit into the commutative diagram
\[\begin{tikzcd}[column sep=large]
\overline D^{ab} \ar[r, "(ab)^{-1}\dif'_0"] \ar[d] & \overline D^{ab} \cotimes_R bN \{-1\} \ar[d] \\
\overline D^b \ar[r, "b^{-1}\dif_0"] \ar[d] & \overline D^b \cotimes_R bN \{-1\} \ar[d] \\
\overline D^b \ar[r, "b^{-1}\dif'_0"] & \overline D^b \cotimes_R ba^{-1}N \{-1\}
\end{tikzcd}\]
and the square on the outside defines the map
\[ \MIC^\tn \bigl( \overline D^{ab} \bigl[ \frac 1 p \bigr], (ab)^{-1}\dif'_0 \bigr) \to \MIC^\tn (\overline D^b \bigl[ \frac 1 p \bigr], b^{-1}\dif'_0) \]
Therefore, this map factors as
\[\MIC^\tn (\overline D^{ab} \bigl[ \frac 1 p \bigr], (ab)^{-1} \dif'_0) \to \MIC^\tn (\overline D^b \bigl[ \frac 1 p \bigr], b^{-1}\dif_0) \to \MIC^\tn (\overline D^b \bigl[ \frac 1 p \bigr], b^{-1}\dif'_0). \]
By Proposition \ref{prop: factor-through connection} (3), we see that the map
\[ \Vect(R) \to \MIC^\tn(\overline D^b \bigl[ \frac 1 p \bigr], b^{-1}\dif_0) \]
is an equivalence. Therefore the map factors through the essential image of $\Vect(R) \to \MIC^\tn (\overline D^b \bigl[ \frac 1 p \bigr], b^{-1}\dif'_0)$, as desired.
\end{proof}

Note that under the assumptions of Theorem \ref{theo: simpson correspondence} the image of the fully faithful functor
\[ \Higgs^\tn\bigl(R \bigl[ \frac 1 p \bigr] ,R \otimes_{\overline P} \cOmega_{\overline P/\overline A} \{-1\}\bigr) \to \Vect\bigl( (R/A)_\Prism, \overline \calO \bigl[ \frac 1 p \bigr] \bigr) \]
is an \'etale local full subcategory of $\Vect\bigl( (R/A)_\Prism, \overline \calO \bigl[ \frac 1 p \bigr] \bigr)$, so we have the following direct consequence (note that fully-faithfulness can be checked locally):

\begin{cor}\label{cor: simpson correspondence etale}
Let $(A,I)$ be a bounded prism and $R/\overline A$ be a ring satisfying Condition \ref{cond: tor -1,0}, such that $R$ is $p$-torsion-free, $\cdL_{R/\overline A}$ is concentrated on degree 0 and $p$-torsion-free. Let $P/A$ be a $\delta$-algebra satisfying Condition \ref{cond: flat + tor 0}. Fix a Cartier divisor $J \subseteq R$ of $R$ such that $(R/J) \bigl[ \frac 1 p \bigr] =0$. Assume that we are given a map of $A$-algebras $P \to R$ such that the induced map $R \otimes_{\overline P}\cOmega_{\overline P/\overline A} \to \cdL_{R/\overline A}$ is an isomorphism after inverting $p$, and its cokernel is annihilated by $J$. Then, the functor
\[ \Higgs^\tn\bigl(R \bigl[ \frac 1 p \bigr], R \cotimes_{\overline P} \cOmega_{\overline P/\overline A} \{-1\}\bigr) \to \Vect\bigl( (R/A)_\Prism, \overline \calO \bigl[ \frac 1 p \bigr] \bigr) \]
given in Proposition \ref{prop: Higgs to crystal} is fully faithful, and for any Cartier divisor $J' \subseteq R$ with $(R/J') \bigl[ \frac 1 p \bigr] =0$, the essential image $\calC_{J'}$ of the composition
\[ \Higgs^\tn\bigl(R \bigl[ \frac 1 p \bigr], J' \cotimes_{\overline P} \cOmega_{\overline P/\overline A} \{-1\}\bigr) \to \Higgs^\tn\bigl(R \bigl[ \frac 1 p \bigr], R \cotimes_{\overline P} \cOmega_{\overline P/\overline A} \{-1\}\bigr) \to \Vect\bigl( (R/A)_\Prism, \overline \calO \bigl[ \frac 1 p \bigr] \bigr) \]
satisfies
\[ \Vect^{\smet{JJ'}}\bigl( (R/A)_\Prism, \overline \calO \bigl[ \frac 1 p \bigr] \bigr) \subseteq \calC_{J'} \subseteq \Vect^{\sm{J'}}\bigl( (R/A)_\Prism, \overline \calO \bigl[ \frac 1 p \bigr] \bigr). \]
\end{cor}

\begin{proof}
Pick any $\delta$-ring $S/P$ satisfying Condition \ref{cond: flat + tor 0}, and take a surjection $S \to R$. Denote $D = \Prism_{R/S}$, and the inclusion
\[ \calC_{J'} \subseteq \Vect^{\sm{J'}}\bigl( (R/A)_\Prism, \overline \calO \bigl[ \frac 1 p \bigr] \bigr) \]
follows directly from the commutative diagram as below:
\[\begin{tikzcd}
\Higgs^\tn\bigl(R \bigl[ \frac 1 p \bigr], J' \cotimes_{\overline P} \cOmega_{\overline P/\overline A} \{-1\}\bigr) \ar[r] \ar[d] & \Higgs^\tn\bigl(R \bigl[ \frac 1 p \bigr], R \cotimes_{\overline P} \cOmega_{\overline P/\overline A} \{-1\}\bigr) \ar[d] \\
\MIC^\tn\bigl( \overline D^{J'} \bigl[ \frac 1 p \bigr], (J')^{-1}\dif \bigr) \ar[r] & \MIC^\tn \bigl( \overline D \bigl[ \frac 1 p \bigr], \dif \bigr).
\end{tikzcd}\]
As for the other side of the inclusion, one only need to observe that $\calC_{J'} \cong \Higgs^\tn\bigl(R \bigl[ \frac 1 p \bigr], J' \cotimes_{\overline P} \cOmega_{\overline P/\overline A} \{-1\}\bigr)$ is a stack in the $p$-completely \'etale topology (by Lemma \ref{lem: Higgs is fpqc stack}), and apply Theorem \ref{theo: simpson correspondence} for $p$-completely \'etale covers of $R$. (Here \'etaleness is necessary, because the condition of Theorem \ref{theo: simpson correspondence} requires calculation of the completed cotangent complex.)
\end{proof}

Recall our calculations in Proposition \ref{prop: Higgs to perfect crystals}. We get the following corollary.

\begin{cor}\label{cor: fully faithfulness on small crystals}
Let $\overline A/\mathbb Z_p^\cyc$ be an integral perfectoid algebra and put $A = \dA_{\inf}(\overline A)$ and $I = \ker \theta$. Let $R/\overline A$ be a ring satisfying Condition \ref{cond: tor -1,0}, such that $R$ is $p$-torsion-free, $\cdL_{R/\overline A}$ is concentrated on degree 0 and $p$-torsion-free. Also, let $x_1,\ldots,x_d$ be elements in $R^\times$ and $J \subseteq R$ be a Cartier divisor with $(R/J) \bigl[ \frac 1 p \bigr] =0$, such that $\dif x_i$ form a basis of $\cdL_{R/\overline A}$, and the cokernel of
\[ \bigoplus_{i=1}^d R \dif x_i \to \cdL_{R/\overline A} \]
is annihilated by $J$. Also, assume that the map
\[ A \langle T_1,\ldots,T_d \rangle \to R \]
sending $T_i$ to $x_i$ is a toric chart of $R$ (see Definition \ref{defi: toric chart}), then the composition
\[ \Vect^{\smet J}\bigl( (R/A)_\Prism, \overline \calO \bigl[ \frac 1 p \bigr] \bigr) \to \Vect\bigl( (R/A)_\Prism, \overline \calO \bigl[ \frac 1 p \bigr] \bigr) \xrightarrow{\Res_{\Spf(R)}} \Vect\bigl( (R/A)_\Prism^\perf, \overline \calO \bigl[ \frac 1 p \bigr] \bigr) \]
is fully faithful.

Moreover, denote $\cdL_{R/\overline A}^\vee$ the classical $R$-module $\Hom_R(\cdL_{R/\overline A},R)$.\footnote{We have $\cdL_{R/\overline A}^\vee \bigl[ \frac 1 p \bigr] \cong \cdL_{R/\overline A} \bigl[ \frac 1 p \bigr]^\vee$, because every $R$-linear map $\cdL_{R/\overline A} \bigl[ \frac 1 p \bigr] \to R \bigl[ \frac 1 p \bigr]$ is bounded by Proposition \ref{prop: finitely generated adic modules}.} For any Cartier divisor $J' \subseteq R$ with $(R/J') \bigl[ \frac 1 p \bigr]=0$, denote
\[ \Vect^{\sm{J'}} \bigl( (R/A)_\Prism^\perf, \overline \calO \bigl[ \frac 1 p \bigr] \bigr) \]
the category of all vector bundles $\calE \in \Vect\bigl( (R/A)_\Prism^\perf, \overline \calO \bigl[ \frac 1 p \bigr] \bigr)$ for which the Sen action of $(J')^{-1} (\zeta_p-1)^{-1}\cdL_{R/\overline A}^\vee \{1\}$ is topologically nilpotent\footnote{We remark that topological nilpotency can be checked $v$-locally, because for a $v$-vector bundle $\calE$ and a locally perfectoid ring $(S,S^+)$, we may take a finite open cover of $\Spa(S,S^+)$ by affinoid perfectoid opens, and then further cover these opens by their $v$-covers. Therefore, we get arbitrarily fine $v$-covers $\bigsqcup_{i=1}^N \Spa(S_i,S_i^+) \to \Spa(S,S^+)$, and we see that $\calE(S) \to \prod \calE(S_i)$ is a closed embedding. Therefore, if an operator $T: \calE \to \calE$ is topologically nilpotent on all $\calE(S_i)$, then it is topologically nilpotent on $\calE(S)$.} \emph{(do not forget the $(\zeta_p-1)$-factor)}. Then, there exists a full subcategory $\calC \subseteq \Vect\bigl( (R/A)_\Prism, \overline \calO \bigl[ \frac 1 p \bigr] \bigr)$ such that for any Cartier divisor $J'$ of $R$ with $(R/J') \bigl[ \frac 1 p \bigr]=0$, we have
\begin{align*}
\Res_{\Spf(R)} \bigl( \calC \cap\Vect^{\sm{J'}} \bigl( (R/A)_\Prism, \overline \calO \bigl[ \frac 1 p \bigr] \bigr) \bigr) & \supseteq \Vect^{\sm{JJ'}} \bigl( (R/A)_\Prism^\perf, \overline \calO \bigl[ \frac 1 p \bigr] \bigr), \\
\Res_{\Spf(R)} \bigl( \Vect^{\smet{JJ'}} \bigl( (R/A)_\Prism, \overline \calO \bigl[ \frac 1 p \bigr] \bigr) \bigr) & \subseteq \Vect^{\sm{J'}} \bigl( (R/A)_\Prism^\perf, \overline \calO \bigl[ \frac 1 p \bigr] \bigr).
\end{align*}
\end{cor}

\begin{proof}
By Theorem \ref{theo: v vector bundles=generalized representations}, evaluating on $R_\infty$ gives an equivalence of categories
\[ \Vect\bigl( (R/A)_\Prism^\perf, \overline \calO \bigl[ \frac 1 p \bigr] \bigr) \to \Rep_\Gamma \bigl(R_\infty \bigl[ \frac 1 p \bigr] \bigr). \]
For the fully faithful part, put $P = A \langle T_1,\ldots,T_d \rangle$. By Corollary \ref{cor: simpson correspondence etale}, we only need to prove that the composition
\begin{align*}
& \Higgs^\tn\bigl( R \bigl[ \frac 1 p \bigr], (R \dif x_1 \oplus \cdots \oplus R\dif x_d)\{-1\} \bigr) \to \Vect\bigl( (R/A)_\Prism, \overline \calO \bigl[ \frac 1 p \bigr] \bigr) \\
\to {} & \Vect\bigl( (R/A)_\Prism^\perf, \overline \calO \bigl[ \frac 1 p \bigr] \bigr) \to \Rep_\Gamma \bigl(R_\infty \bigl[ \frac 1 p \bigr] \bigr)
\end{align*}
is fully faithful.

By Proposition \ref{prop: Higgs to perfect crystals}, this functor sends $(M,\nabla)$ to $M_\infty =R_\infty \otimes_R M$ where $\gamma_i$ acts on $M$ by $\exp((\zeta_p-1) x_i \nabla_i)$ and
\[ \nabla(m) = \sum_{i=1}^d \nabla_i(m) \otimes \frac{\dif x_i} \xi, \quad \xi = \frac{[\ep]-1}{[\ep]^{1/p}-1}. \]
Now, $M$ may be identified with the $\Gamma$-analytic vectors of $M_\infty$, and the operators $\nabla_i$ may be identified with $(\zeta_p-1)^{-1} x_i^{-1} \theta_i$ (note that $(\zeta_p-1) x_i$ is invertible in $R \bigl[ \frac 1 p \bigr]$) for the Sen operators $\theta_i$ (see Theorem \ref{theo: full decompletion} and the definitions before Lemma \ref{lemma: properties of sen actions}). This easily shows that the corresponding functor
\[ \Higgs^\tn\bigl( R \bigl[ \frac 1 p \bigr], R \dif x_1 \oplus \cdots \oplus R\dif x_d \bigr) \to \Rep_\Gamma \bigl(R_\infty \bigl[ \frac 1 p \bigr] \bigr) \]
is fully faithful.

For the essential image part, we take $\calC$ to be the essential image of
\[ \Higgs^\tn\bigl( R \bigl[ \frac 1 p \bigr], (R \dif x_1 \oplus \cdots \oplus R\dif x_d)\{-1\} \bigr) \to \Vect\bigl( (R/A)_\Prism, \overline \calO \bigl[ \frac 1 p \bigr] \bigr), \]
and by the discussions above we see that $\Res_{\Spf(R)}$ is fully faithful on $\calC$. By definition, one sees that the dual basis $\partial_i$ of $\dif x_i$ satisfies
\[ \cdL_{R/\overline A}^\vee \subseteq \langle \partial_1,\ldots,\partial_n \rangle \subseteq J^{-1}\cdL_{R/\overline A}^\vee, \]
and $x_i$ are invertible in $R$.

In fact, $\Res_{\Spf(R)} \bigl( \calC \cap \Vect^{\sm{J'}} \bigl( (R/A)_\Prism, \overline \calO \bigl[ \frac 1 p \bigr] \bigr) \bigr)$ contains the essential image of
\[ \Higgs^\tn\bigl( R \bigl[ \frac 1 p \bigr], (J' \dif x_1 \oplus \cdots \oplus J' \dif x_d)\{-1\} \bigr) \to \Vect\bigl( (R/A)_\Prism^\perf, \overline \calO \bigl[ \frac 1 p \bigr] \bigr), \]
because the map
\[ \Higgs^\tn\bigl( R \bigl[ \frac 1 p \bigr], (J' \dif x_1 \oplus \cdots \oplus J' \dif x_d)\{-1\} \bigr) \to \Vect\bigl( (R/A)_\Prism, \overline \calO \bigl[ \frac 1 p \bigr] \bigr) \]
factors through
\[ \calC \cap \Vect^{\sm{J'}} \bigl( (R/A)_\Prism, \overline \calO \bigl[ \frac 1 p \bigr] \bigr) \]
by Corollary \ref{cor: simpson correspondence etale}. A direct calculations of $\theta_i$ shows that this essential image contains
\[ \Vect^{\sm{JJ'}}\bigl( (R/A)_\Prism^\perf, \overline \calO \bigl[ \frac 1 p \bigr] \bigr). \]
Also, by Corollary \ref{cor: simpson correspondence etale} again, the essential image of $\Res_{\Spf(R)} \bigl( \Vect^{\smet{JJ'}} \bigl( (R/A)_\Prism, \overline \calO \bigl[ \frac 1 p \bigr] \bigr) \bigr)$ is contained in the essential image of
\[ \Higgs^\tn\bigl( R \bigl[ \frac 1 p \bigr], (J' \dif x_1 \oplus \cdots \oplus J' \dif x_d)\{-1\} \bigr) \to \Vect\bigl( (R/A)_\Prism^\perf, \overline \calO \bigl[ \frac 1 p \bigr] \bigr), \]
and a direct calculation of $\theta_i$ shows that this essential image is contained in
\[ \Vect^{\sm{J'}}\bigl( (R/A)_\Prism^\perf, \overline \calO \bigl[ \frac 1 p \bigr] \bigr). \qedhere \]
\end{proof}

\begin{cor}\label{cor: HT comparison revisited}
Let $(A,I)$ be a bounded prism and $R/\overline A$ be a ring satisfying Condition \ref{cond: tor -1,0}. Assume that $R$ is $p$-torsion-free, $\cdL_{R/\overline A}$ is concentrated on degree 0, $p$-torsion-free, and that $\cdL_{R/\overline A} \bigl[ \frac 1 p \bigr]$ is finite projective over $R \bigl[ \frac 1 p \bigr]$. Then,
\begin{enumerate}[label=(\arabic*)]
\item for any $M \in \widehat D(R)$ and integer $n \ge 0$, the map
\[ M \otimes_R^\bL \bigwedge\nolimits_R^n(\cdL_{R/\overline A}) \to M \cotimes_R^\bL \cbigwedge_R^n(\cdL_{R/\overline A}) \]
is an isomorphism; moreover, if $M$ is concentrated on degree 0 of bounded $p^\infty$-torsion, then there exists $k \ge 0$ independent of $n$ and $M$, such that
\[ p^{nk}\hol^i \bigl( M \otimes_R^\bL \bigwedge\nolimits_R^n(\cdL_{R/\overline A}) \bigr) =0, \forall i \ne 0 \]
and, for $M[p^\infty] = M[p^{n_0}]$,
\[ \hol^0 \bigl( M \otimes_R^\bL \bigwedge\nolimits_R^n(\cdL_{R/\overline A}) \bigr) [p^\infty] = \hol^0 \bigl( M \otimes_R^\bL \bigwedge\nolimits_R^n(\cdL_{R/\overline A}) \bigr) [p^{nk+n_0}]. \]

\item the colimit in $D(R)$
\[ \colim_n \Fil_n(\overline\Prism_{R/A}) \]
is already derived $p$-adically complete; therefore, the Hodge--Tate filtration on $\overline\Prism_{R/A}$ is exhaustive in $D(R)$;

\item the Hodge--Tate filtration on the $\dE_\infty$-ring $\overline \Prism_{R/A} \bigl[ \frac 1 p \bigr]$ induces an isomorphism of cohomology rings
\[ \hol^*\bigl( \overline \Prism_{R/A} \bigl[ \frac 1 p \bigr] \bigr) \cong \bigwedge\nolimits^*_{R[\frac 1 p]} \bigl(\cdL_{R/\overline A} \bigl[ \frac 1 p \bigr] \bigr)\]
\end{enumerate}
\end{cor}

\begin{proof}
Note that (3) is a direct consequence of (1) and (2), so we only need to prove (1) and (2).

For (1), denote $N = \cdL_{R/\overline A}$. Let $F$ be a finite free module over $R$ such that we have a surjection $F \bigl[ \frac 1 p \bigr] \to N \bigl[ \frac 1 p \bigr]$ of adic modules. Multiply by a sufficient power of $p$ and take a lifting, we may assume that the surjection comes from a map $q: F \to N$ of $R$-modules. Now, since $N \bigl[ \frac 1 p \bigr]$ is finite projective over $R \bigl[ \frac 1 p \bigr]$, we may take a section $s'$ of the surjection above. The map $s'$ is continuous by Proposition \ref{prop: finitely generated adic modules}, so $s'$ is bounded, and we see that $s'(N)$ is contained in some $p^{-k} F \subseteq F \bigl[ \frac 1 p \bigr]$ for $k \ge 0$ (here we use that $R$ is $p$-torsion-free). Take $s = p^k s': N \to F$, and we see that $qs = p^k$. Denote the functor
\[ G(X) = M \otimes_R^\bL \bigwedge\nolimits_R^n(X), \]
and we may apply $G$ to get $G(q) G(s) \cong G(p^k) \cong p^{nk}$ (the last step is a routine argument of left Kan extensions).

To show that $G(N)$ is derived $p$-adically complete, we may apply the functor $T_p(-) = \RHom_{\mathbb Z}(\mathbb Z[p^{-1}], -)$ (see Lemma \ref{lem: t-structure on complete modules}), and we only need to show that $T_p(G(N)) =0$. Now that $G(F)$ is a direct sum of a finite number of copies of $M$, so $G(F)$ is derived $p$-adically complete, and $T_p(G(F)) =0$. Note that the composition $p^{nk}: T_p(G(N)) \to T_p(G(F)) \to T_p(G(N))$ is invertible (since these are modules over $\mathbb Z[p^{-1}]$). This shows that $T_p(G(N)) =0$.

Now assume that $M$ is concentrated on degree 0 of bounded $p^\infty$-torsion. Observe that
\[ G(F) = M \otimes_R^\bL \bigwedge\nolimits_R^n(F) \]
is a direct sum of a finite number of copies of $M$, therefore $G(F)$ is concentrated on degree 0 and with $G(F)[p^\infty] = G(F)[p^{n_0}]$. Therefore, using the equality $G(q)G(s) = p^{nk}$, we see that for $i \ne 0$ we see that $p^{nk} = 0$ on $\hol^i(G(N))$. For $x \in \hol^0(G(M))[p^\infty]$, we see that $p^{n_0} G(s)(x) =0$, so $p^{nk+n_0} x = G(q) (p^{n_0} G(s) x) =0$.

For (2), pick a $\delta$-algebra $S/A$ satisfying Condition \ref{cond: flat + tor 0} and a map $S \to R$ of $A$-algebras such that $R/\overline S$ satisfies Condition \ref{cond: tor -1}. Denote
\[ \overline D = \overline \Prism_{R/S}, N = \cdL_{R/\overline S}[-1], \Omega = R \cotimes_{\overline S} \cdL_{\overline S/\overline A}, \]
and the differentials $\dif: \overline D \to \overline D \cotimes_R \Omega\{-1\}, \dif_0: \overline D \to \overline D \cotimes_R N\{-1\}$ with $\dif_0$ defined in Proposition \ref{prop: factor-through connection}.

We introduce the following notations: denote $\Comp^{\ge 0} (\Mod_{R,\Flat,\fil})$ the additive $1$-category of complexes $X^\bullet$ over $\Mod_{R,\Flat,\fil}$ with $X^n=0$ for $n<0$. For $X^\bullet \in \Comp^{\ge 0} (\Mod_{R,\Flat,\fil})$, denote $T(X^\bullet)$ the complex with terms
\[ T(X^\bullet)^n = X^n / \bigcup_{k \ge 0} \Fil_k(X^n), \]
and for $a \in R \cap R\bigl[ \frac 1 p \bigr]^\times$, define $S_a(X^\bullet)$ as $S_a(X^\bullet)^n = X^n$, but with differential $a \dif_X$. We always have the natural map $X^\bullet \to S_a(X^\bullet)$ given by multiplication by $a^n$ on degree $n$. Recall that in $D(R)$,
\[ T(X^\bullet)^n \cong \cofib( M \to \widehat M ) = T_p(M)[-1], \]
so $T(X^\bullet)^n$ are modules over $R \bigl[ \frac 1 p \bigr]$. Therefore we see that the natural map $X^\bullet \to S_a(X^\bullet)$ induces an isomorphism on $T(X^\bullet)^n \to T(S_a(X^\bullet))^n$, hence an isomorphism of complexes $T(X^\bullet) \to T(S_a(X^\bullet))$.

We use $U(X^\bullet) \in \Comp^{\ge 0}(R)$ to denote $X^\bullet$ with the structure of filtration forgotten, and $L(X^\bullet) \in \Comp^{\ge 0}(R)$ to be the complex with terms
\[ L(X^\bullet)^n = \bigcup_{k \ge 0} \Fil_k(X^n), \]
then $T(X^\bullet) = \cofib(L(X^\bullet) \to U(X^\bullet))$. If $Y^\bullet \in \Comp^{\ge 0}(R)$ such that the filtration on each $Y^i$ is finite, and the graded pieces are finite projective over $R$, then it is easy to see that for any $X^\bullet \in \Comp^{\ge 0} (\Mod_{R,\Flat,\fil})$, we have natural isomorphisms
\[ U(X^\bullet) \otimes_R U(Y^\bullet) \cong U(X^\bullet \otimes_R Y^\bullet), \quad L(X^\bullet) \otimes_R L(Y^\bullet) \cong L(X^\bullet \otimes_R Y^\bullet), \]
and the map $L(Y^\bullet) \to U(Y^\bullet)$ is an isomorphism of complexes. Therefore,
\[ T(X^\bullet \otimes_R Y^\bullet) \cong T(X^\bullet) \otimes_R L(Y^\bullet). \]
If furthermore $L(Y^\bullet)$ is bounded, then in $D(R)$ we have
\[ T(X^\bullet \otimes_R Y^\bullet) \cong T(X^\bullet) \otimes_R^\bL L(Y^\bullet). \]

In $D(R)$, $T(X^\bullet)$ is the cofibre of the map
\[ \colim \Fil_k(X^\bullet) \to \ccolim \Fil_k(X^\bullet). \]
We wish to show that $T(\DR(\overline D, \dif)) =0$. For this, consider the following exact sequence of $p$-torsion-free $p$-complete $R$-modules
\[ 0 \to N \xrightarrow i \Omega \xrightarrow \pi \cdL_{R/\overline A} \to 0. \]
After inverting $p$, we get a split exact sequence of adic modules (since $\cdL_{R/\overline A} \bigl[ \frac 1 p \bigr]$ is finite projective over $R\bigl[ \frac 1 p \bigr]$). Therefore, the same as above, we may pick an integer $\ell \ge 0$ and maps $s': \cdL_{R/\overline A} \to \Omega$ and $q': \Omega \to N$ with
\[ q's'=0, \pi s' = q'i = iq'+s'\pi = p^\ell. \]
Now take $f: \Omega \to N \oplus F$ and $g: N \oplus F \to \Omega $(recall that $F$ is the finite free module we picked above for $\cdL_{R/\overline A}$, and we have maps $s: \cdL_{R/\overline A} \to F$ and $q: F \to \cdL_{R/\overline A}$ with $qs = p^k$) by 
\[ f= (q',s\pi), \quad g = (p^ki, s'q), \]
and we get $gf = p^k iq' + s'(qs)\pi = p^k(iq'+s'\pi) = p^{k+\ell}$. Denote $a= p^{k+\ell}$, and we get a sequence in $\Comp^{\ge 0} (\Mod_{R,\Flat,\fil})$ as follows:
\[ \DR(\overline D, \dif) \xrightarrow{f_*} \DR(\overline D, f\dif) \xrightarrow{g_*} \DR(\overline D, a\dif). \]
Under the identification $\DR(\overline D, a\dif) \cong S_a(\DR(\overline D,\dif))$, the map $g_*f_*$ is identified with the natural map $\id \to S_a$, hence inducing an isomorphism
\[ g_* f_*: T(\DR(\overline D,\dif) \to \DR(\overline D, a\dif)). \]
Therefore, $T(\DR(\overline D,\dif))$ is a retract of $T(\DR(\overline D, f\dif))$, and we only need to show that $T(\DR(\overline D, f\dif))$ is acyclic. Now, $f\dif: \overline D \to \overline D \cotimes_R (N \oplus F)$ is given by $(q' i \dif_0, s\pi i \dif_0) = (p^k \dif_0, 0)$. Therefore we get a tensor product
\[ \DR(\overline D, f\dif) \cong \DR(\overline D, p^k \dif_0) \otimes_R \DR(R, 0: R \to F) \]
(where $\DR(R, 0: R \to F)$ is a bounded complex all of whose terms are equipped with finite filtrations with terms finite free over $R$). By the discussions of the last paragraph, we see that in $D(R)$
\[ T(\DR(\overline D, p^k \dif_0) \cotimes_R \DR(R, 0: R \to F)) \cong T(\DR(\overline D, p^k\dif_0)) \otimes_R^\bL L(\DR(R, 0: R \to F)). \]
We have $T(\DR(\overline D, p^k\dif_0)) = T(S_{p^k} (\DR(\overline D, p^k\dif_0))) \cong T(\DR(\overline D,\dif_0))$. By Proposition \ref{prop: factor-through connection} (4)), the complexes $\DR(\overline D,\dif_0) / \Fil_m(\DR(\overline D,\dif_0))$ are exact for all $m \ge 0$. Take the filtered colimits, one sees that $T(\DR(\overline D,\dif_0))$ is acyclic, so $T(\DR(\overline D, f\dif))$ is acyclic. Therefore, $T(\DR(\overline D, \dif))$ as a retract of $T(\DR(\overline D, f\dif))$ is acyclic, as desired.
\end{proof}

\section{Prismatic non-abelian Hodge theory}

The goal of this section is to prove Theorem \ref{intro: prism non abelian hodge of geo val rings} and Theorem \ref{intro: main theorem p adic Simpson}.

\subsection{Prismatic non-abelian Hodge theory of geometric valuation rings}

%\subsection{Geometric valuation rings and its Sen theory}

The valuation rings plays a central role in the $p$-adic geometry (or even all $p$-adic mathematics). We will use our theory to understand rational Hodge--Tate prismatic crystals of valuation rings which `comes from geometry'. These special valuation rings were introduced and studied by Tongmu He in \cite{He_2025} and \cite{he2025padicgaloiscohomologyvaluation}.

In the following, we assume that all valuation rings are of characteristic $0$ (i.e. $p$ is non-zero). We begin with the following theorem taking from \cite[Theorem 3.8]{Bouis_2023}.

\begin{theo}\label{theo: valuation rings are cartier smooth}
    Let ${C}^+$ be an integral perfectoid ring such that the underlying ring is a valuation ring, and let
    $f:C^+\to E^+$ be an extension of valuation rings. Then $f$ is $p$-Cartier smooth in the sense of \cite[Definition 2.5 (2)]{Bouis_2023}.
\end{theo}

By the definition of $p$-cotangent smoothness, we immediately get the following corollary:

\begin{cor}
    Keep the notation in Theorem \ref{theo: valuation rings are cartier smooth}, $f$ is $p$-cotangent smooth, i.e. the cotangent complex
    \[\cdL_{E^+/C^+}\]
    is $p$-completely flat over $E^+$. In addition, if $E^+$ is $p$-adically complete, $f$ satisfies Condition \ref{cond: flat + tor 0}.
\end{cor}

Now we fix a complete algebraic closed $p$-adic non-archimedean field $\mathbf{C}$. Define $\calO_{\bfC}=\mathbf C^\circ$ as its valuation ring with maximal ideal $\mathfrak m = \mathbf C^{\circ\circ}$. Our definition of geometric valuation rings is those $p$-complete valuation rings which have almost finitely generated complete differential modules. We first recall some basic commutative algebra of valuation rings.

\begin{theo}\label{theo: almost finitely generated}
    Let $E^+$ be a valuation ring which is $p$-complete. Then:
    \begin{enumerate}[label=(\arabic*)]
        \item $E:=E^+[\frac{1}{p}]$ is a field (and is equipped with a natural structure of a Tate ring);
        \item Any $p$-complete and $p$-torsion free $E^+$-module $M$ such that
\[ \dim_{E} M \bigl[ \frac 1 p \bigr] < +\infty\]
is almost finitely generated with respect to $\sqrt{pE^+} = \bigcup_{n >0} p^{1/n} E^+$. Moreover, for any set $X$ of topological generators of $M$ and $n \in \mathbb Z_+$, we may find $x_1,\ldots,x_d \in X$ such that $M/\langle x_1,\ldots,x_d\rangle$ is annihilated by $p^{1/n}$, and that $x_1,\ldots,x_d$ form a basis of $M \bigl[ \frac 1 p \bigr]$ over $E$.
    \end{enumerate}
\end{theo}

\begin{proof}
    For (1), redefine $E$ to be the fractional field of $E^+$. We only need to prove that for any $f\in E$, $p^Nf\in E^+$ for some $N\geq 0$. We may assume that $f\not\in E^+$, or we can choose $N=0$. Thus, $f^{-1}\in E^+$. Since $\bigcap_{n\geq 1}p^nE^+=0$, there exists an $N\geq 1$ such that $f^{-1}\not \in p^NE^+$. This implies $p^N\in f^{-1}E^+$ as $E^+$ is a valuation ring.

    For (2), Let $M/E^+$ be a $p$-complete and $p$-torsion free module such that
\[\dim_{E}M \bigl[ \frac 1 p \bigr]<+\infty.\]
    
    First assume that $\dim_{E} M \bigl[ \frac 1 p \bigr]=1$, and we may identify $M$ with a submodule of $E$. By Proposition \ref{prop: finitely generated adic modules}, we see that the map $M \bigl[ \frac 1 p \bigr] \to E$ is a homeomorphism, so $M \subseteq E$ is open and bounded. Fix $n \in \mathbb Z_+$. Take the greatest integer $m$ such that $M \subseteq p^{m/n} E^+$. Such integer exists because $M$ is open and bounded. Therefore, since $X$ is a set of topological generators of $M$, we see that $X \nsubseteq p^{(m+1)/n} E^+$. Take $x \in X$ with $x \in X \setminus p^{(m+1)/n} E^+$, and we see that $p^{(m+1)/n} \in xE^+$, and we get $M \subseteq p^{-1/n} x E^+$. Taking $x_1=x$ would suffice.

     For general $M$, denote
\[ d= \dim_{E}M \bigl[ \frac 1 p \bigr]. \]
Take $N = \bigwedge_{E}^d \bigl( M \bigl[ \frac 1 p \bigr] \bigr)$, and $N_0$ the submodule of $N$ generated by all elements of form $x_1 \wedge \cdots \wedge x_d$ for $x_i \in X$. By Proposition \ref{prop: finitely generated adic modules}, the finite-dimensional vector space $\bigl( M \bigl[ \frac 1 p \bigr] \bigr)^{\otimes n/E}$ agrees with the completed tensor product (by the universal property), and the map
\[ \bigl( M \bigl[ \frac 1 p \bigr] \bigr)^{\cotimes n/E} \cong \bigl( M \bigl[ \frac 1 p \bigr] \bigr)^{\otimes n/E} \to N \]
is bounded. It follows that $N_0$ is bounded in $N$. Also, since $N$ is one-dimensional, it is clear that $N_0$ is open in $N$, so $N_0$ is a module of definition for $N$.

By the last paragraph applied to $N_0$, for any $n \in \mathbb Z_+$, there exists $x_1,\ldots,x_d \in X$ such that $N_0 / E^+ (x_1 \wedge \cdots \wedge x_d)$ is annihilated by $p^{1/n}$. In particular, $x_1 \wedge \cdots \wedge x_d$ is nonzero in $N$, so they form a basis of $M \bigl[ \frac 1 p \bigr]$.

For any $x \in M$, write
\[ x = \sum_{i=1}^d \lambda_i x_i, \]
We see that in $N$,
\[ N_0 \ni x_1 \wedge \cdots \wedge x_{i-1} \wedge x \wedge x_{i+1} \wedge \cdots \wedge x_d = \lambda_i x_1 \wedge \cdots \wedge x_d. \]
By our assumptions on $x_1 \wedge \cdots \wedge x_d$, we see that all $\lambda_i \in p^{-1/n} E^+$, so $M / \langle x_1,\ldots,x_d \rangle$ is annihilated by $p^{1/n}$.
\end{proof}

\begin{lemma}\label{lemma: more on val}
    Let $E^+$ be a $p$-complete valuation ring over $O$. Then the following holds.
    \begin{enumerate}[label=(\arabic*)]
        \item The ring $E^\circ$ is the localization of $E^+$ at the prime ideal $\sqrt{pE^+}$.
        \item The inclusion $E^+\to E^\circ$ is an almost isomorphism with respect to the ideal $\sqrt{pE^+} = \bigcup_{n>0}p^{\frac{1}{n}}E^+$. In particular, $E^\circ$ is bounded in $E$.
    \end{enumerate}
    In particular, the map $E^\circ \cotimes_{E^+} \cdL_{E^+/O} \to \cdL_{E^\circ/O}$ is an isomorphism. (Observe that the map $E^\circ \otimes_{E^+} \cdL_{E^+/O} \to E^\circ \cotimes_{E^+} \cdL_{E^+/O}$ is also an isomorphism because $E^\circ/E^+$ is $p$-torsion, or because of Corollary \ref{cor: HT comparison revisited} (1).)
\end{lemma}

\begin{proof}
    Radical ideals in $E^+$ are prime, so $\mathfrak p := \sqrt{pE^+}$ is automatically prime. We first claim that for any $f\in E^+\setminus\mathfrak p$, $f^{-1}$ is power bounded. A direct calculations shows that for all $n \in \mathbb Z_+$,
    \[ pf^{-n} = (p^{-1/n}f)^{-n} \in E^+. \]
    
    On the other hand, for $f \in E^\circ$, if $f \notin E^+$, then $f^{-1} \in E^+$. It follows that $f$ is invertible in $E^\circ$, so $f^{-1} \notin E^{\circ\circ} = \sqrt{pE^+}$. Therefore, $f = (f^{-1})^{-1} \in E^+_\mathfrak p$.
\end{proof}

We now define geometric valuation rings.

\begin{defi}\label{defi:geo val ring}
    A \emph{geometric valuation ring} over $O$ is a $p$-complete valuation ring $\calK^+$ over $O$ such that
    \[\cOmega_{\calK^+/O}[\frac{1}{p}]\]
    is of finite dimension over $\calK=\calK^+$ (note that $\calK$ is a field by Theorem \ref{theo: almost finitely generated}).
\end{defi}

\begin{rmk}
    TBA. Explain the difference between our definition and He's definition.
\end{rmk}

Our next goal is to consider the geometric Sen theory of geometric valuation rings, which we will use to decide the essential image of the restriction functor.
Let us recall some calculation of He (cf. \cite{he2025padicgaloiscohomologyvaluation}).

Fix a geometric valuation ring $\calK^+/O$ with fractional field $\calK$. Recall the notation $\calK^\circ$ from Lemma \ref{lemma: more on val}. Define $\calK^{\circ\circ}$ as the ideal of topologically nilpotent elements.

We are interested in the Sen theory of $\calK^+$. First, we need to construct the normalized trace of a Kummer tower. Recall the following theorem from \cite{he2025padicgaloiscohomologyvaluation}.

\begin{theo}\label{theo: senable val}
    Let $t_1,t_2,\dots,t_d\in \calK^\circ$ be invertible elements and $\pi\in \calK^{\circ\circ}\setminus p^{1/2^d}\calK^\circ$ such that
    \[\dif t_1,\dif t_2,\dots,\dif t_d\]
    form a $\pi$-basis of $\cOmega_{\calK^\circ/O}$ (i.e. They form a basis of $\cOmega_{\calK^\circ/O}[\frac{1}{p}]$ and $\pi\cOmega_{\calK^\circ/O}\subseteq \sum_{j=1}^d\calK^\circ \dif t_j$). Then the following holds.
    \begin{enumerate}[label=(\arabic*)]
        \item For each $j$, $t_j$ is not a $p$-th power of any element of $\calK$.
        \item For any integers $n_1,n_2,\dots,n_d\geq 0$, let $\calK'$ be the finite extension of $\calK$ by adding $t_j^{1/p^{n_j}}$ for all $j$. Then $[\calK':\calK]=p^{n_1+\cdots+n_d}$,
        \[\pi^{2^d-1}\calK'^\circ\subseteq \calK^\circ[t_1^{1/p^{n_1}},\ldots,t^{1/p^{n_d}}] \subseteq \calK'^\circ, \]
        and
        \[\{\dif t_j^{1/p^{n_j}}:j=1,2,\dots,d\}\]
        form a $\pi^{2^d}\calK^{\circ\circ}$-basis of $\cOmega_{\calK'^\circ/O}$ for any choice of $t^{1/p^{n_j}}_j$.
        \item The algebra $\calK^+$ is senable over $O$. More precisely, for any $\pi \in \calK^{\circ\circ} \setminus p^{1/2^d}\calK^\circ$, there exists a set of invertible elements $t_1,\ldots,t_d \in \calK^+$ forming a $\pi$-basis; any for any such $\pi$-basis, they give a toric chart of $\calK^+$ (see Definition \ref{defi: toric chart}).
    \end{enumerate}
\end{theo}

\begin{proof}
    The claim (1) follows from \cite[Proposition 7.3]{he2025padicgaloiscohomologyvaluation}. For (2), by Equation (7.8.2) of the proof of \cite[Lemma 7.8]{he2025padicgaloiscohomologyvaluation}, we find that for any $x\in \calK'^\circ$,
    \begin{equation}\label{eq:explicit normalized trace of val rings}
        [\calK':\calK]^{-1} \Tr_{\calK'/\calK}(x)\in \pi^{1-2^{d}}\calK^\circ.
    \end{equation}

    We first show that for any set of integers $m_1,\ldots,m_d$, if $t_1^{m_1}\cdots t_d^{m_d}$ is a $p$th power of some element in $\calK$, then $m_i$ are all multiples of $p$. Assume otherwise, then we may assume that $m_1=1$. Now take $t'_1 = t_1^{m_1}\cdots t_d^{m_d}$ and $t'_i = t_i$ for $2 \le i \le d$, we may apply (1) to see that $t'_1$ is not a $p$th power of any element in $\calK$, and this is a contradiction.

    For any $1 \le j \le d$, define $S_j$ to be the set $\{p^{-n_j} k\in\dQ:\ k=0,1,\dots,p^{n_j} \}$, and fix a choice of $t_j^{1/p^{n_j}}$ (hence we can define $t_{j}^{k/p^{n_j}}$ for any $k\in \dZ$). It is evident that the set
    \[\biggl\{\prod_{j=1}^dt_j^{\alpha_j}:\alpha_j\in S_j\biggr\}\]
    form a $\calK$-generating set of $\calK'$. Let 
    \[x=\sum_{\alpha_1\in S_1,\alpha_2\in S_2,\dots,\alpha_d\in S_d} \lambda_{\alpha_1,\alpha_2,\dots, \alpha_d} \prod_{j=1}^dt_j^{\alpha_j}\]
    be an arbitrary element in $\calK'^\circ$, where all $\lambda_{{\alpha_1,\alpha_2,\dots, \alpha_d}}\in \calK$. By (\ref{eq:explicit normalized trace of val rings}) and a simple calculation of trace, we have
    \[ \lambda_{\alpha_1,\alpha_2,\dots, \alpha_d} = [\calK':\calK]^{-1} \Tr_{\calK'/\calK}(x\prod_{j=1}^dt_j^{-\alpha_j})\in \pi^{1-2^{d}}\calK^\circ, \]
    so $\prod_{j=1}^dt_j^{\alpha_j}$ indeed form a basis of $\calK' / \calK$, and $\pi^{2^d-1}x \in \calK^\circ[t_1^{1/p^{n_1}},\ldots,t_d^{1/p^{n_d}}]$, as desired.

    For the $\cOmega$ part, denote
    \[ A = \calK^\circ[T_1,\ldots,T_d] / (T_i^{p^{n_i}} - t_i). \]
    Now that $A$ is finite free over $\calK^\circ$, so $A$ is $p$-torsion-free. It is easy to see that $A \bigl[ \frac 1 p \bigr] \to \calK'$ is an isomorphism, so $A$ agrees with its image in $\calK'$, and the map
    \begin{equation}\label{eq: syntomic reduction of field extension}
        A \to \calK^\circ [t_1^{1/p^{n_1}},\ldots,t_d^{1/p^{n_d}}]
    \end{equation}
    is an isomorphism. Now, a direct calculations gives the cotangent complex of $A$ over $\calK^\circ$,
    \[ \cdL_{A/\calK^\circ} \cong \dL_{A/\calK^\circ} \cong (A/p^{n_1}) \oplus \cdots \oplus (A/p^{n_d}), \]
    so the map $A \cotimes_{K^\circ} \cdL_{K^\circ/O} \to \cdL_{A/O}$ is injective with cokernel killed by a power of $p$. In particular, $\cdL_{A/O}$ is of bounded $p^\infty$-torsion, so it is classically $p$-adically complete. By Lemma \ref{lem: p-complete module as classical completion}, $\cdL_{A/O}$ is topologically generated by elements of form $\dif x$ for $x \in A$, and a direct calculations shows that the map
    \[ A \dif T_1 \oplus \cdots \oplus A \dif T_d \to \cdL_{A/O} \]
    is injective with cokernel killed by $\pi$. Now, consider the map $\dif': \calK'^\circ \to \hol^0(\calK'^\circ \cotimes_A^\bL \cdL_{A/O})$ given by $x \mapsto \dif(\pi^{2^{d+1}-2}x)$. For $x,y \in \calK'^\circ$ we have
    \[ \dif'(xy) = \dif((\pi^{2^d-1}x)(\pi^{2^d-1}y)) = (\pi^{2^d-1}x)\dif(\pi^{2^d-1}y) + (\pi^{2^d-1}y) \dif(\pi^{2^d-1}x) = x\dif'(y)+y\dif'(x). \]
    This gives a map $\cOmega_{\calK'/O} \to \hol^0(\calK'^\circ \cotimes_A^\bL \cdL_{A/O})$ such that the composition
    \[ \hol^0(\calK'^\circ \cotimes_A^\bL \cdL_{A/O}) \to \cOmega_{\calK'^\circ/O} \to \hol^0(\calK'^\circ \cotimes_A^\bL \cdL_{A/O}) \]
    is equal to multiplication by $\pi^{2^{d+1}-2}$. Compose with the map $\cdL_{A/O} \to A \dif T_1 \oplus \cdots \oplus A \dif T_d$ given by multiplication by $\pi$, one sees that the map
    \[ \calK' \dif t_1^{1/p^{n_1}} \oplus \cdots \oplus \calK' \dif t_d^{1/p^{n_d}} \to \cOmega_{\calK'^\circ/O} \]
    is injective. It is easy to see that the cokernel of $\hol^0(\calK'^\circ \cotimes_A^\bL \cdL_{A/O}) \to \cOmega_{\calK'^\circ/O}$ is annihilated by $\pi^{2^d-1}$, so the cokernel of the map
    \[ \calK' \dif t_1^{1/p^{n_1}} \oplus \cdots \oplus \calK' \dif t_d^{1/p^{n_d}} \to \cOmega_{\calK'^\circ/O} \]
    is annihilated by $\pi^{2^d}$, as desired.

    At last, we prove (3).

    Given $\pi$, we first constuct the elements $t_1,\ldots,t_d$ as required. Pick $\alpha \in \dQ_+$ with $\pi' := p^{-\alpha} \pi \in \calK^{\circ\circ}$. By Theorem \ref{theo: almost finitely generated}, we may pick elements $t'_1,\ldots,t'_d \in \calK^+$ for which
    \[ \dif t'_1,\ldots,\dif t'_d \]
    form a basis of $\Omega := \cOmega_{\calK^+/O}\bigl[ \frac 1 p \bigr] = \cOmega_{\calK^\circ/O}\bigl[ \frac 1 p \bigr]$ (the equality follows from Lemma \ref{lemma: more on val}), and that the module
    \[ \cOmega_{\calK^\circ/O} / \langle \dif t'_1,\ldots, \dif t'_d \rangle \]
    is annihilated by $\pi'$. Now, put $t_j = 1+p^\alpha t'_j$, and it is easy to verify that $t_j \in (\calK^+)^\times$ form a $\pi$-basis.

    Now, assume that we are given a set of invertible elements $t_1,\ldots,t_d \in \calK^+$ forming a $\pi$-basis, we prove that they form a toric chart. Take $\calK'_n := \calK(t_1^{1/p^n},\ldots,t_d^{1/p^n})$ and $\calK'_\infty := \bigcup \calK'_n$. For each $x \in \calK'^\circ_n$, we may write
\[ \pi^{2^d} \dif x = \sum_{j=1}^d y_j \dif t_j^{1/p^{n_j}} \]
for some $y_j \in \calK'^\circ_n$. Therefore, $\pi^{2^d} \dif x$ equals to zero in $\cOmega_{\calK'^\circ_{n+1}/O} /p$. Since $\calK'^\circ_{n+1}/O$ is $p$-Cartier smooth by Theorem \ref{theo: valuation rings are cartier smooth}. It follows that $\pi^{2^d}x$, as an element in $\calK'^\circ_{n+1}/p$, is a $p$th power. We see that the Frobenius on $\calK'^\circ_\infty/(\pi^{-2^d}p)$ is surjective. A standard argument now shows that $\widehat{\calK'_\infty}$ is a perfectoid field.

By (2) and Lemma \ref{lemma: more on val}, the cokernel of the ring map (the first isomorphism follows from \eqref{eq: syntomic reduction of field extension})
\[ \colim_{n \in \mathbb Z_+} \calK^+ \otimes_{O\langle T_1,\ldots,T_d\rangle} O\langle T_1^{1/p^n},\ldots,T_d^{1/p^n} \rangle \cong \colim_{n \in \mathbb Z_+} \calK^+ [t_1^{1/p^n},\ldots,t_d^{1/p^n}] \to \calK'^\circ_\infty \]
is annihilated by $\pi^{2^d-1}\mathfrak m$. Take the completion, and we see that
\[ \calK^+ \cotimes_{O\langle T_1,\ldots,T_d\rangle} O\langle T_1^{1/p^\infty},\ldots,T_d^{1/p^\infty} \rangle \bigl[ \frac 1 p \bigr] \cong \widehat{\calK'_\infty} \]
and it is perfectoid.
\end{proof}

Let $\overline\calK$ be the algebraic closure of $\calK$ and $\widehat{\overline{\calK}}$ be its completion. The valuation associated to $\calK^+$ extends uniquely to a continuous valuation on $\widehat{\overline{\calK}}$. Let $G=\Gal(\overline{\calK}/\calK)$ be the absolute Galois group of $\calK$.

The ring $\calK^+/O$ is senable by Theorem \ref{theo: senable val}. Hence, for any $v$-vector bundle, we can define its Sen operator. The field $\widehat{\overline{\calK}}$ is perfectoid, hence by Theorem \ref{theo: v vector bundles=generalized representations}, the category of $v$-vector bundles on $\Spa(\calK,\calK^+)$ is equivalent to the category of $\mathbf{Rep}_G(\widehat{\overline{\calK}})$. Combining the above, we have the folloing theorem.

\begin{theo}\label{theo: Sen of val}
    Keep the notation, for any $M\in\mathbf{Rep}_G(\widehat{\overline{\calK}})$, there exists a canonical linear operator
    \[\theta:M\to M\cotimes_{\calK}\cOmega_{\calK^+/O}\bigl[\frac{1}{p}\bigr].\]
\end{theo}

Consider the image of $\cOmega_{\calK^+/O}$ in $\cOmega_{\calK^+/O}\bigl[\frac{1}{p}\bigr]$, define
\[T_{\calK^+/O}:=\{\partial\in \cOmega_{\calK^+/O}[\tfrac{1}{p}]^\vee:\partial(\cOmega_{\calK^+/O})\subseteq\calK^+\}.\]

\begin{defi}
    Let $M\in\mathbf{Rep}_G(\widehat{\overline{\calK}})$ with Sen operator $\theta$. For any $a\in \mathbf{C}^\times$, call $\theta$ \emph{strictly $a$-topologically nilpotent} if there exists $\alpha\in\dQ_{>0}$ such that for any $\partial\in T_{\calK^+/O}$, the endomorphism $p^{-\alpha}a^{-1}\partial\circ\theta$ of $M$ is topologically nilpotent.
\end{defi}

Later, we will see that generalized representations with strictly topologically nilpotent Sen operators.

\begin{lemma}
    Let $M\in \Rep_G(\widehat{\overline{\calK}})$ be a generalized representation with $a$-strictly topologically nilpotent Sen operator $\theta$ for some $a\in\mathbf{C}^\times$. There exists $t_1,t_2,\dots,t_d\in \calK^+$ such that the following holds.
    \begin{enumerate}
        \item The induced map
        \[O\langle T_1,T_2,\dots,T_d \rangle\to \calK^+\]
        is a toric chart;
        \item let $\partial_i$ be the linear function in $\cOmega_{\calK^+/O}[\frac{1}{p}]^\vee$ sending $\dif T_j$ to $\delta_{ij}$. Then the operator $a^{-1}\partial_i\circ \theta$ is topologically nilpotent.
    \end{enumerate}
\end{lemma}

\begin{proof}
    By definition, there exists $\alpha\in\dQ_{>0}$ such that for any $\partial\in T_{\calK^+/O}$, the endomorphism $p^{-\alpha}a^{-1} \partial\circ\theta$ of $M$ is topologically nilpotent. Without losing generality, assume that $\alpha<2^{-d}$. By Theorem \ref{theo: senable val}, we may find elements $t_1,\ldots,t_d \in (\calK^+)^\times$ such that they form a $p^\alpha$-basis of $\calK$, and that they give a toric chart. It is easy to verify that $a^{-1} \partial_i \circ \theta$ is topologically nilpotent, since $\partial_i \in p^{-\alpha} T_{\calK^+/O}$.
\end{proof}

We now apply the idea of Corollary \ref{cor: fully faithfulness on small crystals} to the case of valuation rings.

\begin{theo}\label{theo: prism simpson val ring}
    Let $(\calK,\calK^+)$ be a geometric valuation ring over $O$, and $A:= \AAinf(O), I = \ker \theta$. Then, the map
    \[ \Vect \bigl( (\calK^+/A)_\Prism, \overline\calO\bigl[\frac 1 p \bigr]\bigr) \to \Vect \bigl( (\calK^+/A)_\Prism^\perf, \overline\calO\bigl[\frac 1 p \bigr]\bigr) \]
    restricts to an equivalence of categories on the full subcategories
    \[ \bigcup_{n \in \mathbb Z_+} \Vect^{\sm{p^{1/n}}} \bigl( (\calK^+/A)_\Prism, \overline\calO\bigl[\frac 1 p \bigr]\bigr)  = \bigcup_{n \in \mathbb Z_+} \Vect^{\smet{p^{1/n}}} \bigl( (\calK^+/A)_\Prism, \overline\calO\bigl[\frac 1 p \bigr]\bigr)  \subseteq \Vect \bigl( (\calK^+/A)_\Prism, \overline\calO\bigl[\frac 1 p \bigr]\bigr) \]
    and the full subcategory $\calC'$ of strictly $(\zeta_p-1)$-topologically nilpotent representations on
    \[ \Vect \bigl( (\calK^+/A)_\Prism^\perf, \overline\calO\bigl[\frac 1 p \bigr]\bigr) \cong \Rep_G(\widehat{\overline{\calK}}). \]
\end{theo}

\begin{proof}
This is a direct consequence of Theorem \ref{theo: senable val} (4) and Corollary \ref{cor: fully faithfulness on small crystals}.
\end{proof}

\subsection{Prismatic non-abelian Hodge theory of pre-smooth algebras}
In this section, we will apply the three Corollaries \ref{cor: simpson correspondence etale}, \ref{cor: fully faithfulness on small crystals}, \ref{cor: HT comparison revisited} to pre-smooth algebras.

\begin{defi}[Koszul regular sequence, see for example {\cite[\href{https://stacks.math.columbia.edu/tag/062D}{Tag 062D}]{stacks-project}}]
Let $A$ be a ring and $x_1,\ldots,x_n$ be a sequence in $A$. We say that $x_1,\ldots,x_n$ is Koszul regular, if the Koszul complex
\[ \Kos(A; x_1,\ldots,x_n) \cong A \otimes_{A[y_1,\ldots,y_n]}^\bL A \]
is concentrated on degree 0 (where the two maps $A[y_1,\ldots,y_n] \to A$ are given by $y_i \mapsto 0$ and $y_i \mapsto x_i$).

It is clear that for the simplicial ring $A \otimes_{A[y_1,\ldots,y_n]}^\bL A$, we have
\[ A \otimes_{A[y_1,\ldots,y_n]}^\bL A =0 \Leftrightarrow \pi_0(A \otimes_{A[y_1,\ldots,y_n]}^\bL A) =0 \Leftrightarrow (x_1,\ldots,x_n) = A, \]
so if $x_1,\ldots,x_n$ generates the unit ideal, then it is Koszul regular.

Note that if $A$ is derived $I$-adically complete for some finitely generated module, then the Koszul complex is always $I$-complete, and in this case $\hol^0 \Kos(A;x_1,\ldots,x_n)\cong A/(x_1,\ldots,x_n)$ is always derived $I$-adically complete by Lemma \ref{lem: t-structure on complete modules}.
\end{defi}

\begin{defi}[c.f. {\cite[\href{https://stacks.math.columbia.edu/tag/07D0}{Tag 07D0}]{stacks-project}}]\label{p-completely lci}
Let $R/A$ be $p$-complete rings of bounded $p^\infty$-torsion. We say that $R$ is $p$-completely l.c.i over $A$, if (p-completely) Zariski locally on $R$, there exists a presentation
\[ R \cong A \langle T_1,\ldots,T_n \rangle / (f_1,\ldots,f_m), \]
where $f_1,\ldots,f_m$ is a Koszul regular sequence in $A \langle T_1,\ldots,T_n \rangle$.

In this case, we see that $\cdL_{R/A}$ is a perfect complex over $R$. Indeed, if
\[ R \cong A \langle T_1,\ldots,T_n \rangle / (f_1,\ldots,f_m), \]
then by the cotangent sequence, then $\cdL_{R/A}$ has a presentation
\[ R^{\oplus m} \to R^{\oplus n} \]
with the map given by the Jacobian matrix of $f_1,\ldots,f_m$. 
\end{defi}

\begin{prop}\label{prop: p-completely syntomic are p-completely lci}
Let $S/A$ be $p$-complete rings of bounded $p^\infty$-torsion such that $S/A$ is $p$-completely syntomic (i.e. $S \otimes_A^\bL A/p$ is syntomic over $A/p$). Then, $S/A$ is $p$-completely l.c.i.
\end{prop}

\begin{proof}
Denote $\overline A = A/p$ and $\overline S = S \otimes_A^\bL A/p$. By definition we see that $S$ is $p$-complete and $p$-completely flat over $A$.

For each prime ideal $pS \subseteq \mathfrak p \subseteq S$, by {\cite[\href{https://stacks.math.columbia.edu/tag/07D0}{Tag 07D0}]{stacks-project}}, we may find $s \in S \setminus \mathfrak p$ such that $\overline S[s^{-1}]$ has a presentation for some $0 \le m \le n$
\[ \overline S[s^{-1}] = \overline A[x_1,\ldots,x_n]/(\overline f_1,\ldots,\overline f_m), \]
such that $\overline f_1,\ldots,\overline f_m$ is a Koszul regular sequence. This gives a map $\overline A[x_1,\ldots,x_n] \to \overline S[s^{-1}]$, and we may lift it to a map
\[ \varphi: A\langle x_1,\ldots,x_n\rangle \to \widehat{S[s^{-1}]}. \]
Denote the cofibre of $\varphi$ by $C$, then $C \otimes_A^\bL \overline A \in D^{\le -1}(\overline A)$. It follows that
\[ C = \widehat C = \lim C \otimes_A^\bL A/p^nA \in D^{\le -1}, \]
so $\varphi$ is surjective.

Lift the elements $\overline f_1,\ldots,\overline f_m$ to $f_1,\ldots,f_m \in \ker \varphi$. We get a natural map
\[ \Kos(A \langle x_1,\ldots,x_n \rangle; f_1,\ldots, f_m) \to \hol^0  \Kos(A \langle x_1,\ldots,x_n \rangle; f_1,\ldots, f_m) \to \widehat{S[s^{-1}]}. \]
This map between $p$-complete complexes is a quasi-isomorphism after taking tensor product $-\otimes_A^\bL \overline A$, and we see that it is a quasi-isomorphism by derived Nakayama's lemma.
\end{proof}

\begin{ex}\label{ex: semistable are presmooth}
Let $A$ be a $p$-complete ring of bounded $p^\infty$-torsion and $a \in A$ be any element, $n \le m$ be positive integers, then the ring
\[ R := A \langle x_1,\ldots,x_m \rangle / (x_1 \cdots x_n - a) \]
is $p$-completely syntomic, hence $p$-completely l.c.i over $A$.

If $A$ is $p$-torsion-free, $a \in A \bigl[ \frac 1 p \bigr]^\times$, and $d_i$ is not divisible by $p$ for all $i$, we claim that $R$ and $\cdL_{R/A}$ is also $p$-torsion-free, and $\cdL_{R/A} \bigl[ \frac 1 p \bigr]$ is free over $R \bigl[ \frac 1 p \bigr]$ of basis
\[ \dif x_2,\ldots,\dif x_m. \]
In fact, since $R$ is $p$-completely syntomic (hence $p$-complete and $p$-completely flat) over $A$, it is automatically $p$-torsion-free by Corollary \ref{cor: tensor conc on degree 0}. Now that $a$ is invertible in $R \bigl[ \frac 1 p \bigr]$, so all $x_i$ are invertible in $R \bigl[ \frac 1 p \bigr]$. The module $\cdL_{R/A}$ is calculated as the quotient
\[ F / (z_1 \dif x_1 + \cdots + z_n \dif x_n), \]
where $z_i := x_i^{-1} a$ for $a = x_1 \cdots x_n \in R$, and $F$ is free with basis
\[ \dif x_1, \ldots, \dif x_m. \]
The basis part is evident since all $z_i$ are invertible in $R\bigl[ \frac 1 p \bigr]$, and we still have to show that $\cdL_{R/A}$ is $p$-torsion-free.

Consider the ring
\[ R_0 = A \langle x_2^{\pm 1}, \ldots, x_n^{\pm 1}, x_{n+1},\ldots,x_m \rangle. \]
Now we may consider the map $g: R \to R_0$, sending $x_1$ to $ax_2^{-1} \cdots x_n^{-1}$. We claim that this map is injective. In fact, if $f \in A \langle x_1,\ldots,x_m \rangle$ is sent to $0$ by $g$, write
\[ f = \sum_{\underline k \in \mathbb N^m} c_{\underline k} \underline x^{\underline k}, \]
and define $c_{\underline k} =0$ if any index of $\underline k$ is negative. Then, for any multi-index $k = (k_1,k_2,\ldots,k_m)$, we have
\[ \sum_{j=-\infty}^\infty a^j c_{\underline k + j \underline s} =0, \]
where $\underline s := (\underbrace{1,\ldots,1}_n, 0,\ldots,0)$. This easily implies that the series
\[ \sum_{j=-\infty}^\infty c_{\underline k + j \underline s} T^j \in A\langle T \rangle [T^{-1}] \]
is divisible by $T-a$. Combine them together, and we are able to write $f$ as a multiple of $x_1 \cdots x_n - a$. It is clear that the essential image of $g$ is all power series
\[ f = \sum_{\underline k \in \mathbb Z^{\{2,\ldots,n\}} \times \mathbb N^{\{n+1,\ldots,m\}}} c_{\underline k} \underline x^{\underline k} \]
such that all $a^{-h(\underline k)} c_{\underline k}$ are in $A$, and they form a zero sequence in the cofinite sense (where $h(\underline k) := \max \{ 0, -k_2, -k_3, \ldots, -k_n \}$).

Now, if we are given elements $r_1,\ldots,r_n \in R$ with $x_1r_1 = \cdots = x_nr_n$, then we may apply $g$ to the element $r := r_i/z_i$ (indepedent of $i$). Write
\[ g(r) = \sum_{\underline \ell \in \mathbb Z^{\{2,\ldots,n\}} \times \mathbb N^{\{n+1,\ldots,m\}}} c_{\underline \ell} \underline x^{\underline \ell} \in R_0 \bigl[ \frac 1 p \bigr]. \]
Now that all $g(z_i) g(r)$ are in $R_0$. For $\underline \ell \in \mathbb Z^{\{2,\ldots,n\}} \times \mathbb N^{\{n+1,\ldots,m\}}$, if all indices of $\underline \ell$ are non-negative, then $g(z_1) g(r) \in R_0$ implies that $c_{\underline \ell} \in A$. Otherwise, take $2 \le i \le n$ such that $\ell_i$ is the minimal index of $\ell$, then $g(z_i)g(r) \in R_0$ implies that $a^{-\ell_i-1} (ac_{\underline \ell}) \in A$, and we get $a^{-\ell_i}c_{\underline \ell} \in A$. Therefore, $a^{-h(\underline \ell)} c_{\underline \ell} \in A$ for any $\underline \ell \in \mathbb Z^{\{2,\ldots,n\}} \times \mathbb N^{\{n+1,\ldots,m\}}$, and we conclude that $g(r)$ lies in $g(R)$, and $g$ injective implies that $r \in R$. Therefore, if $u = r_1 \dif x_1 + \cdots + r_m \dif x_m \in F$ and $pu \in \langle z_1 \dif x_1 + \cdots + z_n \dif x_n \rangle$, we see that $px_1r_1 = \cdots = px_nr_n$ and $pr_i=0$ for $i>n$. Since $R$ is $p$-torsion-free, we get $x_1r_1 = \cdots = x_nr_n$ and $r_i=0$ for $i>n$. Take $r=r_i/z_i \in R \bigl[ \frac 1 p \bigr]$ (independent of $i$), so $r \in R$, and we see that $u = r(z_1 \dif x_1 + \cdots + z_n \dif x_n)$. This shows that
\[ \cdL_{R/A} \cong F / (z_1 \dif x_1 + \cdots + z_n \dif x_n) \]
is $p$-torsion-free.
\end{ex}

\begin{prop}\label{prop: composition of p-completely lci are p-completely lci}
For $p$-complete rings $A,B,C$ of bounded $p^\infty$-torsion, if $A\to B$ and $B \to C$ are $p$-completely l.c.i, then their composition $A \to C$ is also $p$-completely l.c.i.
\end{prop}

\begin{proof}
Assume that
\[ \widehat{B[s^{-1}]} \cong A\langle X_1,\ldots,X_n \rangle / (f_1,\ldots,f_m) \]
and
\[ \widehat{C[t^{-1}]} \cong B\langle Y_1,\ldots,Y_{n'} \rangle / (g_1,\ldots,g_{m'}). \]
Then we get a presentation through completed tensor product
\[ \widehat{C[(st)^{-1}]} \cong \Kos\bigl(\widehat{B[s^{-1}]}\langle Y_1,\ldots,Y_{n'} \rangle; g_1,\ldots,g_{m'}\bigr), \]
hence
\[ \widehat{C[(st)^{-1}]} \cong A \langle X_1,\ldots,X_n,Y_1,\ldots,Y_{n'} \rangle / (f_1,\ldots,f_m, h_1,\ldots,h_{m'}), \]
where $h_i$ is a lifting of $g_i$. By {\cite[\href{https://stacks.math.columbia.edu/tag/0669}{Tag 0669}]{stacks-project}}, we see that $f_1,\ldots,f_m,h_1,\ldots,h_{m'}$ is Koszul regular.
\end{proof}

We now restate Corollaries \ref{cor: simpson correspondence etale}, \ref{cor: fully faithfulness on small crystals}, \ref{cor: HT comparison revisited} in the case of l.c.i algebras.

\begin{theo}[Special case of Corollary \ref{cor: simpson correspondence etale}]\label{theo: simpson correspondence etale lci case}
Let $(A,I)$ be a bounded prism and $R/\overline A$ be a $p$-completely l.c.i algebra, such that $R$ is $p$-torsion-free, $\cdL_{R/\overline A}$ is concentrated on degree 0 and $p$-torsion-free. Let $P/A$ be a $\delta$-algebra satisfying Condition \ref{cond: flat + tor 0}. Fix a Cartier divisor $J \to R$ of $R$ such that $(R/J) \bigl[ \frac 1 p \bigr] =0$. Assume that we are given a map of $A$-algebras $P \to R$ such that the induced map $R \otimes_{\overline P}\cOmega_{\overline P/\overline A} \to \cdL_{R/\overline A}$ is an isomorphism after inverting $p$, and its cokernel is annihilated by $J$. Then, the functor
\[ \Higgs^\tn\bigl(R \bigl[ \frac 1 p \bigr], R \cotimes_{\overline P} \cOmega_{\overline P/\overline A} \{-1\}\bigr) \to \Vect\bigl( (R/A)_\Prism, \overline \calO \bigl[ \frac 1 p \bigr] \bigr) \]
given in Proposition \ref{prop: Higgs to crystal} is fully faithful, and for any Cartier divisor $J' \subseteq R$ with $(R/J') \bigl[ \frac 1 p \bigr] =0$, the essential image $\calC_{J'}$ of the composition
\[ \Higgs^\tn\bigl(R \bigl[ \frac 1 p \bigr], J' \cotimes_{\overline P} \cOmega_{\overline P/\overline A} \{-1\}\bigr) \to \Higgs^\tn\bigl(R \bigl[ \frac 1 p \bigr], R \cotimes_{\overline P} \cOmega_{\overline P/\overline A} \{-1\}\bigr) \to \Vect\bigl( (R/A)_\Prism, \overline \calO \bigl[ \frac 1 p \bigr] \bigr) \]
satisfies
\[ \Vect^{\smet{JJ'}}\bigl( (R/A)_\Prism, \overline \calO \bigl[ \frac 1 p \bigr] \bigr) \subseteq \calC_{J'} \subseteq \Vect^{\sm{J'}}\bigl( (R/A)_\Prism, \overline \calO \bigl[ \frac 1 p \bigr] \bigr). \]
\end{theo}

\begin{theo}[cf. Corollary \ref{cor: fully faithfulness on small crystals}]\label{theo: fully faithfulness on small crystal l.c.i case}
Let $(K,K^+)/(\mathbb Q_p^\cyc,\mathbb Z_p^\cyc)$ be a perfectoid field and put $A = \dA_{\inf}(K^+)$ and $I = \ker \theta$. Let $R/K^+$ be a $p$-completely l.c.i algebra satisfying Condition \ref{cond: tor -1,0}, such that $R$ is $p$-torsion-free, $\cdL_{R/K^+}$ is concentrated on degree 0 and $p$-torsion-free. Fix a Cartier divisor $J \subseteq R$ of $R$ such that $(R/J) \bigl[ \frac 1 p \bigr] =0$.

Assume that there exists elements $x_1,\ldots,x_d \in R$ for which the map $R^{\oplus d} \to \cdL_{R/\overline A}$ given by $( \dif x_i )$ is injective with cokernel annihilated by $J$, then the composition
\[ \Vect^{\smet J}\bigl( (R/A)_\Prism, \overline \calO \bigl[ \frac 1 p \bigr] \bigr) \to \Vect\bigl( (R/A)_\Prism, \overline \calO \bigl[ \frac 1 p \bigr] \bigr) \xrightarrow{\Res_{\Spf(R)}} \Vect\bigl( (R/A)_\Prism^\perf, \overline \calO \bigl[ \frac 1 p \bigr] \bigr) \]
is fully faithful.

Moreover, denote $\cdL_{R/\overline A}^\vee$ the classical $R$-module $\Hom_R(\cdL_{R/\overline A},R)$ as in Corollary \ref{cor: fully faithfulness on small crystals}. For any Cartier divisor $J' \subseteq R$ with $(R/J') \bigl[ \frac 1 p \bigr]=0$, denote
\[ \Vect^{\sm{J'}} \bigl( (R/A)_\Prism^\perf, \overline \calO \bigl[ \frac 1 p \bigr] \bigr) \]
the category of all vector bundles $\calE \in \Vect\bigl( (R/A)_\Prism^\perf, \overline \calO \bigl[ \frac 1 p \bigr] \bigr)$ for which the Sen action of $(J')^{-1} (\zeta_p-1)^{-1}\cdL_{R/\overline A}^\vee \{1\}$ is topologically nilpotent \emph{(do not forget the $(\zeta_p-1)$-factor)}. Then, there exists a full subcategory $\calC \subseteq \Vect\bigl( (R/A)_\Prism, \overline \calO \bigl[ \frac 1 p \bigr] \bigr)$ such that for any Cartier divisor $J'$ of $R$ with $(R/J') \bigl[ \frac 1 p \bigr]=0$, we have
\begin{align*}
\Res_{\Spf(R)} \bigl( \calC \cap \Vect^{\sm{J'}} \bigl( (R/A)_\Prism, \overline \calO \bigl[ \frac 1 p \bigr] \bigr) \bigr) & \supseteq \Vect^{\sm{JJ'}} \bigl( (R/A)_\Prism^\perf, \overline \calO \bigl[ \frac 1 p \bigr] \bigr), \\
\Res_{\Spf(R)} \bigl( \Vect^{\smet{JJ'}} \bigl( (R/A)_\Prism, \overline \calO \bigl[ \frac 1 p \bigr] \bigr) \bigr) & \subseteq \Vect^{\sm{J'}} \bigl( (R/A)_\Prism^\perf, \overline \calO \bigl[ \frac 1 p \bigr] \bigr).
\end{align*}
\end{theo}

\begin{proof}
Put $P = A \langle T_1,\ldots,T_d \rangle$. By Corollary \ref{cor: simpson correspondence etale}, the following map
\[ \Higgs^\tn \bigl( R \bigl[ \frac 1 p \bigr], R \cotimes_{\overline P} \cOmega_{\overline P/\overline A} \{-1\} \bigr) \to \Vect\bigl( (R/A), \overline\calO\bigl[ \frac 1 p \bigr] \bigr) \]
is fully faithful. Denote $\calC$ its essential image. More generally, for any Cartier divisor $J' \subseteq R$ with $(R/J') \bigl[ \frac 1 p \bigr] =0$, we denote $\calC_{J'}$ the essential image of the composition
\[ \Higgs^\tn \bigl( R \bigl[ \frac 1 p \bigr], J' \cotimes_{\overline P} \cOmega_{\overline P/\overline A} \{-1\} \bigr) \to \Higgs^\tn \bigl( R \bigl[ \frac 1 p \bigr], R \cotimes_{\overline P} \cOmega_{\overline P/\overline A} \{-1\} \bigr) \to \Vect \bigl( (R/A), \overline\calO\bigl[ \frac 1 p \bigr] \bigr). \]
By Corollary \ref{cor: simpson correspondence etale}, we see that for any Cartier divisor $J' \subseteq R$ with $(R/J') \bigl[ \frac 1 p \bigr] =0$, we have
\[ \Vect^{\smet{JJ'}} \bigl( (R/A), \overline\calO\bigl[ \frac 1 p \bigr] \bigr) \subseteq \calC_{J'} \subseteq \calC \cap \Vect^{\sm{J'}} \bigl( (R/A), \overline\calO\bigl[ \frac 1 p \bigr] \bigr). \]
We only need to prove that $\Res_{\Spf(R)}$ is fully faithful on $\calC$, and for any $J'$ as above, we have
\[ \Vect^{\sm{JJ'}} \bigl( (R/A)^\perf, \overline\calO\bigl[ \frac 1 p \bigr] \bigr) \subseteq \Res_{\Spf(R)} (\calC_{J'}) \subseteq \Vect^{\sm{J'}} \bigl( (R/A)^\perf, \overline\calO\bigl[ \frac 1 p \bigr] \bigr). \]
Now, everything involved are $p$-completely \'etale stacks by Corollary \ref{cor: simpson correspondence etale} (and Lemma \ref{lem: Higgs is fpqc stack}), so we may check this $p$-completely \'etale locally.

Decompose $R$ into Zariski opens cut out by elements of form $y_1\cdots y_d$ (with $y_i \in \{ x_i, x_i-1 \}$), and we may assume that one of $x_i$ and $x_i-1$ is invertible for each $i$. Replace $x_i$ by $x_i-1$ if necessary, and we may assume that all $x_i$ are invertible in $R$.

By the arguments in Corollary \ref{cor: fully faithfulness on small crystals}, we only need to prove that the map
\[ A \langle T_1,\ldots,T_d \rangle \to R \]
sending $T_i$ to $x_i$ is a toric chart (Definition \ref{defi: toric chart}). In other words, we need to show that the ring
\[ \bigl( \overline A \langle T_1^{1/p^\infty}, \ldots, T_d^{1/p^\infty} \rangle \cotimes_{\overline A \langle T_1,\ldots,T_d \rangle} R \bigr) \bigl[ \frac 1 p \bigr] \]
is locally perfectoid. 

By Theorem \ref{theo: lci locally perfectoid}, we only need to show that the map
\[ K^+ \langle T_1,\ldots,T_d \rangle \to R \]
is \'etale (in the sense of \cite[Definition 1.6.5]{Huber_1996}) after inverting $p$. By writing it as the composition
\[ K^+ \langle T_1,\ldots,T_d \rangle \to R \langle T_1,\ldots,T_d \rangle \to R, \]
one sees that it is $p$-completely l.c.i. By our assumptions on $x_i$, we see that the completed cotangent complex
\[ \cdL_{R/K^+ \langle T_1,\ldots,T_d \rangle} \cong \cdL_{R/K^+} / (R \dif x_1 \oplus \cdots \oplus R \dif x_d) \]
is $J$-torsion. Zariski locally on $R$, we may write (with $\mathrm D_d^+ := K^+ \langle T_1,\ldots,T_d \rangle$)
\[ \widehat{R[s^{-1}]} \cong \Kos( \mathrm D_d^+ \langle T'_1,\ldots,T'_n \rangle ; f_1,\ldots,f_m). \]
Examine the Euler characteristic of the perfect complex, we see that $m=n$, and the Jacobian $\frac{\partial (f_1,\ldots,f_n)}{\partial (T'_1,\ldots,T'_n)}$ is annihilated by $J^n$. Therefore, Proposition \cite[Proposition 1.7.1]{Huber_1996} shows that the map
\[ K^+ \langle T_1,\ldots,T_d \rangle \bigl[ \frac 1 p \bigr] \to R \bigl[ \frac 1 p \bigr] \]
is \'etale.
\end{proof}

As a direct corollary, we have:
\begin{cor}\label{cor: fully faithfulness on small crystal l.c.i case revisited}
Let $\mathbf C$ be an algebraically closed $p$-adic non-archimedean field, and $\calO_{\mathbf C}$ be the ring of integers of $\mathbf C$. Put $\Ainf$ the prism $(\AAinf(\calO_{\mathbf C}), \ker \theta)$. Let $R$ be a $p$-completely l.c.i algebra over $\calO_{\mathbf C}$ such that $R$ is pre-smooth over $\calO_{\mathbf C}$ (see Definition \ref{intro: pre smooth}). Denote $\cdL_{R/\calO_{\mathbf C}}^\vee$ the classical $R$-module $\Hom_R(\cdL_{R/\calO_{\mathbf C}},R)$ as in Corollary \ref{cor: fully faithfulness on small crystals}. For any Cartier divisor $J' \subseteq R$ with $(R/J') \bigl[ \frac 1 p \bigr]=0$, denote
\[ \Vect^{\sm{J'}} \bigl( (R/\Ainf)_\Prism^\perf, \overline \calO \bigl[ \frac 1 p \bigr] \bigr) \]
the category of all vector bundles $\calE \in \Vect\bigl( (R/\Ainf)_\Prism^\perf, \overline \calO \bigl[ \frac 1 p \bigr] \bigr)$ for which the Sen action of $(J')^{-1} (\zeta_p-1)^{-1}\cdL_{R/\overline A}^\vee \{1\}$ is topologically nilpotent \emph{(do not forget the $(\zeta_p-1)$-factor)}.

Then, there exists $a \in \mathbf C^\times \cap \calO_{\mathbf C}$ such that the restriction functor
\[ \Res_{\Spf(R)}: \Vect((R/\Ainf)_\Prism, \overline \calO \bigl[ \frac 1 p \bigr]) \to \Vect(\Spd(R,R^+)_v, \widehat \calO) \]
is fully faithful on $\Vect^{\smet a} ((R/\Ainf)_\Prism, \overline \calO \bigl[ \frac 1 p \bigr])$, and for any Cartier divisor $J'$ as above, we have (note that we only use $\Vect^{\smet{J'}}$ instead of $\Vect^{\sm{J'}}$ here)
\begin{align*}
\Vect^{\sm{a^2J'}} \bigl( (R/\Ainf)_\Prism^\perf, \overline \calO \bigl[ \frac 1 p \bigr] \bigr) & \subseteq \Res_{\Spf(R)} \bigl( \Vect^{\smet{aJ'}} \bigl( (R/\Ainf)_\Prism, \overline \calO \bigl[ \frac 1 p \bigr] \bigr) \bigr) \\
& \subseteq \Vect^{\sm{J'}} \bigl( (R/\Ainf)_\Prism^\perf, \overline \calO \bigl[ \frac 1 p \bigr] \bigr).
\end{align*}
\end{cor}

\begin{proof}
Take a $p$-completely Zariski cover
\[ R \to \prod_{i=1}^N \widehat{R[s_i^{-1}]} \]
such that we have maps of $\calO_{\mathbf C}$-algebras
\[ \calO_{\mathbf C} \langle T_1,\ldots,T_{d_i} \rangle \to \widehat{R[s_i^{-1}]} \]
for which the images of $\dif T_i$ freely generate $\cdL_{\widehat{R[f_i^{-1}]}/\calO_{\mathbf C}} \bigl[ \frac 1 p \bigr] \cong \widehat{R[f_i^{-1}]} \cotimes_R \cdL_{R/\calO_{\mathbf{C}}} \bigl[\frac{1}{p}\bigr]$. We may take $a \in \mathbf C^\times \cap \calO_{\mathbf C}$ such that for all $i$, the cokernel of
\[ \widehat{R[f_i^{-1}]} \dif T_1 \oplus \cdots \oplus \widehat{R[f_i^{-1}]} \dif T_{d_i} \to \cdL_{\widehat{R[f_i^{-1}]}/\calO_{\mathbf C}} \bigl[ \frac 1 p \bigr] \]
is annihilated by $a$. We claim that this $a$ is exactly what we need.

Note that the categories
\[ \Vect^{\smet{J'}} \bigl( (R/\Ainf)_\Prism, \overline \calO \bigl[ \frac 1 p \bigr] \bigr), \quad \Vect^{\sm{J'}} \bigl( (R/\Ainf)_\Prism^\perf, \overline \calO \bigl[ \frac 1 p \bigr] \bigr) \]
satisfy descent of the $p$-completely Zariski topology. Therefore, using Theorem \ref{theo: fully faithfulness on small crystal l.c.i case}, we may first argue $p$-completely Zariski locally to show that $\Res_{\Spf(R)}$ is fully faithful on $\Vect^{\smet a} ((R/\Ainf)_\Prism, \overline \calO \bigl[ \frac 1 p \bigr])$. In particular
\[ \Res_{\Spf(R)} \bigl( \Vect^{\smet{aJ'}} \bigl( (R/\Ainf)_\Prism, \overline \calO \bigl[ \frac 1 p \bigr] \bigr) \bigr) \]
is a $p$-completely Zariski stack on $R$, and we may argue $p$-completely Zariski locally again to get the result for the essential image.
\end{proof}

\begin{theo}[Special case of Corollary \ref{cor: HT comparison revisited}]\label{cor: HT comparison revisited lci}
Let $(A,I)$ be a bounded prism and $R / \overline A$ be a $p$-completely l.c.i algebra. Assume that $R$ and $\cdL_{R/\overline A}$ are concentrated on degree 0 and $p$-torsion-free, such that $\cdL_{R/\overline A} \bigl[ \frac 1 p \bigr]$ is finite projective over $R \bigl[ \frac 1 p \bigr]$. Then,
\begin{enumerate}[label=(\arabic*)]
\item for any $M \in \widehat D(R)$ and integer $n \ge 0$, the map
\[ M \otimes_R^\bL \bigwedge\nolimits_R^n(\cdL_{R/\overline A}) \to M \cotimes_R^\bL \cbigwedge_R^n(\cdL_{R/\overline A}) \]
is an isomorphism; moreover, if $M$ is concentrated on degree 0 of bounded $p^\infty$-torsion, then there exists $k \ge 0$ independent of $n$ and $M$, such that
\[ p^{nk}\hol^i \bigl( M \otimes_R^\bL \bigwedge\nolimits_R^n(\cdL_{R/\overline A}) \bigr) =0, \forall i \ne 0 \]
and, for $M[p^\infty] = M[p^{n_0}]$,
\[ \hol^0 \bigl( M \otimes_R^\bL \bigwedge\nolimits_R^n(\cdL_{R/\overline A}) \bigr) [p^\infty] = \hol^0 \bigl( M \otimes_R^\bL \bigwedge\nolimits_R^n(\cdL_{R/\overline A}) \bigr) [p^{nk+n_0}]. \]

\item the colimit in $D(R)$
\[ \colim_n \Fil_n(\overline\Prism_{R/A}) \]
is already derived $p$-adically complete; therefore, the Hodge--Tate filtration on $\overline\Prism_{R/A}$ is exhaustive in $D(R)$;

\item if $\cdL_{R/\overline A}$ is finitely generated over $R$, the Hodge--Tate filtration on the $\dE_\infty$-ring $\overline \Prism_{R/A} \bigl[ \frac 1 p \bigr]$ induces an isomorphism of cohomology rings
\[ \hol^*\bigl( \overline \Prism_{R/A} \bigl[ \frac 1 p \bigr] \bigr) \cong \bigwedge\nolimits^*_{R[\frac 1 p]} \bigl(\cdL_{R/\overline A} \bigl[ \frac 1 p \bigr] \bigr)\]
\end{enumerate}
\end{theo}

\subsection{Prismatic non-abelian Hodge theory of irrational disks}\label{sect: irrational disks}
Denote $\bfC$ the completion of the algebraic closure of $\dQ_p$. Equip $\bfC$ with the standard norm in which $|p| = p^{-1}$. Fix a coherent system of powers of $p$ ($p^r$ for each $r \in \mathbb Q$) in $\bfC$.

For $r \in \dR$, we define the irrational disk over $\calO_{\bfC}$ as follows:
\[ \calO_{\bfC} \bigl\{ \frac{T}{p^r} \bigr\} = \left\{ \sum_{n=0}^\infty a_n \in \bfC [[T]] \middle| p^{-nr}|a_n| \le 1, \lim_{n \to \infty} p^{-nr}|a_n| \to 0 \right\}, \]
which is $p$-complete and $p$-completely flat over $\calO_{\bfC}$. For $r \in \dQ$, it is the completed polynomial ring with parameter $T/p^r$, and for arbitrary $r \in \dR$, we have (note that each term of the colimit is complete, so Lemma \ref{lem: p-complete module as classical completion} guarantees that we can take the classical completion)
\[ \calO_{\bfC} \bigl\{ \frac{T}{p^r} \bigr\} = \ccolim_{s\le r, s \in \mathbb Q} \calO_{\bfC} \bigl\{ \frac{T}{p^s} \bigr\}. \]
In particular, for all $R = \calO_{\bfC} \bigl\{ \frac{T}{p^r} \bigr\}$, the ring $R$ is $p$-complete and $p$-completely flat over $\calO_{\bfC}$, and the module $\cdL_{R/\calO_{\bfC}}$ is $p$-completely flat over $R$.

We will apply our machinaries (Corollary \ref{cor: fully faithfulness on small crystals}) to this case, and we have the following theorem.

\begin{theo}\label{theo: fully faithfulness on small crystals irrational disk}
For $(A,I) = (\dA_{\inf}(\calO_{\bfC}), \ker \theta)$ and $R = \calO_{\bfC} \bigl\{ \frac{T}{p^r} \bigr\}$, the restriction map
\[ \Vect\bigl((R/A)_\Prism, \overline\calO \bigl[ \frac 1 p \bigr]\bigr) \to \Vect\bigl((R/A)_\Prism^\perf, \overline\calO \bigl[ \frac 1 p \bigr]\bigr) \]
induces an equivalence on the full subcategories
\[ \bigcup_{n >0} \Vect^{\sm{p^{1/n}}} \bigl( (R/A)_\Prism, \overline\calO \bigl[ \frac 1 p \bigr] \bigr) = \bigcup_{n >0} \Vect^{\smet{p^{1/n}}} \bigl( (R/A)_\Prism, \overline\calO \bigl[ \frac 1 p \bigr] \bigr) \subseteq \Vect \bigl( (R/A)_\Prism, \overline\calO \bigl[ \frac 1 p \bigr] \bigr) \]
and the full subcategory $\calC'$ of vector bundles $\calE \in \Vect\bigl((R/A)_\Prism^\perf, \overline\calO \bigl[ \frac 1 p \bigr]\bigr)$ for which the Sen action of $p^{-1/n} (\zeta_p-1)^{-1} \cOmega_{R/\overline A}^\vee$ is topologically nilpotent for some $n \in \mathbb Z_+$ (where $\cOmega_{R/\overline A}^\vee$ is defined as the classical $R$-module $\Hom_R(\cOmega_{R/\overline A},R)$).
\end{theo}

\begin{proof}
Fix $s \in \dQ_+$.

Take $t \in \dQ$ with $0<r-t<\min \{ 1,s \}$. Take $P = A \langle X \rangle$ with $\delta(X)=0$, and $P \to R$ be defined by $X \mapsto 1+ p^{-t}T$. We claim that the map $P \to R$ is a toric chart for $R$, such as the cofibre of
\[ \calO_{\bfC} \bigl\{ \frac T {p^r} \bigr\} \dif \bigl( \frac T{p^t} \bigr) \to \cdL_{\calO_{\bfC} \{ \frac T {p^r} \}/\calO_{\bfC}} \]
is concentrated on degree 0 and annihilated by $p^s$. Once we have proved this, the theorem would follow from Corollary \ref{cor: fully faithfulness on small crystals}.

Note that $1+p^{-t}T$ is invertible in $R$, so we only need to show that for the ring
\[ R' := R \cotimes_{\calO_{\bfC}\langle X \rangle} \calO_{\bfC} \langle X^{1/p^\infty} \rangle \]
(with $X \mapsto 1+p^{-t}T$), the Tate ring $R' \bigl[ \frac 1 p \bigr]$ is perfectoid. Rewrite $T = p^t(X-1)$, the ring $R'$ may be rewritten as the ring (here we use that $\calO_{\bfC}\langle T/p^{u+t} \rangle \cong \Kos(\calO_{\bfC}\langle T',T/p^t \rangle; p^uT'-T/p^t )$)
\[ \ccolim_{u \in \dQ, u\le r-t} \Kos( \calO_{\bfC}\langle X^{1/p^\infty},T' \rangle; p^uT'-(X-1)). \]

Denote the perfectoid ring
\[ A' = \calO_{\bfC} \langle X^{1/p^\infty}\rangle. \]
Denote $Y = X^\flat - 1 \in (A')^\flat$, and we see that $Y^\sharp - (X-1) \in pA'$. Observe that $u \le r-t < 1$, so the ring $\Kos( \calO_{\bfC}\langle X^{1/p^\infty},T' \rangle; p^uT'-(X-1))$ is the same as
\[ \Kos(A'\langle T' \rangle; p^uT'- Y^\sharp ). \]
Taking the colimit, and we see that
\[ R' \cong \Kos \bigl( A' \bigl\{ \frac{T'}{p^{r-t}} \bigr\}; T' - Y^\sharp \bigr). \]
Note that the map
\[ \calO_{\bfC} \langle Z^{1/p^\infty} \rangle \to A' \]
sending $Z^\flat$ to $Y$ is an isomorphism. Therefore, we see that
\[ R' \cong \Bigl( \colim_{u \in \dQ, u \le r-t} \calO_{\bfC} [Z^{1/p^\infty}][Z/p^u] \Bigr)^\land. \]
We construct the ring
\[ R'' = \Bigl( \colim_{u \in \dQ, u \le r-t} \calO_{\bfC}[(Z/p^u)^{1/p^\infty}] \Bigr)^\land = \ccolim_{u \in \dQ, u \le r-t} \calO_{\bfC} \langle (Z/p^u)^{1/p^\infty} \rangle, \]
and we see that the cokernel of $R' \to R''$ is annihilated by $p^k$ (where $k \in \dQ$ is any rational number greater or equal than $r-t$). Therefore $R' \bigl[ \frac 1 p \bigr] \cong R'' \bigl[ \frac 1  p \bigr]$. Now, $R''$ is a completed colimit of integral perfectoid rings, so $R''$ is integral perfectoid (because $\Frob: R''/p^{1/p} \to R''/p$ is an isomorphism).

By the above calculations, we see that $P \to R$ is a toric chart for $R$. Observe that for $s' \in \dQ$ and $t < s'\le r$, the cofibre of
\[ \calO_{\bfC} \bigl\{ \frac T {p^{s'}} \bigr\} \dif \bigl( \frac T{p^t} \bigr) \cong \calO_{\bfC} \bigl\{ \frac T {p^{s'}} \bigr\} \cotimes_{\calO_{\bfC} \{ \frac T {p^t} \}} \cdL_{\calO_{\bfC} \{ \frac T {p^t} \} / \calO_{\bfC}} \to \cdL_{\calO_{\bfC} \{ \frac T {p^{s'}} \}/\calO_{\bfC}} \]
is concentrated on degree 0 and annihilated by $p^s$ (since $s>r-t$). Taking the colimit, and we see that the cofibre of
\[ \calO_{\bfC} \bigl\{ \frac T {p^r} \bigr\} \dif \bigl( \frac T{p^t} \bigr) \to \cdL_{\calO_{\bfC} \{ \frac T {p^r} \}/\calO_{\bfC}} \]
is concentrated on degree 0 and annihilated by $p^s$. We are done Corollary \ref{cor: fully faithfulness on small crystals}.
\end{proof}

\appendix
\section{Appendix: Descent theory of $p$-completely faithfully flat ring maps --- elementary proof}
\newcommand\bAlg{\widehat\Alg^\bd}
\newcommand\bMod[1]{\widehat\Mod^\bd_{#1}}
\newcommand\aMod[1]{\Mod^\ad_{{#1}}}

In this appendix, we are going to prove Theorem \ref{thm: descent theory} in an elementary way --- in the case that we are interested in. Also, this appendix contains some useful lemmas that we will frequently use, such as Lemma \ref{lem: basic properties of derived complete modules} and Proposition \ref{prop: finitely generated adic modules}.

Throughout the appendix, denote $\bAlg$ the category of classically $p$-complete commutative rings of bounded $p^\infty$-torsion. For any ring $R \in \bAlg$, we denote $\bMod R$ the category of classically $p$-adically complete $R$-modules of bounded $p^\infty$-torsion, and $\aMod R$ the category of adic modules over $R \bigl[ \frac 1 p \bigr]$ with morphisms being continuous maps (for notational convenience, we use $\aMod R$ instead of $\aMod{R[1/p]}$). A sequence in $\bMod R$ or $\aMod R$ is called exact, if and only if it is exact when regarded as a sequence of $R$-modules.

We will make use of the following lemma:

\begin{lemma}\label{lem: basic properties of derived complete modules}\mbox{}
\begin{enumerate}[label=(\alph*)]
\item Let $C \in D(\mathbb Z)$ be a complex, then $C$ is derived $p$-adically complete if and only if all cohomology groups $\hol^i(C)$ are derived $p$-adically complete;
\item Assume that an abelian group $C$ is derived $p$-adically complete, and $C \bigl[ \frac 1 p \bigr] =0$, then $C$ is annihilated by a power of $p$;
\item Assume that an abelian group $C$ is derived $p$-adically complete, and $C=pC$, then $C=0$.
\end{enumerate}
\end{lemma}

\begin{proof}
The item (a) is (1) and (5) in \cite[\href{https://stacks.math.columbia.edu/tag/091P}{Tag 091P}]{stacks-project}.

The item (b) is \cite[\href{https://stacks.math.columbia.edu/tag/0CQY}{Tag 0CQY}]{stacks-project}.

As for (c), if $C$ is derived $p$-adically complete, then $\Hom (\mathbb Z[1/p],C) =0$. Since $C=pC$, it is easy to see that any group homomorphism $f: \mathbb Z \to C$ extends to a group homomorphism $\mathbb Z[1/p] \to C$. Therefore, $\Hom_\mathbb Z(\mathbb Z,C) =0$, and we see that $C=0$.
\end{proof}

\begin{lemma}\label{lem: left exactness}
The additive categories $\bMod R$ and $\aMod R$ have kernels. Moreover, a complex
\[ 0 \to X \to Y \to Z \]
in either of these categories is left exact, if and only if $X \to \ker ( Y \to Z )$ is an isomorphism.
\end{lemma}

\begin{proof}
First we work in $\bMod R$. Let $f: Y \to Z$ be a map in $\bMod R$, and $X = \ker f$. By \cite[\href{https://stacks.math.columbia.edu/tag/091P}{Tag 091P}]{stacks-project}, we see that $X$ is derived $p$-adically complete. As $X \subseteq Y$, we see that $X$ is of bounded $p^\infty$-torsion, and $X \in \bMod R$. Since $\bMod R \to \Mod_R$ is fully faithful, it detects limits, and we see that $X$ is indeed the kernel of $f$. The `if and only if' statement also follows from the fully faithfulness.

Now we work in $\aMod R$. Let $f: Y \to Z$ be a map in $\aMod R$. We may choose appropriate modules of definition $Y_0 \subseteq Y$ and $Z_0 \subseteq Z$, such that $f(Y_0) \subseteq Z_0$. Take $X_0 = \ker f \cap Y_0$ and $X = \ker f$, and we see that
\[ X = X_0 \bigl[ \frac 1 p \bigr]. \]
Since $Z$ is $p$-torsion free, we see that $p^n X_0 = p^n Y_0 \cap X_0$ for all $n$. Equip $X \subseteq Y$ with the closed subspace topology. By the above discussions, we see that $X$ is adic with module of definition $X_0$. It is easy to see that $X$ is the kernel of $Y \to Z$ in $\aMod R$.

To prove the last statement, we only need to prove that the forgetful functor $\aMod R \to \Mod_R$ is conservative (i.e. detects isomorphisms). This is a direct consequence of Banach's open mapping theorem (Theorem \ref{theo: banach open mapping}).
\end{proof}

\begin{lemma}\label{lem: tensor exact sequence}
For $A \in \bAlg$, $C \in D(A)$ and $M \in \bMod A$, if $M$ is $p$-completely flat over $A$, then we have short exact sequences
\[ 0 \to \hol^0 ( M \cotimes_A^\bL \hol^n(C)) \to \hol^n( M \cotimes_A^\bL C) \to \hol^{-1} ( M \cotimes_A^\bL \hol^{n+1}(C) ) \to 0. \]

For any $A$-module $X$, we have the short exact sequence
\[ 0 \to \clim 1 M \otimes_A (X[p^n]) \to \hol^0 ( M \cotimes_A^\bL X) \to \clim 0 M \otimes_A X/p^n X \to 0. \]
\end{lemma}

\begin{proof}
Let $X$ be any $A$-module. We have the expression
\[ M \cotimes_A^\bL X = \lim \bigl( M \otimes_A^\bL \cofib(X \xrightarrow{p^n} X) \bigr). \]
It follows that $M \cotimes_A^\bL X$ is concentrated on $\left[ -1,1 \right]$. Also,
\[ \hol^1 (M \cotimes_A^\bL X) \cong \clim 1 M \otimes_A X/p^nX =0 \]
since the transition maps are surjective. Therefore, $M \cotimes_A^\bL X$ is concentrated on degree $\left[ -1,0 \right]$. The spectral sequence for $\lim$ gives the short exact sequence
\[ 0 \to \clim 1 M \otimes_A (X[p^n]) \to  \hol^0 ( M \cotimes_A^\bL X) \to \clim 0 M \otimes_A X/p^n X \to 0. \]

For the first part, assume that $n=0$. note that $M \cotimes_A^\bL X$ is concentrated on degree $[-1,0]$ for each $X$. Similarly argue, and we see that for each $-\infty \le a \le b \le \infty$ and $C \in D^{[a,b]} (A),$ we have
\[ M \cotimes_A^\bL C \in D^{[a-1,b]}(A). \]
Therefore we may take the truncation, and assume that $C \in D^{[0,1]}(A).$ Consider the exact triangle
\[ \hol^0(C) \to C \to \hol^1(C)[-1] \to \hol^0(C)[1] \]
It induces a long exact sequence
\begin{gather*}
\hol^{-1}(M \cotimes_A^\bL \hol^1(C)[-1]) \to \hol^0 (M \cotimes_A^\bL \hol^0(C)) \to \hol^0 (M \cotimes_A^\bL C) \\[3pt]
\to \hol^0(M \cotimes_A^\bL \hol^1(C)[-1]) \to \hol^1(M \cotimes_A^\bL \hol^0(C)).
\end{gather*}
By the restrictions on degrees, the modules on both ends are zero, and we get our result as desired.
\end{proof}

\begin{lemma}\label{lem: exactness of flat base change}
Assume that $A \in \bAlg$ and $M \in \bMod A$ is $p$-completely flat. Assume that $X \to Y \to Z \to W$ is exact in $\bMod A$ or $\aMod A$, then its base change
\[ M \cotimes_A X \to M \cotimes_A Y \to M \cotimes_A Z \to M \cotimes_A W\]
is exact \emph{at $M \cotimes_A Y$}.
\end{lemma}

\begin{proof}
In case we are dealing with $\aMod A$, we may choose appropriate modules of definitions $Y_0,Z_0,W_0$ of $Y,Z,W$, and take $X_0 = \ker( Y_0 \to Z_0 ).$ Therefore, we only need to show that if the complex (where $X,Y,Z,W \in \bMod A$, $X$ on degree 0)
\[ C = \left[ \cdots \to 0 \to X \to Y \to Z \to W \to \cdots \right] \]
satisfies the condition that $\hol^1(C) = 0$ and $\hol^2(C) \bigl[ \frac 1 p \bigr] =0,$ then
\[ \hol^1( M \cotimes_A^\bL C) =0. \]
By Lemma \ref{lem: basic properties of derived complete modules} (b), we see that $\hol^2(C)$ is annihilated by $p^N$ for some $N \in \mathbb Z_+$. Using the exact sequence as in Lemma \ref{lem: tensor exact sequence}
\[ 0 \to \hol^0 ( M \cotimes_A^\bL \hol^n(C)) \to \hol^n( M \cotimes_A^\bL C) \to \hol^{-1} ( M \cotimes_A^\bL \hol^{n+1} (C)) \to 0 \]
for $n=1$, we get our desired result.
\end{proof}

\begin{lemma}\label{lem: right exactness implies cokernel}
Assume that $R \in \bAlg$ and
\[ X \to Y \to Z \to 0 \]
be a complex in $\aMod R$. Then, if the complex is exact, then $Z$ is the cokernel of $X \to Y$ in $\aMod R$. Moreover, $\aMod R$ has cokernels, and the cokernel of $f: X \to Y$ is $Y / \overline{\im f}$.

If $Z$ is the cokernel of $X \to Y$ in $\aMod R$, and $M \in \aMod R$ is any module, then $M \cotimes_{R[1/p]} Z$ is the cokernel of
\[ M \cotimes_{R[1/p]} X \to M \cotimes_{R[1/p]} Y \]
in $\aMod R$. Moreover, if
\[ X \to Y \to Z \to 0 \]
is exact, and $M = M' \bigl[ \frac 1 p \bigr]$ for some $p$-complete and $p$-completely flat module $M' / R$, then the sequence
\[ M' \cotimes_R X \to M' \cotimes_R Y \to M' \cotimes_R Z \to 0 \]
is exact.
\end{lemma}

\begin{proof}
We first prove the first part. Denote $f: X \to Y$ and $g: Y \to Z$, then $\im f = \ker g$ is closed in $Y$. Form the topological module $Z' = Y / \im f$. By Banach's open mapping theorem (Theorem \ref{theo: banach open mapping}), we see that $g$ induces an isomorphism $Z' \to Z$. It is now easy to prove that $Z$ is the cokernel of $X \to Y$. We omit the proof that $Y/\overline{\im f}$ is the cokernel of $f: X \to Y$.

For the second part, recall that by Proposition-Definition \ref{prop-defi: base change of adic modules}, we have
\[ \Hom_{\aMod R} (M \cotimes_{R[1/p]} N, X) \cong \Hom_{\aMod R} (N, \underline{\Hom}_{R[\frac 1 p]}^\ad (M,X)), \]
and this gives us the desired exact sequence.

The last part is a direct consequence of Lemma \ref{lem: exactness of flat base change}.
\end{proof}

\begin{prop-defi}\label{prop-defi: k-finiteness}
Let $A$ be a commutative ring and $M$ be an $A$-module. We say that $M$ is $k$-finite ($k \ge 0$), if and only if there exists an exact sequence
\[ F_{k-1} \to \cdots \to F_1 \to F_0 \to M, \]
where $F_i$ are finite projective $A$-modules. Therefore, all modules are $0$-finite, $M$ is $1$-finite if and only if it is finitely generated, and $M$ is $2$-finite if and only if it is finitely presented.

Fix an integer $k \ge 1$. In an exact sequence
\[ 0 \to M' \to M \to M'' \to 0, \]
\begin{enumerate}[label=(\alph*)]
\item if $M'$ and $M''$ are $k$-finite, then $M$ is $k$-finite;
\item if $M$ and $M''$ are $k$-finite, then $M'$ is $(k-1)$-finite;
\item if $M$ is $k$-finite and $M'$ is $(k-1)$-finite, then $M''$ is $k$-finite;
\item if the exact sequence is split, then $M$ is $k$-finite if and only if $M', M''$ are all $k$-finite.
\end{enumerate}
\end{prop-defi}

\begin{proof}
For (a), pick free resolutions $F'_\bullet \to M'$ and $F''_\bullet \to M''$, such that $F'_i$ and $F''_i$ are finite free for $i < k$. Then, the map $M'' \to M'[1]$ gives a map $F''_\bullet \to F'_\bullet[1]$ of complexes, and we get a quasi-isomorphism
\[ M \cong \cone ( F''_\bullet [-1] \to F'_\bullet) = C_\bullet, \]
where $C_n \cong F'_n \oplus F''_n$. It follows that we have an exact sequence
\[ C_{k-1} \to \cdots \to C_0 \to M \to 0 \]
with all $C_i$ being finite free. It follows that $M$ is $k$-finite.

For (c), pick free resolutions $F_\bullet \to M$ and $F'_\bullet \to M'$, such that $F_i$ are finite free for $i<k$, and $F'_i$ are finite free for $i<k-1.$ The map $M' \to M$ gives a map $F'_\bullet \to F_\bullet$ of complexes, and we get a quasi-isomorphism
\[ M'' \cong \cone(F'_\bullet \to F_\bullet) = C_\bullet. \]
We have $C_n = F'_{n-1} \oplus F_n,$ and we see that $C_n$ are finite free for $n<k$. Now, the exact sequence
\[ C_{k-1} \to \cdots \to C_0 \to M'' \to 0 \]
shows that $M''$ is $k$-finite.

For (b), pick free resolutions $F_\bullet \to M$ and $F''_\bullet \to M''$, such that $F_i$ and $F''_i$ are finite free for $i< k$. Then, the map $M \to M''$ gives a map $F_\bullet \to F''_\bullet$ of complexes, and we get a quasi-isomorphism
\[ M' \cong \cone (F_\bullet \to F''_\bullet) [-1] = C_\bullet. \]
We have
\[ C_n = F''_{n+1} \oplus F_n. \]
We get an exact sequence
\[ C_{k-2} \to \cdots \to C_0 \to C_{-1} \oplus M, \]
hence $C_{-1} \oplus M$ is $(k-1)$-finite. By (c) applied to
\[ 0 \to C_{-1} \to C_{-1} \oplus M \to M \to 0, \]
we see that $M$ is $(k-1)$-finite.

As for (d), by (a), we only need to prove the inverse implication. We proceed by induction on $k$. The case $k=1$ is evident, since direct summands of finitely generated modules are finitely generated. Assume that we have proved this for $k-1$, we now prove this for $k$ ($k \ge 2$). Now $M'$ and $M''$ are both $(k-1)$-finite. By (c) applied to the sequences $0 \to M' \to M \to M'' \to 0$ and $0 \to M'' \to M \to M' \to 0$, we see that both $M'$ and $M''$ are $k$-finite.
\end{proof}

\begin{defi}\label{defi: finiteness of adic modules}
Assume that $R \in \bAlg$ and $M \in \aMod R$. We say that $M$ is finitely generated (resp. finitely presented, $k$-finite), if it is finitely generated (resp. finitely presented, $k$-finite) as an abstract $R\bigl[ \frac 1 p \bigr]$-module.
\end{defi}

We know that $\bMod R \to \Mod_R$ is fully faithful. Although $\aMod R \to \Mod_R$ is in general not full, but we have the following proposition:

\begin{prop}\label{prop: finitely generated adic modules}
For any $R \in \bAlg$ and $M,N \in \aMod R$, if $M$ is finitely generated, then the map
\[ \Hom_{\aMod R} (M,N) \to \Hom_{\Mod_R} (M,N) \]
is an isomorphism.

Recall the inner Hom functor in Proposition-Definition \ref{prop-defi: base change of adic modules}. If $M \to N \to Q \to 0$ is an exact sequence of in $\aMod R$, then we have an exact sequence for any $X \in \aMod R$
\[ 0 \to \underline{\Hom}^\ad_{R[\frac 1 p]}(Q,X) \to \underline{\Hom}^\ad_{R[\frac 1 p]}(N,X) \to \underline{\Hom}^\ad_{R[\frac 1 p]}(M,X). \]

If $M$ is finite projective, $R \to S$ is $p$-completely flat in $\bAlg$, and $N$ be as above, then the natural map
\[ S \cotimes_R \underline{\Hom}_{R[\frac 1 p]}^\ad(M,N) \to \underline{\Hom}_{S[\frac 1 p]}^\ad(S \cotimes_R M,S \cotimes_R N) \]
is an isomorphism.
\end{prop}

\begin{proof}
For the first part, we only need to prove that any module map $f: M \to N$ is continuous. Choose a finite set of generators
\[ m_1,\ldots,m_n \]
of $M$, and they generate a finitely generated $A$-submodule $M_0 \subseteq M$. It is easy to see that $M_0$ is open in $M$ by Banach's open mapping theorem (Theorem \ref{theo: banach open mapping}), hence is a module of definition. Now, pick a module of definition $N_0 \subseteq N$. For each $1 \le i \le n$, pick $k_i \in \mathbb Z$ with $f(m_i) \in p^{k_i} N_0$, and we see that
\[ f(M_0) \subseteq p^{\min \left\{ k_i \right\}} N_0, \]
hence $f$ is continuous, as desired.

For the second part, by Lemma \ref{lem: right exactness implies cokernel}, for any module $Y \in \aMod R$ we have an exact sequence
\[ 0 \to \Hom^\ad_{R[\frac 1 p]}(Q \cotimes_{R[\frac 1 p]} Y,X) \to \Hom^\ad_{R[\frac 1 p]}(N \cotimes_{R[\frac 1 p]} Y,X) \to \Hom^\ad_{R[\frac 1 p]}(M \cotimes_{R[\frac 1 p]} Y,X), \]
giving the exact sequence as desired.

For the last part, write $M$ as a direct summand of a finite free $R\bigl[ \frac 1 p \bigr]$-module $F$, we only need to prove the isomorphism for $F$, which is evident.
\end{proof}

From now on, we will consider descent theory.

Let $A \to B$ be a $p$-completely faithfully flat ring map in $\bAlg.$ We also have the category of descent data over $B/A$. Namely, we will denote the category of descent data over $B/A$ for $\bMod{}$ and $\aMod{}$ by $\bMod{B/A}$ and $\aMod{B/A}.$ Elements inside them are pairs $(M,\phi),$ where $M \in \bMod B$ (resp. $M \in \aMod B$) and
\[ \phi: (B \cotimes_A B) \cotimes_{\delta_0^1, B} M \to (B \cotimes_A B) \cotimes_{\delta_1^1, B} M \]
is an isomorphism satisfying the cocycle identity and reduces to $M \to M$ under $B \cotimes_{B\cotimes_A B} -$. We will also denote
\[ B(n) = B^{\cotimes(n+1)/A} \]
the cosimplicial ring associated to $A \to B$. It is usually more convenient to identify the elements in $\bMod{B/A}$ (resp. $\aMod{B/A}$) with `crystals' of $\bMod{}$ (resp. $\aMod{}$) on $B(\bullet),$ i.e. a cosimplicial module $M(\bullet)$ where $M(n) \in \bMod{B(n)}$ (resp. $M(n) \in \aMod{B(n)}$) compatible with the cosimplicial structure, and for each $[m] \to [n]$ the maps
\[ B(n) \cotimes_{B(m)} M(m) \to M(n) \]
are isomorphisms.

\begin{lemma}[Theorem \ref{thm: descent theory} (1)]\label{lem: acyclic cosimplicial}
Let $A \to B$ be a $p$-completely faithfully flat ring map in $\bAlg$ and $M$ be a module in $\bMod A$ or $\aMod A.$ Then, the map
\[ M \to B(\bullet) \cotimes_A M \]
is a quasi-isomorphism.
\end{lemma}

\begin{proof}
For $M \in \aMod A,$ we may choose a module of definition $M_0 \subseteq M$, reducing to the case of $\bMod A.$

By Derived Nakayama's Lemma (see for example \cite[\href{https://stacks.math.columbia.edu/tag/0G1U}{Tag 0G1U}]{stacks-project}), denote the perfect complex $A' = A \otimes_\mathbb Z^\bL \mathbb F_p$, we only need to prove that
\[ A' \otimes_A^\bL M \to \lim_\simp A' \otimes_A^\bL (B(\bullet) \otimes_A^\bL M ) \]
is an isomorphism in $D(A')$. For this, denote $B'(\bullet) = A' \otimes_A^\bL B(\bullet)$ and $M' = A' \otimes_A^\bL M$, and we see that the above map is isomorphic to
\[ M' \to \lim_\simp B'(\bullet) \otimes_{A'}^\bL M'. \]
This is the faithfully flat descent for the faithfully flat connective $\mathbb \dE_\infty$-ring map $A' \to B'$. A simple proof uses the Postnikov tower\footnote{$M'\cong \lim \tau^{\ge -n} M'$ and $B'(m) \otimes_{A'}^\bL \tau^{\ge -n} M' \cong \tau^{\ge -n} (B'(m) \otimes_{A'}^\bL M')$} to reduce to the case for $M' \in D(A')$ discrete, then we are reduced to the classical faithfully flat descent.
\end{proof}

\begin{prop}\label{prop: descent is fully faithful}
Let $A \to B$ be a $p$-completely faithfully flat ring map in $\bAlg$. Then, the functors $\bMod A \to \bMod{B/A}$ and $\aMod A \to \aMod{B/A}$ are fully faithful.
\end{prop}

\begin{proof}
Lemma \ref{lem: acyclic cosimplicial} identifies a module $M \in \bMod A$ (resp. $M \in \aMod A$) with the kernel of
\[ B \cotimes_A M \xrightarrow{\delta_0^1 - \delta_1^1} B(1) \cotimes_A M. \]
This shows that the map is faithful. It is full because whenever we are given a map
\[ \varphi: B(\bullet) \cotimes_A M \to B(\bullet) \cotimes_A N, \]
we are given a map $\varphi_0: M \to N$ which fits into the commutative diagram
\[\xymatrix{
M \ar[r] \ar[d] & N \ar[d] \\
B(\bullet) \cotimes_A M \ar[r]_\varphi & B(\bullet) \cotimes_A N
}\]
It follows that the map $M \xrightarrow{\varphi_0} N \to B(n) \cotimes_A N$ agrees with the map $M \to B(n) \cotimes_A M \xrightarrow{\varphi(n)} B(n) \cotimes_A N.$ Since $M$ generates $B(n) \cotimes_A M$ as a $B(n)$-module, we see that the maps
\[ \id \otimes \varphi_0, \varphi(n): B(n) \cotimes_A M \to B(n) \cotimes_A N \]
agree. Therefore, $\varphi = \id \otimes \varphi_0,$ showing that the functor is full.
\end{proof}

\begin{cor}\label{cor: ff is conservative}
Let $A \to B$ be a $p$-completely faithfully flat ring map in $\bAlg$. Then, for any map $M \to N$ in $\bMod A$ or $\aMod A$, it is an isomorphism if and only if $B \cotimes_A M \to B \cotimes_A N$ is an isomorphism.
\end{cor}

\begin{proof}
This amounts to saying that $\bMod A \to \bMod B$ and $\aMod A \to \aMod B$ are conservative\footnote{functors which detect isomorphisms}. It is the composition of a fully faithful functor $\bMod A \to \bMod{B/A}$ (resp. $\aMod A \to \aMod{B/A}$) with a conservative functor (the forgetful functor) $\bMod{B/A} \to \bMod B$ (resp. $\aMod{B/A} \to \aMod B$), as desired.
\end{proof}

\begin{lemma}\label{lem: faithfully flat descent of exactness}
Assume that $A \to B$ is a $p$-completely faithfully flat map in $\bAlg$, and that $X \to Y \to Z$ is a complex in $\bMod A$ or $\aMod A.$ Assume that $B \cotimes_A X \to B \cotimes_A Y \to B \cotimes_A Z$ is exact at $B \cotimes_A Y$, then $X \to Y \to Z$ is exact at $Y$.
\end{lemma}

\begin{proof}
We first prove this in $\bMod A.$ Consider the complex in $\cD(A)$
\[ C = \left[ \cdots \to 0 \to X \to Y \to Z \to 0 \to \cdots \right], \]
where $Y$ is on the $0$th degree. By our assumption, $B \cotimes_A^\bL C$ is exact on the $0$th degree, i.e.
\[ \hol^0 (B \cotimes_A^\bL C) =0. \]
Using the short exact sequences in Lemma \ref{lem: tensor exact sequence}, we see that $\clim 0 B \otimes_A \hol^0(C)/p^n \hol^0(C) =0.$ Since the transition maps are surjective and $A \to B$ is $p$-completely faithfully flat, we get $\hol^0(C) = p\hol^0(C)$. By Lemma \ref{lem: basic properties of derived complete modules} (a), we see that $\hol^0(C)$ is derived $p$-adically complete, and (c) of the same lemma now shows that $\hol^0(C)=0$, as desired.

We now prove this in $\aMod A.$ Choose appropriate modules of definition $X_0 \subseteq X, Y_0 \subseteq Y, Z_0 \subseteq Z,$ and we may again consider the complex in $\cD(A)$
\[ C_0 = \left[ \cdots \to 0 \to X_0 \to Y_0 \to Z_0 \to 0 \to \cdots \right], \]
By Proposition-Definition \ref{prop-defi: base change of adic modules}, we see that
\[ \hol^0( B \cotimes_A^\bL C_0) \bigl[ \frac 1 p \bigr] =0. \]
By Lemma \ref{lem: basic properties of derived complete modules} (a) and (b), we see that the module $\hol^0( B \cotimes_A^\bL C_0)$ is annihilated by a power of $p$, say, $p^N$. Using the exact sequences Lemma \ref{lem: tensor exact sequence} again, we see that $\hol^0( B \cotimes_A^\bL \hol^0(C_0))$ is annihilated by $p^N$, hence
\[ \clim 0 B \otimes_A \hol^0(C_0) / p^n \hol^0(C_0) \]
is also annihilated by $p^N$. It follows that, since the transition maps are surjective, all $B \otimes_A \hol^0(C_0) / p^n \hol^0(C_0)$ are annihilated by $p^N$. As $B/A$ is $p$-completely faithfully flat, we see that all modules $\hol^0(C_0) / p^n \hol^0(C_0)$ are all annihilated by $p^N$. In particular, $p^N \hol^0(C_0) = p^{N+1} \hol^0(C_0)$. By Lemma \ref{lem: basic properties of derived complete modules} (a), $\hol^0(C_0)$ is derived $p$-adically complete. Apply Lemma \ref{lem: basic properties of derived complete modules} (a) to the complex
\[ \hol^0(C_0) \to \hol^0(C_0) / p^N\hol^0(C_0), \]
we see that $p^N \hol^0(C_0)$ is derived $p$-adically complete. By Lemma \ref{lem: basic properties of derived complete modules} (c) applied to $p^N \hol^0(C_0)$, we get $p^N \hol^0(C_0)=0$. Therefore, $\hol^0\bigl(C_0\bigl[ \frac 1 p \bigr] \bigr) =0$, as desired.
\end{proof}

\begin{cor}\label{cor: finiteness descends along ff}
Let $A \to B$ be a $p$-completely faithfully flat ring map in $\bAlg$, and $M \in \aMod A$ be an adic module. Then,
\begin{enumerate}[label=(\alph*)]
\item if $B \cotimes_A M$ is finitely generated, then so is $M$,
\item if $B \cotimes_A M$ is $k$-finite, then so is $M$.
\end{enumerate}

Lemma \ref{lem: exactness of flat base change} shows the converse of the implications above.
\end{cor}

\begin{proof}
We first prove (a). Take a finite set of generators
\[ n_1,\ldots,n_\ell \]
of $B \cotimes_A M$. Choose a module of definition $M_0 \subseteq M$ such that $n_i \in B \cotimes_A M_0$ for all $i$. Let $M'_0 \subseteq B \cotimes_A M_0$ be the $B$-submodule generated by all $n_i$, then it is a module of definition for $B \cotimes_A M$. Write
\[ n_i = \sum_{j \in J_i} b_{ij} \otimes m_{ij} + n'_i, \]
where $J_i$ is finite, $b_{ij} \in B$ and $n'_i \in pM'_0$. Iterating, and we see that all $n_i$ can be written as the linear combination of $m_{ij}$ with coefficients in $B$. Let $M''_0$ be the submodule of $B \cotimes_A M_0$ generated by all $m_{ij}$ with coefficients in $B$, and we see that
\[ M'_0 \subseteq M''_0 \subseteq B \cotimes_A M_0. \]
It follows that $\left\{ m_{ij} \right\}$ is also a generating set for $B \cotimes_A M$ over $B \bigl[ \frac 1 p \bigr]$. Let $F$ be a free module over $A \bigl[ \frac 1 p \bigr]$ with basis $e_{ij}$ ($1 \le i \le \ell, j \in J_i$). Consider the map
\[ \varphi: F \to M \]
given by $e_{ij} \mapsto m_{ij}.$ It follows that the base change
\[ \id \otimes \varphi: B \cotimes_A F \to B \cotimes_A M \]
is surjective. By Lemma \ref{lem: faithfully flat descent of exactness} applied to $F \to M \to 0$, we see that $\varphi$ is surjective.

We now prove (b). We proceed by induction on $k$. The cases $k=0$ are evident. Assume that we proved (b) for $k-1$, we now prove it for $k$ ($k \ge 1$). By (a), we may take a finite free module $F$ over $A \bigl[ \frac 1 p \bigr]$ and a surjection $F \to M$. Take $K$ to be the kernel of the surjection in $\aMod A$. Its existence is guaranteed by Lemma \ref{lem: left exactness}. Now, by Lemma \ref{lem: exactness of flat base change}, we see that the sequence
\[ 0 \to B \cotimes_A K \to B \cotimes_A F \to B \cotimes_A M \to 0 \]
is exact. Therefore, by Proposition-Definition \ref{prop-defi: k-finiteness}, we see that $B \cotimes_A K$ is $(k-1)$-finite over $B \bigl[ \frac 1 p \bigr]$. It follows that $K$ is $(k-1)$-finite by inductive hypothesis, and therefore $M$ is $k$-finite by Proposition-Definition \ref{prop-defi: k-finiteness}.
\end{proof}

\begin{prop}\label{prop: descent is essentially surjective}
Let $A \to B$ be a $p$-completely faithfully flat ring map in $\bAlg$. Then, the functors $\bMod A \to \bMod{B/A}$ and $\aMod A \to \aMod{B/A}$ are essentially surjective. Combining Proposition \ref{prop: descent is fully faithful}, we see that the functors are equivalences of categories.
\end{prop}

\begin{proof}
Assume that we are given a descent data $(M,\phi)$, where $M$ is in $\bMod B$ or $\aMod B$ and $\phi: B(1) \cotimes_{\delta_0^1,B} M \to B(1) \cotimes_{\delta_1^1,B} M$ is a descent data on $M$. Define $M_0$ to be the equaliser of the maps
\[ 1 \cotimes \id, \phi \circ (1 \cotimes \id): M \to B(1) \cotimes_{\delta_1^1, B} M. \]
Therefore we get a map $M_0 \to M(\bullet).$

For each map $f: [n] \to [m]$, we denote $f_*: [n+1] \to [m+1]$ by
\[ f_*(0) =0, \quad f_*(i+1) = f(i)+1, \forall i \in [n].\]
Therefore, $M(\bullet + 1)$ can be realised as a simplicial object with transition maps given by $f_*$. Under the identification
\[ B \cotimes_A M(n) \cong B(n+1) \cotimes_{\delta_0^n, B(n)} M(n) \cong M(n+1) \cong B(n+1) \cotimes_{q_0^n, B(0)} M \cong B(n) \cotimes_A M, \]
we see that $B \cotimes_A M(\bullet)$ is identified with $B(\bullet) \cotimes_A M$. Lemma \ref{lem: acyclic cosimplicial} now shows that the cosimplicial object $B \cotimes_A M(\bullet)$ has cohomologies concentrated on degree 0. By Lemma \ref{lem: faithfully flat descent of exactness}, we see that the cohomologies of $M(\bullet)$ are concentrated on degree 0.

By Lemma \ref{lem: exactness of flat base change}, we see that the exact sequence
\[ 0 \to M_0 \to M(0) \to M(1) \to \cdots \]
gives rise to an exact sequence
\[ 0 \to B \cotimes_A M_0 \to B \cotimes_A M(0) \to B \cotimes_A M(1) \to \cdots . \]
The natural isomorphism $B \cotimes_A M(\bullet) \cong B(\bullet) \cotimes_A M$ and Lemma \ref{lem: acyclic cosimplicial} gives rise to a natural isomorphism $B \cotimes_A M_0 \cong M$. Now consider the natural map $B(\bullet) \cotimes_A M_0 \to M(\bullet)$ given by $M_0 \to M(\bullet)$. Base change by $A \to B$, and the base change is identified with the isomorphism
\[ B(\bullet) \cotimes_A M \to B \cotimes_A M(\bullet) \]
as above. Therefore, $B \cotimes_A B(\bullet) \cotimes_A M_0 \to B \cotimes_A M(\bullet)$ is an isomorphism \emph{(not merely a quasi-isomorphism!!!)} of cosimplicial modules. It follows that, by Corollary \ref{cor: ff is conservative}, that the map
\[ B(\bullet) \cotimes_A M_0 \to M(\bullet) \]
is an isomorphism of cosimplicial modules. Therefore, $(M,\phi)$ is isomorphic to the canonical descent data on $B \cotimes_A M_0$, as desired.
\end{proof}

We now prove that the maps $\bMod A \to \bMod{B/A}$ and $\aMod A \to \aMod{B/A}$ induces an equivalence of categories on their full subcategories of finite projective modules. More precisely,
\begin{prop}\label{prop: finite projective descent}
Let $A \to B$ be a $p$-completely faithfully flat ring map in $\bAlg$. Then,
\begin{enumerate}[label=(\alph*)]
\item if $M \in \bMod A$ and $B \cotimes_A M$ is finite projective over $B$, then $M$ is finite projective over $A$;
\item if $M \in \aMod A$ and $B \cotimes_A M$ is finite projective over $B \bigl[ \frac 1 p \bigr]$, then $M$ is finite projective over $A\bigl[ \frac 1 p \bigr]$.
\end{enumerate}
\end{prop}

\begin{proof}
We first prove (a). By classical descent theory, all modules $M/p^n M$ are finite projective over $A/p^n A$. Pick a finite free module $F$ over $A$, and a surjection $\overline \pi: F \to M/pM$. Since $F$ is free, we may lift $\overline \pi$ to a map $\pi: F \to M$ of $A$-modules. Since $F,M$ are classically $p$-adically complete, we see that $\pi$ is surjective as a map of $A$-modules. Take a section $e_1: M/pM \to F/pF$ of $\overline \pi$. Assume that we have a section of $\pi_n: F/p^nF \to M/p^n M$:
\[ e_n: M/p^n M \to F/p^n F \]
we may then consider the composition
\[ F/p^{n+1} F \times_{F/p^nF, e_n} M/p^n M \to F/p^{n+1} F \xrightarrow{\pi_{n+1}} M/p^{n+1} M, \]
which is a surjection between $A/p^{n+1} A$-modules by Nakayama's Lemma for nilpotent ideals (it surjects onto $M/p^nM$). Since $M/p^{n+1}M$ is projective over $A/p^{n+1}A$, we may take a section to the above morphism. Compose it with the inclusion
\[ F/p^{n+1} F \times_{F/p^nF, e_n} M/p^n M \to F/p^{n+1} F, \]
and we get a section $e_{n+1}: M/p^{n+1}M \to F/p^{n+1}F$ of $\pi_{n+1}$. It is easy to verify that $e_{n+1}$ lifts $e_n$. Continue this lifting process, and since $M$ and $F$ are classically complete, we get a section $e$ of $\pi$. Therefore, $M$ is a direct summand of $F$, hence $M$ is finite projective over $A$.

We now prove (b). By Corollary \ref{cor: finiteness descends along ff}, we see that $M$ is 3-finite over $A \bigl[ \frac 1 p \bigr]$. We claim that for any $N \in \aMod A,$ the natural map
\[ B \cotimes_A \underline{\Hom}_{A[\frac 1 p]}^\ad(M,N) \to \underline{\Hom}_{B[\frac 1 p]}^\ad(B \cotimes_A M, B \cotimes_A N) \]
is an isomorphism (for this particular ring map $A \to B$, not any map of algebras). Pick a presentation
\[ F_2 \to F_1 \to F_0 \to M \to 0 \]
where $F_i$ are finite free over $A \bigl[ \frac 1 p \bigr].$ By Proposition \ref{prop: finitely generated adic modules}, we see that the maps
\[ B \cotimes_A \underline{\Hom}_{A[\frac 1 p]}^\ad(F_i,N) \to \underline{\Hom}_{B[\frac 1 p]}^\ad(B \cotimes_A F_i, B \cotimes_A N) \]
are isomorphisms. Now consider the exact sequence given by Lemma \ref{lem: exactness of flat base change}
\[ B \cotimes_A F_2 \to B \cotimes_A F_1 \to B \cotimes_A F_0 \to B \cotimes_A M \to 0. \]
Since $B \cotimes_A M$ is finite projective over $B \bigl[ \frac 1 p \bigr]$, we see that the exact sequence splits, and we have an exact sequence
\[ \underline{\Hom}_{B[\frac 1 p]}^\ad(B \cotimes_A F_0, B \cotimes_A N) \to \underline{\Hom}_{B[\frac 1 p]}^\ad(B \cotimes_A F_1, B \cotimes_A N) \to \underline{\Hom}_{B[\frac 1 p]}^\ad(B \cotimes_A F_2, B \cotimes_A N). \]
Therefore, using the isomorphism in Proposition \ref{prop: finitely generated adic modules}, the sequence
\[ B \cotimes_A \underline{\Hom}_{A[\frac 1 p]}^\ad(F_0,N) \to B \cotimes_A \underline{\Hom}_{A[\frac 1 p]}^\ad(F_1,N) \to B \cotimes_A \underline{\Hom}_{A[\frac 1 p]}^\ad(F_2,N) \]
is exact. By Lemma \ref{lem: faithfully flat descent of exactness}, we have an exact sequence
\[ \underline{\Hom}_{A[\frac 1 p]}^\ad(F_0,N) \to \underline{\Hom}_{A[\frac 1 p]}^\ad(F_1,N) \to \underline{\Hom}_{A[\frac 1 p]}^\ad(F_2,N). \]
Combine this with the exact sequence obtained from $F_1 \to F_0 \to M$ (see Proposition \ref{prop: finitely generated adic modules}), we get an exact sequence
\[ 0 \to \underline{\Hom}_{A[\frac 1 p]}^\ad(M,N) \to \underline{\Hom}_{A[\frac 1 p]}^\ad(F_0,N) \to \underline{\Hom}_{A[\frac 1 p]}^\ad(F_1,N) \to \underline{\Hom}_{A[\frac 1 p]}^\ad(F_2,N). \]
Lemma \ref{lem: exactness of flat base change} now shows that the sequence
\[ 0 \to B \cotimes_A \underline{\Hom}_{A[\frac 1 p]}^\ad(M,N) \to B \cotimes_A \underline{\Hom}_{A[\frac 1 p]}^\ad(F_0,N) \to B \cotimes_A \underline{\Hom}_{A[\frac 1 p]}^\ad(F_1,N) \]
is exact. Therefore, we have a commutative diagram with exact rows:
\[\xymatrix{
0 \ar[r] & B \cotimes_A \underline{\Hom}_{A[\frac 1 p]}^\ad(M,N) \ar[r] \ar[d] & B \cotimes_A \underline{\Hom}_{A[\frac 1 p]}^\ad(F_0,N) \ar[r] \ar[d] & B \cotimes_A \underline{\Hom}_{A[\frac 1 p]}^\ad(F_1,N) \ar[d] \\
0 \ar[r] & \underline{\Hom}_{B[\frac 1 p]}^\ad(B \cotimes_A M,B \cotimes_A N) \ar[r] & \underline{\Hom}_{B[\frac 1 p]}^\ad(B \cotimes_A F_0,B \cotimes_A N) \ar[r] & \underline{\Hom}_{B[\frac 1 p]}^\ad(B \cotimes_A F_1,B \cotimes_A N)
}\]
The two vertical morphisms on the right-hand side are isomorphisms by Proposition \ref{prop: finitely generated adic modules}, hence the arrow on the left-hand side
\[ B \cotimes_A \underline{\Hom}_{A[\frac 1 p]}^\ad(M,N) \to \underline{\Hom}_{B[\frac 1 p]}^\ad(B \cotimes_A M, B \cotimes_A N) \]
is an isomorphism\footnote{Another possible approach is to show that $\underline{\Hom}_{B(\bullet)[\frac 1 p]}^\ad(B(\bullet) \cotimes_A M, B(\bullet) \cotimes_A N)$ is a `crystal' on $B(\bullet)$, and use the descent theory as in Proposition \ref{prop: descent is essentially surjective}.}.

Pick a finite free module $F$ over $A \bigl[ \frac 1 p \bigr]$ and a surjection
\[ \varphi: F \to M. \]
Since $B \cotimes_A M$ is finite projective over $B \bigl[ \frac 1 p \bigr]$, there exists a section $B \cotimes_A M \to B \cotimes_A F$ of $\id \otimes \varphi$, and we see that the map induced by $\id \otimes \varphi$
\[ \Hom_{\aMod B}( B \cotimes_A M, B \cotimes_A F) \to \Hom_{\aMod B} (B \cotimes_A M, B \cotimes_A M) \]
is surjective. Consider the map
\[ \underline{\Hom}_{B[\frac 1 p]}^\ad(B \cotimes_A M, B \cotimes_A F) \to \underline{\Hom}_{B[\frac 1 p]}^\ad(B \cotimes_A M, B \cotimes_A B), \]
where the connecting map is induced by $\id \otimes \varphi$. By the universal property, the underlying module of the above map is exactly the map $\Hom_{\aMod B}( B \cotimes_A M, B \cotimes_A F) \to \Hom_{\aMod B} (B \cotimes_A M, B \cotimes_A M)$, and it is a surjection of adic modules. Now, by the previous paragraph, we see that the above map is identified with
\[ B \cotimes_A \underline{\Hom}_{A[\frac 1 p]}^\ad(M,F) \to B \cotimes_A \underline{\Hom}_{A[\frac 1 p]}^\ad(M,M), \]
and Lemma \ref{lem: faithfully flat descent of exactness} shows that the map
\[ \underline{\Hom}_{A[\frac 1 p]}^\ad(M,F) \to \underline{\Hom}_{A[\frac 1 p]}^\ad(M,M) \]
induced by $\varphi: F \to M$ is surjective. By the universal property, taking the underlying modules of the above map gives the surjection
\[ \Hom_{\aMod A}(M,F) \to \Hom_{\aMod A}(M,M). \]
Take an inverse image $s$ of $\id_M$, and it is a section to $\varphi: F \to M$, as desired.
\end{proof}

Combine Proposition \ref{prop: descent is essentially surjective} and Proposition \ref{prop: finite projective descent}, we get the following theorem:

\begin{theo}[Theorem \ref{thm: descent theory} (2)]\label{theo: completed descent}
Let $A,B$ be $p$-complete rings of bounded $p^\infty$-torsion, and $A \to B$ be $p$-completely faithfully flat. Then, finite projective modules descends along $A \to B(\bullet)$ and $A \bigl[ \frac 1 p \bigr] \to B(\bullet) \bigl[ \frac 1 p \bigr]$, where the cosimplicial ring $B(\bullet)$ is given by
\[ B(n) = B^{\cotimes(n+1)/A}. \]
\end{theo}

\bibliographystyle{alpha}
\bibliography{ref}

\end{document}